\documentclass{amsart}

\usepackage[colorlinks=true, citecolor=blue]{hyperref}
\usepackage{bookmark}
    \usepackage{amssymb}
\usepackage{amsmath}
\usepackage{amscd,latexsym}
\usepackage{bookmark}
     \usepackage{amssymb,amsfonts}
\usepackage[all,arc]{xy}
\usepackage{enumerate}
\usepackage{mathrsfs}
\usepackage{amscd}
\usepackage{tikz}
\usepackage{tikz-cd}
\usepackage{amsmath}
\usepackage{latexsym}
\usepackage{mathdots}
\usepackage{amsrefs}
\usepackage{graphicx}
\usepackage{mathtools}
\usepackage{leftidx}
\usepackage{tensor}
\usepackage{bbm}
\usepackage{dsfont}

\newtheorem{theorem}{Theorem}[section]
\newtheorem{lemma}[theorem]{Lemma}
 \newtheorem*{theorem*}{Theorem}
  \newtheorem{corollary}[theorem]{Corollary}
      \newtheorem{proposition}[theorem]{Proposition}

      \newtheorem*{assumption}{Assumption}

\theoremstyle{definition}
\newtheorem*{definition*}{Definition}
\newtheorem{definition}[theorem]{Definition}

\newtheorem{remark}[theorem]{Remark}

\numberwithin{equation}{section}

\begin{document}

\title[The local exterior square $L$-functions]{Derivatives and exceptional poles of the local exterior square $L$-function for $GL_m$}

\author{Yeongseong Jo}
\address{Department of Mathematics, The Ohio State University, Columbus, Ohio 43210}
\email{jo.59@osu.edu}
\thanks{This research was partially supported by NSF grant
     DMS-0968505 through Professor Cogdell.}






\begin{abstract}
Let $\pi$ be an irreducible admissible representation of $GL_m(F)$, where $F$ is a non-archimedean local field of characteristic zero. We follow the method developed by Cogdell and Piatetski-Shapiro to complete the computation of the local exterior square $L$-function $L(s,\pi,\wedge^2)$ in terms of $L$-functions of supercuspidal representations via an integral representation established by Jacquet and Shalika in $1990$. We analyze the local exterior square $L$-functions via exceptional poles and Bernstein and Zelevinsky derivatives. With this result, we show the equality of the local analytic $L$-functions $L(s,\pi,\wedge^2)$ via integral integral representations for the irreducible admissible representation $\pi$ for $GL_m(F)$ and the local arithmetic $L$-functions $L(s, \wedge^2(\phi(\pi)))$ of its Langlands parameter $\phi(\pi)$ via local Langlands correspondence.
\end{abstract}

\maketitle

\section{Introduction}

In the mid 1990's Cogdell and Piatetski-Shapiro embarked on the local analysis of the exterior square $L$-function via Jacquet and Shalika integrals and expected that $L(s,\pi,\wedge^2)$ for irreducible generic representations $\pi$ of $GL_m$ over a nonarchimedean local field of characteristic $0$ can be expressed in terms of $L$-functions for supercuspidal representations. They predicted in \cite{Cog94} that one can also compute exterior square $L$-functions in terms of $L$-functions of the inducing datum over the archimedean places.

\par

In \cite{CoPe}, Cogdell and Piatetski-Shapiro developed a new technique in the mid 1990's to explicitly calculate the local $L$-functions $L(s,\pi \times \sigma)$ defined by the Rankin-Selberg
integrals in terms of $L$-functions for supercuspidal representations. The main two ingredients for this direction are the use of Bernstein-Zelevinsky derivatives \cite{BeZe76,BeZe77,Ze80} and the exceptional poles of $L$-functions. They employ the theory of derivative of Bernstein and Zelevinsky to compute the local $L$-functions $L(s,\pi \times \sigma)$ for the pairs of irreducible generic representations $\pi$ of $GL_n$ and $\sigma$ of $GL_m$ over a nonarchimedean local field of characteristic $0$ in terms of the exceptional $L$-functions of their derivatives. This result let them calculate the $L$-function in the cases where both $\pi$ and $\sigma$ are supercuspidals, as in Gelbert and Jacquet \cite{GeJa78}, and where both $\pi$ and $\sigma$ are square-integrable, as in Jacquet, Piatetski-Shapiro, and Shalika \cite{JaPSSh83}.

\par
We can find the expression for the Rankin-Selberg type $L$-functions $L(s,\pi \times \sigma)$ in the later sections of \cite{JaPSSh83} and so the expression they obtain are not actually new. However they expect that this method works well for other $L$-functions of $GL_n$ construction. In particular, this method will be carried out to investigate the structure of Bump-Friedberg \cite{BuFr89}, symmetric square \cite{BuGi92}, or Asai $L$-functions \cite{Fl88} via integral representations at ramified places. 

\par

 Recently there has been renewed interest in utilizing this method for archimedean and non-archimedean case \cite{Chai15,Ma09,Ma15}. The archimedean derivatives in \cite{ChCo99} of Cogdell and Chang are defined by components in $\mathfrak{n}$-homology, where $\mathfrak{n}$ is the nilradical of some parabolic subalgebra. Then Chai establishes the notion of exceptional poles of different types for the pairs of components of archimedean derivatives  \cite{Chai15}. He studies the poles of Rankin-Selberg product $L$-function $L(s,\pi \times \sigma)$ over the archimedean field in terms of exceptional poles for pairs of derivatives of $\pi$ and $\sigma$. Matringe follows the path of adaptation of this method to complete the calculation of the Bump-Friedberg \cite{BuFr89} or Asai $L$-functions \cite{Fl88} from integral representations for $\pi$, irreducible and generic over non-archimedean field \cite{Ma09,Ma15}. In \cite{Ma15}, he prominently reworks and improves the basic idea of this technique.

\par
There has also been a lot of interests in the study of the local exterior square $L$-functions via integral representations of Jacquet and Shalika \cite{Be14,Ke11,KeRa12,MiYa13}.  In particular, Cogdell and Matringe \cite{Cog,Ma14} prove the local functional equation of the exterior square $L$-functions for any irreducible admissible representations of $GL_m$ based on the Bernsteion-Zelevinsky theory of derivatives \cite{BeZe76,BeZe77,Ze80}, and the theory of linear periods and Shalika periods. The definition of the local exterior square $\varepsilon$- and $\gamma$-factors are now available for all irreducible admissible representations of $GL_m$. 
We expect to prove the multiplicativity and stability of the local $\gamma$-factor. We will return to this in the future.

\par
 In this paper we extend the notion of local exterior square $L$-function $L(s,\pi,\wedge^2)$ to any representations of Whittaker type $\pi$. In the original paper of Jacquet and Shalika \cite{JaSh88}, they mainly focused on irreducible unitary generic representations. We first deduce the rationality of Jacquet-Shalika integrals for representations of Whittaker type. This follows from the asymptotic expansion of Whittaker functions restricted to torus of $GL_m$ in terms of the central characters for successive quotients appearing in a composition series of Bernstein-Zelevinsky derivatives in \cite[Proposition 2.9]{Ma15}. Eventually we define local exterior square $L$-functions for all these representations. We must work with representations of Whittaker type because when we deformed our representation, we necessarily leave the realm of irreducible generic representations. We essentially take the same approach of Cogdell and Piatetski-Shapiro in \cite{CoPe}, and rely on the theory of derivatives of Bernstein and Zelevinsky to compute the local exterior square $L$-functions $L(s,\pi,\wedge^2)$ for $\pi$, irreducible generic representations in terms of the exceptional exterior square $L$-functions of their derivatives. In our case, we have to deal with the study of Shalika functionals. The method makes it possible to compute the local exterior square $L$-functions via integral representations for all irreducible admissible representations of $GL_m$ over the non-archimedean field in terms of $L$-functions for supercuspidal representations. 
 \par
 The main result of this paper are precisely stated as follows. Let $F$ be a non-archimedean local field of characteristic $0$. 
 We denote by $S_{2n}(F)$ the Shalika subgroup of $GL_{2n}(F)$ of matrices of the form $\begin{pmatrix} I_n & X \\ & I_n \end{pmatrix} \begin{pmatrix} g&\\ &g \end{pmatrix}$ where $X$ belongs to the $n$ by $n$ matrices and $g$ lies in $GL_n(F)$. Let $\nu(g)=|\mathrm{det}(g)|$ be the unramified determinantal character of $GL_n(F)$ for any $n$. We fix a nontrivial additive character $\psi : F \rightarrow \mathbb{C}$ and define a character $S_{2n}(F)$ by $\Theta \left( \begin{pmatrix} I_n & X \\ & I_n \end{pmatrix} \begin{pmatrix} g&\\ &g \end{pmatrix} \right)=\psi(\mathrm{Tr} X)$.

\vskip.1in
\noindent
\textbf{Cuspidal case}
\\
Suppose $\rho$ is an irreducible supercuspidal representation of $GL_r(F)$. 
 \begin{enumerate}
\item[$(\mathrm{i})$]  If $r$ is odd, then $L(s,\rho,\wedge^2)=1$.
\item[$(\mathrm{ii})$] If $r$ is even, then 
\[
   L(s,\rho,\wedge^2)=\prod (1-\alpha q^{-s})^{-1}
\]
with the product over all $\alpha=q^{s_0}$ such that $\mathrm{Hom}_{S_{r}(F)}(\rho\nu^{\frac{s_0}{2}},\Theta) \neq 0$.
 \end{enumerate}

\vskip.1in
\noindent
\textbf{Discrete series case}
\par
Let $[\rho,\dotsm,\rho\nu^{\ell-1}]$ be the unique irreducible quotient of the normalized induced representation $\mathrm{Ind}(\rho \otimes \dotsm \otimes  \rho\nu^{\ell-1})$ of $GL_{\ell r}(F)$ from standard parabolic associated to the partition $(r,\dotsm,r)$ with $\rho$ an irreducible supercuspidal representation of $GL_r(F)$. We assume that $\Delta$ is an irreducible square integrable representation of $GL_{\ell r}(F)$ with the segment $\Delta=[\rho\nu^{-\frac{\ell-1}{2}},\dotsm,\rho\nu^{\frac{\ell-1}{2}}]$ and $\rho$ an irreducible unitary supercuspidal representation of $GL_r(F)$.
 \begin{enumerate}
\item[$(\mathrm{i})$] If $\ell$ is even, then
\[
 L(s,\Delta,\wedge^2)=\prod_{i=0}^{\frac{\ell}{2}-1} L(s+2i+1,\rho,\wedge^2) L(s+2i,\rho,\mathrm{Sym}^2).
 \]
\item[$(\mathrm{ii})$] If $\ell$ is odd, then
\[
  L(s,\Delta,\wedge^2)=\prod_{i=0}^{\frac{\ell-1}{2}}L(s+2i,\rho,\wedge^2) \prod_{i=1}^{\frac{\ell-1}{2}}L(s+2i-1,\rho,\mathrm{Sym}^2).
  \]
\end{enumerate}
If $\Delta=\Delta_0 \nu^{s_0}$ is an irreducible quasi-square integrable representation of $GL_{\ell r}(F)$ with $\Delta_0$ an irreducible square integrable representation and $s_0 \in \mathbb{R}$, then 
\[
L(s,\Delta,\wedge^2)=L(s+2s_0,\Delta_0,\wedge^2).
\]

\vskip.1in
\noindent
\textbf{Tempered case}
\par
Suppose that $\pi=\mathrm{Ind}(\Delta_1 \otimes \dotsm \otimes \Delta_t)$ is the normalized induced representation of $GL_m(F)$ from the standard parabolic associated to the partition $(m_1,\dotsm,m_t)$, where $\Delta_i$ are irreducible square integrable representations of $GL_{m_i}(F)$. Then
 \[
   L(s,\pi,\wedge^2)=\prod_{1 \leq k \leq t} L(s,\Delta_k,\wedge^2) \prod_{1 \leq i < j \leq t} L(s, \Delta_i \times \Delta_j).
 \]

We remark that the representation $\pi$ is irreducible by \cite[Theorem 4.2]{Ze80}.

\vskip.1in
\noindent
\textbf{Non-tempered case}
\par
Let $\pi$ be an irreducible admissible representation of $GL_m(F)$. If $\pi=\mathrm{Ind}(\Delta_1 \otimes \dotsm \otimes \Delta_t)$ is generic with $\Delta_i$ irreducible quasi-square-integrable representations, then
\[
  L(s,\pi,\wedge^2)=\prod_{k=1}^t L(s,\Delta_k,\wedge^2) \prod_{1\leq i < j \leq t} L(s,\Delta_i \times \Delta_j).
\]
Suppose that $\pi$ is not generic. Let $\Pi=\mathrm{Ind}(\Delta_1\nu^{u_1} \otimes \dotsm \otimes \Delta_t\nu^{u_t} )$ be the Langlands induced representation with each $\Delta_i$ an irreducible square integrable representation, the $u_i$ real and ordered so that $u_1 \geq \dotsm \geq u_t$ such that $\pi$  is the unique irreducible quotient of $\Pi$. Then
\[
  L(s,\pi,\wedge^2)=\prod_{k=1}^t L(s+2u_k,\Delta_k,\wedge^2) \prod_{1\leq i < j \leq t} L(s+u_i+u_j,\Delta_i \times \Delta_j).
\]

\vskip.1in

 Let $W_F'$ be the Weil-Deligne group of $F$. The local Langlands correspondence for $GL_m(F)$ established by Harris-Taylor \cite{HaTa01} and then Henniart \cite{He02} describes the bijection between the set of isomorphism class of irreducible admissible representations of $GL_m(F)$ and the set of equivalence class of $n$-dimensional Frobenius semisimple complex representations of $W_F'$. For a Frobenius semisimple complex representations of $\phi$ of $W_F'$, the correspondence associates $\phi$ to the irreducible admissible representation $\pi(\phi)$ of $GL_m(F)$. Let $\wedge^2$ be the exterior square representation of $GL_m(\mathbb{C})$. Then we can define Artin's local exterior square $L$-function as $L(s,\wedge^2(\phi))$. There are two more ways to associated the local exterior square $L$-function to an irreducible admissible representation $\pi(\phi)$. The first thing is by analyzing the Euler products that appear the non-constant Fourier coefficients of Eisenstein Series. This method is called the Langlands-Shahidi method in \cite{Sh90} and we denote by $L_{Sh}(s,\pi(\phi),\wedge^2)$. The last $L$-function, denoted by $L(s,\pi(\phi),\wedge^2)$, is defined as the greatest common divisor of $\mathbb{C}[q^{\pm s}]$-fractional ideals. Each rational functions in $\mathbb{C}[q^{\pm s}]$-fractional ideals is obtained as meromorphic continuation of Rankin-Selberg type integral representations, which is constructed by Jacquet and Shalika in \cite{JaSh88}. One might naturally conjecture that these three exterior square $L$-functions are equal $L(s,\wedge^2(\phi))=L_{Sh}(s,\pi(\phi),\wedge^2)=L(s,\pi(\phi),\wedge^2)$. Henniart \cite{He10} proves that the Artin and Langlands-Shahidi $L$-functions are equal $L(s,\wedge^2(\phi))=L_{Sh}(s,\pi(\phi),\wedge^2)$ for every irreducible admissible representation $\pi(\phi)$. Kewat and Raghunathan \cite{KeRa12} recently prove that the functions $L_{Sh}(s,\Delta,\wedge^2)$ and $L(s,\Delta,\wedge^2)$ are equal for $\Delta$ a discrete representation. In this article, the multiplicativity of the analytic exterior square $L$-functions for irreducible generic representations is used to prove the following Theorem of the equality between the exterior square $L$-function through integral representations of Jacquet and Shalika for an irreducible admissible representation (Arithmetic side), and via the local Langlands correspondence and its Langlands parameter (Analytic side).

 \begin{theorem*}
  Let $\phi$ be an $m$-dimensional Frobenius semi simple complex representation of Weil-Deligne group $W_F'$ and $\pi=\pi(\phi)$ the irreducible admissible representation of $GL_m$ associated to $\phi$ under the local Langlands correspondence. Then we have
  \[
    L(s,\wedge^2(\phi))=L(s,\pi(\phi),\wedge^2).
  \]
  \end{theorem*}

Let us briefly describe the contents of this article. In Section $2$, we introduce the Jacquet-Shalika integrals which will define a local exterior square $L$-function $L(s,\pi,\wedge)$ for any representations $\pi$ of Whittaker type on $GL_m(F)$. Section $3$ is devoted to the study of exceptional poles, Shalika groups and Shalika functionals, and inspect a few properties about functional equations for exterior square $L$-functions. We then concern the computational result of the local exterior square $L$-functions $L(s,\pi,\wedge^2)$ in terms of the exceptional exterior square $L$-functions of their even derivatives for the even $GL_{2n}(F)$ or odd derivatives for odd $GL_{2n+1}(F)$ in Section $4$ and $5$. In Section $6$, we employ this factorization formula in order to express the exterior square $L$-functions in the case where $\pi$ is quasi-square integrable in terms of exceptional $L$-functions of supercuspidal representations. In Section $7$, we show that the Jacquet-Shalika integral defining the $L$-function $L(s,\pi_u,\wedge^2)$ for deformed representation $\pi_u$ are rational by using Bernstein's Theorem \cite{Ba98}. By combining the result of Section $3$, the deformation method of Section $7$, and Hartogs' theorem, we are able to prove a weak version of the multiplicativity of $\gamma$-factor in Section $8$. The rest part of Section $8$ is devoted to adapting the methods of \cite{CoPe,Ma09,Ma15} to complete the computation of $L(s,\pi,\wedge)$ for $\pi$ an irreducible admissible representation. 
The multiplicativity formalism in Section $8$ lets us express $L$-functions of square integrable representations in terms of exterior or symmetric square $L$-functions for supercuspidal representations which are in turn replaced by exceptional $L$-functions for supercuspidal representations in Section $6$ and and show the agreement of the local exterior square $L$-function on the analytic side deduced from Jacquet and Shalika integral representations and $L$-function on the arithmetic side via Langlands parameters.

\section{Derivatives and Integral Representation}

 Let $F$ be a nonarchimedean local field of characteristic 0, with ring of integers $\mathcal{O}$, prime ideal $\mathfrak{p}$, and fix a uniformizer $\varpi$ so that $\mathfrak{p}=( \varpi )$. Let $q=|\mathcal{O}/\mathfrak{p}|$ denote the cardinality of the residue class field. We let $v : F^{\times} \rightarrow \mathbb{Z}$ be the associated valuation with $v( \varpi )=1$ and normalize the absolute value so that $|a|=q^{-v(a)}$. We let $\mathcal{M}_n(F)$ be the $n \times n$ matrices, $\mathcal{N}_n(F)$ the subspace of upper triangular matrices of $\mathcal{M}_n(F)$. We denote by $GL_n(F)$ the general linear group of invertible matrices of $\mathcal{M}_n(F)$ and $N_n(F)$ the maximal unipotent subgroup of upper triangular unipotent matrices. Let $K_n=GL_n(\mathcal{O})$ be the standard maximal compact subgroup of $GL_n(F)$. This group admits a filtration by the compact open congruent subgroups
 \[
   K_{n,r}=I_n+\mathcal{M}_n(\mathfrak{p}^r)
 \]
 for $r \geq 1$. We denote by $Z_n(F)$ the center of $GL_n(F)$. We fix a nontrivial additive character $\psi : F \rightarrow \mathbb{C}$ and extend it to a character of $N_n(F)$ by setting $\psi(n)=\psi(n_{1,2}+\dotsm +n_{n-1,n})$ with $n \in N_n(F)$. We shall identify $F^n$ with the space of row vectors of length $n$. We denote by $P_n(F)$ the mirabolic subgroup of $GL_n(F)$ given by 
 \[
  P_n(F)=\left\{ \begin{pmatrix} g_{n-1}&u \\ & 1 \end{pmatrix}\; \middle| \;g_{n-1} \in GL_{n-1}(F), \; ^tu \in F^{n-1} \right\}
\]
and $U_n(F)$ the unipotent radical of $P_n(F)$ given by 
\[
  U_n(F)=\left\{ \begin{pmatrix} I_{n-1} & u \\ & 1\end{pmatrix}\; \middle| \;^tu \in F^{n-1} \right\}  \simeq F^{n-1}.
\]
As a group, $P_n(F)$ has the structure of a semi-direct product $P_n(F) \simeq GL_{n-1}(F) \ltimes U_n(F)$. We restrict the additive character of $N_n(F)$ to $U_n(F)$ by $\psi(u)=\psi(u_{n-1,n})$. Throughout this paper, we abuse notation by letting $GL_n=GL_n(F)$, $N_n=N_n(F)$, etc.

 \par
 
 We denote by $\mathcal{S}(F^n)$ the space of Schwartz-Bruhat functions $\Phi : F^n \rightarrow \mathbb{C}$ which are locally constant and of compact support. Let $e_n=(0,0,\dotsm,0,1) \in F^n$. The Fourier transform on $\mathcal{S}(F^n)$ will be defined by
 \[
    \hat{\Phi}(y)=\int_{F^n} \Phi(x) \psi(x^ty) \; dx
 \]
 for $y \in F^n$. Once $\psi$ is chosen, we assume that the measure on $F^n$ used in this integral is the corresponding self-dual measure. Therefore the Fourier inversion formula takes the form $\hat{\hat{\Phi}}(x)=\Phi(-x)$.

\subsection{Derivatives and Whittaker models}
For the convenience of reader, we summarize Section 2.1 of \cite{CoPe}, which is the well-known facts about the representation theory of $P_n$, following Bernstein and Zelevinsky \cite{BeZe76,BeZe77,Ze80}. Let
Rep($GL_{n}$) be the category of smooth representation of $GL_{n}$, Rep($P_n$) the category of smooth representation of $P_n$, etc.
Consider the four functors: 

\[
\xymatrixcolsep{4pc}
 \xymatrix{
\mathrm{Rep}(P_{n-1})\quad \ar@<1ex>[r]^{ \Phi^+} & \quad
\mathrm{Rep}(P_{n}) \quad \ar@<1ex>[l]^{ \Phi^-} \ar@<-1ex>[r]
_-{\Psi^-} & \quad \mathrm{Rep}(GL_{n-1}) \ar@<-1ex>[l]_-{\Psi^+} }.
\]
$\Phi^{+}$ and $\Psi^{+}$ are induction functors, while
$\Phi^{-}$ and $\Psi^{-}$ are Jacquet functors. More precisely,
given $\tau \in \mathrm{Rep}(P_n)$ on the space $V_{\tau}$, $\Phi^-(\tau)$ is the
normalized representation of $P_{n-1}$ on the space $V_{\tau} \slash
V_{\tau}(U_{n}, \psi)$ with $V_{\tau}(U_{n},
\psi)=\langle \tau(u)v-\psi(u)v\; |\; v \in V_{\tau}, u \in
U_{n} \rangle$. As $V_{\tau} \slash V_{\tau}(U_{n}, \psi)$ is the largest quotient of $V_{\tau}$ on which $U_n$ acts by the character $\psi$ and $P_{n-1}$ is the stabilizer of the additive character $\psi$ in $GL_{n-1}$, the normalized action $\sigma=\Phi^-(\tau)$ of $P_{n-1}$ is given by
\[
 \sigma(p)(v+V_{\tau}(U_{n}, \psi))=|\mathrm{det}(p)|^{-\frac{1}{2}}(\tau(p)v+V_{\tau}(U_{n}, \psi))
\]
for $v \in V_{\tau}$ and $p \in P_{n-1}$. Likewise $\Psi^-(\tau)$ is a normalized representation
of $GL_{n-1}$ on the space $V_{\tau} \slash V_{\tau}(U_{n},
\bf{1})$ with $V_{\tau}(U_{n}, {\bf{1}})=\langle \tau(u)v-v\; |\;
v \in V_{\tau}, u \in U_{n} \rangle$. Since $V_{\tau}(U_{n}, {\bf{1}})$ is the largest quotient on which $U_n$ acts trivially and $GL_{n-1}$ stabilizes $V_{\tau}(U_n,{\bf{1}})$, $\sigma=\Psi^-(\tau)$ is the normalized action of $GL_{n-1}$ given by
\[
 \sigma(g)(v+V_{\tau}(U_{n},{\bf{1}}))=|\mathrm{det}(g)|^{-\frac{1}{2}}(\tau(g)v+V_{\tau}(U_{n}, {\bf{1}}))
\]
for $v \in V_{\tau}$ and $g \in GL_{n-1}$. Given a representation
$(\sigma, V_{\sigma})$ of $P_{n-1}$, we denote
$\tau=\Phi^+(\sigma)$ by
$\Phi^+(\sigma)=\text{ind}_{P_{n-1}U_n}^{P_n}(|\mathrm{det}(g)|^{\frac{1}{2}}\sigma
\otimes \psi)$, where $``$ind$"$ is a compactly supported induction.
Similarly given  a representation $(\sigma, V_{\sigma})$ of
$GL_{n-1}$, we denote $\tau=\Psi^+(\sigma)$ by
$\Psi^+(\sigma)=\text{ind}_{GL_{n-1}U_n}^{P_n}(|\mathrm{det}(g)|^{\frac{1}{2}}\sigma
\otimes {\bf 1})$. In this case
$P_n=GL_{n-1}U_n$, so that $\text{ind}_{GL_{n-1}U_n}^{P_n}(|\mathrm{det}(g)|^{\frac{1}{2}}\sigma
\otimes {\bf 1})=|\mathrm{det}(g)|^{\frac{1}{2}}\sigma
\otimes {\bf 1}$. Bernstein and Zelevinsky establish the basic properties of these functors that are extracted from \cite{BeZe77, CoPe}:

\begin{enumerate}
\item $\Phi^+$ and $\Psi^+$ send irreducible representations to irreducible representations.
\item Any irreducible representation of $\tau$ of $P_n$ is of the from $\tau \simeq {(\Phi^+)}^{k-1}\Psi^+(\rho)$ with $\rho$ an irreducible representation of $GL_{n-k}$. The index $k$ and the representation $\rho$ are completely determined by $\tau$.
\item The derivatives: Let $\tau \in \mathrm{Rep}(P_n)$. For each $k=1,2,\dotsm,n$ there are representations $\tau_{(k)} \in \mathrm{Rep}(P_{n-k})$ and $\tau^{(k)} \in \mathrm{Rep}(GL_{n-k})$ associated to $\tau$ by $\tau_{(k)} = (\Phi^-)^k(\tau)$ and $\tau^{(k)}=\Psi^-(\Phi^-)^{k-1}(\tau)$. $\tau^{(k)}$ is called the $k^{th}$ derivative of $\tau$.
\item The filtration by derivatives: Any $\tau \in \mathrm{Rep}(P_n)$ admits a natural filtration of $P_n$ submodules $0 \subset \tau_{n} \subset \dotsm \subset \tau_2 \subset \tau_1=\tau$ such that $\tau_k=(\Phi^+)^{k-1}(\Phi^-)^{k-1}(\tau)$. The successive quotients are completely determined by the derivatives of $\tau$ since $\tau_k\slash\tau_{k+1}=(\Phi^+)^{k-1}\Psi^+(\tau^{(k)})$.
\end{enumerate}
Let $\pi \in \mathrm{Rep}(GL_n)$. The $0^{th}$ derivative is $\pi$ itself, that is $\pi^{(0)}=\pi$. If we set $\tau=\pi_{(0)}=\pi|_{\mathrm{P}_n}$ then $\pi_{(k)}=\tau_{(k)}$ and $\pi^{(k)}=\tau^{(k)}$ for $k=1,\dotsm,n$. Diagrammatically:
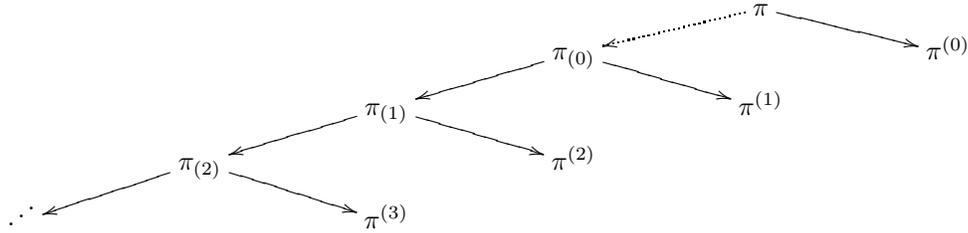
\[
  \xymatrix@C=4pc@L=.2pc@R=.2pc{    
                                                    &                                       &                                        &                                            &\pi \ar@{.>}[dl] \ar[dr] &     \\
                                                    &                                      &                                        &         \pi_{(0)} \ar[dl]  \ar[dr] &                            & \pi^{(0)}    \\
                                                    &                                      & \pi_{(1)}  \ar[dl] \ar[dr]     &                                           & \pi^{(1)}                &      \\                                                                           
                                                     & \pi_{(2)}  \ar[dl]  \ar[dr] &                                        &     \pi^{(2)}                            &                            &   \\
                  \iddots                         &                                     &   \pi^{(3)}                        &                                              &                           &     \\
  }
\]
where the leftward dotted and the first rightward arrow illustrate the restriction to $P_{n}$ and identity functor respectively, the rest of all leftward arrow represent an application of $\Phi^-$, and the remainder of the rightward arrow describe an application of $\Psi^-$.

\par

 Let $(\pi,V_{\pi})$ be an admissible representation of $GL_n$. We denote by $\mathrm{Hom}_{N_n}(V_{\pi},\psi)$ the space of all linear functionals $\lambda$ satisfying $\lambda(\pi(n)v)=\psi(n)\lambda(v)$ for all $v \in V_{\pi}$ and $n \in N_n$. The admissible representation $\pi$ is called non-degenerate if the dimension of the space $\mathrm{Hom}_{N_n}(V_{\pi},\psi)$ is not zero. In this case, a nontrivial linear functional $\lambda \in \mathrm{Hom}_{N_n}(V_{\pi},\psi)$ (up to scalars) is called a Whittaker functional on $V_{\pi}$. If $\pi$ is irreducible and $n \geq 2$, then Gelfand and Kazhdan in \cite{GeKa72} prove that $\mathrm{dim}(\mathrm{Hom}_{N_n}(V_{\pi},\psi))=0$ or $1$. $\pi$ is called generic  \cite[Section 2]{JaPSSh79} if it is irreducible and $\mathrm{dim}(\mathrm{Hom}_{N_n}(V_{\pi},\psi))=1$. Let $\text{Ind}_{N_n}^{GL_n}(\psi)$ denote the full space of smooth functions $W : GL_n \rightarrow \mathbb{C}$ which satisfy $W(ng)=\psi(n)W(g)$ for all $n \in N_n$. $GL_n$ acts on this by right translation. Frobenius reciprocity implies the isomorphism between $\mathrm{Hom}_{N_n}(V_{\pi},\psi)$ and $\mathrm{Hom}_{GL_n}(V_{\pi},\mathrm{Ind}_{N_n}^{GL_n}(\psi))$. If $\mathrm{dim}(\mathrm{Hom}_{N_n}(V_{\pi},\psi))=1$, we denote the image of the intertwining operator $V_{\pi} \rightarrow \text{Ind}_{N_n}^{GL_n}(\psi)$ by $\mathcal{W}(\pi,\psi)$ and call it Whittaker model of $\pi$. A map defined by $v \in V_{\pi} \mapsto W_v(g)=\lambda(\pi(g)v) \in \mathcal{W}(\pi,\psi)$ is not necessarily injective. We divide out by kernel of this map to obtain that the quotient space of $\pi$ is isomorphic to Whittaker model $\mathcal{W}(\pi,\psi)$ or embeds into the space $\mathrm{Ind}_{N_n}^{GL_n}(\psi)$. 
 
 \par
 Let $\pi^{\iota}$ be the representation of $GL_n$ on the same space $V
_{\pi}$ with the action $\pi^{\iota}(g)=\pi(^tg^{-1})$. We denote by $(\widetilde{\pi},V_{\widetilde{\pi}})$ the contragredient representation of $(\pi,V_{\pi})$. If $\pi$ is irreducible, then $\pi^{\iota}=\widetilde{\pi}$ from \cite{GeKa72}. Let $\displaystyle w=\begin{pmatrix} &&1\\& \iddots &\\1&& \end{pmatrix}$ denote the long Weyl elements of $GL_n$.  We define $\widetilde{W}$ by $\widetilde{W}(g)=W(w\;{^tg^{-1}}) \in \mathcal{W}(\pi^{\iota},\psi^{-1})$.

 \par
 Let $(\pi,V_{\pi})$ be an irreducible admissible representation of $GL_n$. Let $\nu(g)=|\text{det}(g)|$ be the unramified determinantal character of $GL_n$ for any $n$. We say that $\pi$ is square integrable if its central character is unitary and
\[
  \int_{Z_n \backslash GL_n} |\langle \pi(g)v,\tilde{v} \rangle|^2 dg < \infty,
\]
for all $v \in V_{\pi}$ and $\tilde{v} \in V_{\widetilde{\pi}}$. Here a function of the form $g \mapsto \langle \pi(g)v,\tilde{v} \rangle$ for $g \in GL_n$ is a matrix coefficient. We call  a representation quasi-square-integrable or a discrete series if  it is some twist of the square integrable representation by the unramified determinantal character $\nu^{s}$ for some $s \in \mathbb{C}$. If $\rho$ is an irreducible supercuspidal representation of $GL_r$, the normalized induced representation $\text{Ind}(\rho \otimes \rho\nu \otimes \dotsm \otimes \rho\nu^{\ell-1})$ from standard parabolic associated to the partition $(r,\dotsm,r)$ has the unique irreducible quotient, that we denote by $[\rho,\rho\nu,\dotsm,\rho\nu^{\ell-1}]$. We call such a representation of $GL_n$ a segment. According to Bernstein and Zelevinsky \cite[Theorem 9.3]{Ze80,BeZe77}, we have the following Theorem.

\begin{theorem}[Bernstein and Zelevinsky]
$\Delta$ is an irreducible quasi-square integrable representation of $GL_n$ if and only if 
\[
\Delta \simeq [\rho,\rho\nu,\dotsm,\rho\nu^{\ell-1}]
\]
 for some segment $[\rho,\rho\nu,\dotsm,\rho\nu^{\ell-1}]$.
\end{theorem}

 \par
 Let $\pi$ be an admissible and induced representation of $GL_n$ of the form
\[
  \pi=\mathrm{Ind}_Q^{GL_n}(\Delta_1 \otimes \dotsm \otimes \Delta_t)
\]
where each $\Delta_i$ is an irreducible quasi-sqare integrable representation of $GL_{n_i}$ and the induction is normalized from the standard parabolic $Q$ associated to the partition $(n_1,\dotsm,n_t)$ of $n$. The earlier work of Rodier in \cite{Ro73} shows that $\pi$ is non-degenerate and $\mathrm{dim}(\mathrm{Hom}_{N_n}(V_{\pi},\psi))=1$. $\pi$ needed not to be irreducible but it admits unique Whittaker models $\mathcal{W}(\pi,\psi)$. Following \cite{JaPSSh83}, we would like to introduce a wider class of representations called representations of Whittaker type which is more than just generic representations. 

 \begin{definition*}
 We say that the admissible representation $\pi$ of $GL_n$ is a representation of Whittaker type, if the dimension of the space $\mathrm{Hom}_{N_n}(V_{\pi},\psi)$ is $1$. We denote by
\[
   \mathrm{Ind}(\Delta_1 \otimes \dotsm \otimes \Delta_t)
\]
the normalized induction from the standard parabolic subgroup associated to the partition $(n_1,\dotsm,n_t)$ of $n$, where each $\Delta_i$ is an irreducible quasi-sqare integrable representation of $GL_{n_i}$.
\end{definition*}

The normalized induced representation $\pi=\mathrm{Ind}(\Delta_1 \otimes \dotsm \otimes \Delta_t)$ of $GL_n$ is included in the representation of Whittaker type, but these are the only representations of Whittaker type needed for the application. If $\pi$ is irreducible and generic, $V_{\pi}$ is isomorphic to $\mathcal{W}(\pi,\psi)$. 
 There is one case where the representation $\pi$ is a possibly reducible representation of Whittaker type and still is isomorphic to its Whittaker model $\mathcal{W}(\pi,\psi)$. These are the induced representation of Langlands type. An induced representation of Langlands type is a representation of the form
\[
   \pi=\mathrm{Ind}(\Delta_1\nu^{u_1} \otimes \dotsm \otimes \Delta_t\nu^{u_t})
\]
where each $\Delta_i$ is now an irreducible square integrable representation of $GL_{n_i}$, each $u_i$ is real and they are ordered so that $u_1 \geq u_2 \geq \dotsm \geq u_t$. The induction is normalized from the standard parabolic associated to the partition $(n_1,\dotsm,n_t)$ of $n$. From the work of Jacquet and Shalika \cite{JaSh83}, we know that for these induced representation of Langlands type the map $v \in V_{\pi} \mapsto W_v(g)=\lambda(\pi(g)v) \in \mathcal{W}(\pi,\psi)$ is an isomorphism. Now we introduce a realization in terms of the Whittaker model for $(\Phi^-)^{k-1}(\pi|_{P_n})=\pi_{(k-1)}$ from Section 1 in \cite{CoPe}.

\begin{proposition}[Cogdell and Piatetski-Shapiro]
 \label{mira-model}
Let $\pi$ be an induced representation of Langlands type of $GL_n$ and $k$ an integer between $1$ and $n-1$. We may view $P_{n-k+1}$ as a subgroup of $GL_n$ via the embedding $p \mapsto \mathrm{diag}(p,I_{k-1})$. The representation $\pi_{(k-1)}$ of $P_{n-k+1}$ has as a model the space of functions 
\[
\mathcal{W}(\pi_{(k-1)},\psi)=\left\{|\mathrm{det}(g)|^{-\frac{k-1}{2}}W\begin{pmatrix}g&\\&I_k \end{pmatrix} \middle|\; W \in \mathcal{W}(\pi,\psi),\; g \in GL_{n-k}  \right\}
\]
with the action being by right translation. 
\end{proposition}
As the subspace of $\mathcal{W}(\pi_{(k-1)},\psi)$, we have a similar characterization in terms of the Whittaker model for $\Phi^+(\pi_{(k)})$.
 \begin{proposition}[Cogdell and Piatetski-Shapiro]
  \label{mira-model2}
Let $\pi$ be an induced representation of Langlands type of $GL_n$ and $k$ an integer between $1$ and $n-1$. We may view $P_{n-k+1}$ as a subgroup of $GL_n$ via the embedding $p \mapsto \mathrm{diag}(p,I_{k-1})$. As $P_{n-k+1}$ modules, $\Phi^+(\pi_{(k)})$ has the space of functions
and 
\[
\begin{split}
\Phi^+(\pi_{(k)}) &\simeq \left\{ |\mathrm{det}(g)|^{-\frac{k-1}{2}}W\begin{pmatrix}g&\\&I_k \end{pmatrix} \right| \;W \in \mathcal{W}(\pi,\psi),\; g \in GL_{n-k}, \mathrm{\;and\; there\;} \\
 &\phantom{****} \mathrm{exists}\; N>0\;\mathrm{such\;that}  
 \left. \mathrm{if}\max_i\{|g_{n-k,i}|\}<q^{-N}\;\mathrm{then}\; W\begin{pmatrix}g&\\&I_k \end{pmatrix} \equiv 0 \right\} \\
& = V_{\pi_{(k-1)}}(U_{n-k+1},\bf{1}).
\end{split}
\]
\end{proposition}
We will refer to previous two models as Whittaker models for $\pi_{(n-k)}$ and $\Phi^+(\pi_{(k)})$, respectively. Let $\pi=\mathrm{Ind}(\Delta_1 \otimes \dotsm \otimes \Delta_t)$ be an irreducible generic representation of $GL_n$. As the quasi-square integrable representations $\Delta_i$ can be reordered to be in the representation of Langlands type without changing $\pi$, Proposition \ref{mira-model} and \ref{mira-model2} remain valid for irreducible generic representations. We recall the following results about the connection between restriction of Whittaker functions to $P_{n-k+1}$ and derivatives from Section 1 of \cite{CoPe} to prove a factorization of local exterior square $L$-functions described in Section $4$ and $5$.

\begin{proposition} [Cogdell and Piatetski-Shapiro]
\label{asym}
Let $\pi=\mathrm{Ind}(\Delta_1 \otimes \dotsm \otimes \Delta_t)$ be an irreducible generic representation of $GL_n$ and $k$ an integer between $1$ and $n-1$. Let $\tau=\pi_{(k-1)}$ and $V_{\tau}$ the space for $\tau$. Let $\pi_{0}^{(k)} \subset \pi^{(k)}$ be an irreducible sub-representation of $\pi^{(k)}$, the normalized quotient representation on $V_{\pi^{(k)}}=V_{\tau} \slash V_{\tau}(U_{n-k+1},\bf{1})$. Let $g \in GL_{n-k}$ and $p : V_{\tau} \rightarrow V_{\pi^{(k)}}$ the normalized projection map and let $V_{\tau_0}=p^{-1}(V_{\pi_0^{(k)}})$. We may view $P_{n-k+1}$ as a subgroup of $GL_n$ via the embedding $g \mapsto \mathrm{diag}(g,I_{k-1})$. For every $W_{\circ} \in \mathcal{W}(\pi_{0}^{(k)},\psi)$ there is a $W \in \mathcal{W}(\tau_0,\psi)$ and $\Phi_{\circ} \in \mathcal{S}(F^{n-k})$ with $\Phi_{\circ}(0) \neq 0$ such that
\[
W_{\circ}(g)\Phi_{\circ}(e_{n-k} g)=|\mathrm{det}(g)|^{-\frac{k}{2}}W\begin{pmatrix} g& \\ &I_k \end{pmatrix} \Phi_{\circ}(e_{n-k} g).
\]
Moreover, for every $W \in \mathcal{W}(\tau_0,\psi)$ and every $\Phi_{\circ} \in \mathcal{S}(F^{n-k})$ locally constant and supported in a sufficiently small neighborhood of $0 \in F^{n-k}$, there is a $W_{\circ} \in \mathcal{W}(\pi_0^{(k)},\psi)$ such that
\[
W\begin{pmatrix} g& \\ &I_k \end{pmatrix} \Phi_{\circ}(e_{n-k} g)=|\mathrm{det}(g)|^{\frac{k}{2}}W_{\circ}(g)\Phi_{\circ}(e_{n-k} g).
\]
\end{proposition}

 Let $(\pi,V_{\pi}) \in \mathrm{Rep}(GL_n)$. Then $\pi^{(n)}$ is a representation of $GL_0$, that is, just a vector space, since $GL_0$ is a trivial group. The space of $\pi^{(n)}$ is $V_{\pi} \slash V_{\pi}(N_n,\psi)$ with $V_{\pi}(N_n,\psi)=\langle \pi(n)v-\psi(n)v \;|\; v \in V_{\pi},\;n \in N_n \rangle$. $V_{\pi} \slash V_{\pi}(N_n,\psi)$ is the maximal quotient of $V_{\pi}$ on which $N_n$ acts via the character $\psi$. Therefore the dual linear space $(V_{\pi} \slash V_{\pi}(N_n,\psi))^*$ is the space of Whittaker functionals on $\pi$ and we have $\mathrm{dim}(\pi^{(n)}) \leq 1$ if $\pi$ is irreducible \cite{GeKa72}. Finally let $\pi=\mathrm{Ind}(\Delta_1 \otimes \dotsm \otimes \Delta_t)$ be a representation of Whittaker type. For the nontrivial $\psi$-Whittaker functional $\lambda$, we set
\[
  \mathcal{W}_{(0)}(\pi,\psi)=\{ W(p)=\lambda(\pi(p)v)  \; | \; p \in P_n, v \in V_{\pi} \}.
\]
and let $\tau$ be the representation of $P_n$ in this space by right translation. Let $\mathrm{ind}^{P_{n}}_{N_{n}}(\psi)$  be the space of all smooth functions $W : GL_n \rightarrow \mathbb{C}$, which are compactly supported modulo $N_n$ and $W(np)=\psi(n)W(p)$ for all $p \in P_n$ and $n \in N_n$. For $\tau$ a representation of $P_n$, the bottom and top pieces of filtration by derivatives $0 \subset \tau_n \subset \dotsm \dotsm \subset \tau_1=\tau$ such that $\tau_k=(\Phi^+)^{k-1}(\Phi^-)^{k-1}(\tau)$ implies that $(\tau,\mathcal{W}_{(0)}(\pi,\psi))$ contains $(\tau_n,\mathrm{ind}_{N_n}^{P_n}(\psi))$ as $P_n$ modules. We record it as the following Lemma.
 
 \begin{lemma}
 \label{kirillov}
 Let $\pi=\mathrm{Ind}(\Delta_1 \otimes \dotsm \otimes \Delta_t)$ be a representation of $GL_n$ of Whittaker type. For the non-trivial $\psi$ Whittaker functional $\lambda$, let $\mathcal{W}_{(0)}(\pi,\psi)=\{ W(p)=\lambda(\pi(p)v) \in \mathcal{W}(\pi,\psi)  |  p \in P_n, v \in V_{\pi} \}$ with the action $\tau$ of $P_n$ being by right translation. Then as $P_n$ modules, $(\tau_n,\mathrm{ind}_{N_n}^{P_n}(\psi))$ is contained in $(\tau,\mathcal{W}_{(0)}(\pi,\psi))$.
 \end{lemma}

 For a representation of Whittaker type $\pi=\mathrm{Ind}(\Delta_1 \otimes \dotsm \otimes \Delta_t)$, the derivative $\pi^{(k)}$ are computed by Bernstein and Zelevinsky \cite{BeZe77, Ze80}.
The result is as follows.

 \begin{theorem}[Bernstein and Zelevinsky]
 \label{deriviative}
 Let $\rho$ be an irreducible supercuspidal representation of $GL_r$. Let $\Delta=[\rho,\dotsm,\rho\nu^{\ell-1}]$ be a quasi-square integrable representation of $GL_n$ with $n=\ell r$.
  \begin{enumerate}
 \item[$(\mathrm{1})$] $\rho^{(0)}=\rho$, $\rho^{(k)}=0$ for $1 \leq k \leq r-1$ and $\rho^{(r)}=\bf{1}$.
  \item[$(\mathrm{2})$] $\Delta^{(k)}=0$ if $k$ is not a multiple of $r$, $\Delta^{(0)}=\Delta$, $\Delta^{(kr)}=[\rho \nu^k,\dotsm,\rho \nu^{\ell-1}]$ for $1 \leq k \leq \ell-1$, and $\Delta^{(\ell r)}=\bf{1}$.  
   \item[$(\mathrm{3})$] 
 If $\pi=\mathrm{Ind}(\Delta_1 \otimes \dotsm \otimes \Delta_t)$ is a representation of Whittaker type of $GL_n$, then we have a filtration of $\pi^{(k)}$ for $1 \leq k \leq n$ whose successive quotients are the representations $\mathrm{Ind}(\Delta_1^{(k_1)} \otimes \dotsm \otimes \Delta_t^{(k_t)})$ with $k=k_1+\dotsm+k_t$.
 \end{enumerate}
 \end{theorem}

 We end this Section with giving the relation between Whittaker functions and Schwartz-Bruhat functions, which we will be using.
 
 \begin{lemma}
 \label{key relation}
 Let $\pi=\mathrm{Ind}(\Delta_1 \otimes \dotsm \otimes \Delta_t)$ be a representation of Whittaker type of $GL_n$. Let $\lambda$ be the unique $\psi$-Whittaker functional on $(\pi,V_{\pi})$. For each $v \in V_{\pi}$, set $W_v(g)=\lambda(\pi(g)v)$ for any $g \in GL_n$.
 Let $k < n$. Given $v \in V_{\pi}$ and a smooth function $\Phi \in \mathcal{S}(F^k)$ there is $v_0 \in V_{\pi}$ such that, for any $g \in GL_k$,
 \[
   W_{v_0} \begin{pmatrix} g&\\ & I_{n-k} \end{pmatrix}=W_v \begin{pmatrix} g&\\&I_{n-k} \end{pmatrix} \Phi(e_kg).
 \]  
 Conversely, given $v \in V_{\pi}$, there are $v_0 \in V_{\pi}$ and Schwartz-Bruhat function $\Phi \in \mathcal{S}(F^k)$ such that
 \[
   W_v \begin{pmatrix} g&\\ &I_{n-k} \end{pmatrix}= W_{v_0} \begin{pmatrix}  g&\\&I_{n-k} \end{pmatrix}  \Phi(e_kg).
 \]
 
 \end{lemma}

 \begin{proof}
 For the first part, we apply the Fourier transform of $\Phi$ and set
 \[
  v_0=\int_{F^k} \pi \begin{pmatrix} I_k & u & \\ &1& \\ &&I_{n-k-1} \end{pmatrix} v\; \hat{\Phi}(-^tu) \; du.
 \]
 Since $\hat{\Phi}$ is locally constant and of compact support, this integral is essentially a finite sum. As $W_{v_0}(g)=\lambda(\pi(g)v_0)$ and $W_v(g)=\lambda(\pi(g)v)$ we obtain
 \[
 \begin{split}
   W_{v_0} \begin{pmatrix} g&&\\ &1&\\&&I_{n-k-1} \end{pmatrix}&=\int_{F^k} W_v \left( \begin{pmatrix} g&&\\ &1&\\&&I_{n-k-1} \end{pmatrix}  \begin{pmatrix} I_k & u & \\ &1& \\ &&I_{n-k-1} \end{pmatrix}  \right)
  \hat{\Phi}(-^tu)  du\\
  &=W_v \begin{pmatrix} g&&\\ &1&\\&&I_{n-k-1} \end{pmatrix}\; \int_{F^k} \psi(e_kgu)  \hat{\Phi}(-^tu) du\\
  &=W_v \begin{pmatrix} g&&\\ &1&\\&&I_{n-k-1} \end{pmatrix}\; \Phi(e_kg).
 \end{split}
 \]
 \par
 For the second assertion, we proceed similarly. Let $v \in V_{\pi}$ be given. Since $\pi$ is a smooth representation, there is a compact open subgroup $H$ of $F^k$, the space of column vectors of size $k$, for which
 \[
   W_v\begin{pmatrix} g&&\\ &1&\\&&I_{n-k-1} \end{pmatrix}=W_v\left( \begin{pmatrix} g&&\\ &1&\\&&I_{n-k-1} \end{pmatrix}  \begin{pmatrix} I_k & u & \\ &1& \\ &&I_{n-k-1} \end{pmatrix}  \right)
 \]
  for all $u \in H$. If $\hat{\Phi}$ is chosen to be an appropriate scalar multiple of characteristic function of $H$, we obtain
  \[
    \begin{split}
    &W_v\left( \begin{pmatrix} g&&\\ &1&\\&&I_{n-k-1} \end{pmatrix}  \begin{pmatrix} I_k & u & \\ &1& \\ &&I_{n-k-1} \end{pmatrix}  \right) \\
    &=\int_{F^k} W_v \left( \begin{pmatrix} g&&\\ &1&\\&&I_{n-k-1} \end{pmatrix}  \begin{pmatrix} I_k & u & \\ &1& \\ &&I_{n-k-1} \end{pmatrix}  \right)
  \hat{\Phi}(-^tu)  du.
 \end{split}
  \]
 Letting $v=v_0$,  we have the desired relation:
  \[
    W_v \begin{pmatrix} g&&\\ &1&\\&&I_{n-k-1} \end{pmatrix}=W_{v_0} \begin{pmatrix} g&&\\ &1&\\&&I_{n-k-1} \end{pmatrix} \Phi(e_kg).
  \]
 \end{proof}

\begin{remark}
In this paper, we will often say $``$representation$"$ instead of $``$admissible representation$"$.
\end{remark}

\subsection{Measure on the quotient space} 
We shall now give a brief description on quasi-invariant measure in Section 9 of \cite{We41} (cf. \cite[Appendix]{Wa72} and \cite[Chapter 2]{Fo95}). 
Let $G$ be a locally compact group and $H$ a closed subgroup. Let $C_c^{\infty}(G)$ be the space of functions $\displaystyle f : G \longrightarrow \mathbb{C}$ which are locally constant and of compact support. We denote by $d_L g$ a left Haar measure on $G$. Let $d_R g$ be a right Haar measure on $G$. Let $\delta_G$ denote the modular function on $G$ satisfying $d_Lg(xh)=\delta_G(h)^{-1}d_Lg(x)$ for $x,h \in G$. Let $\rho$ be a strictly positive continuous function on $G$, bounded above and below on compact subsets such that $\displaystyle \rho(hg)=\frac{\delta_H(h)}{\delta_G(h)} \rho(g)$ for all $h \in H$ and $g \in G$. A function with these properties is called a rho-function. Suppose that $\mu$ is a positive Radon measure on the quotient space $H \backslash G$. For $x \in G$, we define right translation $\mu_x$ of $\mu$ by $\mu_x(Hg)=\mu(Hgx)$ for $Hg \in H \backslash G$. $\mu$ is called quasi-invariant if the measure $\mu_x$ for $x \in G$ are all mutually absolutely continuous. Due to Weil in Section 9 of \cite{We41}, if we fix a rho-function $\rho$, there is associated quasi-invariant measure  $\mu_{\rho}$ such that
\[
  \int_{G} f(g) \rho(g) d_R g=\int_{H \backslash G} \int_H f(hg) d_R h\; d \mu_{\rho} (g)
\]
with $f \in C_c(G)$. 
For $f \in C_c^{\infty}(G)$, we denote by $f^H$ the function in $C_c^{\infty}(H \backslash G)$ defined by
$\displaystyle
  f^H(g)=\int_H f(hg) d_R h.
$
If $G$ is a locally profinite group, then it is proven in \cite[3.4 Proposition]{BuHe06} that the map $f \mapsto f^H$ is surjective. Let $G$ be a unimodular group, and let $P$ and $K$ be closed subgroups such that $P \cap K$ is compact and $G=PK$. Let $d_Lp$ and $d_Rk$ be left Haar measure on $P$ and right Haar measure on $K$. From Proposition 2.1.5 in \cite{Bu97} the Haar measure on $G$ is given by
\[
   \int_G f(g)dg=\int_P \int_K f(pk)  d_Rk\; d_Lp.
\]
Because $P \cap K$ is compact, we normalize the Haar measure on $P \cap K$ so the volume of $P \cap K$ is equal to one.

\par
Let us write $g \in GL_n$ as $g=zpk$ with $z \in Z_n,\; p \in P_n,$ and $k \in K_n$. We refer to this as the partial Iwasawa decomposition. Fix a function $f$ in $C_c^{\infty}(GL_n)$. We normalize measures on $K_n, P_n$ and $Z_n$ with respect to $K_n$. That is, we endow $dk, dp$ and $dz$ with unit volume on the sets $P_n \cap K_n$, and $Z_n \cap K_n$ so that
\[
   \int_{P_n \cap K_n} dk=\int_{P_n \cap K_n} dp=\int_{Z_n \cap K_n} dz=1.
\]
If $GL_n=Z_n P_n K_n$ is the partial Iwasawa decomposition, then the corresponding decomposition of the Haar measure on $GL_n$ is
\[
   \int_{GL_n} f(g) dg =\int_{K_n} \int_{P_n}  \int_{Z_n}  f(zpk) dz \; d_Lp \; dk=\int_{K_n} \int_{P_n}  \int_{Z_n} f(zpk) \delta_{P_n}(p)^{-1} dz d_Rp dk.
\]
The following Lemma describes the partial Iwasawa decomposition of the Haar measure on the quotient space $N_n \backslash GL_n$.

\begin{lemma}
\label{zpkdec}
Let $dg$ be an right Haar measure on $N_n \backslash GL_n$, $dp$ an right Haar measure on $N_n \backslash P_n$, $dz$ a Haar measure on $Z_n$, and $dk$ a Haar measure on $K_n$.
For all $f \in C_c^{\infty}(N_n \backslash GL_n)$, we have
\[
  \int_{N_n \backslash GL_n} f(g) dg=\int_{K_n} \int_{N_n \backslash P_n} \int_{Z_n}  f(zpk)|\mathrm{det}(p)|^{-1} dzdpdk. 
 \]
\end{lemma}

\begin{proof}
For $\Psi \in C_c^{\infty}(GL_n)$, the function $\Psi^{N_n}$ in $C_c^{\infty}(N_n \backslash GL_n)$ is defined by $\displaystyle \Psi^{N_n}(g)=\int_{N_n} \Psi(ng) dn$.
Since $GL_n$ and $N_n$ are unimodular, there is a unique $GL_n$-right invariant Haar measure $\mu$ on $N_n \backslash GL_n$ up to a constant. 
From surjection of the map $\Psi \mapsto \Psi^{N_n}$ by \cite[3.4 Proposition]{BuHe06}, there is a function $F \in C_c^{\infty}(GL_n)$  such that $f=F^{N_n}$.
The function $F$ also satisfies Weil integration formula in Section 9 of \cite{We41}, which is given by
\[
   \int_{GL_n} F(g) dg = \int_{N_n \backslash GL_n} \int_{N_n} F(ng) dn\; d\mu=  \int_{N_n \backslash GL_n} f(g) d\mu.
\]
We denote by abuse of notation again $dg$ for $GL_n$-right invariant Haar measure $d\mu$ on $N_n \backslash GL_n$. Since $\delta_{P_n}(p)=|\mathrm{det}(p)|$ and $\delta_{P_n}|_{N_n}=\delta_{N_n}=1$, $N_n \backslash P_n$ admits a $P_n$-right invariant Haar measure $dp$ by abusing notation. Weil integration formula in Section 9 of \cite{We41} again asserts that
\[
  \int_{P_n} F(p) d_R p = \int_{N_n \backslash P_n} \int_{N_n} F(np) dn\; dp.
\]
We split the Haar measure on $P_n$ associated with the measure on quotient space $N_n \backslash P_n$, which is given by
\[
\begin{split}
    &\int_{N_n \backslash GL_n} f(g) d\mu=\int_{K_n} \int_{P_n}  \int_{Z_n} F(zpk) \delta_{P_n}(p)^{-1} dz \; d_Rp \; dk\\
   &=\int_{K_n} \int_{N_n \backslash P_n}  \int_{Z_n}  \int_{N_n} F(npzk)\delta_{P_n}(p)^{-1} dn  dz dp dk\\
   &=\int_{K_n} \int_{N_n \backslash P_n} \int_{Z_n}  f(zpk)\delta_{P_n}(p)^{-1}\; dz  dp dk.
\end{split}
\]
This establishes the desired equality.

\end{proof}

We now present another Lemma on the space $P_n\backslash GL_n \simeq F^n-\{0 \}$ that is widely known but is not explained in the literature.

\begin{lemma}
Let $u$ be a Haar measure on $F^n$, $d^{\times} a$ a Haar measure on $F^{\times}$, and $dk$ a Haar measure on $K_n$. For any $f \in C_c^{\infty}(F^n)$, we have
\[
 \int_{K_n} \int_{F^{\times}}  f(e_nak) |a|^nd^{\times}a dk=\int_{F^n-\{ 0\}} f(u) du.
\]
\end{lemma}

\begin{proof}
\label{pkdec}
 Recall that $\nu(g)=|\mathrm{det}(g)|$ for $g \in GL_n$. Then $\nu$ is a rho-function of $P_n$, since $\delta_{GL_n}=1$ and $\delta_{P_n}(p)=|\mathrm{det}(p)|$. Let $\mu_{\nu}$ be a quasi-invariant measure on $P_n \backslash GL_n$ with corresponding rho-function $\nu$. For $\Psi \in C_c^{\infty}(GL_n)$, the function $\Psi^{P_n}$ defined by $\displaystyle \Psi^{P_n}(g)=\int_{P_n} \Psi(pg) dp$ is an element of $C_c^{\infty}(P_n \backslash GL_n)$. The map $\Phi \mapsto \Phi^{P_n}$ is surjective by \cite[3.4 Proposition]{BuHe06}  and so, for $f(e_n\cdot) \in C_c^{\infty}(P_n \backslash GL_n)$ there is $F \in C_c^{\infty}(GL_n)$ such that $f(e_ng)=F^{P_n}(g)$ for all $g \in P_n \backslash GL_n$.
 On the one hand if $GL_n=Z_nP_nK_n$ is the partial Iwasawa decomposition, the corresponding decomposition of the Haar measure on $GL_n$ is
 \[
 \label{GLsplit1}
 \tag{2.1}
\begin{split}
   &\int_{GL_n} F(g) |\mathrm{det}(g)| dg=\int_{K_n} \int_{P_n} \int_{Z_n} F(zpk) |\mathrm{det}(zpk)| dz d_Lp dk\\
   &= \int_{K_n} \int_{F^{\times}}  \int_{P_n} F(pak) |a|^n|\mathrm{det}(p)| \delta_{P_n}^{-1}(p) dp dk d^{\times}a=\int_{K_n} \int_{F^{\times}} f(e_nak) |a|^nd^{\times}a dk,
  \end{split}
\]
 where $d_L p$ is a left Haar measure on $P_n$. On the other hand $F \in C_c^{\infty}(GL_n)$ also satisfies Weil integration formula in Section 9 of \cite{We41}, which is given by
\[
 \label{GLsplit2}
 \tag{2.2}
  \int_{GL_n} F(g) |\mathrm{det}(g)| dg=\int_{P_n \backslash GL_n} \int_{P_n} F(pg) dp d\mu_{\nu}(g)=\int_{P_n \backslash GL_n} f(e_ng) d\mu_{\nu}(g).
\]
Combined \eqref{GLsplit1} with \eqref{GLsplit2}, we obtain 
\[
  \int_{K_n} \int_{F^{\times}} f(e_nak) |a|^nd^{\times}a dk=\int_{P_n \backslash GL_n} f(e_ng) d\mu_{\nu}(g).
\]
Then the quotient space $P_n \backslash GL_n$  is homeomorphic to $F^n-\{ 0\}$ by $g \mapsto e_ng$. As $\mu_{\nu}$ on the coset space $P_n \backslash GL_n$ and $u$ on $F^n$
are both $|\mathrm{det}(h)|$-invariant measure for $h \in GL_n$, we can deduce from Weil \cite[Chapter 2, Section 9]{We41} that
\[
  \int_{P_n \backslash GL_n} f(e_ng) d\mu_{\nu}(g)=\int_{F^{n}-\{0 \}} f(u) du.
\]
In turn we identify $P_n \backslash GL_n$ with $F^n-\{ 0\}$ by $g \mapsto e_ng$ as homogeneous $GL_n$-spaces and that the quasi invariant measure $\mu_{\nu}$ transforms to the Haar measure $u$ on $F^n$. This completes the proof of the Lemma.
\end{proof}

\subsection{Integral representation for exterior square $L$-functions}

 Let $\pi=\mathrm{Ind}(\Delta_1 \otimes \dotsm \otimes \Delta_t)$ be a representation of $GL_m$ of Whittaker type. Let us denote by $\omega_{\pi}$ the central character of $\pi$. As $\pi^{(k)}=\mathrm{Ind}(\Delta_1 \otimes \dotsm \otimes \Delta_t)^{(k)}$ has a filtration whose successive quotients are the $\mathrm{Ind}(\Delta_1^{(k_1)} \otimes \dotsm \otimes \Delta_t^{(k_t)})$ with $k=k_1+\dotsm+k_t$ by Theorem \ref{deriviative}, we denote by $\pi_{i_k}^{(k)}$ nonzero successive quotients $\mathrm{Ind}(\Delta_1^{(k_1)} \otimes \dotsm \otimes \Delta_t^{(k_t)})$.
 For $0 < k < m$, let $r_k$ be the cardinality of be the set
 \[
  \{ (k_1,\dotsm,k_t) \;|\; k=k_1+\dotsm+k_t  \; \text{and}\;  \mathrm{Ind}(\Delta_1^{(k_1)} \otimes \dotsm \otimes \Delta_t^{(k_t)}) \neq 0  \; \text{for}\; k_i \geq 0  \}.
\] 
 We call $(\omega_{\pi_{i_k}^{(k)}})$, $1 \leq i_k \leq r_k$  the family of the central characters for all nonzero successive quotients $\pi_{i_k}^{(k)}=\mathrm{Ind}(\Delta_1^{(k_1)} \otimes \dotsm \otimes \Delta_t^{(k_t)})$ with $k=k_1+\dotsm+k_t$. If $r_k=0$, then the family is empty. We begin with reviewing some basic properties about smooth Whittaker functions $W \in \mathcal{W}(\pi,\psi)$ in \cite{CoKiMu04}. They satisfy
\begin{enumerate}
\item[$(\mathrm{i})$] $W(ng)=\psi(n)W(g)$ for all $n \in N_m$ and $g \in GL_m$. 
\item[$(\mathrm{ii})$] $W(gk')=W(g)$ for $k' \in K^{\circ}$ and for some $K^{\circ} \subset GL(\mathcal{O})=K_m$.
\end{enumerate}
Using the Iwasawa decomposition on $GL_m$, they are essentially controlled by their behavior on
\[
  A_m=\left\{ a=\begin{pmatrix} a_1a_2 \dotsm a_m&&&\\ &a_2a_3 \dotsm a_m&&\\ &&\ddots&\\&&&a_m \end{pmatrix}\; \middle| \;a_i \in F^{\times} \right\}.
  \]
 Under these simple roots coordinates, $a_i$, $1 \leq i \leq m-1$ are the standard simple roots for $GL_m$. The basic result for the asymptotic expansion of Whittaker functions in \cite[Proposition 2.9]{Ma15} which is proved by Matringe is given in the following proposition. The formula describes the relation between asymptotic of Whittaker functions and central characters for successive quotient occurring in a composition series of Bernstein-Zelevinsky derivatives. This proposition is an refinement of analytic property of the Whittaker functions in \cite[Proposition 2.2]{JaPSSh79}.

\begin{proposition}[Matringe]
\label{asymWh}
Let $\pi=\mathrm{Ind}(\Delta_1 \otimes \dotsm \otimes \Delta_t)$ be a representation of $GL_m$ of Whittaker type. Let $S$ be the set of positive integer $k$ such that $1 \leq k \leq m-1$ and the derivative $\pi^{(k)}$ is nonzero. We set $T=\{1,2,\dotsm,m-1 \}-S$. We then denote for each $k \in S$, by  $(\omega_{\pi_{i_k}^{(k)}})_{i_k}$, the family of the central characters for all nonzero successive quotients of the form $\pi_{i_k}^{(k)}=\mathrm{Ind}(\Delta_1^{(k_1)} \otimes \dotsm \otimes \Delta_t^{(k_t)})$ with $k=k_1+\dotsm+k_t$ appearing in the composition series of $\pi^{(k)}$. Let $\mathcal{S}_0(F)$ be the space of Schwartz-Bruhat functions $\Phi$ such that $\Phi(0)=0$. For every $W \in \mathcal{W}(\pi,\psi)$, $W(\mathrm{diag}(a_1a_2\dotsm a_m,a_2a_3\dotsm a_m,\dotsm,a_{m-1}a_m,a_m))$ is a linear combination of functions of the form
\[
\tag{2.3}
\label{asymWh-expression}
  \omega_{\pi}(a_m) \prod_{m-k \in S} \omega_{\pi_{i_{m-k}}^{(m-k)}}(a_k) |a_k|^{\frac{k(m-k)}{2}} v (a_k)^{m_k} \Phi_k(a_k) \prod_{m-\ell \in T} \Phi_{\ell}(a_{\ell}),
\]
for $m_k$ non-negative integers depending on $\pi$, $\Phi_k$ in $\mathcal{S}(F)$, and $\Phi_{\ell}$ in $\mathcal{S}_0(F)$. We set $v^{m_k} \equiv 1$ for $m_k=0$.
\end{proposition}

Notice that we can determine the behavior in $a_i$ by using the central characters. Furthermore the central characters $\omega_{\pi}$ and $\omega_{\pi_{i_{m-k}}^{(m-k)}}$, and the powers of valuation $m_k$ depend only on $\pi$ not the choice of $W$. 

\begin{remark}
We emphasize that Proposition \ref{asymWh} differs from the one in Proposition 2.9 of \cite{Ma15}. This is because there is a minor mistake in Theorem 2.1 of \cite{Ma11-2} and asymptotic of Whittaker functions for $GL_m$ in \cite[Proposition 2.9]{Ma15} is derived from \cite[Theorem 2.1]{Ma11-2}. The problem in  \cite[Theorem 2.1]{Ma11-2} is that the contribution of the zero derivatives to the asymptotic expansion of Whittaker functions on the torus is neglected. In fact we
just need to put Schwartz-Bruhat functions vanishing at zero in the corresponding term of  
the product for zero derivatives. The author in \cite{Ma11-2} corrects the statement and the proof of \cite[Theorem 2.1]{Ma11-2} in $``$https://arxiv.org/abs/1004.1315v2$"$.

\end{remark}

Following the notation in Section 4 of \cite{JaSh88}, we define
\[
  m(a_1,\;a_2,\;\dotsm,\;a_m)=\text{diag}(a_1a_2\dotsm a_{m},\;a_2a_3\dotsm a_{m},\;\dotsm,\;a_{m-1}a_m,\;a_m) \in A_m.
\]
Let $B_m$ be the Borel subgroup of $GL_m$ consisting of upper triangular matrices.  Let $d_L b$ and $d_R b$ denote a left Haar measure and a right Haar measure on $B_m$. Then if
\[
  b=\begin{pmatrix} 1 & x_{12} & \dotsm & x_{1m} \\  & 1 & \dotsm &x_{2m} \\ && \ddots & \vdots \\ &&&1  \end{pmatrix} \begin{pmatrix} a_1a_2 \dotsm a_m&&&\\ &a_2a_3 \dotsm a_m&&\\ &&\ddots\\&&&a_m \end{pmatrix} \in B_m,
\]
we define $\delta_{B_m}$ the modular character of $B_m$, which is the ratio of right and left Haar measures, by
\[
d_Rb=\delta_{B_m}(b)d_Lb, \quad \quad \quad \delta_{B_m}(b)=\prod_{i=1}^m |a_i|^{i(m-i)}. 
\]

\subsubsection{The even case $m=2n$.}
Let $\pi$ be an irreducible admissible generic unitary representation of $GL_{2n}$. Let $\sigma_{2n}$ be the permutation matrix given by
\[
\sigma_{2n} = \begin{pmatrix} 1& 2 & \dotsm & n & | & n+1 & n+2 & \dotsm & 2n \\ 1& 3 & \dotsm & 2n-1 & | & 2 & 4 & \dotsm & 2n \end{pmatrix}.
\]
 In 1990, Jacquet and Shalika \cite{JaSh88} establish an integral representation for the exterior square $L$-function for $GL_{2n}$ in the following way:
\[
\begin{split}
&J(s,W,\Phi)\\
& = \int_{N_n\backslash GL_n} \int_{\mathcal{N}_n\backslash \mathcal{M}_n} W \left(\sigma_{2n}\begin{pmatrix} I_n & X \\ & I_n  \end{pmatrix} \begin{pmatrix} g&\\&g \end{pmatrix} \right) \psi^{-1}(\mathrm{Tr} X) dX \; \Phi(e_n g) |\mathrm{det}(g)|^s\; dg,
\end{split}
\]
where $W \in \mathcal{W}(\pi,\psi)$ and $\Phi \in \mathcal{S}(F^n)$. It is proved in Section 7 of \cite{JaSh88} that there is $r_{\pi}$ in $\mathbb{R}$ such that the above integrals $J(s,W,\Phi)$ converge for $\text{Re}(s) > r_{\pi}$. 

\par
Let $(\pi,V_{\pi})$ be a representation of $GL_{2n}$ of Whittaker type.  In this section, we need to understand what type of functions of $s$ these local integrals $J(s,W,\Phi)$ are for the representations of Whittaker type. 
We define the Shalika subgroup $S_{2n}$ of $GL_{2n}$ by
\[
S_{2n}=\left\{ \begin{pmatrix} I_n & Z \\ & I_n \end{pmatrix} \begin{pmatrix} g&\\ &g \end{pmatrix} \middle| \; Z \in \mathcal{M}_n, \; g \in GL_n \right\}.
\]
Let us define an action of the Shalika group $S_{2n}$ on $\mathcal{S}(F^n)$ by 
\[
R\begin{pmatrix} \begin{pmatrix} I_n & Z \\ & I_n \end{pmatrix}\begin{pmatrix} g &  \\ & g \end{pmatrix}  \end{pmatrix}\Phi(x)=\Phi(xg)
\]
 for $\Phi \in \mathcal{S}(F^n)$. For any character $\chi$ of $F^{\times}$, $\chi$ can be uniquely decomposed as $\chi=\chi_0\nu^{s_0}$, where $\chi_0$ is a unitary character and $s_0$ is a real number. To proceed from this point, let us make a convention that, we use the notion $s_0=\text{Re}(\chi)$ for the real part of the exponent of the character $\chi$. Now we introduce the existence Theorem of Jacquet and Shalika in \cite{JaSh88}.

\begin{theorem}
\label{integral-rational}
Let $\pi=\mathrm{Ind}(\Delta_1 \otimes \dotsm \otimes \Delta_t)$ be a representation of Whittaker type of $GL_{2n}$. For every $0 \leq k \leq 2n-1$, let $(\omega_{\pi_{i_k}^{(k)}})_{i_k=1,\dotsm,r_k}$ be the family of the central characters for all nonzero successive quotients of the form $\pi_{i_k}^{(k)}=\mathrm{Ind}(\Delta_1^{(k_1)} \otimes \dotsm \otimes \Delta_t^{(k_t)})$ with $k=k_1+\dotsm+k_t$ appearing in the composition series of $\pi^{(k)}$. Let $W \in \mathcal{W}(\pi,\psi)$ and $\Phi \in \mathcal{S}(F^n)$.
\begin{enumerate}
\item[$(\mathrm{i})$] If for all $1 \leq k \leq n$ and all $1 \leq i_{2k} \leq r_{2k}$, we have $\mathrm{Re}(s) >-\frac{1}{k}\mathrm{Re}(\omega_{\pi_{i_{2n-2k}}^{(2n-2k)}})$, then each local integral $J(s,W,\Phi)$ converges absolutely.
\item[$(\mathrm{ii})$] Each $J(s,W,\Phi) \in \mathbb{C}(q^{-s})$ is a rational function of $q^{-s}$ and hence $J(s,W,\Phi)$ as a function of $s$ extends meromorphically to all of $\mathbb{C}$.
\item[$(\mathrm{iii})$] Each $J(s,W,\Phi)$ can be written with a common denominator determined only by the representation $\pi$. Hence the family has $``$bounded denominators$"$.
\end{enumerate}
\end{theorem}

\begin{proof}
The proof of the theorem is essentially identical to that of Jacquet and Shalika \cite{JaSh88} or Belt \cite[Lemma 3.2]{Be14}. We reproduce it here for completeness. Throughout this proof, we denote by $\varrho$ right translation. Applying the Iwasawa decomposition, we obtain
\[
\begin{split}
  &J(s,W,\Phi)=\int_{K_n}  \int_{A_{n}}  \int_{\mathcal{N}_n \backslash \mathcal{M}_n} W \left( \sigma_{2n} \begin{pmatrix} I_n & X \\ & I_n \end{pmatrix} \begin{pmatrix} ak & \\ & ak\end{pmatrix}\right) \delta_{B_n}^{-1}(a) |\mathrm{det}(a)|^s\\
  &\phantom{*********************}\times \psi^{-1}(\mathrm{Tr} X) dX \int_{F^{\times}} \omega_{\pi}(z)|z|^{ns} \Phi(e_nzk) d^{\times}z\; da\; dk,
\end{split}
\]
where $a=m(a_1,a_2,\dotsm,a_{n-1},1)$. Take $K_n^{\circ} \subset K_n$ a compact subgroup which stabilizes $W$ and $\Phi$. Write $K_n=\cup_ik_iK_n^{\circ}$ and let $W_i=\varrho\begin{pmatrix} k_i &\\&k_i  \end{pmatrix}W$ and $\Phi_i=R \begin{pmatrix} k_i & \\ & k_i \end{pmatrix} \Phi$. Then our integral can be decomposed as a finite sum of the form 
\begin{equation}
\label{Iwa-dec}
\tag{2.4}
\begin{split}
  &J(s,W,\Phi)=c \sum_i  \int_{A_{n}}  \int_{\mathcal{N}_n \backslash \mathcal{M}_n}  W_i  \left( \sigma_{2n} \begin{pmatrix} I_n & X \\ & I_n \end{pmatrix}  \begin{pmatrix} a & \\ & a \end{pmatrix} \right)
  \delta_{B_n}^{-1}(a) |\mathrm{det}(a)|^s \\
   &\phantom{*******************}\times \psi^{-1}(\mathrm{Tr} X) dX da \int_{F^{\times}} \omega_{\pi}(z)|z|^{ns} \Phi_i(e_nz) d^{\times}z,
\end{split}
\end{equation}
with $c  > 0$ the volume of $K_n^{\circ}$. Continuing our analysis, we assume that $\displaystyle \mathrm{Re}(s)+\frac{1}{n}\mathrm{Re}(\omega_{\pi}) > 0$ for the convergence of Tate integrals. We first need to remove the unipotent integration appearing in $J(s,W,\Phi)$, employing the following Lemma.

\begin{lemma}
\label{rational dec}
For $0 \leq k \leq n$, $a=m(a_1,a_2,\dotsm,a_{n-1},1) \in A_n$, and $W \in \mathcal{W}(\pi,\psi)$, we define
\[
  \mathfrak{J}_k(W;a)= \int_{\mathcal{N}_k \backslash \mathcal{M}_k} W\left( \sigma_{2n} \begin{pmatrix} I_{k} &&X&\\ &I_{n-k}&& \\ &&I_{k}& \\ &&&I_{n-k} \end{pmatrix}
  \begin{pmatrix} a&\\ &a \end{pmatrix} \right) \psi^{-1}(\mathrm{Tr} X)dX.
\]
Then there are finitely many  vectors $W_i \in \mathcal{W}(\pi,\psi)$ such that
\[
   \mathfrak{J}_k(W;a)=\sum_i \frac{1}{|a_1||a_2|^2\dotsm|a_{k-1}|^{k-1}}  \mathfrak{J}_{k-1}(W_i;a).
\]
\end{lemma}

The proof of this Lemma is postponed to the end. With the notation in Lemma \ref{rational dec}, we can restate the integral of \eqref{Iwa-dec} in the following way:
\[
\label{new-dec}
\tag{2.5}
  J(s,W,\Phi)=c\sum_i \int_{A_n}  \mathfrak{J}_n(W_i;a) \int_{F^{\times}} \omega_{\pi}(z)|z|^{ns} \Phi_i(e_nz) d^{\times}z\; \delta_{B_n}^{-1}(a) |\mathrm{det}(a)|^s da.
\]
Utilizing the descending induction on $k$ until $k=1$ in Lemma \ref{rational dec}, our integral $ \mathfrak{J}_n(W_i;a)$ in \eqref{new-dec} can be written as 
 \[
  \mathfrak{J}_n(W_i;a)=\prod_{k=1}^{n-1}|a_k|^{-k(n-k)} W'_i  \left( \sigma_{2n} \begin{pmatrix} a&\\ &a \end{pmatrix} \right)=\prod_{k=1}^{n-1}|a_k|^{-k(n-k)} \varrho(\sigma_{2n})W'_i   \begin{pmatrix} b \end{pmatrix},
 \]
where we have written $b=\sigma_{2n} (\mathrm{diag}(a,a)) \sigma_{2n}^{-1}=\text{diag}(b_1,b_2,\dotsm,b_{2n})$. Then $b_{2k-1}=b_{2k}=a_ka_{k+1}\dotsm a_{n-1}$ for $k=1,2,\dotsm,n-1$ and $b_{2n-1}=b_{2n}=1$ and $W'_i \in \mathcal{W}(\pi,\psi)$. Inserting this expression into \eqref{new-dec}, we obtain
\[
\begin{split}
 &J(s,W,\Phi)\\
 &=c\sum_i \int_{A_n} \prod_{k=1}^{n-1}|a_k|^{-k(n-k)} \varrho(\sigma_{2n})W'_i   \begin{pmatrix} b \end{pmatrix}  \int_{F^{\times}} \omega_{\pi}(z)|z|^{ns} \Phi_i(e_nz) d^{\times}z\; \delta_{B_n}^{-1}(a) |\mathrm{det}(a)|^s da.
\end{split}
\]
According to Proposition \ref{asymWh}, we can express each Whittaker functions $ \varrho(\sigma_{2n})W'_i $ as a finite sum of the form in \eqref{asymWh-expression}. Hence $J(s,W,\Phi)$ converges absolutely as soon as the integrals of the form
\[
\begin{split}
 &\int_{F^{\times}}  \int_{A_{n}} \prod_{i=1}^{n-1}|a_i|^{-i(n-i)}  \prod_{k=1}^{2n-1} \omega_{\pi_{j_{2n-k}}^{(2n-k)}}\left(\frac{b_k}{b_{k+1}}\right) \left|\frac{b_k}{b_{k+1}}\right|^{\frac{k(n-k)}{2}}  v \left( \frac{b_k}{b_{k+1}} \right)^{m_k}     \Phi_{k}    \left(\frac{b_k}{b_{k+1}}\right) \\
 &\phantom{**************************}  \times  \delta_{B_n}^{-1}(a)|\text{det}(a)|^s  \omega_{\pi}(z)|z|^{ns} \Phi(e_nz) da d^{\times}z
 \end{split}
\]
converges absolutely where the index $j_{2n-k}$ runs over $1 \leq j_{2n-k} \leq r_{2n-k}$ for $1 \leq k \leq 2n-1$ and $\Phi_k$ belongs to $\mathcal{S}(F)$. But the evaluation
\[
 \omega_{\pi_{j_{2n-2k+1}}^{(2n-2k+1)}}\left(\frac{b_{2k-1}}{b_{2k}}\right) \left|\frac{b_{2k-1}}{b_{2k}}\right|^{\frac{(2k-1)(n-2k+1)}{2}}  v \left( \frac{b_{2k-1}}{b_{2k}} \right)^{m_{2k-1}}     \Phi_{2k-1}    \left(\frac{b_{2k-1}}{b_{2k}}\right) 
\]
is $0$ if $m_{2k-1} > 0$ or is $\Phi_{2k-1}(1)$ if $m_{2k-1}=0$, and
\[
\begin{split}
  \prod_{k=1}^{n-1} |a_k|^{-k(n-k)} |a_k|^{\frac{2k(2n-2k)}{2}}\delta_{B_n}^{-1}(a)|\text{det}(a)|^s&=\prod_{k=1}^{n-1} |a_k|^{-2k(n-k)} |a_k|^{k(2n-2k)}|a_k|^{ks} \\
  &=\prod_{k=1}^{n-1}|a_k|^{ks}.
\end{split}
\]
Therefore the integral $J(s,W,\Phi)$ in \eqref{new-dec} becomes a sum of products of integrals of the form
\[
\tag{2.6}
\label{product}
  \left( \int_{F^{\times}} \omega_{\pi}(z)|z|^{ns} \Phi(e_nz) d^{\times}z \right) \times \left( \prod_{k=1}^{n-1} \int_{F^{\times}} \omega_{\pi_{j_{2n-2k}}^{(2n-2k)}}(a_k)v(a_k)^{m_{2k}} \Phi_{2k}(a_k) |a_k|^{ks} d^{\times}a_k \right).
\]
We set a real number $r_{\pi}=\max_{k,j_k}\{-\frac{1}{k}\text{Re}(\omega_{\pi_{j_{2n-2k}}^{(2n-2k)}}) \}$. Each product of Tate integrals in \eqref{product} converges absolutely for $\text{Re}(s) > r_{\pi}$ and hence $J(s,W,\Phi)$ do so for any $W \in \mathcal{W}(\pi,\psi)$ and $\Phi \in \mathcal{S}(F^n)$. Furthermore all factors of Tate integrals in \eqref{product} are a sum of geometric series which converges to $Q(q^{-s})(1-\alpha_J q^{-s})^{-\beta_J}$ by Tate in \cite{Tate}, where $Q(X) \in \mathbb{C}[X]$ and $\alpha_J$, $\beta_J$ rely only on the central characters $\omega_{\pi}$, $\omega_{\pi_{j_{2k}}^{(2k)}}$ and the power of valuation $m_{2k}$. Since $\omega_{\pi}$, $\omega_{\pi_{j_{2k}}^{(2k)}}$ and $m_{2k}$ depends only on the representation $\pi$, the integral $J(s,W,\Phi)$ spans a subspace of $\mathbb{C}(q^{-s})$ having a common denominator as asserted in (ii) and (iii).

\end{proof}

We are left with proving Lemma \ref{rational dec} below. We give a proof based on \cite{Be14}.

\begin{proof}[Proof of Lemma \ref{rational dec}]
To reduce the notation burden, we express $\mathfrak{J}_{k+1}(W;a)$ as a finite sum of $\mathfrak{J}_k(W';a)$ for some $W' \in \mathcal{W}(\pi,\psi)$.
Write $X=\begin{pmatrix}  Z & \\ y & 0 \end{pmatrix}$ in the integral $\mathfrak{J}_{k+1}(W;a)$ for $Z \in \mathcal{N}_k \backslash \mathcal{M}_k$ and $y \in F^k$. Then the $\mathcal{N}_{k+1} \backslash \mathcal{M}_{k+1}$-integration equals to 
\[
\begin{split}
  &\mathfrak{J}_{k+1}(W;a)\\
  &=\int_{\mathcal{N}_k \backslash \mathcal{M}_k} \int_{F^k} W\left( \sigma_{2n} \begin{pmatrix} I_{k} &&&Z&& \\ &1&&y&& \\ &&I_{n-k-1}&&& \\ &&&I_{k}&& \\ &&&&1& \\&&&&&I_{n-k-1} \end{pmatrix}
  \begin{pmatrix} a&\\ &a \end{pmatrix} \right) \psi^{-1}(\mathrm{Tr} X)dy dX.
\end{split}
\]
Consider the following simple matrix multiplication :
\[
 \begin{pmatrix} I_{k} &&&Z&& \\ &1&&y&& \\ &&I&&& \\ &&&I_{k}&& \\ &&&&1& \\&&&&&I \end{pmatrix}\hspace*{-2mm} 
  \begin{pmatrix} a&\\ &a \end{pmatrix}
  \hspace*{-1mm} =\hspace*{-1mm} \begin{pmatrix} I_k&&&Z&& \\ &1&&&& \\ &&I&&& \\ &&&I_k&& \\ &&&&1& \\&&&&&I \end{pmatrix}\hspace*{-2mm} \begin{pmatrix} a&\\ &a \end{pmatrix}\hspace*{-2mm}
  \begin{pmatrix} I_k&&&&& \\ &1&&y'&& \\ &&I&&& \\ &&&I_k&& \\ &&&&1& \\&&&&&I \end{pmatrix},
\]
where $y'=(a_1a_2\dotsm a_ky_1,a_2\dotsm a_ky_2,\dotsm,a_ky_k)$ and $I$ is the identity matrix of size $n-k-1$.
After the change of variables, we obtain the integral $\mathfrak{J}_{k+1}(W;a)$ as the expression :
\[
\begin{split}
&\mathfrak{J}_{k+1}(W;a)=\frac{1}{|a_1||a_2|^2\dotsm|a_k|^k}  \int_{\mathcal{N}_k \backslash \mathcal{M}_k} \int_{F^k}
\\
&\phantom{*****}  W\left( \sigma_{2n} \begin{pmatrix} I_{k} &&Z&\\ &I_{n-k}&& \\ &&I_{k}& \\ &&&I_{n-k}  \end{pmatrix}
  \begin{pmatrix} a&\\ &a \end{pmatrix} \begin{pmatrix} I_{k} &&&&& \\ &1&&y&& \\ &&I_{n-k-1}&&& \\ &&&I_{k}&& \\ &&&&1& \\&&&&&I_{n-k-1} \end{pmatrix} \right) \\
&\phantom{*******************************************}  \psi^{-1}(\mathrm{Tr} X)dy dX.
\end{split}
\]
Since $W$ is smooth, we can choose a compact open subset $U$ of $F^k$ for which
\[
  W \left( g \begin{pmatrix} I_{k} &&&&& \\ &1&&&& \\ &&I_{n-k-1}&&& \\ &&&I_{k}&u& \\ &&&&1& \\&&&&&I_{n-k-1} \end{pmatrix} \right)=W(g)
\]
for all $u \in U$. We choose $\varphi$ to be an appropriate scalar multiple of the characteristic function of $U$ so that
\[
  W(g)=\int_{F^k} W \left( g \begin{pmatrix} I_{k} &&&&& \\ &1&&&& \\ &&I_{n-k-1}&&& \\ &&&I_{k}&u& \\ &&&&1& \\&&&&&I_{n-k-1} \end{pmatrix} \right) \varphi(^tu) du.
\]
We observe that
\[
\tag{2.7}
\label{conjugation}
\begin{split}
&\begin{pmatrix} I_{k} &&Z&\\ &I_{n-k}&& \\ &&I_{k}& \\ &&&I_{n-k}  \end{pmatrix}
  \begin{pmatrix} a&\\ &a \end{pmatrix} 
  \begin{pmatrix} I_{k} &&&&& \\ &1&&y&& \\ &&I&&& \\ &&&I_{k}&& \\ &&&&1& \\&&&&&I \end{pmatrix} 
   \begin{pmatrix} I_{k} &&&&& \\ &1&&&& \\ &&I&&& \\ &&&I_{k}&u& \\ &&&&1& \\&&&&&I \end{pmatrix} \\
   &= \begin{pmatrix} I_{k} &&&&Zu'& \\ &1&&&yu& \\ &&I&&& \\ &&&I_{k}&u'& \\ &&&&1& \\&&&&&I \end{pmatrix}
  \begin{pmatrix} I_{k} &&Z&\\ &I_{n-k}&& \\ &&I_{k}& \\ &&&I_{n-k}  \end{pmatrix} 
  \begin{pmatrix} a&\\ &a \end{pmatrix} \begin{pmatrix} I_k &&&&& \\ &1&&y&& \\ &&I&&& \\ &&&I_k&& \\ &&&&1& \\&&&&&I\end{pmatrix} 
\end{split}
\]
where $u'=(a_1a_2\dotsm a_ku_1,a_2\dotsm a_ku_2,\dotsm, a_ku_k)$ is in $F^k$ and $I$ is the identity matrix of size $n-k-1$.
Since the conjugation under $\sigma_{2n}$ of the first matrix on the right hand side in $\eqref{conjugation}$ under the character of $\psi$ gives us that
\[
  \psi \left( \sigma_{2n} \begin{pmatrix} I_{k} &&&&Zu'& \\ &1&&&yu& \\ &&I_{n-k-1}&&& \\ &&&I_{k}&u'& \\ &&&&1& \\&&&&&I_{n-k-1} \end{pmatrix} \sigma_{2n}^{-1} \right)=\psi(yu),
\]
we obtain the following form of expression
\[
\begin{split}
&\mathfrak{J}_{k+1}(W;a)=\frac{1}{|a_1||a_2|^2\dotsm|a_k|^k} \int_{\mathcal{N}_k \backslash \mathcal{M}_k} \int_{F^k}\\
&\phantom{**}  W\left( \sigma_{2n} \begin{pmatrix} I_{k} &&Z&\\ &I_{n-k}&& \\ &&I_{k}& \\ &&&I_{n-k}  \end{pmatrix}
  \begin{pmatrix} a&\\ &a \end{pmatrix} \begin{pmatrix} I_{k} &&&&& \\ &1&&y&& \\ &&I_{n-k-1}&&& \\ &&&I_{k}&& \\ &&&&1& \\&&&&&I_{n-k-1} \end{pmatrix} \right) \hat{\varphi}(y)
 dy \\
  &\phantom{***********************************************}\psi^{-1}(\mathrm{Tr} X)dX.
\end{split}
\]
Set 
\[
 \rho(\hat{\varphi})W(g)=\int_{F^k} W \left( g\begin{pmatrix} I_{k} &&&&& \\ &1&&y&& \\ &&I_{n-k-1}&&& \\ &&&I_{k}&& \\ &&&&1& \\&&&&&I_{n-k-1} \end{pmatrix} \right) \hat{\varphi}(y)dy.
\]
As $\hat{\varphi}$ is locally constant and of compact support, $\rho(\hat{\varphi})W$ is in fact a finite sum of right translates of $W$ and so $\rho(\hat{\varphi})W$ is in $\mathcal{W}(\pi,\psi)$. Therefore
\[
 \mathfrak{J}_{k+1}(W;a)=\frac{1}{|a_1||a_2|^2\dotsm|a_k|^k}  \mathfrak{J}_{k}(\rho(\hat{\varphi})W;a).
\]

\end{proof}

 We denote by $\Theta$ the character of $S_{2n}$ given by $\Theta \left( \begin{pmatrix} I_n & Z \\ & I_n \end{pmatrix}\begin{pmatrix} g &  \\ & g \end{pmatrix} \right)=\psi(\mathrm{Tr}Z)$. We denote two types of generators of $S_{2n}$  by
\[
d_1 =\begin{pmatrix} d &\\ &d \end{pmatrix} \quad \text{and} \quad d_2 = \begin{pmatrix} I_n & Z\\ & I_n \end{pmatrix}
\]
for $d \in GL_n$ and $Z \in \mathcal{M}_n$. We can check that in the realm of convergence
\[
\tag{2.8}
\label{invariance}
 \begin{split}
 &J(s,\pi(d_1)W, R(d_1)\Phi) \\
 &=\int_{N_n\backslash GL_n} \int_{\mathcal{N}_n \backslash \mathcal{M}_n} W \begin{pmatrix}\sigma_{2n}\begin{pmatrix} I_n&X \\
 &I_n
  \end{pmatrix} \begin{pmatrix} gd &\\&gd \end{pmatrix} \end{pmatrix}
\psi^{-1}(\mathrm{Tr}X)dX \; \Phi(e_n gd) |\mathrm{det}(g)|^s dg\\
&=|\mathrm{det}(d)|^{-s}\\
&\phantom{****}\times \int_{N_n\backslash GL_n}
\int_{\mathcal{N}_n \backslash \mathcal{M}_n} W
\begin{pmatrix}\sigma_{2n}\begin{pmatrix} I_n&X \\
 &I_n
  \end{pmatrix} \begin{pmatrix} g &\\&g \end{pmatrix} \end{pmatrix}
\psi^{-1}(\mathrm{Tr}X)dX \; \Phi(e_n g) |\mathrm{det}(g)|^s dg\\
&=|\mathrm{det}(d_1)|^{-\frac{s}{2}}J(s,W, \Phi)
\end{split}
\]
and
\[
\begin{split}
&J(s,\pi(d_2)W, R(d_2)\Phi) \\
&=\int_{N_n\backslash GL_n} \int_{\mathcal{N}_n \backslash \mathcal{M}_n}W \begin{pmatrix}\sigma_{2n}\begin{pmatrix} I_n&X \\
 &I_n
  \end{pmatrix} \begin{pmatrix} g & gZ\\&g \end{pmatrix} \end{pmatrix} \psi^{-1}(\mathrm{Tr}X)dX \; \Phi(e_n g) |\mathrm{det}(g)|^s dg\\
 &=\int_{N_n\backslash GL_n} \int_{\mathcal{N}_n \backslash \mathcal{M}_n}\hspace*{-1mm}  W\hspace*{-1mm}   \begin{pmatrix}\hspace*{-1mm} \sigma_{2n}\hspace*{-1mm} \begin{pmatrix} I_n&gZg^{-1}+X \\
 &I_n
  \end{pmatrix}\hspace*{-1mm}  \begin{pmatrix} g & \\&g \end{pmatrix} \hspace*{-1mm}  \end{pmatrix} \hspace*{-1mm}  \psi^{-1}(\mathrm{Tr}X)dX\Phi(e_n g) |\mathrm{det}(g)|^s dg \\
   &=\psi(\mathrm{Tr}Z) J(s,W, \Phi).\\
 \end{split}
\]
Let
\[
 \mathcal{J}(\pi)=\left \langle J(s,W,\Phi)|\; W \in \mathcal{W}(\pi,\psi), \Phi \in \mathcal{S}(F^n) \right \rangle
\]
denote the $\mathbb{C}$-linear span of the local integrals $J(s,W,\Phi)$. The quasi invariance of $J(s,W,\Phi)$ under the generators of $S_{2n}$ shows that the linear subspace $\mathcal{J}(\pi)$ of $\mathbb{C}(q^{-s})$ is closed under multiplication by $q^s$ and $q^{-s}$ and thus is in fact a $\mathbb{C}[q^{\pm s}]$-fractional ideal in $\mathbb{C}(q^{-s})$.

\par

To obtain a generator of $\mathcal{J}(\pi)$ having numerator $1$, we show that the fractional ideal $\mathcal{J}(\pi)$ contains $1$.
Let $r > 0$ be a large integer such that $\psi$ is trivial on $\mathfrak{p}^r$.
If $\pi$ is an unitary irreducible generic representation of $GL_{2n}$, Belt explicitly construct an appropriate Whittaker function $W_{\pi} \in \mathcal{W}(\pi,\psi)$ in \cite[Lemma 2.3]{Be14} such that the following integral
\[
  J_{(0)}(s,W_{\pi})=\int_{N_n \backslash P_n} \int_{\mathcal{N}_n \backslash \mathcal{M}_n} W_{\pi} \left(\sigma_{2n} \begin{pmatrix} I_n & X \\ & I_n \end{pmatrix} \begin{pmatrix} p & \\ & p \end{pmatrix} \right)  |\mathrm{det}(p)|^{s-1} \psi^{-1}(\mathrm{Tr}X) dX dp
\]
is $(\mathrm{Vol}(\mathfrak{p}^r))^{n-1}\mathrm{Vol}((N_{n-1}\cap K_{n-1,r}) \backslash K_{n-1,r}))\mathrm{Vol}(\mathcal{N}_{n-1}(\mathfrak{p}^r)\backslash \mathcal{M}_{n-1}(\mathfrak{p}^r)) \neq 0$, which only depends on $\psi$. To extend Lemma 2.3 of Belt to any representations of Whittaker type on $GL_{2n}$, we need to check that the proof of Lemma 2.3 in Section 4.2 of Belt \cite{Be14} can be safely replaced by the representation of Whittaker type. He defines the smooth function on $GL_{2n-1}$ by
\[
  \varphi_{\circ}(g)=\frac{1}{\mathrm{Vol}(N_{2n-1} \cap K_{2n-1,r})} \int_{N_{2n-1}} \mathbbm{1}_{K_{2n-1,r}}(ng) \psi^{-1}(n) dn,
\]
where $\mathbbm{1}_{K_{2n-1,r}}$ is a characteristic function of $K_{2n-1,r}$ and explains that $\varphi_{\circ}$ is an element of $\mathrm{ind}_{N_{2n-1}}^{GL_{2n-1}}(\psi)$. If we look at the proof of this Lemma, there is one point to utilize unitary irreducible generic assumption that the so-called Kirillov model of $\pi$
\[
  \mathcal{W}_{(0)}(\pi,\psi)=\{ W(p) \; | \; W \in \mathcal{W}(\pi,\psi),\; p \in P_{2n} \}
\]
contains the space $\mathrm{ind}_{N_{2n-1}}^{GL_{2n-1}}(\psi)$ proven by Gelfand and Kazhdan in \cite{GeKa72}. He employ this Theorem to assert the existence of  $W_{\pi} \in \mathcal{W}(\pi,\psi)$ such that 
\[
 W_{\pi} \begin{pmatrix} g& \\ & 1 \end{pmatrix}=\varphi_{\circ}(g)
\] 
for all $g \in GL_{2n-1}$. In virtue of Lemma \ref{kirillov}, we can generalize it from unitary irreducible generic representations to representations of Whittaker type. All the other arguments in the proof of Lemma 2.3 in Section 4.2 are valid for any representations of Whittaker type. Therefore we can record this result in the following Lemma.

\begin{lemma}[Belt]
\label{nonvanishing-even}
Let $\pi$ be a representation of Whittaker type of $GL_{2n}$. There exists $W \in \mathcal{W}(\pi,\psi)$ such that $J_{(0)}(s,W)$ is a non-zero constant which only depends on $\psi$.
\end{lemma}

Fix $W \in \mathcal{W}(\pi,\psi)$ such that $J_{(0)}(s,W)$ is a non-zero constant as in the above lemma. Using the partial Iwasawa decomposition to $g \in GL_n$ in Lemma \ref{zpkdec}, we see that
\[
\begin{split}
J(s,W,\Phi)
&=\int_{K_n} \int_{N_n\backslash P_n} \int_{\mathcal{N}_n \backslash \mathcal{M}_n} W \begin{pmatrix} \sigma_{2n} \begin{pmatrix} I_n & X \\ & I_n \end{pmatrix} \begin{pmatrix} pk &  \\ & pk \end{pmatrix} \end{pmatrix} \\ 
  &\phantom{*********}\times |\mathrm{det}(p)|^{s-1} \psi^{-1}(\mathrm{Tr} X) dX\int_{F^{\times}} \omega_{\pi}(z) |z|^{ns} \Phi (e_nzk) d^{\times} z\; dp\; dk.
\end{split}
\]
Let $K_{n,r} \subset K_n$ be a compact open congruent subgroup which stabilize $W$. Now choose $\Phi$ to be the characteristic function of $e_nK_{n,r}$. If $\Phi (e_nzk) \neq 0$, then $z \in \mathcal{O}^{\times}$. As
\[
\label{zp-invariance}
\tag{2.9}
\begin{split}
&\int_{N_n\backslash P_n}\omega_{\pi}(z) W \begin{pmatrix} \sigma_{2n} \begin{pmatrix} I_n & X \\ & I_n \end{pmatrix} \begin{pmatrix} pk &  \\ & pk \end{pmatrix} \end{pmatrix} |z|^{ns}  \Phi (e_nzk) 
|\mathrm{det}(p)|^{s-1} dp \\
  &\phantom{****************}=\int_{N_n\backslash P_n} W \begin{pmatrix} \sigma_{2n} \begin{pmatrix} I_n & X \\ & I_n \end{pmatrix} \begin{pmatrix} p &  \\ & p \end{pmatrix}\end{pmatrix} |\mathrm{det}(p)|^{s-1} dp
  \end{split}
   \]
for all $z \in \mathcal{O}^{\times}$ and $k \in K_n$ such that $zI_nk \in (P_n \cap K_n)K_{n,r}$ and the normalization of $dp$ gives the unit volume on the set $P_n \cap K_n$ from Section 2.2, the integral becomes  $J(s,W,\Phi)=\text{Vol}((P_n \cap K_n)K_{n,r}))J_{(0)}(s,W)=\text{Vol}((P_n \cap K_n) \backslash (P_n \cap K_n)K_{n,r})J_{(0)}(s,W)=\text{Vol}(e_nK_{n,r})J_{(0)}(s,W)$. Putting this together we have the following.

\begin{theorem}
Let $\pi$ be a representation of $GL_{2n}$ of Whittaker type. The family of local integrals $\mathcal{J}(\pi)=\langle J(s,W,\Phi) \rangle$ is a $\mathbb{C}[q^s,q^{-s}]$-ideal of $\mathbb{C}(q^{-s})$ containing the constant $1$.
\end{theorem}

We can now define the local exterior square $L$-function $L(s,\pi,\wedge^2)$ for representation of $GL_{2n}$ of Whittaker type $\pi$. To do this, note that $\mathbb{C}[q^s,q^{-s}]$ is a PID and hence $\mathcal{J}(\pi)$ is a principal fractional ideal. Since this ideal contains $1$, we may choose a unique generator of $\mathcal{J}(\pi)$ of the form $P(q^{-s})^{-1}$ with $P(X) \in \mathbb{C}[X]$ having $P(0)=1$.

\begin{definition*}
Let $\pi$ be a representation of Whittaker type of $GL_{2n}$. We define
\[
  L(s,\pi,\wedge^2)=\frac{1}{P(q^{-s})}
\]
to be the unique normalized generator of $\mathcal{J}(\pi)$.
\end{definition*}

\subsubsection{The odd case $m=2n+1$.}

Let $\pi$ be an irreducible admissible generic unitary representations of $GL_{2n+1}$. We let $\sigma_{2n+1}$ be the permutation matrix associated with 
\[
\sigma_{2n+1} = \begin{pmatrix} 1& 2 & \dotsm & n & | & n+1 & n+2 & \dotsm & 2n & 2n+1 \\ 1& 3 & \dotsm & 2n-1 & | & 2 & 4 & \dotsm & 2n & 2n+1 \end{pmatrix}. 
\]
 In 1990, Jacquet and Shalika \cite{JaSh88} establish an integral representation for the exterior square $L$-function for $GL_{2n+1}$ in the following way:
\[
\begin{split}
&J(s,W)\\
&= \int_{N_n\backslash GL_n} \int_{\mathcal{N}_n\backslash \mathcal{M}_n} W \left(\sigma_{2n+1}\begin{pmatrix} I_n & X &\\ & I_n &\\&&1 \end{pmatrix} \begin{pmatrix} g&&\\&g&\\&&1 \end{pmatrix} \right) \psi^{-1}(\mathrm{Tr} X) dX  |\mathrm{det}(g)|^{s-1} dg,
\end{split}
\]
where $W \in \mathcal{W}(\pi,\psi)$. It was proven in Section 9 of \cite{JaSh88} that there is $r_{\pi}$ in $\mathbb{R}$ such that the above integrals $J(s,W,\Phi)$ converge for $\text{Re}(s) > r_{\pi}$.

\par

Let $\pi=\mathrm{Ind}(\Delta_1 \otimes \dotsm \otimes \Delta_t)$ be a representation of $GL_{2n+1}$ of Whittaker type. In this section, we shall describe the definition of local exterior square $L$-function associated to $\pi$. Applying the Iwasawa decomposition to $g \in GL_n$, we get
\[
\begin{split}
 &J(s,W) 
 =\int_{K_n} \int_{A_n} \int_{\mathcal{N}_n \backslash \mathcal{M}_n} W \left( \sigma_{2n+1} \begin{pmatrix} I_n&X&\\&I_n&\\&&1 \end{pmatrix}
  \begin{pmatrix} ak&&\\&ak&\\&&1\end{pmatrix}  \right)\\
  &\phantom{***************************} \times \delta_{B_n}^{-1}(a) |\text{det}(a)|^{s-1} \psi^{-1}(\text{Tr}X) dX da dk,
\end{split}
\]
where $a=m(a_1,\dotsm,a_n)$. Along the line in the proof of the even case, $J(s,W)$ can be decomposed as a finite sum of integrals of the form 
\[
  \prod_{k=1}^{n}|a_k|^{-k(n-k)} \int_{A_n} W' \left( \sigma_{2n+1} \begin{pmatrix} a&&\\&a&\\&&1\end{pmatrix}  \right) \delta_{B_n}^{-1}(a) |\text{det}(a)|^{s-1} da.
\]
for some $W' \in \mathcal{W}(\pi,\psi)$. After conjugating by $\sigma_{2n+1}$, the previous integral can be written as
\[
 \prod_{k=1}^{n}|a_k|^{-k(n-k)} \int_{A_n} \pi(\sigma_{2n+1})W' (b)  \delta_{B_n}^{-1}(a) |\text{det}(a)|^{s-1} da,
 \]
 where $b=\text{diag}(b_1,b_2,\dotsm,b_{2n},1)$ with $b_{2k-1}=b_{2k}=a_ka_{k+1}\dotsm a_n$ for $k=1,2,\dotsm,n$. Substituting the asymptotic expansion of Whittaker function of Proposition \ref{asymWh} in the previous integral, $J(s,W)$ is a finite sum of integrals of the form
\[
  \int_{A_n}  \prod_{k=1}^{n}|a_k|^{-k(n-k)} \omega_{\pi_{i_{2n+1-2k}}^{(2n+1-2k)}}(a_k) |a_k|^{\frac{2k(2n+1-2k)}{2}} v(a_k)^{m_{2k}}  \Phi_{2k}(a_k)  \delta_{B_n}^{-1}(a)|\text{det}(a)|^{s-1} da,
 \]
 where $ \Phi_{2k}$ lies in $\mathcal{S}(F)$, the index $i_{2n+1-2k}$ runs over $1 \leq i_{2n+1-2k} \leq r_{2n+1-2k}$ for $1 \leq k \leq n$, and $v^{m_{2k}}=1$ for $m_{2k}=0$. But we have
 \[
 \begin{split}
  \prod_{k=1}^n |a_k|^{-k(n-k)} |a_k|^{\frac{2k(2n+1-2k)}{2}}\delta_{B_n}^{-1}(a)|\text{det}(a)|^{s-1}&=\prod_{k=1}^{n} |a_k|^{-2k(n-k)} |a_k|^{k(2n+1-2k)}|a_k|^{k(s-1)}\\
  &=\prod_{k=1}^{n}|a_k|^{ks}.
\end{split}
\]
Finally, our integral $J(s,W)$ becomes a sum of integrals of the form
 \[
\tag{2.10}
\label{product-odd}
   \prod_{k=1}^n \int_{F^{\times}} \omega_{\pi_{i_{2n+1-2k}}^{(2n+1-2k)}}(a_k) v(a_k)^{m_{2k}}   \Phi_{2k}(a_k) |a_k|^{ks} d^{\times}a_k.
\]
 We let a real number $r_{\pi}=\max_{k,i_{k}}\{-\frac{1}{k}\text{Re} (\omega_{\pi_{i_{2n+1-2k}}^{(2n+1-2k)}}) \}$. Each product of Tate integrals in \eqref{product-odd} converges absolutely for $\mathrm{Re}(s) > r_{\pi}$ and hence $J(s,W)$ do so for any $W \in \mathcal{W}(\pi,\psi)$. Furthermore all factors of Tate integrals in \eqref{product-odd} are essentially a sum of geometric series which converges to $Q(q^{-s})(1-\alpha_J q^{-s})^{-\beta_J}$, where $Q(X) \in \mathbb{C}[X]$ and $\alpha_J$, $\beta_J$ rely only on the central characters $\omega_{\pi_{i_{2k-1}}^{(2k-1)}}$ and the power of valuation $m_{2k}$. Since $\omega_{\pi_{i_{2k-1}}^{(2k-1)}}$ and $m_{2k}$ depends only on the representation $\pi$, the integral $J(s,W)$ spans a subspace of $\mathbb{C}(q^{-s})$ having a common denominator. Collecting this information, we arrive at the following theorem.

\begin{theorem}
\label{integral-rational-odd}
Let $\pi=\mathrm{Ind}(\Delta_1 \otimes \dotsm \otimes \Delta_t)$ be a representation of Whittaker type of $GL_{2n+1}$. For every $1 \leq k \leq 2n$, let $(\omega_{\pi_{i_k}^{(k)}})_{i_k=1,\dotsm,r_k}$ be the family of the central characters for all nonzero successive quotients of the form $\pi_{i_k}^{(k)}=\mathrm{Ind}(\Delta_1^{(k_1)} \otimes \dotsm \otimes \Delta_t^{(k_t)})$ with $k=k_1+\dotsm+k_t$ appearing in the composition series of $\pi^{(k)}$. Let $W \in \mathcal{W}(\pi,\psi)$.
\begin{enumerate}
\item[$(\mathrm{i})$] If we have $\mathrm{Re}(s) >-\frac{1}{k}\mathrm{Re}(\omega_{\pi_{i_{2n+1-2k}}^{(2n+1-2k)}})$ for all $1 \leq k \leq n$ and all $1 \leq i_{2k-1} \leq r_{2k-1}$, then each local integral $J(s,W)$ converges absolutely.
\item[$(\mathrm{ii})$] Each $J(s,W) \in \mathbb{C}(q^{-s})$ is a rational function of $q^{-s}$ and hence $J(s,W)$ as a function of $s$ extends meromorphically to all $\mathbb{C}$.
\item[$(\mathrm{iii})$] Each $J(s,W)$ can be written with a common denominator determined by the representation $\pi$. Hence the family has $``$bounded denominators$"$.
\end{enumerate}
\end{theorem}

Let $\mathcal{J}(\pi)=\left \langle J(s,W)\;|\; W \in \mathcal{W}(\pi,\psi) \right \rangle$ denote the $\mathbb{C}$-linear span of the local integrals $J(s,W)$. As in the even case, one can check that
\[
\begin{split}
  &J(s,\pi \left(\begin{pmatrix} I_n& Z& \\ & I_n&\\ && 1 \end{pmatrix} \begin{pmatrix} h&& \\ &h&\\&&1 \end{pmatrix} \begin{pmatrix} I_n& &y \\ & I_n&\\ && 1 \end{pmatrix}\right)W) \\
  &\phantom{******************************}=|\mathrm{det}(h)|^{-(s-1)}\psi(\mathrm{Tr}Z)J(s,W),
\end{split}
\]
for all $Z \in M_n$, $h \in GL_n$ and $y \in \mathcal{M}_{n,1}$ with $\mathcal{M}_{n,1}$ the space of $n \times 1$ matrices. The quasi invariance of $J(s,W)$ under the transformation above shows that the $\mathbb{C}$-vector space $\mathcal{J}(\pi)$ generated by the integral $J(s,W)$ is in fact a $\mathbb{C}[q^{\pm s}]$-fractional ideal in $\mathbb{C}(q^{-s})$. To define the normalized generator of the fractional ideal $\mathcal{J}(\pi)$ formed by the family of the local integral $J(s,W)$, we need the following result.

\begin{lemma}
\label{nonvanishing-odd}
Let $\pi$ be a representation of Whittaker type of $GL_{2n+1}$. There exists $W \in \mathcal{W}(\pi,\psi)$ such that $J(s,W)$ is a non-zero constant which only depends on $\psi$.
\end{lemma}

\begin{proof}
The proof of this lemma is obtained by repeating the proof of \cite[Lemma 2.3]{Be14} of Belt in our setting. The author treats the even case, but the method transfers completely.
We fix a positive integer $r > 0$ sufficiently large such that $\psi$ is trivial on $\mathfrak{p}^r$. We observe that if $X \in \mathcal{M}_n(\mathfrak{p}^r)$, $\psi(\mathrm{Tr}X)=1$. Let $\mathbbm{1}_{K_{n,r}}$ be the characteristic function of $K_{n,r}$. Let $\varphi_{\circ}$ denote the smooth function on $GL_{2n}(F)$ defined by
\[
  \varphi_{\circ}(g)=\frac{1}{\mathrm{Vol}(N_{2n}(F) \cap K_{2n,r})} \int_{N_{2n}(F)} \mathbbm{1}_{K_{2n,r}}(ng) \psi^{-1}(n) dn.
\]
We can evaluate this function explicitly. If $g=n_0k_0$ for $n_0 \in N_{2n}(F)$ and $k_0 \in K_{2n,r}$, then
\[
\begin{split}
 \varphi_{\circ}(g)&=\frac{\psi(n_0)}{\mathrm{Vol}(N_{2n} \cap K_{2n,r})}  \int_{N_{2n}(F)} \mathbbm{1}_{K_{2n,r}}(nk_0) \psi^{-1}(n) dn \\
 &=\frac{\psi(n_0)}{\mathrm{Vol}(N_{2n}(F) \cap K_{2n,r})} \int_{N_{2n}(F)\cap K_{2n,r}} dn=\psi(n_0),
\end{split}
\]
since $\psi$, being trivial on $\mathfrak{p}^r$, is in fact trivial on $N_{2n}(F) \cap K_{2n,r}$. However if $g \notin N_{2n}(F)K_{2n,r}$, then for all $n \in N_{2n}(F)$, $\mathbbm{1}_{K_{2n,r}}(ng)=0$, and thus $\varphi_{\circ}(g)=0$. Therefore we have
\[
 \varphi_{\circ}(g) =
  \begin{cases}
    \psi(n)       & \quad \text{if } g=nk \text{ for } n \in N_{2n}(F),\; k \in K_{2n,r}\\
     0 & \quad \text{otherwise.} \\
  \end{cases}
\]
Then $\varphi_{\circ}$ is in the space $\mathrm{ind}^{GL_{2n}(F)}_{N_{2n}(F)}(\psi)$ and is invariant under right translation by the element $K_{2n,r}$. The result of Gelfand and Kazdan in \cite[ Theorem F]{GeKa72} or Lemma \ref{kirillov} assert that there exists $W_{\circ} \in \mathcal{W}(\pi,\psi)$ such that
\[
  W_{\circ} \begin{pmatrix} g & \\ & 1 \end{pmatrix}=\varphi_{\circ}(g).
\]
for all $g \in GL_{2n}$. Setting $\pi(\sigma_{2n})W_{\circ}=W$, we obtain
\[
\label{J-Phi}
\tag{2.11}
\begin{split}
  &J(s,W) \\
  &=\int_{N_n(F)\backslash GL_n(F)} \int_{\mathcal{N}_n(F)\backslash \mathcal{M}_n(F)} \varphi_{\circ} \left( \sigma_{2n} \begin{pmatrix} g& Xg \\ &g \end{pmatrix} \sigma_{2n}^{-1} \right)
  \psi^{-1}(\mathrm{Tr}X) dX |\mathrm{det}(g)|^{s-1} dg.
\end{split}
\]
Assume that
$\displaystyle
 \varphi_{\circ} \left( \sigma_{2n} \begin{pmatrix} g& Xg \\ &g \end{pmatrix} \sigma_{2n}^{-1} \right) \neq 0
 $. Then 
 $\displaystyle \sigma_{2n} \begin{pmatrix} g& Xg \\ &g \end{pmatrix}\sigma_{2n}^{-1}=nk$ for some $n \in N_{2n}(F)$ and $k \in K_{2n,r}$. Setting $n'=\sigma_{2n}^{-1}n^{-1} \sigma_{2n}$, we have
 \[
   \sigma_{2n} n'  \begin{pmatrix} g& Xg \\ &g \end{pmatrix}\sigma_{2n}^{-1}=k \in K_{2n,r} \quad \mathrm{or} \quad n'  \begin{pmatrix} g& Xg \\ &g \end{pmatrix} \in \sigma_{2n}^{-1}K_{2n,r}\sigma_{2n}=K_{2n,r}.
 \]
 Let $\mathcal{V}_n(F)$ be the subspace of strictly upper triangular matrices or upper triangular matrices with $0$'s along the diagonal of $\mathcal{M}_n(F)$. The matrix $n'$ in fact has the form
 \[
   n'=\sigma_{2n}^{-1}n^{-1} \sigma_{2n}=\begin{pmatrix} n_1 &y \\ t&n_2 \end{pmatrix},
 \]
 where $n_1,n_2 \in N_n(F)$, $t \in \mathcal{V}_n(F)$ and $y \in \mathcal{N}_n(F)$. Using this expression for $n'$, we have
  \begin{equation}
  \label{V-form}
 \tag{2.12}
  n'  \begin{pmatrix} g& Xg \\ &g \end{pmatrix}=\begin{pmatrix} n_1 &y \\ t&n_2 \end{pmatrix} \begin{pmatrix} g& Xg \\ &g \end{pmatrix}=\begin{pmatrix} n_1g&n_1Xg+yg \\ tg& tXg+n_2g \end{pmatrix} \in K_{2n,r}.
 \end{equation}
From the top left corner of the matrix $n_1g \in K_{n,r}$, $g$ belonging to $n_1^{-1}K_{n,r}$ is essentially an element of $N_n(F)K_{n,r}$. Thus the support of $\varphi_{\circ}$ restricts the outer most integral in $J(s,W)$ of \eqref{J-Phi} to be taken over the image $N_n(F) \backslash N_n(F)K_{n,r}$ of $K_{n,r}$ in $N_n(F) \backslash GL_n(F)$
\[
  J(s,W)= \int_{N_n(F) \backslash N_n(F)K_{n,r}} \int_{\mathcal{N}_n(F)\backslash \mathcal{M}_n(F)} \varphi_{\circ} \left( \sigma_{2n} \begin{pmatrix} g& Xg \\ &g \end{pmatrix} \sigma_{2n}^{-1} \right) \psi^{-1}(\mathrm{Tr}X) dX dg.
 \]
We observe that there is a homeomorphism $N_n(F) \backslash N_n(F)K_{n,r} \simeq (N_n(F) \cap K_{n,r})\backslash K_{n,r}$ as $K_{n,r}$-homogeneous spaces given by $N_n(F)k \mapsto  (N_n(F) \cap K_{n,r})k$ in Weil \cite[Chapter 1, Section 3]{We41}. Then the uniqueness Theorem of $K_{n,r}$-right invariant measure on the quotient space $g \in N_n(F) \backslash N_n(F)K_{n,r}$  in Weil \cite[Chapter 2, Section 9]{We41} asserts that
\[
  J(s,W)=\int_{(N_n(F) \cap K_{n,r})\backslash K_{n,r}} \int_{\mathcal{N}_n(F)\backslash \mathcal{M}_n(F)}  \varphi_{\circ} \left( \sigma_{2n} \begin{pmatrix} k& Xk \\ &k \end{pmatrix} \sigma_{2n}^{-1} \right) \psi^{-1}(\mathrm{Tr}X) dX dk.
 \]
Since $\displaystyle  \sigma_{2n} \begin{pmatrix} k&  \\ &k \end{pmatrix}\sigma_{2n}^{-1}$ is in $K_{2n,r}$ for $k \in K_{n,r}$, and $\varphi_{\circ}$ is invariant on the right by the element of $K_{2n,r}$, $J(s,W)$ is the same as
 \[
 \begin{split}
   &J(s,W)\\
   &=\mathrm{Vol}((N_n(F) \cap K_{n,r})\backslash K_{n,r}) \int_{\mathcal{N}_n(F)\backslash \mathcal{M}_n(F)} \varphi_{\circ} \left( \sigma_{2n} \begin{pmatrix} I_n& X \\ &I_n \end{pmatrix} \sigma_{2n}^{-1} \right)
 \psi^{-1}(\mathrm{Tr}X) dX.
 \end{split}
 \]
 Now we evaluate this integral explicitly. As in the matrix form \eqref{V-form},
 \[
   n'  \begin{pmatrix} I_n& X \\ &I_n \end{pmatrix}=\begin{pmatrix} n_1&n_1X+y \\ t& tX+n_2 \end{pmatrix}
 \]
 is in $K_{2n,r}$, where $n_1,n_2 \in N_n(F)$, $t \in \mathcal{V}_n(F)$ and $y \in \mathcal{N}_n(F)$. We consider the matrix on the right hand side. Focusing on the upper right entry,
 we get $n_1X+y \in \mathcal{M}_n(\mathfrak{p}^r)$. However, $n_1$ being the upper left corner, is in $K_{n,r}$, so that $X+n_1^{-1}y$ is in $K_{n,r}\mathcal{M}_n(\mathfrak{p}^r)$, which is again contained in $\mathcal{M}_n(\mathfrak{p}^r)$. Since $n_1^{-1}y$ is upper triangular, the entries below the diagonal in $X$ must be in $\mathfrak{p}^r$. As $X$ is an element of $\mathcal{N}_n(F)\backslash \mathcal{M}_n(F)$ whose representatives are given by strictly lower triangular matrices, $X$ is in the image of $\mathcal{M}_n(\mathfrak{p}^r)$ in $\mathcal{N}_n(F)\backslash \mathcal{M}_n(F)$. Therefore the integral $J(s,W)$ can be taken over the compact space $\mathcal{N}_n(\mathfrak{p}^r)\backslash \mathcal{M}_n(\mathfrak{p}^r)$.
In this case, from our choice of $r$ at the beginning, $\psi(\mathrm{Tr}X)=1$ and $\displaystyle  \sigma_{2n} \begin{pmatrix} I_n & X  \\ &I_n \end{pmatrix}\sigma_{2n}^{-1}$
belongs to $\sigma_{2n}K_{2n,r}\sigma_{2n}^{-1}=K_{2n,r}$. Hence 
\[
\varphi_{\circ} \left( \sigma_{2n} \begin{pmatrix} I_n& X \\ &I_n \end{pmatrix} \sigma_{2n}^{-1} \right)
 \psi^{-1}(\mathrm{Tr}X) \equiv 1,
 \]
which implies that
\[
 J(s,W)
 =\mathrm{Vol}((N_n(F) \cap K_{n,r})\backslash K_{n,r}) \mathrm{Vol}(\mathcal{N}_n(\mathfrak{p}^r)\backslash \mathcal{M}_n(\mathfrak{p}^r)) \neq 0.
\]
$J(s,W)$ is in fact a non-zero constant which only depends on $\psi$.

\end{proof}

Now we can make a choice $W \in \mathcal{W}(\pi,\psi)$ so that the integral $J(s,W)$ is one. Hence the $\mathbb{C}[q^{\pm s}]$-fractional ideal $\mathcal{J}(\pi)$ contains $1$. Putting this together, we have the following.

\begin{theorem}
Let $\pi$ be a representation of $GL_{2n+1}$ of Whittaker type. The family of local integrals $\mathcal{J}(\pi)=\langle J(s,W) \rangle$ is a $\mathbb{C}[q^s,q^{-s}]$-ideal of $\mathbb{C}(q^{-s})$ containing $1$.
\end{theorem}

The ideal $\mathcal{J}(\pi)$ has a unique generator of the form $P(q^{-s})^{-1}$ with $P(X) \in \mathbb{C}[X]$ having $P(0)=1$. We define the local exterior square $L$-function for the representations of Whittaker type on $GL_{2n+1}$ in the following way.

\begin{definition*}
Let $\pi$ be a representation of Whittaker type of $GL_{2n+1}$. We define
\[
  L(s,\pi,\wedge^2)=\frac{1}{P(q^{-s})}
\]
to be the unique normalized generator of $\mathcal{J}(\pi)$.
\end{definition*}

\section{Shalika Functional and Exceptional Poles}

In order to compute $L(s,\pi,\wedge^2)$ we can exploit the following elementary characterization of the polynomial $P(q^{-s})$. In terms of divisibility, $P(q^{-s})$ is the minimal polynomial in $q^{-s}$ such that $P(q^{-s})J(s,W,\Phi)$, in the case $m=2n$, or $P(q^{-s})J(s,W)$, in the case $m=2n+1$, is entire function of $s$ for all $W \in \mathcal{W}(\pi,\psi)$ and $\Phi \in \mathcal{S}(F^n)$ if necessary. Once we normalize so that $P(0)=1$, this characterizes $P(q^{-s})$ and hence $L(s,\pi,\wedge^2)$. With this description of $L(s,\pi,\wedge^2)$ in hand, we call the local Euler factor a function of the form $P(q^{-s})^{-1}$, where $P(X) \in \mathbb{C}[X]$ is polynomial satisfying $P(0)=1$.

\subsection{Exceptional poles of the $L$-function for $GL_{2n}$}
Let $\pi=\mathrm{Ind}(\Delta_1 \otimes \dotsm \otimes \Delta_t)$ be a representation of Whittaker type of $GL_{2n}$. We now begin to analyze the poles of the rational functions in $\mathcal{J}(\pi)$, since these poles and their order will determine $P(q^{-s})=L(s,\pi,\wedge^2)^{-1}$. As this ideal is linearly spanned by the $J(s,W,\Phi)$ it will be sufficient to understand the poles of these integrals. Suppose there is a function in $\mathcal{J}(\pi)$ having a pole of order $d$ at $s=s_0$ and that this is the highest order pole of the family at $s=s_0$. Consider a rational function defined by an individual integral $J(s,W,\Phi)$. Then the Laurent expansion at $s=s_0$ will be of the form 
\[
\tag{3.1} 
\label{Excep-Laurent}
J(s,W,\Phi)=\frac{B_{s_0}(W,\Phi)}{(q^s-q^{s_0})^d}+\mathrm{higher\; order\; terms.}
\]
The coefficient of the leading term, $B_{s_0}(W,\Phi)$, will define a non-trivial bilinear from on $\mathcal{W}(\pi,\psi)\times \mathcal{S}(F^n)$ satisfying $B_{s_0}(\pi(h)W,R(h)\Phi)=|\mathrm{det}(h)|^{-\frac{s_0}{2}}\Theta(h) B_{s_0}(W,\Phi)$ with $h =\begin{pmatrix} I_n &Z \\& I_n \end{pmatrix} \begin{pmatrix} g &\\ & g \end{pmatrix} \in S_{2n}$. The space $\mathcal{S}(F^n)$ has a small filtration $\{0 \} \subset \mathcal{S}_0(F^n) \subset  \mathcal{S}(F^n),$ where $\mathcal{S}_0(F^n)=\{\Phi \in \mathcal{S}(F^n)\; |\; \Phi(0)=0\}$. This filtration is stable under the action $R$.

\begin{definition*} 
 The pole $s=s_0$ of $L(s,\pi,\wedge^2)$ of the family $\mathcal{J}(\pi)$ is called exceptional if the associated bilinear form $B_{s_0}$, the coefficient of the highest order pole of the family $\mathcal{J}(\pi)$, vanished identically on $\mathcal{W}(\pi,\psi) \times \mathcal{S}_0(F^n)$. 
\end{definition*}

If $s_0$ is an exceptional pole of $\mathcal{J}(\pi)$ then the bilinear form $B_{s_0}$ factors to a non-zero bilinear form on $\mathcal{W}(\pi,\psi) \times (\mathcal{S}(F^n)\slash \mathcal{S}_0(F^n))$. The quotient $\mathcal{S}(F^n)\slash \mathcal{S}_0(F^n)$ is isomorphic to $\mathbb{C}$ via the map $\Phi \rightarrow \Phi(0)$.

\begin{definition*}
 A linear functional $\Lambda$ on $\mathcal{W}(\pi,\psi)$ is called a Shalika functional if it satisfies $\Lambda(\pi(h)W)=\psi(\mathrm{Tr}Z)\Lambda(W)$ for all $h=\begin{pmatrix} I_{n} &Z \\& I_{n} \end{pmatrix} \begin{pmatrix} g &\\ & g \end{pmatrix} \in S_{2n}$ and $W \in \mathcal{W}(\pi,\psi)$. We say that a linear functional $\Lambda_s$ on $\mathcal{W}(\pi,\psi)$ is a twisted Shalika functional if $\Lambda_s(\pi(h)W)=|\mathrm{det}(h)|^{-\frac{s}{2}}\psi(\mathrm{Tr}Z)\Lambda_s(W)$ for $h=\begin{pmatrix} I_{n} &Z \\& I_{n} \end{pmatrix} \begin{pmatrix} g &\\ & g \end{pmatrix} \in S_{2n}$ and $W \in \mathcal{W}(\pi,\psi)$.
\end{definition*}

If $s_0$ is an exceptional pole of $\mathcal{J}(\pi)$, then the bilinear form $B_{s_0}$ can be written as $B_{s_0}(W,\Phi)=\Lambda_{s_0}(W)\Phi(0)$ with $\Lambda_{s_0} \in \mathrm{Hom}_{S_{2n}}(\mathcal{W}(\pi,\psi),|\cdot|^{-\frac{s_0}{2}}\Theta)$ a twisted Shalika functional. If the ideal $\mathcal{J}(\pi)$ has an exceptional pole of order $d_{s_0}$ at $s=s_0$ and this is the highest order pole of the family at $s=s_0$, then this pole contributes a factor of $(1-q^{s_0}q^{-s})^{d_{s_0}}$ to $L(s, \pi,\wedge^2)^{-1}$. As these factors $(1-q^{s_0}q^{-s})^{d_{s_0}}$ are relatively prime in $\mathbb{C}[q^{\pm s}]$ for distinct exceptional poles $s=s_0$, we establish the following Definition.

\begin{definition*}
Let $L_{ex}(s,\pi,\wedge^2)^{-1}$ denote the product of these factors $(1-q^{s_0}q^{-s})^{d_{s_0}}$ as $s_0$ runs over the exceptional poles of $\mathcal{J}(\pi)$ with $d_{s_0}$ the maximal order of the pole at $s=s_0$.
\end{definition*}

Applying the partial Iwasawa decomposition to $g \in GL_n$, we may decompose the integral $J(s,W,\Phi)$ as:
\[
\tag{3.2}
\label{pIwasawa}
\begin{split}
J(s,W,\Phi)
&=\int_{K_n} \int_{N_n\backslash P_n} \int_{\mathcal{N}_n \backslash \mathcal{M}_n} W \begin{pmatrix} \sigma_{2n} \begin{pmatrix} I_n & X \\ & I_n \end{pmatrix} \begin{pmatrix} pk &  \\ & pk \end{pmatrix} \end{pmatrix}\\ 
&\phantom{****}\times |\mathrm{det}(p)|^{s-1} \psi^{-1}(\mathrm{Tr} X) dX \int_{F^{\times}} \omega_{\pi}(z) |z|^{ns} \Phi (e_nzk) d^{\times}z dp dk.
\end{split}
\]
Focusing on integration over $N_n \backslash P_n$ and $\mathcal{N}_n \backslash \mathcal{M}_n$, we formally define the following integral.

\begin{definition*}
Let $W$ belong to $\mathcal{W}(\pi,\psi)$. We define the following integral:
\[
  J_{(0)}(s,W)=\int_{N_n \backslash P_n} \int_{\mathcal{N}_n \backslash \mathcal{M}_n} W \left(\sigma_{2n} \begin{pmatrix} I_n & X \\ & I_n \end{pmatrix} \begin{pmatrix} p & \\ & p \end{pmatrix} \right)  |\mathrm{det}(p)|^{s-1} \psi^{-1}(\mathrm{Tr}X) dX dp.
\]
\end{definition*}
 We can identify $GL_{n-1}$ with a subgroup of $GL_n$ via $g \mapsto \mathrm{diag}(g,1)$. Under this embedding, $N_{n-1}$, $K_{n-1}$ and $A_{n-1}$ are considered as a subgroup of $GL_n$. We view integration over $N_n \backslash P_n$ as $N_{n-1} \backslash GL_{n-1}$. Since $GL_{n-1}=N_{n-1}A_{n-1}K_{n-1}$, the integral $J_{(0)}(s,W)$ becomes
 \[
 \begin{split}
   &J_{(0)}(s,W)=\int_{K_{n-1}}  \int_{A_{n-1}}  \int_{\mathcal{N}_n \backslash \mathcal{M}_n} W \left( \sigma_{2n} \begin{pmatrix} I_n & X \\ & I_n \end{pmatrix} \begin{pmatrix} ak & \\ & ak\end{pmatrix}\right) \\
   &\phantom{****************************}\delta_{B_{n-1}}^{-1}(a) |\mathrm{det}(a)|^{s-1}\psi^{-1}(\mathrm{Tr} X) dX da dk,
\end{split}
 \]
with $a=m(a_1,\dotsm,a_{n-1})$. Arguing as in the proof of Theorem \ref{integral-rational}, $J_{(0)}(s,W)$ is convergent for $\mathrm{Re}(s)$ large, and defines a rational function in $\mathbb{C}(q^{-s})$ with a common denominator depending only on the representation. As for the quasi-invariance, note that the integrals $J_{(0)}(s,W)$ naturally satisfy
\[
\label{Pn-invariant}
\tag{3.3}
  J_{(0)} ( s, \pi \left( \begin{pmatrix} I_n& Z \\ & I_n \end{pmatrix} \begin{pmatrix} p& \\ &p \end{pmatrix} \right) W ) =\psi(\mathrm{Tr}Z)|\mathrm{det}(p)|^{-(s-1)} J_{(0)}(s,W)
\] 
for $Z \in \mathcal{M}_n$ and $p \in P_n$. The vector space generated by the functions $J_{(0)}(s,W)$ for $W \in \mathcal{W}(\pi,\psi)$ is a fractional ideal $\mathcal{J}_{(0)}(\pi)$ of $\mathbb{C}[q^{\pm s}]$ which has a unique generator of the form $\frac{1}{P(q^{-s})}$ where $P(X)$ is a polynomial in $\mathbb{C}[X]$ with $P(0)=1$ in virtue of Lemma \ref{nonvanishing-even}.

\begin{definition*}
Let $\mathcal{J}_{(0)}(\pi)$ denote the span of the rational functions defined by the integrals $J_{(0)}(s,W)$ and let $L_{(0)}(s,\pi,\wedge^2)$ be the Euler factor which generates the $\mathbb{C}[q^{\pm s}]$-fractional ideal $\mathcal{J}_{(0)}(\pi)$ in $\mathbb{C}(q^{-s})$.
\end{definition*}

We describe the containment $\mathcal{J}_{(0)}(\pi) \subset \mathcal{J}(\pi)$ of $\mathbb{C}[q^{\pm s}]$-fractional ideals, and explain the relation between the exceptional poles of $\mathcal{J}(\pi)$ and the two stage filtration $\mathcal{J}_{(0)}(\pi) \subset \mathcal{J}(\pi)$. We first introduce the notation from \cite{Ma09,Ma10,Ma15} about $L$-function. We denote
\[
  L^{(0)}(s,\pi,\wedge^2)=\frac{L(s,\pi,\wedge^2)}{L_{(0)}(s,\pi,\wedge^2)}.
\]

\begin{proposition}
\label{L0 function}
Let $\pi$ be a representation of $GL_{2n}$ of Whittaker type. The fractional ideal $\mathcal{J}_{(0)}(\pi)$ is spanned by the integral $J(s,W,\Phi)$ for $W$ in $\mathcal{W}(\pi,\psi)$ and $\Phi \in \mathcal{S}_0(F^n)$. As $\mathbb{C}[q^{\pm s}]$-fractional ideals in $\mathbb{C}(q^{-s})$, we have the inclusion $\mathcal{J}_{(0)}(\pi) \subset \mathcal{J}(\pi)$. The Euler factor $L_{(0)}(s,\pi,\wedge^2)^{-1}$ divides $L(s,\pi,\wedge^2)^{-1}$, and the quotient
\[
 L^{(0)}(s,\pi,\wedge^2)=\frac{L(s,\pi,\wedge^2)}{L_{(0)}(s,\pi,\wedge^2)}
\]
has simple poles, which are exactly the exceptional pole of $L(s,\pi,\wedge^2)$.
\end{proposition}

\begin{proof}
Take $K_n^{\circ} \subset K_n$ a compact open subgroup which stabilizes $\Phi$ and $W$ in the sense that $\varrho(\mathrm{diag}(k,k))W=W$ for all $k \in K_n^{\circ}$, where $\varrho$ denotes right translation. Write $K_n=\cup_ik_iK_n^{\circ}$ and let $W_i=\varrho \begin{pmatrix} k_i & \\ & k_i \end{pmatrix}W$ and $\Phi_i=R\begin{pmatrix} k_i & \\ & k_i \end{pmatrix} \Phi$. Then our integral in \eqref{pIwasawa} can be decomposed as a finite sum of the form
\[
\tag{3.4}
\label{pIwasawa-dec}
\begin{aligned} 
&J(s,W,\Phi)=\lambda_1\sum_i \int_{N_n \backslash P_n} \int_{\mathcal{N}_n \backslash \mathcal{M}_n} W_i \left(\sigma_{2n} \begin{pmatrix} I_n & X \\ & I_n \end{pmatrix} \begin{pmatrix} p & \\ & p \end{pmatrix} \right) \\
&\phantom{*************}\times  |\mathrm{det}(p)|^{s-1} \psi^{-1}(\mathrm{Tr}X) dX dp
 \int_{F^{\times}} \omega_{\pi}(z) |z|^{ns} \Phi_i(e_nz) d^{\times} z 
\end{aligned}
\]
with $\lambda_1 > 0$ the volume of $K_n^{\circ}$. If $\Phi$ is in $\mathcal{S}_0(F^n)$, then each $\Phi_i(0)=0$ so that Tate integrals $\displaystyle  I(s,\omega_{\pi},\Phi_i)=\int_{F^{\times}} \omega_{\pi}(z) |z|^{ns} \Phi_i(e_nz) d^{\times} z$ in \eqref{pIwasawa-dec} becomes a polynomial in $\mathbb{C}[q^{\pm s}]$. From \eqref{pIwasawa-dec}, we are able to see that $J(s,W,\Phi)$ is a finite sum of integrals of the form $P(q^{\pm s})J_{(0)}(s,W)$ with $P(X) \in \mathbb{C}[X]$ if $\Phi \in \mathcal{S}_0(F^n)$. Since $\mathcal{J}_{(0)}(\pi)$ is closed under multiplication by $q^s$ and $q^{-s}$, it belongs to $\mathcal{J}_{(0)}(\pi)$. Conversely, for any integrals $J_{(0)}(s,W)$ let $K_{n,r}$ be a compact open congruent subgroup which stabilizes $W$ in the sense that $\varrho(\mathrm{diag}(k,k))W=W$ for all $k \in K_{n,r}$. If $\Phi \in \mathcal{S}_0(F^n)$ is the characteristic function of $e_n K_{n,r}$, then our integral reduce to $J(s,W,\Phi)=\lambda_2 J_{(0)}(s,W)$ in virtue of \eqref{zp-invariance} and \eqref{pIwasawa} with $\lambda_2 > 0$ the volume of $e_n K_{n,r}$. Therefore $\mathcal{J}_{(0)}(\pi)$ is the vector space spanned by the integrals $J(s,W,\Phi)$ for $W \in \mathcal{W}(\pi,\psi)$ and $\Phi \in \mathcal{S}_0(F^n)$. It implies that we have the inclusion $\mathcal{J}_{(0)}(\pi) \subset \mathcal{J}(\pi)$ as $\mathbb{C}[q^{\pm s}]$-fractional ideals and $L_{(0)}(s,\pi,\wedge)^{-1}$ divides $L(s,\pi,\wedge)^{-1}$ in $\mathbb{C}[q^{\pm s}]$.

\par
We now claim that $L(s,\pi,\wedge^2)$ has an exceptional pole at $s=s_0$ if and only if $L^{(0)}(s,\pi,\wedge^2)$ has a pole at $s=s_0$. If $s=s_0$ is an exceptional pole of maximal order $d_{s_0}$, then the coefficient of the leading term $B_{s_0}$ in \eqref{Excep-Laurent} becomes $0$ for all $\Phi \in \mathcal{S}_0(F^n)$. Thus every rational functions in $\mathcal{J}_{(0)}(\pi)$ spanned by the integrals $J(s,W,\Phi)$ for $W \in \mathcal{W}(\pi,\psi)$ and $\Phi \in \mathcal{S}_0(F^n)$ has the pole of highest order at most $d_{s_0}-1$. As $L_{(0)}(s,\pi,\wedge^2)$ has the pole which is strictly less than $d_{s_0}$ at $s=s_0$,  $L^{(0)}(s,\pi,\wedge^2)$ has a pole at $s=s_0$. For the other implication, if $L^{(0)}(s,\pi,\wedge^2)$ has a pole at $s=s_0$ then the maximal order $d_{s_0}$ of the pole $s_0$ for $L(s,\pi,\wedge^2)$ is strictly greater than the one of the pole $s_0$ for $L_{(0)}(s,\pi,\wedge)$. Then any rational functions in $J_{(0)}(s,W)$ has a pole of order at most $d_{s_0}-1$ at $s=s_0$. Since $\mathcal{J}_{(0)}(\pi)$ is spanned by the integrals $J(s,W,\Phi)$ for $W \in \mathcal{W}(\pi,\psi)$ and $\Phi \in \mathcal{S}_0(F^n)$, the first residual term $B_{s_0}$ corresponding to a pole of highest order $d_{s_0}$ of the Laurent expansion of any function $J(s,W,\Phi)$ with $\Phi \in \mathcal{S}_0(F^n)$ in \eqref{Excep-Laurent} must be zero, and the pole $s_0$ is exceptional. Thus the exceptional poles of $L(s,\pi,\wedge^2)$ are actually those of $\displaystyle \frac{L(s,\pi,\wedge^2)}{L_{(0)}(s,\pi,\wedge^2)}=L^{(0)}(s,\pi,\wedge^2)$.
\par

According to \eqref{pIwasawa-dec}, any integral $J(s,W,\Phi)$ is a finite sum of $J_{(0)}(s,W_i)I(s,\omega_{\pi},\Phi_i)$, and hence $L(s,\pi,\wedge^2)^{-1}=L_{(0)}(s,\pi,\wedge^2)^{-1}L^{(0)}(s,\pi,\wedge^2)^{-1}$ divides the product of $L$-functions $L_{(0)}(s,\pi,\wedge^2)^{-1}L(ns,\omega_{\pi})^{-1}$. Then we have 
\[
L_{(0)}(s,\pi,\wedge^2)^{-1}L(ns,\omega_{\pi})^{-1}=Q(q^{-s})L_{(0)}(s,\pi,\wedge^2)^{-1}L^{(0)}(s,\pi,\wedge^2)^{-1}
\]
 for some polynomial $Q(X) \in \mathbb{C}[X]$. Cancelling out the common factor $L_{(0)}(s,\pi,\wedge^2)^{-1}$, $L^{(0)}(s,\pi,\wedge^2)^{-1}$ divides the factor $L(ns,\omega_{\pi})^{-1}$, it therefore has simple poles, as this $L$-function $L(ns,\omega_{\pi})$ from Tate integrals $I(s,\omega_{\pi},\Phi)$ always has simple poles by Tate \cite{Tate}. This proves the finial assertion.
\end{proof}

We know that $\displaystyle L_{ex}(s,\pi,\wedge^2)^{-1}=\prod_{s_i}(1-q^{s_i-s})^{d_i}$, where the $s_i$'s are the exceptional poles of $L(s,\pi,\wedge^2)$ and the $d_i$'s their orders in $L(s,\pi,\wedge^2)$. Then $L^{(0)}$ can be expressed as $\displaystyle L^{(0)}(s,\pi,\wedge^2)^{-1}=\prod_{s_i}(1-q^{s_i-s})$.

\subsection{The local functional equation for exterior square $L$-functions.} In this section, we review the local functional equation for exterior square $L$-functions from Matringe and Cogdell \cite{Ma14,Cog}.
We denote by $M_{2n}$ the standard Levi of $GL_{2n}$ associated to the partition $(n,n)$ of $2n$. Let $w_{2n}=\sigma_{2n}$ and then let $H_{2n}=w_{2n}M_{2n}w_{2n}^{-1}$. When $m=2n+1$ is odd, let $w_{2n+1}=w_{2n+2}|_{GL_{2n+1}}$ so that 
\[
 w_{2n+1}=\begin{pmatrix} 1& 2 & \dotsm & n+1 & | & n+2 & n+3 & \dotsm & 2n+1 \\ 1& 3 & \dotsm & 2n+1 & | & 2 & 4 & \dotsm & 2n \end{pmatrix}. 
\]
In the odd case, $\sigma_{2n+1} \neq w_{2n+1}$. We let $M_{2n+1}$ denote the standard Levi subgroup of $GL_{2n+1}$ associated to the partition $(n+1,n)$ of $2n+1$ and set $H_{2n+1}=w_{2n+1}M_{2n+1}w_{2n+1}^{-1}$ as in the even case. Note that $H_m$ are compatible in the sense that $H_{m} \cap GL_{m-1}=H_{m-1}$. We denote by $\delta_m$ the character of $H_m$ given by $\displaystyle \delta_m : w_m \begin{pmatrix} g_1 & \\ & g_2 \end{pmatrix}w_m^{-1} \mapsto \frac{|\text{det}(g_1)|}{|\text{det}(g_2)|}$. We define the character of $H_m$ by
\[
\tag{3.5}
\label{chi-relation}
 \chi_m =
  \begin{cases}
        1                  & \quad \text{if } m=2n \text{ is even}\\
       \delta_m       & \quad \text{if } m=2n+1 \text{ is odd}\\
  \end{cases}
   ,\quad \mu_m =
  \begin{cases}
      \delta_m     & \quad \text{if } m=2n \text{ is even}\\
      1                 & \quad \text{if } m=2n+1 \text{ is odd}\\
 \end{cases}
.
\]
Hence $\chi_m|_{H_{m-1}}=\mu_{m-1}$, and $\mu_{m}|_{H_{m-1}}=\chi_{m-1}$. We recall the proposition which follows from \cite[Proposition 4.14]{Ma15} or \cite[Proposition 3.1]{Ma14}.
\begin{proposition}[Matringe]
\label{func-homo}
 Let $\sigma$ be an irreducible representation of $P_{m-1}$ and $\xi$ be a character of $P_m \cap H_m$. If $m=2n$, we have

\[
\mathrm{Hom}_{P_{2n} \cap H_{2n}}(\Phi^+(\sigma),\;\xi) \simeq  \mathrm{Hom}_{P_{2n-1} \cap H_{2n-1}}(\sigma,\;\xi\delta_{2n}^{\frac{1}{2}}).
\]
If $m=2n+1$, we obtain
\[
     \mathrm{Hom}_{P_{2n+1} \cap H_{2n+1}}(\Phi^+(\sigma),\;\xi) \simeq  \mathrm{Hom}_{P_{2n} \cap H_{2n}}(\sigma,\;\xi\delta_{2n+1}^{-\frac{1}{2}}).    \\
\]
\end{proposition}

We emphasize that if $\pi=\mathrm{Ind}(\Delta_1 \otimes \dotsm \otimes \Delta_t)$ is a representation of Whittaker type of $GL_m$, then
$\pi^{\iota}=\text{Ind}(\widetilde{\Delta}_t \otimes \dotsm \otimes \widetilde{\Delta}_1)$ is again of Whittaker type.

\subsubsection{The even case $m=2n$.}

Let $\pi=\mathrm{Ind}(\Delta_1 \otimes \dotsm \otimes \Delta_t)$ be  a representation of $GL_{2n}$ of Whittaker type. Let $\tau_{2n}$ be the matrix $\begin{pmatrix} &I_n \\ I_n& \\ \end{pmatrix}$ and $\varrho$ right translation on the Whittaker model $\mathcal{W}(\pi,\psi)$. As in \eqref{invariance}, the bilinear form $\displaystyle B_{s, \pi, \psi} : (W,\Phi) \mapsto \frac{J(s,W,\Phi)}{L(s,\pi,\wedge^2)}$ on $\mathcal{W}(\pi,\psi) \times \mathcal{S}(F^n)$ satisfies $B_{s, \pi, \psi}(\varrho(h)W, R(h)\Phi)=|h|^{-\frac{s}{2}}\Theta(h) B_{s, \pi, \psi}(W, \Phi)$ for $h \in S_{2n}$, which belongs to the space $\text{Hom}_{S_{2n}}(\mathcal{W}(\pi,\psi)\otimes \mathcal{S}(F^n),|\cdot|^{-\frac{s}{2}}\Theta)$. The local functional equation for irreducible generic representations in the even case is given in the paper of Matringe \cite{Ma14}. Cogdell and Matringe only use the assumption that $\pi=\mathrm{Ind}(\Delta_1 \otimes \dotsm \otimes \Delta_t)$ is the representation of Whittaker type in Remark of Section 3.4 of \cite{Cog} to establish the uniqueness of functionals in the space $\mathrm{Hom}_{S_{2n}}(\pi \otimes \mathcal{S}(F^n),|\cdot|^{-\frac{s}{2}}\Theta)$ for almost all $s$. As the surjection from $\pi$ to its Whittaker model $\mathcal{W}(\pi,\psi)$ induces an injection of $\mathrm{Hom}_{S_{2n}}(\mathcal{W}(\pi,\psi) \otimes \mathcal{S}(F^n),|\cdot|^{-\frac{s}{2}}\Theta)$ to $\mathrm{Hom}_{S_{2n}}(\pi \otimes \mathcal{S}(F^n),|\cdot|^{-\frac{s}{2}}\Theta)$, we obtain the following statement for the $1$-dimensionality result.

\begin{proposition}[Matringe]
Let $\pi=\mathrm{Ind}(\Delta_1 \otimes \dotsm \otimes \Delta_t)$ be  a representation of $GL_{2n}$ of Whittaker type. For all values of $q^{-s}$ except a finite number,
the space $\mathrm{Hom}_{S_{2n}}(\mathcal{W}(\pi,\psi) \otimes \mathcal{S}(F^n),|\cdot|^{-\frac{s}{2}}\Theta)$ is of dimension at most $1$.
\end{proposition}

We present the invariance property of bilinear form corresponding to the dual representation $\pi^{\iota}$ to deduce the local functional equation for exterior square $L$-functions for $GL_{2n}$.

\begin{lemma}
 The bilinear form on $\mathcal{W}(\pi,\psi) \times \mathcal{S}(F^n)$ which is defined by $C_{s, \pi, \psi} : (W, \Phi) \mapsto B_{1-s, \pi^{\iota}, \psi^{-1}}(\varrho(\tau_{2n})\widetilde{W}, \hat{\Phi})$ belongs to the space 
 $\mathrm{Hom}_{S_{2n}}(\mathcal{W}(\pi,\psi) \otimes \mathcal{S}(F^n),|\cdot|^{-\frac{s}{2}}\Theta)$.
\end{lemma}

\begin{proof}
For $h \in GL_{2n}$, we define $h^{\tau}$ by the conjugation of the matrix $\tau h \tau^{-1}$. We set ${^t}{h}{^{-\tau_{2n}}}={^t}{(h^{-1})}{^{\tau_{2n}}}$ for $h \in GL_{2n}$. We can check that the map $s \mapsto {^t}{(s^{-1})}{^{\tau_{2n}}}:={^t}{s}{^{-\tau_{2n}}}$ defines an automorphism of the group $S_{2n}$, because for the two generators of $S_{2n}$, we have
\[
\begin{pmatrix} g &\\ &g \end{pmatrix} \mapsto \begin{pmatrix}  {^t}{g}{^{-1}}& \\ &{^t}{g}{^{-1}} \end{pmatrix} \quad \text{and} \quad \begin{pmatrix} I_n & X\\ & I_n \end{pmatrix} \mapsto \begin{pmatrix}  I_n & -X^{t} \\&I_n  \end{pmatrix}.
\]
Let $\mathcal{F}$ denote the Fourier transform on $\mathcal{S}(F^{n})$. We define two generators of $S_{2n}$ by
\[
s_1 =\begin{pmatrix} h &\\ &h \end{pmatrix} \quad \text{and} \quad s_2 = \begin{pmatrix} I_n & X\\ & I_n \end{pmatrix}.
\]
Then for $\Phi \in \mathcal{S}(F^n)$
\[
\begin{split}
\mathcal{F}(R(s_1) \Phi) &= \int_{F^n} \Phi(uh)\; \psi (u{}^ty) du= |h|^{-1} \int_{F^n} \Phi(u)\; \psi (uh^{-1}{}^ty)du= |h|^{-1} \hat{\Phi}(y\;({^t}{h}{^{-1}})) \\ &= |s_1|^{-\frac{1}{2}}R({^t}{s_1}{^{-\tau_{2n}}})\mathcal{F}(\Phi)
\end{split}
\]
and $\mathcal{F}(R(s_2) \Phi) = \mathcal{F}(\Phi) = |s_2|^{-\frac{1}{2}}R({^t}{s_2}{^{-\tau_{2n}}})\mathcal{F}(\Phi)$ since $|s_2|=1$ and $R({^t}{s_2}{^{-\tau_{2n}}})\mathcal{F}(\Phi)=\mathcal{F}(\Phi)$. For Whittaker model $\mathcal{W}(\pi^{\iota},\psi^{-1})$ of the dual representation $\pi^{\iota}$, we obtain
\[
\begin{split}
   \widetilde{\varrho(h)W}(g)&=(\varrho(h)W)(w\;{^tg^{-1}})=W(w^tg^{-1}h)=W(w^t(g^th^{-1})^{-1})\\
   &=\widetilde{W}(g^th^{-1})=\varrho(^th^{-1})\widetilde{W}(g),
\end{split}
\]
for $g,h \in GL_{2n}$ and $W \in \mathcal{W}(\pi,\psi)$. We define another bilinear form. Now for $h \in S_{2n}$, we can find
\[
 \begin{split}
 C_{s, \pi, \psi}(\varrho(h)W, R(h)\Phi) &= B_{1-s, \pi^{\iota}, \psi^{-1}}(\varrho(\tau_{2n})\widetilde{\varrho(h)W},\mathcal{F}(R(h)\Phi))\\
 &=B_{1-s, \pi^{\iota}, \psi^{-1}}(\varrho(\tau_{2n})\varrho({^t}{h}^{-1})\widetilde{W},|h|^{-\frac{1}{2}}R({^t}{h}^{-\tau_{2n}})\hat{\Phi})\\
 &=|h|^{-\frac{1}{2}}B_{1-s, \pi^{\iota}, \psi^{-1}}(\varrho({^t}{h}^{-\tau_{2n}})\varrho(\tau_{2n})\widetilde{W},R({^t}{h}^{-\tau_{2n}})\hat{\Phi})\\
 &=|h|^{-\frac{1}{2}+\frac{1-s}{2}}\Theta(h)C_{s, \pi, \psi}(W, \phi)=|h|^{-\frac{s}{2}}\Theta(h)C_{s, \pi, \psi}(W, \phi),
 \end{split}
\]
which is also an element of $\text{Hom}_{S_{2n}}(\mathcal{W}(\pi,\psi)\otimes \mathcal{S}(F^n),|\cdot|^{-\frac{s}{2}}\Theta)$. 

\end{proof}

From the quasi-invariance of two bilinear forms $B_{s, \pi, \psi}$ and $C_{s, \pi, \psi}$, we have the following functional equation associated with the Weyl element $\tau_{2n}$. 

\begin{theorem}
\label{local functional eq}
Let $\pi=\mathrm{Ind}(\Delta_1 \otimes \dotsm \otimes \Delta_t)$ be  a representation of $GL_{2n}$ of Whittaker type. There exists an invertible element $\varepsilon(s,\pi,\wedge^2,\psi)$ of $\mathbb{C}[q^{\pm s}]$, such that for every $W$ in $\mathcal{W}(\pi,\psi)$, and every $\Phi$ in $\mathcal{S}(F^n)$, we have the following functional equation
\[
 \varepsilon(s,\pi,\wedge^2,\psi)\frac{J(s,W,\Phi)}{L(s,\pi,\wedge^2)}=\frac{J(1-s,\varrho(\tau_{2n})\widetilde{W},\hat{\Phi})}{L(1-s,\pi^{\iota},\wedge^2)},
 \]
 where $\varrho$ is right translation.
\end{theorem}
We then define the local exterior square $\gamma$-factor.
 \begin{definition*}
 Let $\pi=\mathrm{Ind}(\Delta_1 \otimes \dotsm \otimes \Delta_t)$ be  a representation of $GL_{2n}$ of Whittaker type. We set
 \[
   \gamma(s,\pi,\wedge^2,\psi)=\frac{\varepsilon(s,\pi,\wedge^2,\psi)L(1-s,\pi^{\iota},\wedge^2)}{L(s,\pi,\wedge^2)}.
 \]
\end{definition*}

\par

Let $\pi=\text{Ind}(\Delta_1 \otimes \dotsm \otimes \Delta_t)$ be an irreducible generic representation of $GL_{2n}$. Upon the analysis of proof of \cite[Theorem 4.1]{Ma14}, we will determine the exceptional points $s$ in $\mathbb{C}$ on which the uniqueness of functionals in the space $\text{Hom}_{S_{2n}}(\pi \otimes \mathcal{S}(F^n),|\cdot|^{-\frac{s}{2}}\Theta)$ potentially fails. These points play a role in deformation argument, which we will present in detail in Section $7$. 

 \begin{proposition}
 \label{even-multi one}
Let $\pi=\mathrm{Ind}(\Delta_1 \otimes \dotsm \otimes \Delta_t)$ be an irreducible generic representation of $GL_{2n}$. For every $1 \leq k \leq 2n$, suppose that all the nonzero successive quotients $\pi^{(2n-k)}_i=\mathrm{Ind}(\Delta_1^{(a_1)}\otimes\dotsm \otimes \Delta_t^{(a_t)})$ of the derivatives $\pi^{(2n-k)}$ are generic and irreducible. The space $ \mathrm{Hom}_{S_{2n}}(\pi \otimes \mathcal{S}(F^n),\; |\cdot|^{-\frac{s}{2}}\Theta)$ is of dimension at most $1$, except possibly when there exist $1 \leq k \leq n$ and one of the central characters of the nonzero Jordan--H\"older constituents $\pi_i^{(2n-2k)}$ such that
\[
  \omega_{\pi^{(2n-2k)}_i}(\varpi)=q^{ks}.
  \]
 \end{proposition}

Before we continue with the proof of Proposition \ref{even-multi one}, let us give the following Proposition deduced from Proposition 4.3 and Theorem 4.1 in \cite{Ma14}.

\begin{proposition}[Matringe]
\label{even-embedding}
Let $\pi$ be an irreducible generic representation of $GL_{2n}$. There exists a real number $r_\pi$ depending on $\pi$ such that for all $s_0 > r_{\pi}$ we have a linear injection $\mathrm{Hom}_{P_{2n} \cap S_{2n}}(\pi, \Theta)$ into the space $\mathrm{Hom}_{P_{2n} \cap H_{2n}}(\pi,\delta_{2n}^{-s_0})$.
\end{proposition}

\begin{proof}[Proof of Proposition \ref{even-multi one}]
We imitate the proof of \cite[Lemma 4.5]{Ma15}. Proposition \ref{deriviative} ensures that the subquotients of $\pi^{(2n-k)}$ for $0 \leq k \leq 2n-1$ are the representations $\text{Ind}(\Delta_1^{(a_1)}\otimes\dotsm \otimes \Delta_t^{(a_t)})$ with $\sum_{i=1}^t a_i=2n-k$. We suppose that all the subquotients $\text{Ind}(\Delta_1^{(a_1)}\otimes\dotsm \otimes \Delta_t^{(a_t)})$ of the derivatives $\pi^{(2n-k)}$ for $1 \leq k \leq 2n-1$ are generic and irreducible. The reason of making this additional assumption is so that the central characters of all the nonzero successive quotient $\text{Ind}(\Delta_1^{(a_1)}\otimes\dotsm \otimes \Delta_t^{(a_t)})$ of $\pi^{(2n-k)}$ can be explicitly written in terms of central characters of the inducing data $ \Delta_i^{(a_i)}$. The space $\text{Hom}_{S_{2n}}(\pi,|\cdot|^{-\frac{s}{2}}\Theta)$ is zero except when $\pi$ and $s$ are related by $\omega_{\pi}(\varpi)=q^{ns}$. Except for those values of $q^{-s}$, the space $\text{Hom}_{S_{2n}}(\pi\otimes \mathcal{S}(F^n),|\cdot|^{-\frac{s}{2}}\Theta)$ is equal to $\text{Hom}_{S_{2n}}(\pi\otimes \mathcal{S}_0(F^n),|\cdot|^{-\frac{s}{2}}\Theta)$ because of the following short exact sequence
\[
\begin{split}
  0 \rightarrow \text{Hom}_{S_{2n}}(\pi,|\cdot|^{-\frac{s}{2}}\Theta)& \rightarrow \text{Hom}_{S_{2n}}(\pi\otimes \mathcal{S}(F^n),|\cdot|^{-\frac{s}{2}}\Theta) \\
  &\rightarrow \text{Hom}_{S_{2n}}(\pi\otimes \mathcal{S}_0(F^n),|\cdot|^{-\frac{s}{2}}\Theta) \rightarrow 0.
\end{split}
\]
Following the series of isomorphism in \cite[Theorem 4.1]{Ma14}, we obtain
\[
\begin{split}
\mathrm{Hom}_{S_{2n}}(\pi \otimes \mathcal{S}_{0}(F^{n}), |\cdot|^{-\frac{s}{2}}\Theta) &\simeq \mathrm{Hom}_{S_{2n}}(\pi \otimes \mathrm{ind}_{S_{2n} \cap P_{2n}}^{S_{2n}}({\bf 1}), |\cdot|^{-\frac{s}{2}}\Theta) \\
&\simeq  \mathrm{Hom}_{S_{2n}} (\pi,\mathrm{Ind}^{S_{2n}}_{S_{2n}\cap P_{2n}}(|\cdot|^{-\frac{s-1}{2}}\Theta)) \\
&\simeq \mathrm{Hom}_{P_{2n}\cap S_{2n}}(\pi, |\cdot|^{-\frac{s-1}{2}}\Theta)
\end{split}
\]
the first isomorphism identifying $S_{2n} \cap P_{2n}\backslash S_{2n}$ with $F^n -\{0\}$, the second isomorphism using adjointness, and the last isomorphism by Frobenius reciprocity law. We are correcting the power of  $-\frac{s}{2}+1$ in the third and forth space in \cite[Theorem 4.1]{Ma14}.

\par
Suppose that $s_0$ is a real number such that $s_0 > r_{\pi}$. From the previous Proposition \ref{even-embedding}, the space $\text{Hom}_{P_{2n} \cap S_{2n}}(\pi,\; |\cdot|^{-\frac{s-1}{2}}\Theta)$ embeds as a subspace of the vector space $\text{Hom}_{P_{2n} \cap H_{2n}}(\pi, |\cdot|^{-\frac{s-1}{2}}\delta_{2n}^{-s_0})$. Now the restriction of $\pi$ to $P_{2n}$ has a filtration by derivative with each successive quotient of the form $(\Phi^+)^{2n-k-1}\Psi^+(\pi^{(2n-k)})$ for $0 \leq k \leq 2n-1$. We know from Theorem \ref{deriviative} that each derivatives $\pi^{(2n-k)}$ has a composition series whose subquotients are the representations $\text{Ind}(\Delta_1^{(a_1)}\otimes\dotsm \otimes \Delta_t^{(a_t)})$ with $\sum_{i=1}^t a_i=2n-k$. Since the functors $\Phi^+$ and $\Psi^+$ are exact, we can obtain the filtration of $\pi|_{P_{2n}}$ with the factors
$(\Phi^+)^{2n-k-1}\Psi^+(\text{Ind}(\Delta_1^{(a_1)}\otimes\dotsm \otimes \Delta_t^{(a_t)}))$ for $0 \leq k \leq 2n-1$ and $\sum_{i=1}^t a_i=2n-k$. Then we have
\[
 \begin{split}
 & \text{dim}(\text{Hom}_{P_{2n} \cap H_{2n}}(\pi, |\cdot|^{-\frac{s-1}{2}}\delta_{2n}^{-s_0}))\\
 &\phantom{****} \leq \sum \text{dim}(\text{Hom}_{P_{2n} \cap H_{2n}}((\Phi^+)^{2n-k-1}\Psi^+(\text{Ind}(\Delta_1^{(a_1)}\otimes\dotsm \otimes \Delta_t^{(a_t)}),\;|\cdot|^{-\frac{s-1}{2}} \delta_{2n}^{-s_0} )).
  \end{split}
 \]
 By applying the Proposition \ref{func-homo} repeatedly, we can deduce that
 \[
 \tag{3.6}
 \label{even-func-points}
\begin{split}
 & \text{Hom}_{P_{2n} \cap H_{2n}}((\Phi^+)^{2n-k-1}\Psi^+(\text{Ind}(\Delta_1^{(a_1)}\otimes\dotsm \otimes \Delta_t^{(a_t)}),\;|\cdot|^{-\frac{s-1}{2}} \delta_{2n}^{-s_0}  ) \\
 & \simeq \text{Hom}_{P_{k+1} \cap H_{k+1}}(\Psi^+(\text{Ind}(\Delta_1^{(a_1)}\otimes\dotsm \otimes \Delta_t^{(a_t)}),\;|\cdot|^{-\frac{s-1}{2}} \chi_{k+1}^{\frac{1}{2}} \delta_{k+1}^{-s_0} ) \\
 & \simeq \text{Hom}_{GL_{k} \cap H_{k+1}}(\text{Ind}(\Delta_1^{(a_1)}\otimes\dotsm \otimes \Delta_t^{(a_t)}),\;|\cdot|^{-\frac{s}{2}} \mu_{k}^{\frac{1}{2}} \delta_{k}^{-s_0} ). \\
 \end{split}
\]
The character $\mu_k$ defined in \eqref{chi-relation} guarantees that if the central characters $\prod_i \omega_{\Delta_i^{(a_i)}}(\varpi) $ of the nonzero vector space $\text{Ind}(\Delta_1^{(a_1)}\otimes\dotsm \otimes \Delta_t^{(a_t)})$ for all $\sum_i a_i=2n-k$ and $1 \leq k \leq 2n-1$ satisfy the relations 
$\prod_i \omega_{\Delta_i^{(a_i)}}(\varpi) \neq |\varpi|^{-s_0-\frac{ks}{2}}$ when $k$ is odd and $\prod_i \omega_{\Delta_i^{(a_i)}}(\varpi) \neq |\varpi|^{-\frac{ks}{2}}$ when $k$ is even, the last space in \eqref{even-func-points} is zero. For $k=0$, the dimension of the last space in \eqref{even-func-points} is less than or equal to $1$, because the representation is the trivial representation of the trivial group. As we have a linear injection of the space $\mathrm{Hom}_{P_{2n} \cap S_{2n}}(\pi,|\cdot|^{-\frac{s-1}{2}}\Theta)$ into the space $\text{Hom}_{P_{2n} \cap H_{2n}}(\pi, |\cdot|^{-\frac{s-1}{2}}\delta_{2n}^{-s_0})$ for any real number $s_0$ greater than $r_{\pi}$, the $1$-dimensionality of the space $\mathrm{Hom}_{P_{2n} \cap S_{2n}}(\pi,|\cdot|^{-\frac{s-1}{2}}\Theta)$ potentially fails when a complex number $s$ satisfies one of the following conditions:

\begin{enumerate}
\item[$(\mathrm{1})$] For each $\epsilon > 0$, there exists $t$-tuples $(a_1,\dotsm,a_t)$ of nonzero space $\text{Ind}(\Delta_1^{(a_1)}\otimes\dotsm \otimes \Delta_t^{(a_t)})$ such that 
$
\prod_{i=1}^t \omega_{\Delta_i^{(a_i)}}(\varpi) = |\varpi|^{-r_{\pi}-\epsilon-\frac{ks}{2}}
$
with $1\leq  \sum_{i=1}^t a_i \leq 2n-1$ and $\sum_{i=1}^t a_i$ is an odd number.
\item[$(\mathrm{2})$] There exists $(a_1,\dotsm,a_t)$ of nonzero space $\text{Ind}(\Delta_1^{(a_1)}\otimes\dotsm \otimes \Delta_t^{(a_t)})$ which satisfies $\prod_{i=1}^t \omega_{\Delta_i^{(a_i)}}(\varpi)=|\varpi|^{-\frac{ks}{2}}$ with $1\leq  \sum_{i=1}^t a_i \leq 2n-1$ and $\sum_{i=1}^t a_i$ is an even number.
\end{enumerate}

We show that the first case does not occur. Let $\mathcal{K}$ be the set of all odd integers $k$ such that $1 \leq k \leq 2n-1$ and $\mathcal{T}$ the set of  $t$-tuples $(a_1,\dotsm,a_t)$ of all nonzero constituents $ \text{Ind}(\Delta_1^{(a_1)}\otimes\dotsm \otimes \Delta_t^{(a_t)})$ of $\pi^{(2n-k)}$ such that $\sum_i a_i=2n-k$ and $k \in \mathcal{K}$. Since two sets $\{-\frac{2}{k}\; |\; k \in \mathcal{K} \}$ and
$\{r_{\pi}+\sum_{i=1}^t\mathrm{Re}(\omega_{\Delta_i^{(a_i)}}) \;| \;(a_1,\dotsm,a_t) \in \mathcal{T} \}$ are finite, we can we can choose $\epsilon_1,\epsilon_2 > 0$ such that the intersection
\[
  \bigcap_{j=1,2} \left\{ -\frac{2}{k}\left(\epsilon_j+r_{\pi}+\sum_{i=1}^t\mathrm{Re}(\omega_{\Delta_i^{(a_i)}})\right)\; \middle|\; k \in \mathcal{K},\; (a_1,\dotsm,a_t) \in \mathcal{T}  \right\} 
\]
is empty. Therefore the condition $(1)$ is not valid for any $s \in \mathbb{C}$. The condition $(2)$ ensures that outside the finite number of points $s$ satisfying $\prod_{i=1}^t \omega_{\Delta_i^{(a_i)}}(\varpi)=|\varpi|^{-\frac{ks}{2}}$ with $\sum_{i=1}^t a_i=2n-k$ when $k$ is even for $1 \leq k \leq 2n-1$, the dimension of the space $\mathrm{Hom}_{P_{2n} \cap S_{2n}}(\pi,|\cdot|^{-\frac{s-1}{2}}\Theta)$ is at most one. 
Collecting all information, we complete our proof.


\end{proof}

The intersection of the domain $\mathrm{Re}(s) \neq -\frac{1}{k}\mathrm{Re}(\omega_{\pi_i^{(2n-2k)}})$ for $1 \leq k \leq n$ on which the $1$-dimensionality of the space $ \mathrm{Hom}_{S_{2n}}(\pi \otimes \mathcal{S}(F^n),\; |\cdot|^{-\frac{s}{2}}\Theta)$ holds contains the intersection of domains $\mathrm{Re}(s) > -\frac{1}{k}\mathrm{Re}(\omega_{\pi_i^{(2n-2k)}})$ for $1 \leq k \leq n$, on which $J(s,W,\Phi)$ is given by an absolutely convergence integral from Theorem \ref{integral-rational}.

\subsubsection{The odd case $m=2n+1$.}
Let $\pi=\mathrm{Ind}(\Delta_1 \otimes \dotsm \otimes \Delta_t)$ be a representation of Whittaker type of $GL_{2n+1}$. For $\Phi \in \mathcal{S}(F^n)$ and $W \in \mathcal{W}(\pi,\psi)$, we introduce another families of integrals appearing in the local functional equation for $GL_{2n+1}$ in \cite{Cog}.

\begin{definition*}
Let $W$ belong to $\mathcal{W(\pi,\psi)}$ and $\Phi$ belong to $\mathcal{S}(F^n)$. We define following integrals:
\[
\begin{split}
&J(s,W,\Phi)\\
&= \int_{N_n \backslash GL_n} \int_{\mathcal{N}_n \backslash \mathcal{M}_n} \int_{F^n}\; W \left(\sigma_{2n+1} \begin{pmatrix} I_n & X &\\ & I_n &\\&&1 \end{pmatrix} \begin{pmatrix} g&&\\&g&\\&&1 \end{pmatrix} \begin{pmatrix} I_n&&\\ &I_n&\\ &x&1 \end{pmatrix} \right)\\
&\phantom{******************************}\Phi(x) dx\; \psi^{-1}(\mathrm{Tr} X) dX  |\mathrm{det}(g)|^{s-1} dg.
\end{split}
\]
\end{definition*}

Let $\displaystyle \rho(\Phi)W(g)=\int_{F^n} W\left( g \begin{pmatrix} I_n &&\\ &I_n&\\ &x&1 \end{pmatrix} \right) \Phi(x)\; dx$. 
Since $\Phi$ is locally constant and of compact support, $\rho(\Phi)W$ is a finite sum of right translations of $W$ and so we can write $\rho(\Phi)W(g)$ as a finite sum $\rho(\Phi)W(g)=\sum_i W_i(g)$
for some functions $W_i \in \mathcal{W}(\pi,\psi)$ which are right translates of $W$. The integral $J(s,W,\Phi)$ converges absolutely as soon as the integrals $J(s,W_i)$ will do for all $i$, because $J(s,W,\Phi)=\sum_i J(s,W_i)$. But according to Theorem  \ref{integral-rational-odd} this will be the case if the integral $J(s,W,\Phi)$ satisfies the condition in the following Lemma.

\begin{lemma}
\label{convergence-odd}
Let $\pi=\mathrm{Ind}(\Delta_1 \otimes \dotsm \otimes \Delta_t)$ be a representation of Whittaker type of $GL_{2n+1}$. For every $1 \leq k \leq 2n$, let $(\omega_{\pi_{i_k}^{(k)}})_{i_k=1,\dotsm,r_k}$ be the family of the central characters for all nonzero successive quotients of the form $\pi_{i_k}^{(k)}=\mathrm{Ind}(\Delta_1^{(k_1)} \otimes \dotsm \otimes \Delta_t^{(k_t)})$ with $k=k_1+\dotsm+k_t$ appearing in the composition series of $\pi^{(k)}$. Let $W$ belong to $\mathcal{W}(\pi,\psi)$ and $\Phi$ belong to $\mathcal{S}(F^n)$.
If for all $1 \leq k \leq n$ and all $1 \leq i_{2k-1} \leq r_{2k-1}$, we have $\mathrm{Re}(s) >-\frac{1}{k}\mathrm{Re}(\omega_{\pi_{i_{2n+1-2k}}^{(2n+1-2k)}})$, then each local integral $J(s,W,\Phi)$ converges absolutely.
\end{lemma}

Let $\mathcal{M}_{a,b}$ denote the space of $a \times b$ matrices with entries in $F$. We define the odd Shalika subgroup $S_{2n+1}$ of $GL_{2n+1}$ :
\[
S_{2n+1}=\left\{ \begin{pmatrix} g & Z &y\\ & g & \\ &x&1 \end{pmatrix}  \middle| \; Z \in \mathcal{M}_n, \; g \in GL_n,\; x \in \mathcal{M}_{1,n}, \; y \in \mathcal{M}_{n,1} \right\}.
\]
We denote by $\Theta$ the character of $P_{2n+1} \cap S_{2n+1}$  given by
\[
   \Theta \begin{pmatrix} \begin{pmatrix} I_n & Z & \\ &I_n & \\ && 1 \end{pmatrix} \begin{pmatrix} g&& \\ &g& \\ &&1 \end{pmatrix} \begin{pmatrix} I_n&&y \\ &I_n& \\ &&1 \end{pmatrix} \end{pmatrix} = \psi(\mathrm{Tr}Z).
\]
We define an action of the Shalika subgroup $S_{2n+1}$ on $\mathcal{S}(F^n)$ by
\begin{align*}
&\bullet R \begin{pmatrix} g&& \\ &g& \\ &&1 \end{pmatrix} \Phi(x)= \Phi(xg) \quad & \bullet\; R \begin{pmatrix} I_n&Z& \\ &I_n& \\ &&1 \end{pmatrix} \Phi(x)&= \psi^{-1}(\mathrm{Tr}Z) \Phi(x)  \\
&\bullet R \begin{pmatrix} I_n&&y_0 \\ &I_n& \\ &&1 \end{pmatrix} \Phi(x)= \psi(xy_0)\Phi(x)  & \bullet\; R \begin{pmatrix} I_n&& \\ &I_n& \\ &x_0&1 \end{pmatrix} \Phi(x)&= \Phi(x+x_0)  \\
\end{align*}
for $\Phi \in \mathcal{S}(F^n)$. We recall that $F^n$ is realized as row vectors, so the multiplication $xy_0$ above makes sense. 
We denote the four types of generators of $S_{2n+1}$ by 
\[
d_1 =\begin{pmatrix} h &&\\ &h&\\ &&1 \end{pmatrix}, d_2 = \begin{pmatrix} I_n & Z& \\ & I_n&\\&&1 \end{pmatrix}, d_3 = \begin{pmatrix} I_n & &y_0 \\ & I_n&\\&&1 \end{pmatrix}, \mathrm{and}\;  d_4 = \begin{pmatrix} I_n & & \\ & I_n&\\&x_0&1 \end{pmatrix}.
\]
for $h \in GL_n$. We write the formula for $\mathrm{Re}(s)$ large to guarantee convergence
\[
\tag{3.7}
\label{invariance-odd}
 \begin{split}
 &J(s,\pi(d_1)W, R(d_1)\Phi)\\
 &= \int_{N_n \backslash GL_n} \int_{\mathcal{N}_n \backslash \mathcal{M}_n} \int_{F^n}\; W \left(\sigma_{2n+1} \begin{pmatrix} I_n & X &\\ & I_n &\\&&1 \end{pmatrix} \begin{pmatrix} gh&&\\&gh&\\&&1 \end{pmatrix} \begin{pmatrix} I_n&&\\ &I_n&\\ &xh&1 \end{pmatrix} \right)\\
&\phantom{**************************}\times \Phi(xh) dx\; \psi^{-1}(\mathrm{Tr} X) dX  |\mathrm{det}(g)|^{s-1} dg\\
&= |\mathrm{det}(d_1)|^{-\frac{s}{2}}J(s,W, \Phi),
\end{split}
\]
\[
 \begin{aligned}
 &J(s,\pi(d_2)W, R(d_2)\Phi)\\
 &=\hspace*{-2mm}\int_{N_n \backslash GL_n} \int_{\mathcal{N}_n \backslash \mathcal{M}_n} \int_{F^n}  W\hspace*{-1mm} \left(\sigma_{2n+1} \begin{pmatrix} I_n & gZg^{-1}+X &\\ & I_n &\\&&1 \end{pmatrix} \begin{pmatrix} g&&\\&g&\\&&1 \end{pmatrix} \begin{pmatrix} I_n&&\\ &I_n&\\ &x&1 \end{pmatrix}\hspace*{-1mm} \right)&\\
&\phantom{************************} \times \psi^{-1}(\mathrm{Tr} Z) \Phi(x) dx\; \psi^{-1}(\mathrm{Tr} X) dX  |\mathrm{det}(g)|^{s-1} dg\\
&=J(s,W, \Phi),\\
\end{aligned}
\]
\[
 \begin{aligned}
 &J(s,\pi(d_3)W, R(d_3)\Phi)  \\
 &=\hspace*{-1mm} \int_{N_n \backslash GL_n} \int_{\mathcal{N}_n \backslash \mathcal{M}_n} \int_{F^n} W\hspace*{-1mm} \left(\sigma_{2n+1}\hspace*{-1mm} \begin{pmatrix} I_n & X-gy_0xg^{-1} &gy\\ & I_n &\\&&1 \end{pmatrix}\hspace*{-1mm} \begin{pmatrix} g&&\\&g&\\&&1 \end{pmatrix}\hspace*{-1mm} \begin{pmatrix} I_n&&\\ &I_n&\\ &x&1 \end{pmatrix}\hspace*{-1mm} \right)\\
&\phantom{**************************}\times \psi(xy_0)\Phi(x) dx\; \psi^{-1}(\mathrm{Tr} X) dX  |\mathrm{det}(g)|^{s-1} dg\\
&=\psi^{-1}(\mathrm{Tr}\; gy_0xg^{-1}) \psi(xy_0)J(s,W, \Phi)\\
&=J(s,W, \Phi),\\
\end{aligned}
\]
and
\[
 \begin{split}
 &J(s,\pi(d_4)W, R(d_4)\Phi)\\
 &= \int_{N_n \backslash GL_n} \int_{\mathcal{N}_n \backslash \mathcal{M}_n} \int_{F^n}\; W \left(\sigma_{2n+1} \begin{pmatrix} I_n & X &\\ & I_n &\\&&1 \end{pmatrix} \begin{pmatrix} g&&\\&g&\\&&1 \end{pmatrix} \begin{pmatrix} I_n&&\\ &I_n&\\ &x+x_0&1 \end{pmatrix} \right)\\
&\phantom{**************************}\times \Phi(x+x_0) dx\; \psi^{-1}(\mathrm{Tr} X) dX  |\mathrm{det}(g)|^{s-1} dg\\
&=J(s,W, \Phi),
\end{split}
\]
After $J(s,W,\Phi)$ extends by meromorphic continuation into a function in $\mathbb{C}(q^{-s})$, $(W,\Phi) \mapsto J(s,W,\Phi)$ regarded as a bilinear form on $\mathcal{W}(\pi,\psi) \times \mathcal{S}(F^n)$ is in fact an element of the vector space $\mathrm{Hom}_{S_{2n+1}}(\mathcal{W}(\pi,\psi) \otimes \mathcal{S}(F^{n}), |\cdot|^{-\frac{s_0}{2}}\Theta)$. Cogdell and Matringe in \cite{Cog} showed that two families of integrals $J(s,W)$ and $J(s,W,\Phi)$ span a same fractional ideal $\mathcal{J}(\pi)$ of $\mathbb{C}[q^{\pm s}]$, generated by an Euler factor $L(s,\pi,\wedge^2)$. We summarize this in the following statement.

\begin{proposition}
\label{odd-Lfunction}
Let $\mathcal{J}(\pi)$ be the span of rational functions defined by the integrals $J(s,W)$. Then the family of integrals $J(s,W,\Phi)$ span the same $\mathbb{C}[q^{\pm s}]$-fractional ideal $\mathcal{J}(\pi)$ in $\mathbb{C}(q^{-s})$, which is generated by the Euler factor $L(s,\pi,\wedge^2)$.  
\end{proposition}

Let $\tau_{2n+1}$ be the Weyl element $\begin{pmatrix} &I_n&\\ I_n&&\\ &&1 \end{pmatrix}$. Cogdell and Matringe essentially establish in Proposition 3.5 of \cite{Cog} that if $\pi=\mathrm{Ind}(\Delta_1 \otimes \dotsm \otimes \Delta_t)$ is representation of $GL_{2n+1}$ of Whittaker type, then the space $\mathrm{Hom}_{S_{2n+1}}(\pi\otimes \mathcal{S}(F^n),|\cdot|^{-\frac{s}{2}})$ is of dimension at most $1$ for almost all $s$. As in the even case, the surjection from $\pi$ to its Whittaker model $\mathcal{W}(\pi,\psi)$ induces an injection of the space $\mathrm{Hom}_{S_{2n+1}}(\mathcal{W}(\pi,\psi) \otimes \mathcal{S}(F^n),|\cdot|^{-\frac{s}{2}})$ to the space $\mathrm{Hom}_{S_{2n+1}}(\pi \otimes \mathcal{S}(F^n),|\cdot|^{-\frac{s}{2}})$. The bilinear form on $\mathcal{W}(\pi,\psi) \times \mathcal{S}(F^n)$ can be uniquely inflated to the bilinear form on $\pi \times \mathcal{S}(F^n)$. From the $1$-dimensionality of the space $\mathrm{Hom}_{S_{2n+1}}(\mathcal{W}(\pi,\psi) \otimes \mathcal{S}(F^n),|\cdot|^{-\frac{s}{2}})$, we have the following local functional equations in the odd case for Whittaker type.

\begin{theorem}[Cogdell and Matringe]
\label{local-func-odd}
Let $\pi=\mathrm{Ind}(\Delta_1 \otimes \dotsm \otimes \Delta_t)$ is representation of $GL_{2n+1}$ of Whittaker type. There exists an invertible element $\varepsilon(s,\pi,\wedge^2,\psi)$ of $\mathbb{C}[q^{\pm s}]$, such that for every $W$ in $\mathcal{W}(\pi,\psi)$, and every $\Phi$ in $\mathcal{S}(F^n)$, we have the following functional equation
\[
 \varepsilon(s,\pi,\wedge^2,\psi)\frac{J(s,W,\Phi)}{L(s,\pi,\wedge^2)}=\frac{J(1-s,\varrho(\tau_{2n+1})\widetilde{W},\hat{\Phi})}{L(1-s,\pi^{\iota},\wedge^2)},
\]
where $\varrho$ is right translation.
\end{theorem}
We then define the local exterior square $\gamma$-factor.
 \begin{definition*}
 Let $\pi=\mathrm{Ind}(\Delta_1 \otimes \dotsm \otimes \Delta_t)$ be  a representation of $GL_{2n+1}$ of Whittaker type. We set
 \[
   \gamma(s,\pi,\wedge^2,\psi)=\frac{\varepsilon(s,\pi,\wedge^2,\psi)L(1-s,\pi^{\iota},\wedge^2)}{L(s,\pi,\wedge^2)}.
 \]
\end{definition*}

\par
Let $\pi=\text{Ind}(\Delta_1 \otimes \dotsm \otimes \Delta_t)$ be an irreducible generic representation of $GL_{2n+1}$. Upon the analysis of proof of \cite[Theorem 3.1]{Cog}, we are going to investigate the exceptional points in $\mathbb{C}$ on which the space of quasi-invariant functionals in $\text{Hom}_{S_{2n+1}}(\pi \otimes \mathcal{S}(F^n),|\cdot|^{-\frac{s}{2}}\Theta)$ potentially fails. This will be a consequence of the technique in the proof of \cite[Theorem 4.5]{Ma15}. 

\begin{proposition}
\label{odd-multi one}
Let $\pi=\mathrm{Ind}(\Delta_1 \otimes \dotsm \otimes \Delta_t)$ be an irreducible generic representation of $GL_{2n+1}$. For every $1 \leq k \leq 2n+1$, suppose that all the nonzero successive quotients $\pi^{(2n+1-k)}_i=\mathrm{Ind}(\Delta_1^{(a_1)}\otimes\dotsm \otimes \Delta_t^{(a_t)})$ of the derivatives $\pi^{(2n+1-k)}$ are generic and irreducible. The space $\mathrm{Hom}_{S_{2n+1}}(\pi \otimes \mathcal{S}(F^n),|\cdot|^{-\frac{s}{2}}\Theta)$ is of dimension at most $1$, except possibly when there exist $1 \leq k \leq n$ and one of the central characters of the nonzero Jordan--H\"older constituents $\pi^{(2n+1-2k)}_i$ such that
\[
  \omega_{\pi^{(2n+1-2k)}_i}(\varpi)=q^{ks}.
\]
\end{proposition}

Before we start with the proof of Proposition \ref{odd-multi one}, we introduce the result analogue to Proposition \ref{even-embedding} deduced from Proposition 3.4 and 3.5 in \cite{Cog}.
\begin{proposition}[Cogdell and Matringe]
\label{odd-embedding}
Let $\pi$ be an irreducible generic representation of $GL_{2n+1}$. There exists a real number $r_\pi$ depending only on $\pi$ such that for all $s_0 > r_{\pi}$ the vector space $\mathrm{Hom}_{P_{2n+1} \cap S_{2n+1}}(\pi, \Theta)$ embeds as a subspace of $\mathrm{Hom}_{P_{2n+1} \cap H_{2n+1}}(\pi,\delta_{2n+1}^{-s_0})$.
\end{proposition}

\begin{proof}[Proof of Proposition \ref{odd-multi one}]
We being with the isomorphism
 \[
  \text{Hom}_{S_{2n+1}}(\pi \otimes \mathcal{S}(F^n),|\cdot|^{-\frac{s}{2}}\Theta) \simeq  \text{Hom}_{P_{2n+1} \cap S_{2n+1}}(\pi,\;|\cdot|^{-\frac{s-1}{2}}\Theta)
\]
explained in the proof of \cite[Theorem 4.5]{Cog}. Suppose that $s_0$ is a real number such that $s_0 > r_{\pi}$. From the previous Proposition \ref{odd-embedding}, the vector space $\text{Hom}_{P_{2n+1} \cap S_{2n+1}}(\pi,\; |\cdot|^{-\frac{s-1}{2}}\Theta)$ embeds as a subspace of $\text{Hom}_{P_{2n+1} \cap H_{2n+1}}(\pi, |\cdot|^{-\frac{s-1}{2}}\delta_{2n+1}^{-s_0})$. Now the restriction of $\pi$ to $P_{2n+1}$ has a filtration by derivative with each successive quotient of the form $(\Phi^+)^{2n-k}\Psi^+(\pi^{(2n+1-k)})$ for $0 \leq k \leq 2n$. We know from Theorem \ref{deriviative} that each derivatives $\pi^{(2n+1-k)}$ has a composition series whose subquotients are the representations $\text{Ind}(\Delta_1^{(a_1)}\otimes\dotsm \otimes \Delta_t^{(a_t)})$ with $\sum_{i=1}^t a_i=2n+1-k$. Since the functors $\Phi^+$ and $\Psi^+$ are exact, we can obtain the filtration of $\pi|_{P_{2n+1}}$ with the factors
$(\Phi^+)^{2n-k}\Psi^+(\text{Ind}(\Delta_1^{(a_1)}\otimes\dotsm \otimes \Delta_t^{(a_t)}))$ for $0 \leq k \leq 2n$ and $\sum_{i=1}^t a_i=2n+1-k$. Then we have
 \[
 \begin{split}
 & \text{dim}(\text{Hom}_{P_{2n+1} \cap H_{2n+1}}(\pi, |\cdot|^{-\frac{s-1}{2}}\delta_{2n+1}^{-s_0})) \\
& \leq \sum \text{dim}(\text{Hom}_{P_{2n+1} \cap H_{2n+1}}((\Phi^+)^{2n-k}\Psi^+(\text{Ind}(\Delta_1^{(a_1)}\otimes\dotsm \otimes \Delta_t^{(a_t)}),\;|\cdot|^{-\frac{s-1}{2}} \delta_{2n+1}^{-s_0})).
\end{split}
 \]
By applying the proposition \ref{func-homo} repeatedly, we can deduce that
\[
\tag{3.8}
 \label{odd-func-points}
\begin{split}
 & \text{Hom}_{P_{2n+1} \cap H_{2n+1}}((\Phi^+)^{2n-k}\Psi^+(\text{Ind}(\Delta_1^{(a_1)}\otimes\dotsm \otimes \Delta_t^{(a_t)}),\;|\cdot|^{-\frac{s-1}{2}}\delta_{2n+1}^{-s_0}) \\
 & \simeq \text{Hom}_{P_{k+1} \cap H_{k+1}}(\Psi^+(\text{Ind}(\Delta_1^{(a_1)}\otimes\dotsm \otimes \Delta_t^{(a_t)}),\;\mu_{k+1}^{-\frac{1}{2}} |\cdot|^{-\frac{s-1}{2}}\delta_{k+1}^{-s_0} ) \\
 & \simeq \text{Hom}_{GL_k \cap H_{k+1}}(\text{Ind}(\Delta_1^{(a_1)}\otimes\dotsm \otimes \Delta_t^{(a_t)}),\;\chi_{k}^{-\frac{1}{2}}|\cdot|^{-\frac{s}{2}} \delta_{k}^{-s_0} ). \\
 \end{split}
\]
The character $\chi_k$ defined in \eqref{chi-relation} guarantees that if the central characters $\prod_i \omega_{\Delta_i^{(a_i)}}(\varpi)$ of the nonzero vector space $\text{Ind}(\Delta_1^{(a_1)}\otimes\dotsm \otimes \Delta_t^{(a_t)})$ for all $\sum_i a_i=2n+1-k$ and $1 \leq k \leq 2n$ satisfy the relations 
$\prod_i \omega_{\Delta_i^{(a_i)}}(\varpi) \neq |\varpi|^{-s_0-\frac{1}{2}-\frac{ks}{2}}$ when $k$ is odd and $\prod_i \omega_{\Delta_i^{(a_i)}}(\varpi) \neq |\varpi|^{-\frac{ks}{2}}$ when $k$ is even, the last space in \eqref{odd-func-points} is zero. For $k=0$, the dimension of the last space in \eqref{odd-func-points} is less than or equal to $1$, because the representation is the trivial representation of the trivial group. As we have a linear injection of the space $\mathrm{Hom}_{P_{2n+1} \cap S_{2n+1}}(\pi,|\cdot|^{-\frac{s-1}{2}}\Theta)$ into the space $\text{Hom}_{P_{2n+1} \cap H_{2n+1}}(\pi, |\cdot|^{-\frac{s-1}{2}}\delta_{2n+1}^{-s_0})$ for any real number $s_0$ greater than $r_{\pi}$, the $1$-dimensionality of the space $\mathrm{Hom}_{P_{2n+1} \cap S_{2n+1}}(\pi,|\cdot|^{-\frac{s-1}{2}}\Theta)$ potentially fails when a complex number $s$ satisfies one of the following conditions:

\begin{enumerate}
\item[$(\mathrm{1})$] For each $\epsilon > 0$, there exists $t$-tuples $(a_1,\dotsm,a_t)$ of nonzero space $\text{Ind}(\Delta_1^{(a_1)}\otimes\dotsm \otimes \Delta_t^{(a_t)})$ such that $\prod_{i=1}^t \omega_{\Delta_i^{(a_i)}}(\varpi) = |\varpi|^{-r_{\pi}-\epsilon-\frac{1}{2}-\frac{ks}{2}}0$ with $1\leq  \sum_{i=1}^t a_i \leq 2n$ and $\sum_{i=1}^t a_i$ is an even number.
\item[$(\mathrm{2})$] There exists $(a_1,\dotsm,a_t)$ of nonzero space $\text{Ind}(\Delta_1^{(a_1)}\otimes\dotsm \otimes \Delta_t^{(a_t)})$ which satisfies $\prod_{i=1}^t \omega_{\Delta_i^{(a_i)}}(\varpi) = |\varpi|^{-\frac{ks}{2}}$ with $1\leq  \sum_{i=1}^t a_i \leq 2n$ and $\sum_{i=1}^t a_i$ is an odd number.
\end{enumerate}

We show that the first case does not occur. Let $\mathcal{K}$ be the set of all odd integers such that $1 \leq k \leq 2n$ and $\mathcal{T}$ the set of all $t$-tuples $(a_1,\dotsm,a_t)$ of all nonzero constituents $ \text{Ind}(\Delta_1^{(a_1)}\otimes\dotsm \otimes \Delta_t^{(a_t)})$ of $\pi^{(2n+1-k)}$ such that $\sum_i a_i=2n+1-k$ and $k \in \mathcal{K}$. Since two sets $\{-\frac{2}{k}\;|\; k \in \mathcal{K} \}$ and $\{ r_{\pi}+\frac{1}{2}+\sum_{i=1}^t\text{Re}(\omega_{\Delta_i^{(a_i)}}) \;|\; (a_1,\dotsm,a_t) \in \mathcal{T} \}$ are finite, we can choose $\epsilon_1,\epsilon_2 > 0$ such that the intersection
\[
  \bigcap_{j=1,2} \left\{ -\frac{2}{k}\left(\epsilon_j+r_{\pi}+\frac{1}{2}+\sum_{i=1}^t\mathrm{Re}(\omega_{\Delta_i^{(a_i)}})\right)\; \middle|\; k \in \mathcal{K},\;(a_1,\dotsm,a_t) \in \mathcal{T} \right\} 
\]
is empty. Therefore the condition $(1)$ is not valid for any $s \in \mathbb{C}$. The condition $(2)$ ensures that outside the finite number of points $s$ satisfying $\prod_{i=1}^t \omega_{\Delta_i^{(a_i)}}(\varpi)=|\varpi|^{-\frac{ks}{2}}$ with $\sum_{i=1}^t a_i=2n+1-k$ when $k$ is even for $1 \leq k \leq 2n$, the dimension of the space $\mathrm{Hom}_{P_{2n+1} \cap S_{2+1n}}(\pi,|\cdot|^{-\frac{s-1}{2}}\Theta)$ is at most one. Collecting all information, we completes the proof.

\end{proof}

The intersection of the domain $\mathrm{Re}(s)\neq -\frac{1}{k}\mathrm{Re}(\omega_{\pi_i^{(2n+1-2k)}})$ for $1 \leq k \leq n$ on which the $1$-dimensionality of the space $\text{Hom}_{S_{2n+1}}(\pi \otimes \mathcal{S}(F^n),|\cdot|^{-\frac{s}{2}}\Theta)$ holds contains the intersection of domains $\mathrm{Re}(s) > -\frac{1}{k}\mathrm{Re}(\omega_{\pi_i^{(2n+1-2k)}})$ for $1 \leq k \leq n$, on which $J(s,W)$ is given by an absolutely convergence integral from Theorem \ref{integral-rational-odd}.

\section{Factorization Formula of Exterior Square $L$-Functions for $GL_{2n}$}

In this Section we introduce the factorization result of exterior square $L$-function in terms of exceptional poles of even derivatives of irreducible generic representations of $GL_{2n}$, which will be used to compute $L(s,\pi,\wedge^2)$ for irreducible generic representations $\pi$. We derive a filtrations of $\mathbb{C}[q^{\pm s}]$-fractions ideals defined by various families of integrals. It turns out that this is convenient to analyze the location of the poles of the family $\mathcal{J}_{(0)}(\pi)$ and eventually associate these poles with exceptional poles of $L$-functions for even derivatives.

\par

We shall review certain aspects of the local theory of the exterior square $L$-function on $GL_2$ and complete the computation for $GL_2$ in terms exceptional poles. The case of $GL_2$ is simply the local theory of Tate's integral for central characters. Let $\pi$ be a representation of  $GL_2$. of Langlands type. For $W \in \mathcal{W}(\pi,\psi)$ and $\Phi \in \mathcal{S}(F^2)$, our integral becomes
 \[
 \begin{split}
 &J(s,W,\Phi) \\
 &= \int_{N_1\backslash GL_1} \int_{\mathcal{N}_1\backslash \mathcal{M}_1} W \left(\sigma_{2}\begin{pmatrix} 1 & X \\ & 1  \end{pmatrix} \begin{pmatrix} g&\\&g \end{pmatrix} \right) \psi^{-1}(\mathrm{Tr} X) dX \; \Phi(e_1 g) |\mathrm{det}(g)|^s \; dg.
\end{split}
 \]
Since $N_1$ is trivial, $\sigma_2=I_2$ and $\mathcal{N}_1=\mathcal{M}_1$, our integral simply degenerates into
 \[
   J(s,W,\Phi) =\int_{F^{\times}} W \begin{pmatrix} z& \\ &z \end{pmatrix} \Phi(z) |z|^s \; d^{\times}z .
 \]
Pulling off the central character, we get
\[
J(s,W,\Phi) = W(I_2) \int_{F^{\times}} \omega_{\pi}(z) \Phi(z) |z|^s d^{\times} z= W(I_2) I(s,\omega_{\pi}, \Phi),\; 
\]
where $\displaystyle I(s,\omega_{\pi}, \Phi)=\int_{F^{\times}} \omega_{\pi}(z) \Phi(z) |z|^s d^{\times} z$ is a Tate integral. From the Tate thesis \cite{Tate}, we have $L(s,\pi,\wedge^2)=L(s,\omega_{\pi})$. Since $\Phi \in \mathcal{S}(F)$, there exists integers $n_1>n_2$ such that $\Phi(z)=0$ for $v(z)<n_2$ and $\Phi(z)=\Phi(0)$ for $v(z)\geq n_1$. Then we obtain
\[
I(s,\omega_{\pi}, \Phi)=\int_{n_2 \leq v(z) \leq n_1}  \omega_{\pi}(z) \Phi(z) |z|^s d^{\times} z+\int_{v(z) \geq n_1}  \omega_{\pi}(z) \Phi(z) |z|^s d^{\times} z.
\]
For the first part, the integration is actually a finite sum, and hence it is a polynomial in $\mathbb{C}[q^{\pm s}]$. For the second part we get
\[
\int_{v(z) \geq n_1}  \omega_{\pi}(z) \Phi(z) |z|^s d^{\times} z=\Phi(0)\sum_{n \geq n_1} q^{-ns} \omega_{\pi}(\varpi)^n\int_{\mathcal{O}^{\times}} \omega_{\pi}(z) d^{\times}z
\]
for $\text{Re}(s)$ large enough. We note that the poles of $L(s,\omega_{\pi})$ are always exceptional because the geometric series above becomes zero for all $\Phi \in \mathcal{S}_0(F)$. We call $\omega_{\pi}$ unramified if $\omega_{\pi}(\mathcal{O}^{\times})\equiv1$. Exploiting this notation, we have the following result \cite[Proposition 23.2]{BuHe06}

\begin{proposition}
 \label{basic}
Let $\pi$ be a representation of $GL_2$ of Langlands type. Then the associated exterior square $L$-function of $GL_2$ is given by
\[
L(s,\pi,\wedge^2)=L(s,\omega_{\pi})=
    \begin{cases}
    1      & \quad \text{if $\omega_{\pi}$ is ramified},\\
    \frac{1}{1-\omega_{\pi}(\varpi)q^{-s}} & \quad \text{if $\omega_{\pi}$ is unramified.} \\
  \end{cases}
\]
Moreover we have $L(s,\pi,\wedge^2)=L_{ex}(s,\pi,\wedge^2)$.
\end{proposition}

Let us denote $W_F'$ be the Weil-Deligne group of $F$. Harris-Taylor \cite{HaTa01} and then Henniart \cite{He02} prove the local Langlands correspondence which states that there is natural bijection from the set of equivalence classes of the Frobenius semi-simple complex representations $\phi : W_F' \rightarrow GL_n(\mathbb{C})$ to the set of isomorphism classes of irreducible admissible complex representations of $GL_n$. Suppose that $\pi(\phi)$ is the corresponding irreducible admissible representation to $\phi$. Let $\wedge^2 : GL_n(\mathbb{C}) \rightarrow GL_N(\mathbb{C})$ with $N=\frac{n(n-1)}{2}$ be the exterior square representation. We let $L(s,\wedge^2(\phi))$ be the exterior square Artin's $L$-function. For $n=2$ we know that $\wedge^2 \phi=\mathrm{det}(\phi)$ and $L(s,\pi(\phi),\wedge^2)=L(s,\omega_{\pi(\phi)})$. Since $\mathrm{det}(\phi)$ and $\omega_{\pi(\phi)}$ correspond under the local Langlands correspondence, we obtain the agreement for the local arithmetic and analytic exterior square $L$-functions, that is, $L(s,\wedge^2(\phi))=L(s,\pi(\phi),\wedge^2)$.

\subsection{Families of integrals and $\mathbb{C}[q^{\pm s}]$-fractional ideals} 

Assume that $n > 1$. Throughout Section $4.1$ to Section $4.3$, $\pi$ is an induced representations of $GL_{2n}$ of Langlands type. The reason of making this assumption is so that we can utilize Whittaker models for $\pi_{(k-1)}$ and $\Phi^+(\pi_{(k)})$ in Proposition \ref{mira-model} and \ref{mira-model2}. We introduce different types of Jacquet-Shalika integrals associated to Whittaker model for $\pi_{(k-1)}$ in Proposition \ref{mira-model} to utilize the representation of smaller mirobolic subgroup $P_{2n-k+1}$ 
for $1 \leq k \leq 2n-1$.

\begin{definition*}
For $n > 1$ and $W \in \mathcal{W}(\pi,\psi)$, we define the following integrals:
 \begin{enumerate}
\item[$(\mathrm{i})$] $J_{(2m-1)}(s,W)$
\[
\begin{split}
&=\int_{N_{n-m} \backslash {GL}_{n-m}}  \int_{\mathcal{N}_{n-m} \backslash \mathcal{M}_{n-m}} W \begin{pmatrix}  \sigma_{2n-2m+1}  \begin{pmatrix} I_{n-m} & X & \\ & I_{n-m} &\\ &&1  \end{pmatrix} \begin{pmatrix}  g&&  \\ &g& \\ &&1 \end{pmatrix} &\hspace*{-5mm}  \\&\hspace*{-5mm} I_{2m-1} \end{pmatrix} \\
&\phantom{*************************************} \psi^{-1}(\mathrm{Tr} X) dX\; |\mathrm{det}(g)|^{s-2m} dg
\end{split}
\]
\end{enumerate}
with $0 < m < n$ and
 \begin{enumerate}
\item[$(\mathrm{ii})$] $ J_{(2m)}(s,W)$
\[
\begin{split}
&= \int_{N_{n-m} \backslash P_{n-m}}  \int_{\mathcal{N}_{n-m} \backslash \mathcal{M}_{n-m}} W \begin{pmatrix}  \sigma_{2n-2m}  \begin{pmatrix} I_{n-m} & X \\ & I_{n-m}  \end{pmatrix} \begin{pmatrix}  p&  \\ &p \end{pmatrix} &\hspace*{-3mm} \\&\hspace*{-3mm} I_{2m} \end{pmatrix} \\
 &\phantom{************************************}  \psi^{-1}(\mathrm{Tr} X) dX\; |\mathrm{det}(p)|^{s-2m-1} dp
 \end{split}
\]
\end{enumerate}
with $0 \leq m < n$. Let us denote by $\mathcal{J}_{(k)}(\pi)$ the span of rational functions defined by the integrals $J_{(k)}(s,W)$ for $0 \leq k \leq 2n-2$.

\end{definition*}

We discuss the absolute convergence of integrals $J_{(2m-1)}(s,W)$ and $J_{(2m)}(s,W)$ for $\mathrm{Re}(s)$ large. We view the integration over $N_{n-m} \backslash P_{n-m}$ as $N_{n-m-1} \backslash GL_{n-m-1}$ in $J_{(2m)}(s,W)$. We identify $GL_{n-m-1}$ with a subgroup of $GL_{n-m}$ via $g \mapsto \mathrm{diag}(g,1)$. Then $N_{n-m-1}$, $K_{n-m-1}$ and $A_{n-m-1}$ are embedded in $GL_{n-m}$ through $GL_{n-m-1}$.  Since $GL_{n-m-1}=N_{n-m-1}A_{n-m-1}K_{n-m-1}$, the integral $J_{(2m)}(s,W)$ becomes
\[
\begin{split}
&J_{(2m)}(s,W)=\int_{K_{n-m-1}}  \int_{A_{n-m-1}}  \int_{\mathcal{N}_{n-m} \backslash \mathcal{M}_{n-m}} \hspace*{-1mm} W\hspace*{-1mm}  \begin{pmatrix}\hspace*{-1mm}   \sigma_{2n-2m}  \begin{pmatrix} I_{n-m} & X \\ & I_{n-m}  \end{pmatrix} \hspace*{-1mm}  \begin{pmatrix}  ak&  \\ &ak \end{pmatrix} &\hspace*{-5mm}  \\&\hspace*{-5mm}  I_{2m} \end{pmatrix} \\
&\phantom{***************************}\delta_{B_{n-m-1}}^{-1}(a) |\mathrm{det}(a)|^{s-2m-1}\psi^{-1}(\mathrm{Tr} X) dX da dk
\end{split}
\]
with $a=m(a_1, \dotsm, a_{n-m-1})$. In a similar manner to the previous integral, we decompose $GL_{n-m}=N_{n-m}A_{n-m}K_{n-m}$ and obtain
\[
\begin{split}
&J_{(2m-1)}(s,W)\\
&=\int_{K_{n-m}}  \int_{A_{n-m}}   \int_{\mathcal{N}_{n-m} \backslash \mathcal{M}_{n-m}} W \begin{pmatrix}  \sigma_{2n-2m+1}  \begin{pmatrix} I_{n-m} & X & \\ & I_{n-m} &\\ &&1  \end{pmatrix} \begin{pmatrix}  ak&&  \\ &ak& \\ &&1 \end{pmatrix} &\hspace*{-5mm} \\&\hspace*{-5mm} I_{2m-1} \end{pmatrix} \\
&\phantom{******************************} \delta_{B_{n-m}}^{-1}(a) |\mathrm{det}(a)|^{s-2m}\psi^{-1}(\mathrm{Tr} X) dX da dk
\end{split}
\]
with $a=m(a_1,\dotsm,a_{n-m})$. The same argument in the proof of Theorem \ref{integral-rational} shows that there is a real number $r_{\pi}$, depending only on the representation $\pi$, such that for all $W$ and $s$ with $\mathrm{Re}(s) > r_{\pi}$, the integrals $J_{(2m-1)}(s,W)$ and $J_{(2m)}(s,W)$ are convergent absolutely and define rational functions in $\mathbb{C}(q^{-s})$ with common denominators relying on the representation. In the domain of absolute convergence, $J_{(2m-1)}(s,W)$ and $J_{(2m)}(s,W)$ may be regarded as a linear form on $\mathcal{W}(\pi,\psi)$ satisfying
\[
  J_{(2m-1)}(s,\pi ( \begin{pmatrix} r& \\ &I_{2m-1} \end{pmatrix} )W )=|\mathrm{det}(r)|^{-\frac{s_0-2m}{2}} \Theta(r) J_{(2m-1)}(s,W)
\]
for all $r \in P_{2n-2m+1} \cap S_{2n-2m+1}$ and
\[
  J_{(2m)}(s,\pi ( \begin{pmatrix} r& \\ &I_{2m} \end{pmatrix} )W )=|\mathrm{det}(r)|^{-\frac{s_0-2m-1}{2}} \Theta(r) J_{(2m)}(s,W).
\]
for all $r \in P_{2n-2m} \cap S_{2n-2m}$. As linear spaces $\mathcal{J}_{(2m-1)}(\pi)$ and $\mathcal{J}_{(2m)}(\pi)$ of $\mathbb{C}(q^{-s})$ are closed under multiplication by $q^s$ and $q^{-s}$, $\mathcal{J}_{(2m-1)}(\pi)$ and $\mathcal{J}_{(2m)}(\pi)$ are in fact $\mathbb{C}[q^{\pm s}]$-fractional ideals.

\begin{lemma}
Assume that $n > 1$. Let $\pi$ be an induced representations of $GL_{2n}$ of Langlands type. For $0 \leq k \leq 2n-2$, $J_{(k)}(s,W)$ converge absolutely for all $W \in \mathcal{W}(\pi,\psi)$ when $\mathrm{Re}(s)$ is sufficiently large, and define rational functions in $\mathbb{C}(q^{-s})$. Moreover $\mathcal{J}_{(k)}(\pi)$ are $\mathbb{C}[q^{\pm s}]$-fractional ideals in $\mathbb{C}(q^{-s})$. 
\end{lemma}

We will reduce later our integral $J_{(2m)}(s,W)$ as a form of integrals below.

\begin{definition*}
For $n > 1$ and $0 < m < n$, we define the following integrals:
\[
\begin{split}
  &J_{(2m-1)}(s,W,\Phi)=\int_{N_{n-m} \backslash GL_{n-m}}  \int_{\mathcal{N}_{n-m} \backslash \mathcal{M}_{n-m}} \int_{F^{n-m}}\\ 
  &\phantom{*****} W \begin{pmatrix}  \sigma_{2n-2m+1}  \begin{pmatrix} I_{n-m} & X & \\ & I_{n-m} &\\ &&1  \end{pmatrix} \begin{pmatrix}  g&&  \\ &g& \\ &&1 \end{pmatrix}\begin{pmatrix}  I_{n-m} &&  \\ &I_{n-m} & \\ &y&1 \end{pmatrix} & \\& I_{2m-1} \end{pmatrix} \\
&\phantom{*******************************} \Phi(y)dy \psi^{-1}(\mathrm{Tr} X) dX\; |\mathrm{det}(g)|^{s-2m} dg
\end{split}
\]
with $W \in \mathcal{W}(\pi,\psi)$ and $\Phi \in \mathcal{S}(F^{n-m})$. 
\end{definition*}

The following Lemma shows that the functions $(W,\Phi) \mapsto J_{(2m-1)}(s,W,\Phi)$ span the fractional ideal $\mathcal{J}_{(2m-1)}(\pi)$ in $\mathbb{C}(q^{-s})$.

\begin{lemma}
\label{J2m-same}
 Assume that $n > 1$. Let $\pi$ be an induced representations of $GL_{2n}$ of Langlands type. As a $\mathbb{C}[q^{\pm s}]$-fractional ideal, the family of integrals $J_{(2m-1)}(s,W)$ and $J_{(2m-1)}(s,W,\Phi)$ span the sam ideal $\mathcal{J}_{(2m-1)}(\pi)$ in $\mathbb{C}(q^{-s})$ for $0 < m \leq n$. Moreover the fractional ideal $\mathcal{J}_{(2m-1)}(\pi)$ is spanned by the integral $J_{(2m-1)}(s,W)$ for $W$ in $\mathcal{W}(\pi_{(2m-1)},\psi)$.
\end{lemma}

\begin{proof}
For $\Phi$ a characteristic function of a small neighborhood of zero, $J_{(2m-1)}(s,W)$ is equal to a positive multiple of $J_{(2m-1)}(s,W,\Phi)$ by the smoothness of $\pi$. Hence $\langle J_{(2m-1)}(s,W,\Phi) \rangle \supset \mathcal{J}_{(2m-1)}(\pi)$. For $\Phi \in  \mathcal{S}(F^{n-m})$, we define $\rho(\Phi)W$ by
\[
 \rho(\Phi)W(g)=\int_{F^{n-m}} W \left(g \begin{pmatrix} \begin{pmatrix} I_{n-m} &&\\ &I_{n-m}&\\&y&1  \end{pmatrix} &\\ & I_{2m-1}  \end{pmatrix} \right) \Phi(y) dy.
\]
with $W \in \mathcal{W}(\pi,\psi)$. Then $\rho(\Phi)W$ becomes a finite sum of right translations of $W$ and so we can write $\rho(\Phi)W$ as $\rho(\Phi)W(g)=\sum_i W_i(g)$ for $g \in GL_{2n}$ and some $W_i \in \mathcal{W}(\pi,\psi)$ which are right translations of $W$. As the integrals $J_{(2m-1)}(s,W_i)$ converge absolutely for all $W_i \in \mathcal{W}(\pi,\psi)$ when $\mathrm{Re}(s)$ is sufficiently large, $J_{(2m-1)}(s,W,\Phi)$ will do so, because $J(s,W,\Phi)=\sum_i J(s,W_i)$ in the realm of absolute convergence. Moreover we have $\langle J_{(2m-1)}(s,W,\Phi) \rangle \subset\mathcal{J}_{(2m-1)}(\pi)$.  Analogously to quasi-invariance property of $J(s,W,\Phi)$ in \eqref{invariance-odd}, $J_{(2m-1)}(s,W,\Phi)$ defined by meromorphic continuation is a bilinear form on $\mathcal{W}(\pi,\psi) \times \mathcal{S}(F^{n-m})$ satisfying
\[
 J_{(2m-1)}(s,\pi ( \begin{pmatrix} r& \\ &I_{2m-1} \end{pmatrix} )W, R(r)\Phi )=|\mathrm{det}(r)|^{-\frac{s_0-2m}{2}} \Theta(r) J_{(2m-1)}(s,W,\Phi),
\]
for all $r \in P_{2n-2m+1} \cap S_{2n-2m+1}$, where the action $R$ of Shalika subgroup $S_{2n-2m+1}$ on $\mathcal{S}(F^{n-m})$ is defined in Section $3.3$. Furthermore according to Proposition \ref{mira-model}, 
since $W$ in $J_{(2m-1)}(s,W)$ enter into these integrals through their restriction to $GL_{2n-2m}$, they depend on the image of $W$ in $\mathcal{W}(\pi_{(2m-1)},\psi)$. This concludes the proof.
\end{proof}

\subsection{An equality of $\mathbb{C}[q^{\pm s}]$-fractional ideals in $\mathbb{C}(q^{-s})$.}

Assume that $n > 1$. Let $\pi$ be an induced representations of $GL_{2n}$ of Langlands type. Our goal in this Section is to explain the equality  $\mathcal{J}_{(2m)}(\pi)=\mathcal{J}_{(2m+1)}(\pi)$ of fractional ideals in $\mathbb{C}[q^{\pm s}]$.

\begin{proposition}
\label{even-equal}
We assume that $n > 1$. Let $\pi$ be an induced representations of $GL_{2n}$ of Langlands type. As $\mathbb{C}[q^{\pm s}]$-fractional ideals in $\mathbb{C}(q^{-s})$, we have the equalities $\mathcal{J}_{(2m+1)}(\pi)=\mathcal{J}_{(2m)}(\pi) $ for $0 \leq m < n-1$.
\end{proposition}

\begin{proof}
We begin with showing the inclusion $\mathcal{J}_{(2m)}(\pi) \subset \langle J_{(2m+1)}(s,W,\Phi) | W \in \mathcal{W}(\pi,\psi), \Phi \in \mathcal{S}(F^{n-m-1}) \rangle$. Throughout this proof, let $I$ denote the identity matrix of size $n-m-1$. Following Section 7 of \cite{JaSh88}, the integral 
over $N_{n-m} \backslash P_{n-m}$ in $J_{(2m)}(s,W)$ can be reduced to an integral over $N_{n-m-1} \backslash GL_{n-m-1}$ and the integral over $\mathcal{N}_{n-m} \backslash \mathcal{M}_{n-m}$ in $J_{(2m)}(s,W)$ can be regarded as an integral over $\mathcal{N}_{n-m-1} \backslash \mathcal{M}_{n-m-1} \times F^{n-m-1}$.  For $Z \in \mathcal{N}_{n-m} \backslash \mathcal{M}_{n-m}$, we write $Z=\begin{pmatrix} X&\\y&0 \end{pmatrix}$ with $X \in  \mathcal{N}_{n-m-1} \backslash \mathcal{M}_{n-m-1}$ and $y \in F^{n-m-1}$. Then our integral $J_{(2m)}(s,W)$ can be written as
\[
\begin{split}
J_{(2m)}(s,W)&=\int_{N_{n-m-1} \backslash GL_{n-m-1}} \int_{\mathcal{N}_{n-m-1} \backslash \mathcal{M}_{n-m-1}} \int_{F^{n-m-1}}\\
&\phantom{*********} \times W \begin{pmatrix} \sigma_{2n-2m} \begin{pmatrix} I &&X&\\&1&y&\\&&I&\\&&&1 \end{pmatrix} \begin{pmatrix} g&&&\\ &1&&\\&&g&\\&&&1 \end{pmatrix} &\\ & I_{2m} \end{pmatrix} \\
& \phantom{************************} dy \psi^{-1}(\mathrm{Tr} X) dX |\mathrm{det}(g)|^{s-2m-1}  dg.
\end{split}
\]
Consider the following simple matrix multiplication :
\[
\begin{pmatrix} I &&X&\\&1&y&\\&&I&\\&&&1 \end{pmatrix} \begin{pmatrix} g&&&\\ &1&&\\&&g&\\&&&1 \end{pmatrix}
 =\begin{pmatrix} I&&X&\\ &1&&\\&&I&\\&&&1 \end{pmatrix}  \begin{pmatrix} g&&&\\&1&&\\&&g&\\&&&1 \end{pmatrix}\begin{pmatrix} I&&&\\&1&yg&
\\&&I&\\&&&1 \end{pmatrix},
\]
Applying the change of variables of $y$ to $yg^{-1}$, we get the expression :
\[
\begin{split}
J_{(2m)}(s,W)&=\int_{N_{n-m-1} \backslash GL_{n-m-1}} \int_{\mathcal{N}_{n-m-1} \backslash \mathcal{M}_{n-m-1}} \int_{F^{n-m-1}}\\
& \phantom{***} 
W \begin{pmatrix}\sigma_{2n-2m}  \begin{pmatrix} I&&X&\\ &1&&\\&&I&\\&&&1 \end{pmatrix}  \begin{pmatrix} g&&&\\&1&&\\&&g&\\&&&1 \end{pmatrix}\begin{pmatrix} I&&&\\&1&y&
\\&&I&\\&&&1 \end{pmatrix} &\\ & I_{2m} \end{pmatrix}\\
&\phantom{***********************}  dy \psi^{-1}(\mathrm{Tr} X) dX |\mathrm{det}(g)|^{s-2m-2}  dg,
\end{split}
\]
Set $\omega_{2n-2m}= \begin{pmatrix} \sigma_{2n-2m-2}^{-1}&\\ & I_2 \end{pmatrix}  \sigma_{2n-2m}$. Since $\sigma_{2n-2m}$ and $\displaystyle \begin{pmatrix} \sigma_{2n-2m-2} & \\ & I_2 \end{pmatrix}^{-1}$ are
\[
\begin{split}
 &\left( \begin{matrix} 1 & 2 & \cdots & n-m-1 & n-m & |  \\ 
                                   1& 3 & \cdots & 2n-2m-3 & 2n-2m-1 & | 
                               \end{matrix} \right. \\
  &  \phantom{****************} \left. \begin{matrix}  | & n-m+1 & n-m+2 & \dotsm \quad 2n-2m-1  & 2n-2m \\  
                                                                               | & 2 & 4 & \dotsm \quad 2n-2m-2  & 2n-2m \end{matrix} \right)                                                     
\end{split}
\]
and
\[
\begin{split}
&\left( \begin{matrix} 1 & 2 & \dotsm \quad n-m-1 & |  \\
                                  1 & 3 & \dotsm \quad  2n-2m-3 & |   \end{matrix} \right. \\
 &  \phantom{******} \left. \begin{matrix} | &  n-m & n-m+1 & \dotsm \quad 2n-2m-2 & | & 2n-2m-1 & 2n-2m \\
                                                                             | &  2 & 4 & \dotsm \quad 2n-2m-2 & | & 2n-2m-1 & 2n-2m  \end{matrix} \right)^{-1},
\end{split}
\]
respectively, the composition of two permutations implies that
\[
\begin{split}
\omega_{2n-2m}&=\left( \begin{matrix}  1 &  \dotsm & n-m-1 & | & n-m & |  \\
                                     1 & \dotsm & n-m-1 & | & 2n-2m-1 & | 
         \end{matrix} \right. \\
   & \phantom{*********}    \left. \begin{matrix}     | & n-m+1 & n-m+2 \quad \dotsm \quad  2n-2m-1 \quad | \quad 2n-2m \\
                                                                       | & n-m & n-m+1 \quad \dotsm \quad  2n-2m-2 \quad | \quad 2n-2m
                                        \end{matrix}    \right).
\end{split}
\]
Conjugating by $\omega_{2n-2m}$ yields that
\[
\begin{split}
& \sigma_{2n-2m}  \begin{pmatrix} g&&&\\ &1&&\\&&g&\\&&&1 \end{pmatrix}   \sigma_{2n-2m}^{-1} \\
& =\begin{pmatrix} \sigma_{2n-2m-2} & \\ & I_2 \end{pmatrix} \left\{ \omega_{2n-2m}  \begin{pmatrix} g&&&\\ &1&&\\&&g&\\&&&1 \end{pmatrix} \omega_{2n-2m}^{-1} \right\} \begin{pmatrix} \sigma_{2n-2m-2} & \\ & I_2 \end{pmatrix}^{-1} \\
  &=\hspace*{-1mm} \begin{pmatrix} \sigma_{2n-2m-2} &\hspace*{-1mm} \\ &\hspace*{-1mm} I_2 \end{pmatrix}\hspace*{-1mm}   \begin{pmatrix} g&&\hspace*{-1mm}\\ &g&\hspace*{-1mm}\\&&\hspace*{-1mm}I_2 \end{pmatrix}  \hspace*{-1mm}   \begin{pmatrix} \sigma_{2n-2m-2} &\hspace*{-1mm} \\ &\hspace*{-1mm} I_2 \end{pmatrix}^{-1} 
\hspace*{-1mm}  =\hspace*{-1mm} \begin{pmatrix} \hspace*{-1mm} \sigma_{2n-2m-2}  \begin{pmatrix} g &  \\ & g  \end{pmatrix} \sigma_{2n-2m-2}^{-1} &\hspace*{-1mm}\\&\hspace*{-1mm} I_2 \end{pmatrix}. 
 \end{split}
\]
Similarly, we have
\[
 \sigma_{2n-2m}\hspace*{-1mm}   \begin{pmatrix} I_{n-m}&\hspace*{-1mm} &X&\\ &\hspace*{-1mm} 1&&\\&\hspace*{-1mm} &I_{n-m}&\\&\hspace*{-1mm} &&1 \end{pmatrix} \hspace*{-1mm}   \sigma_{2n-2m}^{-1} 
 =\begin{pmatrix}  \sigma_{2n-2m-2}  \begin{pmatrix} I_{n-m-1} & X \\ & I_{n-m-1}  \end{pmatrix} \sigma_{2n-2m-2}^{-1} &\\& I_2 \end{pmatrix}
\]
and
\[
 \sigma_{2n-2m}\hspace*{-1mm}   \begin{pmatrix} I_{n-m}&&&\\ &1&y&\\&&I_{n-m}&\\&&&1 \end{pmatrix} \hspace*{-1mm}   \sigma_{2n-2m}^{-1} 
 =\begin{pmatrix}  \sigma_{2n-2m-1} \hspace*{-1mm}  \begin{pmatrix} I_{n-m-1} &\hspace*{-2mm}  &\hspace*{-2mm}  \\ &\hspace*{-2mm}  I_{n-m-1}&\hspace*{-2mm} \\&\hspace*{-2mm} y&\hspace*{-2mm} 1  \end{pmatrix}\hspace*{-1mm}  \sigma_{2n-2m-1}^{-1} &\hspace*{-2mm} \\&\hspace*{-2mm}  1 \end{pmatrix}. 
\]
Taken together, we are left with
\[
\tag{4.1}
\label{J2m integral}
 \begin{split}
J_{(2m)}(s,W)&=\int_{N_{n-m-1} \backslash GL_{n-m-1}} \int_{\mathcal{N}_{n-m-1} \backslash \mathcal{M}_{n-m-1}} \int_{F^{n-m-1}}\\
&\phantom{***} 
W^{\circ} \begin{pmatrix}\sigma_{2n-2m-1}  \begin{pmatrix} I&X&\\ &I&\\&&1 \end{pmatrix}  \begin{pmatrix} g&&\\&g&\\&&1 \end{pmatrix}\begin{pmatrix} I&&\\&I&
\\&y&1 \end{pmatrix} &\\ & I_{2m+1} \end{pmatrix}\\
&\phantom{**********************}  dy \psi^{-1}(\mathrm{Tr} X) dX |\mathrm{det}(g)|^{s-2m-2}  dg,
\end{split}
\]
where $W^{\circ}=\pi \begin{pmatrix} \omega_{2n-2m} & \\ & I_{2m} \end{pmatrix}W$. From the second part of Lemma \ref{key relation}, there exists a function $W^1 \in \mathcal{W}(\pi,\psi)$ and a function $\Phi$ in $\mathcal{S}(F^{n-m-1})$ such that for $g \in GL_{n-m-1}$, $X \in \mathcal{M}_{n-m-1}$, and $y \in F^{n-m-1}$
\[
\tag{4.2}
\label{J2m relation}
\begin{split}
&W^{\circ} \begin{pmatrix}\sigma_{2n-2m-1}  \begin{pmatrix} I&X&\\ &I&\\&&1 \end{pmatrix}  \begin{pmatrix} g&&\\&g&\\&&1 \end{pmatrix}\begin{pmatrix} I&&\\&I&
\\&y&1 \end{pmatrix} &\\ & I_{2m+1} \end{pmatrix}
\\
& =W^1 \begin{pmatrix}\sigma_{2n-2m-1}  \begin{pmatrix} I&X&\\ &I&\\&&1 \end{pmatrix}  \begin{pmatrix} g&&\\&g&\\&&1 \end{pmatrix}\begin{pmatrix} I&&\\&I&
\\&y&1 \end{pmatrix} &\\ & I_{2m+1} \end{pmatrix} \Phi(y).
\end{split}
\]
If we insert this relation between $W^{\circ}$ and $W^1$ into our integral for $J_{(2m)}(s,W)$ in \eqref{J2m integral}, we obtain $J_{(2m)}(s,W)=J_{(2m+1)}(s,W^1,\Phi)$ for $\mathrm{Re}(s) \gg 0$. We extend this equality meromorphically to rational functions in $\mathbb{C}(q^{-s})$.  This says that $J_{(2m)}(s,W)$ is equal to the integral $J_{(2m+1)}(s,W^1,\Phi)$ for appropriate choice of $W^1 \in \mathcal{W}(\pi,\psi)$ and $\Phi \in \mathcal{S}(F^{n-m-1})$. Hence we have $\mathcal{J}_{(2m)}(\pi) \subset \langle J_{(2m+1)}(s,W,\Phi)\; |\; W \in \mathcal{W}(\pi,\psi), \Phi \in \mathcal{S}(F^{n-m-1}) \rangle$ of $\mathbb{C}[q^{\pm s}]$-fractional ideals.
\par

 It remains to establish the other containment $\mathcal{J}_{(2m)}(\pi) \supset\langle J_{(2m+1)}(s,W,\Phi) \rangle$. Given $W^1 \in \mathcal{W}(\pi,\psi)$ and $\Phi \in \mathcal{S}(F^{n-m-1})$, by the first part of Lemma \ref{key relation}, we may choose $W^{\circ} \in \mathcal{W}(\pi,\psi)$ satisfying the formula in \eqref{J2m relation}. Reproduce the steps \eqref{J2m integral}$-$\eqref{J2m relation} backward to land at
$J_{(2m+1)}(s,W^1,\Phi)=J_{(2m)}(s,\pi \begin{pmatrix}\omega_{2n-2m}^{-1}&\\ & I_{2m} \end{pmatrix} W^{\circ})$. Altogether each individual integral $J_{(2m+1)}(s,W^1,\Phi)$ is actually equal to an integral $J_{(2m)}(s,W)$ for appropriate choice of $W \in \mathcal{W}(\pi,\psi)$. Therefore we obtain that $\mathcal{J}_{(2m)}(\pi) \supset \langle J_{(2m+1)}(s,W,\Phi) \rangle$ as $\mathbb{C}[q^{\pm s}]$-fractional ideals.

\par
 As $\mathcal{J}_{(2m+1)}(\pi)=\langle J_{(2m+1)}(s,W,\Phi)\; |\; W \in \mathcal{W}(\pi,\psi), \Phi \in \mathcal{S}(F^{n-m-1}) \rangle$ by Lemma \ref{J2m-same}, we have the equalities $\mathcal{J}_{(2m+1)}(\pi)=\mathcal{J}_{(2m)}(\pi) $ of $\mathbb{C}[q^{\pm s}]$-fractional ideals.

\end{proof}

\subsection{An inclusion of $\mathbb{C}[q^{\pm s}]$-fractional ideals in $\mathbb{C}(q^{-s})$.}

Assume that $n > 1$. Let $\pi$ be an induced representations of $GL_{2n}$ of Langlands type. Our task for this Section is to describe the inclusion $\mathcal{J}_{(2m)}(\pi) \subset \mathcal{J}_{(2m-1)}(\pi)$ of $\mathbb{C}[q^{\pm s}]$-fractional ideals in $\mathbb{C}(q^{-s})$.

\begin{proposition}
\label{J2m-char}
We assume that $n > 1$. Let $\pi$ be an induced representations of $GL_{2n}$ of Langlands type. For $1 \leq m < n$, $\mathcal{J}_{(2m)}(\pi)$ is the $\mathbb{C}[q^{\pm s}]$-fractional ideal spanned by the integrals $J_{(2m-1)}(s,W)$ for $W$ belonging to $\mathcal{W}(\pi_{(2m-1),2},\psi)$. Furthermore we have the inclusion $\mathcal{J}_{(2m)}(\pi) \subset \mathcal{J}_{(2m-1)}(\pi)$ of $\mathbb{C}[q^{\pm s}]$-fractional ideals. 

\end{proposition}

\begin{proof}
We will employ $z$ to denote $zI_n=\text{diag}(z,z,\dotsm,z) \in Z_n$ to lighten the notation burden. Due to the partial Iwasawa decomposition on $GL_{n-m}$ in Section $2.2$, our integral $J_{(2m-1)}(s,W)$ becomes
\[
\begin{split}
J_{(2m-1)}(s,W)&=\int_{K_{n-m}} \int_{N_{n-m} \backslash P_{n-m}}  \int_{Z_{n-m}}\int_{\mathcal{N}_{n-m} \backslash \mathcal{M}_{n-m}}  \\
&\phantom{*****} W \begin{pmatrix}  \sigma_{2n-2m+1}  \begin{pmatrix} I_{n-m} & X& \\ & I_{n-m}&\\&&1  \end{pmatrix} \begin{pmatrix}  zpk&&  \\ &zpk&\\&&1 \end{pmatrix} & \\& I_{2m-1} \end{pmatrix} \\
&\phantom{**************} \times \psi^{-1}(\mathrm{Tr} X) dX |\mathrm{det}(z)|^{s-2m} dz |\mathrm{det}(p)|^{s-2m-1} dp dk.
\end{split}
\]
Let $K=\{ \mathrm{diag}(k,k,I_{2m})\;|\; k \in K_{n-m}\}$ be a subgroup of $K_{2n}$ and $K^{\circ} \subset K_{2n}$ a compact open subgroup which stabilize $W$. Put $K_{\circ}=K \cap K^{\circ}$ and write $K=\cup_i k_i K_{\circ}$, because $K_{\circ}$ is compact and open in $K$. Then our integral $J_{(2m-1)}(s,W)$ becomes a finite sum of the form
\[
\tag{4.3}
\label{dec-zpart}
\begin{split}
&J_{(2m-1)}(s,W)\\
&=\sum_i c_1 \int_{Z_{n-m}}  \int_{N_{n-m} \backslash P_{n-m}}  \int_{\mathcal{N}_{n-m} \backslash \mathcal{M}_{n-m}} W_i \begin{pmatrix}  \sigma_{2n-2m}  \begin{pmatrix} I_{n-m} & X \\ & I_{n-m}  \end{pmatrix} \begin{pmatrix}  zp&  \\ &zp \end{pmatrix} &\hspace*{-3mm}  \\&\hspace*{-3mm} I_{2m} \end{pmatrix}\\
&\phantom{*************************} \times \psi^{-1}(\mathrm{Tr} X) dX |\mathrm{det}(z)|^{s-2m} dz |\mathrm{det}(p)|^{s-2m-1} dp
\end{split}
\]
with $c_1 > 0$ the volume of $K_{\circ}$ and $W_i=\pi \begin{pmatrix} k_i && \\ & k_i &\\&&I_{2m}\end{pmatrix} W$. We now investigate a property of the support of $W_i$ in the $z$. By Lemma \ref{key relation}, there exists $W^i_{\circ} \in \mathcal{W}(\pi,\psi)$ and $\Phi_i \in \mathcal{S}(F^{n-m})$ such that for all $z \in Z_{n-m}$, $X \in \mathcal{M}_{n-m}$ and $p \in P_{n-m}$, we get
\[
\begin{split}
 &W_i \begin{pmatrix}  \sigma_{2n-2m+1}  \begin{pmatrix} I_{n-m} & X& \\ & I_{n-m}&\\&&1  \end{pmatrix} \begin{pmatrix}  zp&&  \\ &zp&\\&&1 \end{pmatrix} & \\& I_{2m-1} \end{pmatrix} \\
&\phantom{*******} =W^i_{\circ} \begin{pmatrix}  \sigma_{2n-2m+1}  \begin{pmatrix} I_{n-m} & X& \\ & I_{n-m}&\\&&1  \end{pmatrix} \begin{pmatrix}  zp&&  \\ &zp&\\&&1 \end{pmatrix} & \\& I_{2m-1} \end{pmatrix} \Phi_i(e_{n-m}z)
\end{split}
\]
Since each $\Phi_i$ has compact support on $F^{n-m}$, we can select $M$ large enough so that all of $W_i$ vanish for all $|z|>q^M$. We now suppose that $W$ belongs to $\mathcal{W}(\pi_{(2m-1),2},\psi)$. We have seen in Section 2.1 that $\Phi^+(\pi_{(2m)})=\Phi^+\Phi^-(\pi_{(2m-1)})=\pi_{(2m-1),2}$. We embed $P_{2n-2m+1}$ into $GL_{2n}$ by $p \mapsto \mathrm{diag}(p,I_{2m-1})$. From the characterization of Whittaker models $W \in \mathcal{W}(\pi_{(2m-1),2},\psi)$ on $g=\sigma_{2n-2m}  \begin{pmatrix} I_{n-m} & X \\ & I_{n-m}  \end{pmatrix} \begin{pmatrix}  zpk&  \\ &zpk \end{pmatrix}$ in Proposition \ref{mira-model2} as $P_{2n-2m+1}$-module, there exists a positive number $N$ such that
\[
 W \begin{pmatrix}  \sigma_{2n-2m+1}  \begin{pmatrix} I_{n-m} & X& \\ & I_{n-m}&\\&&1  \end{pmatrix} \begin{pmatrix}  zpk&&  \\ &zpk&\\&&1 \end{pmatrix} & \\& I_{2m-1} \end{pmatrix}=0
\]
for all $|z| < q^{-N}$. Since $\mathcal{W}(\pi_{(2m-1),2},\psi)$ is invariant under the action of $P_{2n-2m+1}$ and $\mathrm{diag}(k,k,I_{2m})$ is in $P_{2n-2m+1}$ for all $k \in K_{n-m}$, all of $W_i$ are elements of $\mathcal{W}(\pi_{(2m-1),2},\psi)$, which again vanish when $|z| < q^{-N}$. Thus each of $W_i$ has a compact multiplicative support in the variable $z$, given by
$\{z\; | \;q^{-N} \leq |z| \leq q^M \}$. Setting $z=\varpi^ju$ for $u \in \mathcal{O}^{\times}$, we reach the following finite sum of the form
\[
\tag{4.4}
\label{z-support1}
\begin{split}
 &J_{(2m-1)}(s,W)=\sum_{i,-M \leq j \leq N} c_1 q^{-(2n-2m)(s-2m)j}
  \int_{N_{n-m} \backslash P_{n-m}}  \int_{\mathcal{N}_{n-m} \backslash \mathcal{M}_{n-m}} \int_{\mathcal{O}^{\times}} \\
 &\phantom{****}  \pi \left( \begin{pmatrix} (\varpi^ju) I_{2n-2m}&\\ &I_{2m} \end{pmatrix} \right)W_i \begin{pmatrix}  \sigma_{2n-2m}  \begin{pmatrix} I_{n-m} & X \\ & I_{n-m}  \end{pmatrix} \begin{pmatrix}  p&  \\ &p \end{pmatrix} & \\& I_{2m} \end{pmatrix} \\
 &\phantom{*****************************} d^{\times} u \psi^{-1}(\mathrm{Tr} X) dX |\mathrm{det}(p)|^{s-2m-1} dp,
\end{split}
\]
because $|\mathrm{det}(\varpi^juI_{2n-2m})|^{s-2m}=q^{-(2n-2m)(s-2m)j}$ in \eqref{dec-zpart}. We define a compact open subgroup in $K$ by $K'=\{ \mathrm{diag}(uI_{2n-2m},I_{2m}) \; | \; u \in \mathcal{O}^{\times} \}$.
We repeat the step \eqref{dec-zpart} with $K'$ replacing $K$.
The integral turns out to be a finite sum of the form
\[
\label{z-support2}
\tag{4.5}
\begin{split}
 &J_{(2m-1)}(s,W)=\sum_{\ell,-M \leq j \leq N} c_2 q^{-(2n-2m)(s-2m)j} 
                \int_{N_{n-m} \backslash P_{n-m}}  \int_{\mathcal{N}_{n-m} \backslash \mathcal{M}_{n-m}}  \\
 &\phantom{****}  \pi \left( \begin{pmatrix} (\varpi^ju) I_{2n-2m}&\\ &I_{2m} \end{pmatrix} \right) W'_{\ell} \begin{pmatrix}  \sigma_{2n-2m}  \begin{pmatrix} I_{n-m} & X \\ & I_{n-m}  \end{pmatrix} \begin{pmatrix}  p&  \\ &p \end{pmatrix} & \\& I_{2m} \end{pmatrix}\\
 &\phantom{********************************}   \psi^{-1}(\mathrm{Tr} X) dX |\mathrm{det}(p)|^{s-2m-1} dp.
\end{split}
\]  
with $c_2 > 0$ the volume term and $W'_{\ell}$ some right translation of $W_i$ by $\begin{pmatrix} uI_{2n-2m}& \\ &I_{2m} \end{pmatrix}$ with $u \in \mathcal{O}^{\times}$. We show from \eqref{z-support2} that $J_{(2m-1)}(s,W)$ is a finite sum of integrals of the form $P(q^{\pm s})J_{(2m)}(s,W_1)$ for some $P(X) \in \mathbb{C}[X]$ and $W_1 \in \mathcal{W}(\pi,\psi)$ if $W \in \mathcal{W}(\pi_{(2m-1),2},\psi)$, so that $\mathcal{J}_{(2m)}(\pi) \supset \langle J_{(2m-1)}(s,W) \; | \; W \in \mathcal{W}(\pi_{(2m-1),2},\psi) \rangle$.

\par
Conversely, for any integral $J_{(2m)}(s,W)$ with $W \in \mathcal{W}(\pi,\psi)$, let $K_{n-m,r}$ be a compact open subgroup of $K_{n-m}$ such that $\pi(\mathrm{diag}(k,k,I_{2m}))W=W$ for all $k \in K_{n-m,r}$. Let $\Phi$ be a characteristic function of $e_{n-m}K_{n-m,r}$. As $(0,\dotsm,0) \notin e_{n-m}K_{n-m,r}$, $\Phi$ lies in $\mathcal{S}_0(F^{n-m})$. We have the following integration formula, derived using partial Iwasawa decomposition on $GL_{n-m}$ in Section 2.2
\[
\begin{split}
 & \int_{N_{n-m} \backslash GL_{n-m}}  \int_{\mathcal{N}_{n-m} \backslash \mathcal{M}_{n-m}}  
  W \begin{pmatrix}  \sigma_{2n-2m}  \begin{pmatrix} I_{n-m} & X \\ & I_{n-m}  \end{pmatrix} \begin{pmatrix}  g&  \\ &g \end{pmatrix} & \\& I_{2m} \end{pmatrix} \\
  &\phantom{*******************************} \psi^{-1}(\mathrm{Tr} X) dX \Phi(e_{n-m}g) |\mathrm{det}(g)|^{s-2m} dg \\
  &= \int_{K_{n-m}} \int_{Z_{n-m}} \int_{N_{n-m} \backslash P_{n-m}}   \int_{\mathcal{N}_{n-m} \backslash \mathcal{M}_{n-m}} 
   \hspace*{-1mm}W \hspace*{-1mm} \begin{pmatrix}  \sigma_{2n-2m}  \begin{pmatrix} I_{n-m} & X \\ & I_{n-m}  \end{pmatrix} \begin{pmatrix}  pzk&  \\ &pzk \end{pmatrix} &\hspace*{-5mm}  \\&\hspace*{-5mm}  I_{2m} \end{pmatrix} \\
   &\phantom{******************}  \psi^{-1}(\mathrm{Tr} X) dX  \Phi(e_{n-m}zk) |\mathrm{det}(p)|^{s-2m-1} dp |\mathrm{det}(z)|^{s-2m} dz dk.
  \end{split}
\]
According to \eqref{zp-invariance}, this formula leads to
\[
\tag{4.6}
\label{switch-J2m}
\begin{split}
 & J_{(2m)}(s,W)\\
 &=\frac{1}{\mathrm{Vol}(e_{n-m}K_{n-m,r})} \int_{N_{n-m} \backslash GL_{n-m}}  \int_{\mathcal{N}_{n-m} \backslash \mathcal{M}_{n-m}}  
  \hspace*{-2mm}W\hspace*{-1mm} \begin{pmatrix}  \sigma_{2n-2m} \hspace*{-1mm} \begin{pmatrix} I_{n-m} & X \\ & I_{n-m}  \end{pmatrix}\hspace*{-1mm} \begin{pmatrix}  g&  \\ &g \end{pmatrix} &\hspace*{-5mm} \\&\hspace*{-5mm} I_{2m} \end{pmatrix} \\
  &\phantom{*****************************} \times \psi^{-1}(\mathrm{Tr} X) dX \Phi(e_{n-m}g) |\mathrm{det}(g)|^{s-2m} dg.
  \end{split}
\]
We select $W_{\circ} \in \mathcal{W}(\pi,\psi)$ by virtue of Lemma \ref{key relation} such that
\[
\tag{4.7}
\label{relation-J2m}
\begin{split}
 & W \begin{pmatrix}  \sigma_{2n-2m}  \begin{pmatrix} I_{n-m} & X \\ & I_{n-m}  \end{pmatrix} \begin{pmatrix}  g&  \\ &g \end{pmatrix} & \\& I_{2m} \end{pmatrix} \Phi(e_{n-m}g) \\
 &=W_{\circ}\begin{pmatrix}  \sigma_{2n-2m+1}  \begin{pmatrix} I_{n-m} & X& \\ & I_{n-m}&\\&&1  \end{pmatrix} \begin{pmatrix}  g&&  \\ &g&\\&&1 \end{pmatrix} & \\& I_{2m-1} \end{pmatrix} 
  \end{split}
\]
for all $g \in GL_{n-m}$ and $X \in \mathcal{M}_{n-m}$. We apply the formula \eqref{relation-J2m} to \eqref{switch-J2m} in order to reach $J_{(2m)}(s,W)=c_3J_{(2m-1)}(s,W_{\circ})$ with $\displaystyle c_3=\frac{1}{\mathrm{Vol}(e_{n-m}K_{n-m,r})} > 0$. As $\Phi(0)=0$, $W_{\circ}$ vanishes when the last row of $g$ in \eqref{relation-J2m} satisfies the estimates $\max_i\{ |g_{n-m,i}|\} < q^{-N}$ for some positive $N$. By Proposition \ref{mira-model2}, $J_{(2m-1)}(s,W_{\circ})$ only relies on the image $W_{\circ}$ in $\mathcal{W}(\pi_{(2m-1),2},\psi)$. We conclude that $J_{(2m)}(s,W)=c_3J_{(2m-1)}(s,W_{\circ})$ for $W_{\circ}$ corresponding to functions in $\mathcal{W}(\pi_{(2m-1),2},\psi)$ so that $\mathcal{J}_{(2m)}(\pi) \subset \langle J_{(2m-1)}(s,W) \; | \; W \in \mathcal{W}(\pi_{(2m-1),2},\psi) \rangle$. Combined this with the result on the previous paragraph, we complete the first part of proof.

\par
For the second assertion, we know from Lemma \ref{J2m-same} that $\mathcal{J}_{(2m-1)}$ is generated by the integral $J_{(2m-1)}(s,W)$ for $W \in \mathcal{W}(\pi_{(2m-1)},\psi)$. Since $\mathcal{W}(\pi_{(2m-1),2},\psi)$ is $P_{2n-2m+1}$-submodule of $\mathcal{W}(\pi_{(2m-1)},\psi)$ by Proposition \ref{mira-model2}, we obtain the desired inclusion $\mathcal{J}_{(2m)}(\pi) \subset \mathcal{J}_{(2m-1)}(\pi)$ of $\mathbb{C}[q^{\pm s}]$-fractional ideals. 
\end{proof}

If we combine Proposition \ref{L0 function} and Propostion \ref{even-equal} with our analysis of the inclusion of fractional ideals $\mathcal{J}_{(2m)}(\pi) \subset \mathcal{J}_{(2m-1)}(\pi)$ in Proposition \ref{J2m-char}, the family of rational functions defined by Jacquet-Shalika integrals have the descending chain $\mathcal{J}_{(2n-2)}(\pi)  \subset \dotsm \subset \left[ \mathcal{J}_{(2m+1)}(\pi) = \mathcal{J}_{(2m)}(\pi) \right] \subset \dotsm \subset\left[\mathcal{J}_{(1)}(\pi)=\mathcal{J}_{(0)}(\pi)\right] \subset \mathcal{J}(\pi)$ of $\mathbb{C}[q^{\pm s}]$-fractional ideals. As $N_1=P_1=\{1\}$ and $\mathcal{N}_1=\mathcal{M}_1$, the family of rational functions in $\mathcal{J}_{(2n-2)}(\pi)$ are $\mathbb{C}[q^{\pm s}]$-linear combination of $J_{(2m)}(s,W)=W(I_{2n})$ with $W \in \mathcal{W}(\pi,\psi)$, resulting in $\mathcal{J}_{(2n-2)}(\pi)=\mathbb{C}[q^{\pm s}]$. Summarizing this analysis, we have the following result.

 \begin{proposition}
 \label{filtration-even}
We assume that $n > 1$. Let $\pi$ be an induced representations of $GL_{2n}$ of Langlands type. For $0 \leq m \leq n-2$, we have the filtration 
\[
  \mathcal{J}_{(2n-2)}(\pi)  \subset    \dotsm \subset \left[ \mathcal{J}_{(2m+1)}(\pi) = \mathcal{J}_{(2m)}(\pi) \right] \subset \dotsm \subset \left[\mathcal{J}_{(1)}(\pi)=\mathcal{J}_{(0)}(\pi)\right] \subset \mathcal{J}(\pi)
\]
 of $\mathbb{C}[q^{\pm s}]$-fractional ideals in $\mathbb{C}(q^{-s})$. The bottom piece of this filtration $\mathcal{J}_{(2n-2)}(\pi)$ is $\mathbb{C}[q^{\pm s}]$.
\end{proposition}

\subsection{The exceptional poles of the even derivatives.} 

In this section, we deduce the factorization of $L(s,\pi,\wedge^2)$ for an irreducible generic representation $\pi$  in terms of exceptional $L$-functions for even derivatives. For this, we associate the exceptional poles of $\mathcal{J}(\pi^{(2m)})$ with the inclusion of fractional ideals $\mathcal{J}_{(2m)}(\pi) \subset \mathcal{J}_{(2m-1)}(\pi)$ in $\mathbb{C}(q^{-s})$. 
\par

Assume that $n > 1$. Let $\pi$ be an induced representations of $GL_{2n}$ of Langlands type. Let $0 \leq k \leq 2n-2$. Suppose there is a function in $\mathcal{J}_{(k)}(\pi)$ having a pole of order $d$ at $s=s_0$ and that this is the highest order pole of the family at $s=s_0$. Consider the rational function defined by an individual integral $J_{(k)}(s,W)$. Then the Laurent expansion at that point will be of the form 
\[
J_{(k)}(s,W)=\frac{\Lambda_{(k),s_0}(W)}{(q^s-q^{s_0})^d}+\mathrm{higher\; order\; terms.}
\]
By a realization of Whittaker model for $\pi_{(k)}$ in Proposition \ref{mira-model}, as the integral $J_{(k)}(s,W)$ only depends on the functions $W$ through their restriction to $P_{2n-k}$, $\Lambda_{(k),s_0}$ belongs to the vector space $\mathrm{Hom}_{P_{2n-k} \cap S_{2n-k}}(\mathcal{W}(\pi_{(k)},\psi), |\cdot|^{-\frac{s_0-1}{2}}\Theta)$.

\begin{definition*}
For $0 \leq k \leq 2n-2$, a linear functional $\Lambda_{(k),s}$ on $\mathcal{W}(\pi_{(k)},\psi)$ is called a twisted Shalika functional with respect to a Shlika subgroup $P_{2n-k} \cap S_{2n-k}$ if it satisfies the following quasi-invariance 
\[
\Lambda_{(k),s}(\pi_{(k)}(\begin{pmatrix} h &\\ & I_{k} \end{pmatrix})W)=|\mathrm{det}(h)|^{-\frac{s-1}{2}}\Theta(h)\Lambda_{(k),s}(W)
\]
 for all $h \in P_{2n-k} \cap S_{2n-k}$.
\end{definition*}

To proceed, we would make the following assumption.

\begin{assumption}
For $n > 1$, let $\pi=\mathrm{Ind}(\Delta_1 \otimes \dotsm \otimes \Delta_t)$ be an irreducible generic representation of $GL_{2n}$. For $0 \leq m \leq 2n$, we assume that all derivatives $\pi^{(m)}$ of $\pi$ are completely reducible with irreducible generic subquotients of the form $\pi_{i}^{(m)}=\mathrm{Ind}(\Delta_1^{(m_1)} \otimes \dotsm \otimes \Delta_t^{(m_t)})$ and $m=m_1+\dotsm+m_t$.  
\end{assumption}

We will remove this condition in Section $7$ by a deformation argument. Under this assumption, $\Delta_i$ are rearranged to be induced representation of Langlands type without changing $\pi$, and we can compute the derivatives of the derivatives. For $m>1$ each irreducible constituents $\pi_{i}^{(m)}=\mathrm{Ind}(\Delta_1^{(m_1)} \otimes \dotsm \otimes \Delta_t^{(m_t)})$ is an irreducible generic representation of smaller group $GL_{2n-m}$. Moreover their derivatives $(\pi_{i}^{(m)})^{(k)}$ are completely reducible and all successive quotients are irreducible and of the form $\mathrm{Ind}((\Delta_1^{(m_1)})^{(k_1)} \otimes \dotsm \otimes (\Delta_t^{(m_t)})^{(k_t)})$ with $k=k_1+\dotsm+k_t$. However Bernstein and Zelevinsky in Theorem \ref{deriviative} assert that the derivatives for the segments are associative in the sense that $(\Delta_j^{(m_j)})^{(k_j)}\simeq \Delta_j^{(m_j+k_j)}$. As a consequence of associativity, every irreducible subquotients of the form $\mathrm{Ind}((\Delta_1^{(m_1)})^{(k_1)} \otimes \dotsm \otimes (\Delta_t^{(m_t)})^{(k_t)}) \simeq \mathrm{Ind}(\Delta_1^{(m_1+k_1)} \otimes \dotsm \otimes \Delta_t^{(m_t+k_t)})$ in turn appear amongst those of the derivatives $\pi^{(m+k)}$ of $\pi$. Each irreducible constituents $\pi_{i}^{(m)}$ satisfying our assumption eventually makes it possible to exploit the induction on $n$ of $GL_{2n}$. Our assumption is also included in the case where the segments are in $``$general position$"$ as in Section 7. This assumption make it possible to prove the following factorization Theorem of $L_{(0)}(s,\pi,\wedge^2)$ in terms of exceptional $L$-functions of the even derivatives of $\pi$ that we seek.

\begin{theorem}
\label{prod-L0}
For $n > 1$, let $\pi=\mathrm{Ind}(\Delta_1 \otimes \dotsm \otimes \Delta_t)$ be an irreducible generic representation of $GL_{2n}$ such that all derivatives $\pi^{(m)}$ of $\pi$ are completely reducible with irreducible generic Jordan--H\"older constituents of the form $\pi_{i}^{(m)}=\mathrm{Ind}(\Delta_1^{(m_1)} \otimes \dotsm \otimes \Delta_t^{(m_t)})$ and $m=m_1+\dotsm+m_t$. For each $m$, $i$ is indexing the partition of $m$. Then
\[
\label{prod-L0-1}
\tag{4.8}
L_{(0)}(s,\pi,\wedge^2)^{-1}=l.c.m_{m,i}\{ L_{ex}(s,\pi_i^{(2m)},\wedge^2)^{-1} \}
\]
where the least common multiple is with respect to divisibility in $\mathbb{C}[q^{\pm s}]$ and is taken over all $m$ with $0 < m < n$ and for each $m$ all constituents $\pi_i^{(2m)}$ of
$\pi^{(2m)}$. 

\end{theorem}

\begin{proof}
Throughout this proof, all of the divisibility and the least common multiple are respect to divisibility in $\mathbb{C}[q^{\pm s}]$. We start with showing that a pole at $s=s_0$ of order $d$ of the reciprocal of $l.c.m_{m,i}\{ L_{ex}(s,\pi_i^{(2m)},\wedge^2)^{-1} \}$ occurs as the pole at $s=s_0$ of order $d$ of $L_{(0)}(s,\pi,\wedge^2)$. Let $s=s_0$ be a pole of order $d$ of $L_{ex}(s,\pi_i^{(2m)},\wedge^2)$ for some $0 < m < n$ and $i$. Since $s=s_0$ is an exceptional pole, it is a pole of an integral $J(s,W_{\circ},\phi)$ for some $W_{\circ} \in \mathcal{W}(\pi_i^{(2m)},\psi)$ and $\phi \in \mathcal{S}(F^{n-m})$ which does not vanish at zero. We can in fact choose $\phi$ such that $\phi(0)=1$. Let $\Phi_{\circ}$ be a characteristic function of a sufficiently small neighborhood of zero. We can further replace $\phi$ by $\Phi_{\circ}$ because $s=s_0$ occurs as the pole of order strictly less than $d$ of $J(s,W_{\circ},\phi-\Phi_{\circ})$ and $\phi-\Phi_{\circ}$ vanishes at $0$.
Let $p : V_{\pi_{(2m-1)}} \rightarrow V_{\pi^{(2m)}}$ be the normalized map in Proposition \ref{asym} and let $V_{\tau_i}=p^{-1}(V_{\pi_i^{(2m)}})$. Then $\tau_i$ is a subrepresentation of $\pi_{(2m-1)}$. We apply Proposition \ref{asym} to $\sigma_{2n-2m} \begin{pmatrix} I_{n-m}&X\\ & I_{n-m} \end{pmatrix} \begin{pmatrix} g&\\ &g \end{pmatrix}$, $W_{\circ} \in \mathcal{W}(\pi_i^{(2m)},\psi)$,  and $\Phi_{\circ} \in \mathcal{S}(F^{n-m})$. Then there exists a function $W \in \mathcal{W}(\tau_i,\psi)$ satisfying the following equality:
 \begin{equation}
\label{filtration-derivative}
\tag{4.9}
\begin{split}
&W \begin{pmatrix} \sigma_{2n-2m} \begin{pmatrix}
I_{n-m}&X\\ &I_{n-m} \end{pmatrix} \begin{pmatrix} g&\\ &g \end{pmatrix}&\\&I_{2m}
\end{pmatrix} \Phi_{\circ}(e_{n-m} g) \\
&= W_{\circ} \left( \sigma_{2n-2m}  \begin{pmatrix}
I_{n-m}&X\\ &I_{n-m} \end{pmatrix} \begin{pmatrix} g&\\ &g \end{pmatrix} \right) |\mathrm{det}(g)|^{2m} \Phi_{\circ}(e_{n-m}g).
\end{split}
\end{equation}
In virtue of Lemma \ref{key relation}, we can find $W^{\circ} \in \mathcal{W}(\pi,\psi)$ such that
  \begin{equation}
\label{filtration2}
\tag{4.10}
\begin{split}
&W^{\circ} \begin{pmatrix} \sigma_{2n-2m} \begin{pmatrix}
I_{n-m}&X\\ &I_{n-m} \end{pmatrix} \begin{pmatrix} g&\\ &g \end{pmatrix}&\\&I_{2m}
\end{pmatrix} \\
& =W \begin{pmatrix} \sigma_{2n-2m} \begin{pmatrix}
I_{n-m}&X\\ &I_{n-m} \end{pmatrix} \begin{pmatrix} g&\\ &g \end{pmatrix}&\\&I_{2m}
\end{pmatrix} \Phi_{\circ}(e_{n-m} g). 
\end{split}
\end{equation}
Combining \eqref{filtration-derivative} with \eqref{filtration2}, we can write $J(s,W_{\circ},\Phi)$ in terms of $W^{\circ}$, namely
\[
\begin{split}
 &J(s,W_{\circ},\Phi) \\
 &= \int_{N_{n-m} \backslash {GL}_{n-m}}  \int_{\mathcal{N}_{n-m} \backslash \mathcal{M}_{n-m}} W^{\circ} \begin{pmatrix}  \sigma_{2n-2m}  \begin{pmatrix} I_{n-m} & X \\ & I_{n-m}  \end{pmatrix} \begin{pmatrix}  g&  \\ &g \end{pmatrix} & \\& I_{2m} \end{pmatrix}\\
 &\phantom{********************************}\times \psi^{-1}(\mathrm{Tr} X) dX\; |\mathrm{det}(g)|^{s-2m} dg \\
 &=J_{(2m-1)}(s,W^{\circ}).
 \end{split}
\]
$J_{(2m-1)}(s,W^{\circ})$ has a pole at $s=s_0$ of order $d$ and this gives rise to a pole of the family $\mathcal{J}_{(2m-1)}(\pi)$. As a consequence of Proposition \ref{filtration-even}, $\mathcal{J}_{(2m-1)}(\pi) \subset \mathcal{J}_{(0)}(\pi)$ asserts that $s=s_0$ contributes to the pole of order $d$ at least of the family $\mathcal{J}_{(0)}(\pi)$. According to Definition of $L_{(0)}(s,\pi,\wedge^2)$ in Section 3.1, the Euler factor $L_{(0)}(s,\pi,\wedge^2)$ which generates the $\mathbb{C}[q^{\pm s}]$-fractional ideal $\mathcal{J}_{(0)}(\pi)$ in $\mathbb{C}(q^{-s})$ has a pole of order $d$ at least at $s=s_0$. This proves that the right hand side divides the left hand side in equality \eqref{prod-L0-1}.

\par

 It remains to establish that any pole at $s=s_0$ of order $d$ of $L_{(0)}(s,\pi,\wedge^2)$ appears as an exceptional pole at $s=s_0$ of order $d$ of $L_{ex}(s,\pi_i^{(2m)},\wedge^2)$ for some $0 < m < n$ and $i$.
 As $L_{(0)}(s,\pi,\wedge^2)$ is completely determined by the poles of $\mathcal{J}_{(0)}(\pi)$ and their order, let $d$ be the maximal order at the pole $s=s_0$ of the family $\mathcal{J}_{(0)}(\pi)$. By the filtration of $\mathbb{C}[q^{\pm s}]$-fractional ideals in Proposition \ref{filtration-even}, we can select the smallest index $m$, $0 < m < n$ such that the pole $s=s_0$ has the highest order $d$ for the family $\mathcal{J}_{(2m-1)}(\pi)$ whereas of order $d-1$ at most of the family $\mathcal{J}_{(2m)}(\pi)$. Since the ideal $\mathcal{J}_{(2m-1)}(\pi)$ is linearly spanned by the rational functions defined by the integrals $J_{(2m-1)}(s,W)$, this pole must occur with order $d$ for some rational functions $J_{(2m-1)}(s,W)$ with $W \in \mathcal{W}(\pi,\psi)$. Then this will have a Laurent expansion at $s=s_0$:
\[
J_{(2m-1)}(s,W)=\frac{\Lambda_{(2m-1),s_0}(W)}{(q^s-q^{s_0})^d} + \;\mathrm{higher\;order\;terms.}
\]
Applying Proposition \ref{J2m-char}, $\Lambda_{(2m-1),s_0}$ vanishes for all $W \in \mathcal{W}(\pi_{(2m-1),2},\psi) \subset \mathcal{W}(\pi_{(2m-1)},\psi)$. As a representation of $P_{2n-2m+1}$, 
$\pi_{(2m-1)} \slash \pi_{(2m-1),2}=\Psi^+(\pi^{(2m)})$. Therefore $\Lambda_{(2m-1),s_0}(W)$ is an element of $\mathrm{Hom}_{P_{2n-2m+1}\cap S_{2n-2m+1}}(\mathcal{W}(\Psi^+(\pi^{(2m)}),\psi),|\cdot|^{-\frac{s_0-1}{2}}\Theta)$ and defines a non-trivial twisted Shalika functional on $\mathcal{W}(\Psi^+(\pi^{(2m)}),\psi)$ for a Shalika subgroup $P_{2n-2m+1}\cap S_{2n-2m+1}$. Through the normalized projection map $p : V_{\pi_{(2m-1)}} \rightarrow \oplus_i V_{\pi_i^{(2m)}}$, there must be an $i$ such that the Shalika functional $\Lambda_{(2m-1),s_0}$ determined by the pole $s_0$ must restrict non trivially to $W$ in $\mathcal{W}(\tau_i,\psi)$ with $V_{\tau_i}=p^{-1}(V_{\pi_i^{(2m)}})$.
 According to Proposition \ref{asym}, there exists function $W_{\circ} \in \mathcal{W}(\pi_i^{(2m)},\psi)$ such that for every Schwartz-Bruhat function 
$\Phi_{\circ} \in \mathcal{S}(F^{n-m})$ which is the characteristic function of a sufficiently small neighborhood of $0 \in F^{n-m}$ we have the equality \eqref{filtration2}. We can decompose the integral $J_{(2m-1)}(s,W)$ into two parts $J_{(2m-1)}(s,W)=J^0_{(2m-1)}(s,W)+J^1_{(2m-1)}(s,W)$,  where
\[
\begin{split}
&J_{(2m-1)}^0(s,W)= \int_{N_{n-m} \backslash GL_{n-m}} \int_{\mathcal{N}_{n-m} \backslash \mathcal{M}_{n-m}}  W \begin{pmatrix} \sigma_{2n-2m} \begin{pmatrix}
I_{n-m}&X\\ &I_{n-m} \end{pmatrix} \begin{pmatrix} g&\\ &g \end{pmatrix}&\hspace*{-5mm} \\&\hspace*{-5mm} I_{2m}
\end{pmatrix}\\
&\phantom{******************************} \Phi_{\circ}(e_{n-m}g) |\mathrm{det}(g)|^{s-2m}  \psi^{-1}(\mathrm{Tr} X) dX \; dg
\end{split}
\]
and
\[
\begin{split}
&J_{(2m-1)}^1(s,W)= \int_{N_{n-m} \backslash GL_{n-m}} \int_{\mathcal{N}_{n-m} \backslash \mathcal{M}_{n-m}}  W \begin{pmatrix} \sigma_{2n-2m} \begin{pmatrix}
I_{n-m}&X\\ &I_{n-m} \end{pmatrix} \begin{pmatrix} g&\\ &g \end{pmatrix}&\hspace*{-5mm}\\&\hspace*{-5mm}I_{2m}
\end{pmatrix}\\
&\phantom{**************************} (1-\Phi_{\circ}(e_{n-m}g)) |\mathrm{det}(g)|^{s-2m}  \psi^{-1}(\mathrm{Tr} X) dX \; dg.
\end{split}
\]
Due to the partial Iwasawa decomposition $GL_{n-m}=P_{n-m}Z_{n-m}K_{n-m}$ in Section 2.2, the second part of Lemma \ref{key relation} implies that there exists $M > 0$ such that
\[
 W \begin{pmatrix} \sigma_{2n-2m} \begin{pmatrix}
I_{n-m}&X\\ &I_{n-m} \end{pmatrix} \begin{pmatrix} pzk&\\ &pzk \end{pmatrix}&\\&I_{2m}
\end{pmatrix}\equiv 0
\]
for all $|z|>q^M$. In the integral $J_{(2m-1)}^1(s,W)$, as $1-\Phi_{\circ}$ vanishes at zero, we repeat the steps from \eqref{z-support1} to \eqref{z-support2} in order to rewrite the integral $J_{(2m-1)}^1(s,W)$ as a finite sum of integrals of the form $J_{(2m)}(s,W_i)$, which only depends on the image of $W_i$ in $\mathcal{W}(\pi_{(2m)},\psi)$ from Proposition \ref{mira-model}. However the twisted Shalika functional $\Lambda_{(2m-1),s_0}(W)$ restricts to zero on this space. Hence the integral of the form $J_{(2m-1)}^1(s,W)$ cannot contribute to the order $d$ at $s=s_0$.

\par
 We may write $J_{(2m-1)}^0(s,W)$ in terms of $W_{\circ}$ in \eqref{filtration-derivative}, namely
\[
\begin{split}
&J_{(2m-1)}^0(s,W)=\int_{N_{n-m} \backslash GL_{n-m}} \int_{\mathcal{N}_{n-m} \backslash \mathcal{M}_{n-m}} W_{\circ} \left( \sigma_{2n-2m}  \begin{pmatrix}
I_{n-m}&X\\ &I_{n-m} \end{pmatrix} \begin{pmatrix} g&\\ &g \end{pmatrix} \right)\\ \
&\phantom{***************************} \Phi_{\circ}(e_{n-m}g) |\mathrm{det}(g)|^s \psi^{-1}(\mathrm{Tr} X) dX \; dg.
\end{split}
\]
This integral realizes standard Jacquet-Shalika integrals $J_{(2m-1)}^0(s,W)=J(s,W_{\circ},\Phi_{\circ})$ for $\pi_i^{(2m)}$ and hence $L(s,\pi_i^{(2m)},\wedge^2)$ has a pole of order at least $d$ at $s=s_0$. Altogether it was shown that any pole at $s=s_0$ of order $d$ of $L_{(0)}(s,\pi,\wedge^2)$ occurs as the pole at $s=s_0$ of order at least $d$ of $L(s,\pi_i^{(2m)},\wedge^2)$ for some $0 < m < n$ and $i$.

\par

To complete the proof, we employ induction on $n$ of $GL_{2n}$. For $n=2$, we know from Proposition \ref{basic} that
$L(s,\pi_i^{(2)},\wedge^2)=L_{ex}(s,\pi_i^{(2)},\wedge^2)$. Therefore the left hand side divides the right hand side in equality \eqref{prod-L0-1} for $n=2$, because any pole at $s=s_0$ of order $d$ of $L_{(0)}(s,\pi,\wedge^2)$ appears in the pole at $s=s_0$ of order at least $d$ of $L_{ex}(s,\pi_i^{(2)},\wedge^2)$ for some $i$. To proceed, we make the following induction hypothesis.

\vskip.2in
\noindent
\textbf{Induction Hypothesis:} For $n > 2$, let $0 < m < n$. For every irreducible generic representation of $\pi$ of $GL_{2m}$ such that all derivatives $\pi^{(k)}$ of $\pi$ are completely reducible with irreducible Jordan--H\"older constituents of the form $\pi_i^{(k)}=\mathrm{Ind}(\Delta_1^{(k_1)} \otimes \dotsm \otimes \Delta_t^{(k_t)})$, $L_{(0)}(s,\pi,\wedge^2)^{-1}$ divides $l.c.m_{k,i}\{ L_{ex}(s,\pi_i^{(2k)},\wedge^2)^{-1} \}$, where the least common multiple is with respect to divisibility in $\mathbb{C}[q^{\pm s}]$ and is taken over all $k$ with $0 < k < m$ and for each $k$ all constituents $\pi_i^{(2k)}$ of $\pi^{(2k)}$.   

\vskip.2in
$L(s,\pi_i^{(2m)},\wedge^2)^{-1}$ is the least common multiple in $\mathbb{C}[q^{\pm s}]$ of $L_{(0)}(s,\pi_i^{(2m)},\wedge^2)^{-1}$ and $L_{ex}(s,\pi_i^{(2m)},\wedge^2)^{-1}$ because any poles of $\displaystyle L^{(0)}(s,\pi_i^{(2m)},\wedge^2)=\frac{L(s,\pi_i^{(2m)},\wedge^2)}{L_{(0)}(s,\pi_i^{(2m)},\wedge^2)}$ are the exceptional poles of $L(s,\pi_i^{(2m)},\wedge^2)$ by Proposition \ref{L0 function} and $L_{ex}(s,\pi_i^{(2m)},\wedge^2)$ has the maximal order for each exceptional poles.

\par

Under the induction hypothesis on $n-m$ of $GL_{2(n-m)}$, $L_{(0)}(s,\pi_i^{(2m)},\wedge^2)^{-1}$ divides the Euler factor $l.c.m_{k,j_i}\{ L_{ex}(s,(\pi_i^{(2m)})_{j_i}^{(2k)},\wedge^2)^{-1} \}$ where the least common multiple is taken over all $k$ with $0 < k < n-m$ and for each $k$ all constituents $(\pi_i^{(2m)})_{j_i}^{(2k)}$ of $(\pi_i^{(2m)})^{(2k)}$. However as all irreducible constituents $(\pi_i^{(2m)})_{j_i}^{(2k)}$ appear amongst those $\pi_{\ell}^{(2m+2k)}$ of the derivatives $\pi^{(2m+2k)}$ of $\pi$, we also obtain that $L(s,\pi_i^{(2m)},\wedge^2)^{-1}$ divides $l.c.m_{k,\ell}\{ L_{ex}(s,\pi_{\ell}^{(2m+2k)},\wedge^2)^{-1} \}$, where $k$ runs over $0 \leq k < n-m$ and $\pi_{\ell}^{(2m+2k)}$ are irreducible constituents of $\pi^{(2m+2k)}$. This in turn implies that the pole at $s=s_0$ of order at least $d$ of $L(s,\pi_i^{(2m)},\wedge^2)$ for any $0 < m < n$ and $i$ occurs with the same pole of order at least $d$ of the the reciprocal of $l.c.m_{k,i}\{ L_{ex}(s,\pi_i^{(2k)},\wedge^2)^{-1} \}$, where $k$ now takes over all $0 < k < n$ and $\pi_i^{(2k)}$ are irreducible constituents of $\pi^{(2k)}$. Altogether collecting our analysis of the poles of the Euler factor $L_{(0)}(s,\pi,\wedge^2)$, any pole at $s=s_0$ of order $d$ of $L_{(0)}(s,\pi,\wedge^2)$ occurs as the pole at $s=s_0$ of order at least $d$ of the the reciprocal of $l.c.m_{k,i}\{ L_{ex}(s,\pi_i^{(2k)},\wedge^2)^{-1} \}$. This proves our assertion for induction.

\par
 As one side of \eqref{prod-L0-1} divides the other side of  \eqref{prod-L0-1} and vice versa, and both sides of \eqref{prod-L0-1} are Euler factors, they are equal.

\end{proof}

As we have already used in the proof above, $L(s,\pi,\wedge^2)^{-1}$ is the least common multiple of $L_{ex}(s,\pi,\wedge^2)^{-1}$ and $L_{(0)}(s,\pi,\wedge^2)^{-1}$ by Proposition \ref{L0 function}. If we combine this with Theorem \ref{prod-L0}, we arrive at the following Theorem.

\begin{theorem}
\label{prod-L}
Let $\pi=\mathrm{Ind}(\Delta_1 \otimes \dotsm \otimes \Delta_t)$ be an irreducible generic representation of $GL_{2n}$ such that all derivatives $\pi^{(m)}$ of $\pi$ are completely reducible with irreducible generic Jordan--H\"older constituents of the form $\pi_{i}^{(m)}=\mathrm{Ind}(\Delta_1^{(m_1)} \otimes \dotsm \otimes \Delta_t^{(m_t)})$ and $m=m_1+\dotsm+m_t$. For each $m$, $i$ is indexing the partition of $m$. Then
\[
L(s,\pi,\wedge^2)^{-1}=l.c.m_{m,i}\{ L_{ex}(s,\pi_i^{(2m)},\wedge^2)^{-1} \}
\]
where the least common multiple is with respect to divisibility in $\mathbb{C}[q^{\pm s}]$ and is taken over all $m$ with $m=0,1, \dotsm, n-1$ and for each $m$ all constituents $\pi_i^{(2m)}$ of
$\pi^{(2m)}$. 
\end{theorem}

Diagrammatically:
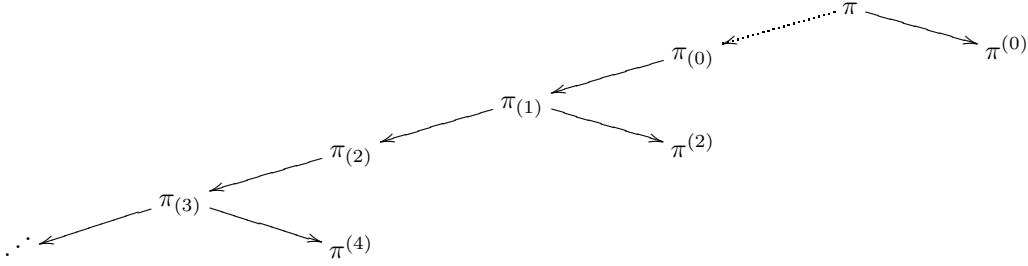
\[
  \xymatrix@C=3.5pc@L=.2pc@R=.2pc{ 
                            &                                       &                         &                                            &                                    &   \pi \ar@{.>}[dl] \ar[dr]   &               \\        
                             &                                      &                            &                                        &         \pi_{(0)} \ar[dl]     &                                        & \pi^{(0)} \\
                             &                                      &                           & \pi_{(1)}  \ar[dl] \ar[dr]      &                                     &                                        &\\                                                                           
                            &                                       & \pi_{(2)}  \ar[dl]   &                                        &         \pi^{(2)}                 &                                       & \\
                            &\pi_{(3)} \ar[dl]   \ar[dr]   &                            &                                       &                                        &                                       &  \\
                \iddots &                                       &  \pi^{(4)}             &                                       &                                        &                                       &  \\
   }
\]
where the leftward dotted and the first rightward arrow illustrate the restriction to $P_{2n}$ and identity functor resepctively, the rest of all leftward arrow represent an application of $\Phi^-$, and the remain of the rightward arrows an application of $\Psi^-$. The filtration $\mathcal{J}_{(2n-2)}(\pi)  \subset \dotsm \subset \left[ \mathcal{J}_{(2m+1)}(\pi) = \mathcal{J}_{(2m)}(\pi) \right] \subset \dotsm \subset\left[\mathcal{J}_{(1)}(\pi)=\mathcal{J}_{(0)}(\pi)\right] \subset \mathcal{J}(\pi)$ in $\mathbb{C}(q^{-s})$ for $0 \leq m \leq n-2$ is associated to all leftward arrow of functors. All rightward arrow of functors describe all even derivatives of $\pi$ which contribute all poles of local exterior square $L$-function $L(s,\pi,\wedge^2)$ for $GL_{2n}$.

\section{Factorization Formula of Exterior Square $L$-Functions for $GL_{2n+1}$}

\subsection{A filtration of $\mathbb{C}[q^{\pm s}]$-fractional ideals in $\mathbb{C}(q^{-s})$.}
 Let $\pi$ be an induced representation of $GL_{2n+1}$ of Langlands type. In this section, we will analyze the location of the poles of the family $\mathcal{J}(\pi)$ in terms of the filtration of fractional ideals of $\mathbb{C}[q^{\pm s}]$. The  $\mathbb{C}[q^{\pm s}]$-fractional ideals are defined by different families of integrals. The descending chain is convenient to relate the poles occurring on $\mathcal{J}(\pi)$ with exceptional poles of $L$-functions for odd derivatives. The manipulation of integrals resembles the proof for even case and we describe the necessary modification.

 \par
 We examine the case for $n=0$. Then we can drop the integration over $N_n\backslash GL_n$ and $\mathcal{N}_n\backslash \mathcal{M}_n$ in $J(s,W)$. Hence our integral representation for $GL_1$ can degenerates into $J(s,W)=W(1)$. The family of rational functions in $\mathcal{J}(\pi)$ are $\mathbb{C}[q^{\pm s}]$-linear combination of $J(s,W)=W(1)$, resulting in $\mathcal{J}(\pi)=\mathbb{C}[q^{\pm s}]$. 
The Jacquet-Shalika exterior square $L$-function of $\pi$ is the Euler factor which generate $\mathcal{J}(\pi)$. We therefore take $L(s,\pi,\wedge^2)$ to be $1$ for the exterior square $L$-function of $GL_1$. 

\begin{proposition}
 \label{GL_1}
 Let $\pi$ be a character of $GL_1$. Then $L(s,\pi,\wedge^2)=1$.
 \end{proposition}

 \par
 We assume that $n > 0$. We introduce different types of Jacquet-Shalika integrals in order to utilize the representation of smaller mirobolic subgroup $P_{2n+1-m}$ for $0 \leq m < 2n+1$.

\begin{definition*}
Let $n > 0$ and $W \in \mathcal{W}(\pi,\psi)$. For $0 \leq m < n$, we define the following integrals:
 \begin{enumerate}
\item[$(\mathrm{i})$] $J_{(2m)}(s,W)$
\[
\begin{split}
&=\int_{N_{n-m} \backslash {GL}_{n-m}}  \int_{\mathcal{N}_{n-m} \backslash \mathcal{M}_{n-m}} W \begin{pmatrix}  \sigma_{2n-2m+1}  \begin{pmatrix} I_{n-m} & X & \\ & I_{n-m} &\\ &&1  \end{pmatrix} \begin{pmatrix}  g&&  \\ &g& \\ &&1 \end{pmatrix} &\hspace*{-5mm}  \\&\hspace*{-5mm}  I_{2m} \end{pmatrix} \\
&\phantom{************************************} \psi^{-1}(\mathrm{Tr} X) dX\; |\mathrm{det}(g)|^{s-2m-1} dg
\end{split}
\]
 \end{enumerate}
and
 \begin{enumerate}
\item[$(\mathrm{ii})$] $ J_{(2m)}(s,W,\Phi)$
\[
\begin{split}
  &=\int_{N_{n-m} \backslash GL_{n-m}}  \int_{\mathcal{N}_{n-m} \backslash \mathcal{M}_{n-m}} \int_{F^{n-m}}\\ 
  &\phantom{*********} W \begin{pmatrix}  \sigma_{2n-2m+1}  \begin{pmatrix} I_{n-m} & X & \\ & I_{n-m} &\\ &&1  \end{pmatrix} \begin{pmatrix}  g&&  \\ &g& \\ &&1 \end{pmatrix}\begin{pmatrix}  I_{n-m} &&  \\ &I_{n-m} & \\ &y&1 \end{pmatrix} & \\& I_{2m} \end{pmatrix} \\
&\phantom{********************************} \Phi(y)dy \psi^{-1}(\mathrm{Tr} X) dX\; |\mathrm{det}(g)|^{s-2m-1} dg
\end{split}
\]
 \end{enumerate}
with $\Phi \in \mathcal{S}(F^{n-m})$. In addition, for $0 \leq m < n$ we set
\begin{enumerate}
\item[$(\mathrm{iii})$] $J_{(2m+1)}(s,W)$
\[
\begin{split}
& = \int_{N_{n-m} \backslash P_{n-m}}  \int_{\mathcal{N}_{n-m} \backslash \mathcal{M}_{n-m}} W \begin{pmatrix}  \sigma_{2n-2m}  \begin{pmatrix} I_{n-m} & X \\ & I_{n-m}  \end{pmatrix} \begin{pmatrix}  p&  \\ &p \end{pmatrix} & \\& I_{2m+1} \end{pmatrix} \\
&\phantom{************************************} \psi^{-1}(\mathrm{Tr} X) dX |\mathrm{det}(p)|^{s-2m-2} dp.
\end{split}
\]
\end{enumerate}
Let $\mathcal{J}_{(k)}(\pi)$ denote the span of rational functions defined by the integrals $J_{(k)}(s,W)$ for $0 \leq k \leq 2n-1$.
\end{definition*}

The proof of absolute convergence, rational functions in $\mathbb{C}(q^{-s})$ with common denominators, and quasi-invariance for theses integrals $J_{(2m)}(s,W)$, $J_{(2m)}(s,W,\Phi)$ and $J_{(2m+1)}(s,W)$ is similar to the proof in Section 4.1 of the even case. We summarize this in the following statement.

\begin{lemma}
Let $\pi$ be an induced representation of $GL_{2n+1}$ of Langlands type. For $0 \leq m < n$, $J_{(2m)}(s,W)$, $J_{(2m)}(s,W,\Phi)$ and $J_{(2m+1)}(s,W)$ converge absolutely for all $W \in \mathcal{W}(\pi,\psi)$ when $\mathrm{Re}(s)$ is sufficiently large, and define rational functions in $\mathbb{C}(q^{-s})$. For $0 \leq k \leq 2n-1$, $\mathcal{J}_{(k)}$ are $\mathbb{C}[q^{\pm s}]$-fractional ideals in $\mathbb{C}(q^{-s})$.
As a $\mathbb{C}[q^{\pm s}]$-fractional ideal, the family of integrals $J_{(2m)}(s,W)$ and $J_{(2m)}(s,W,\Phi)$ span the same ideal $\mathcal{J}_{(2m)}(\pi)$ in $\mathbb{C}(q^{-s})$. Furthermore the fractional ideal $\mathcal{J}_{(2m)}(\pi)$ is spanned by the integrals $J_{(2m)}(s,W)$ for $W$ in $\mathcal{W}(\pi_{(2m)},\psi)$.

\end{lemma}

We assume $n > 1$. The proof of the following equality of $\mathbb{C}[q^{\pm s}]$-fractional ideals is totally similar to the even case. Essentially, we exchange $I_{2m}$ to $I_{2m+1}$
and $I_{2m+1}$ to $I_{2m+2}$ in the lower right hand block without changing the integrations over $N_{n-m-1} \backslash GL_{n-m-1}$, $\mathcal{N}_{n-m-1} \backslash \mathcal{M}_{n-m-1}$, and $F^{n-m-1}$ of the integral $J_{(2m)}(s,W)$ on the proof of Proposition \ref{even-equal}. As in Proposition \ref{even-equal}, this establishes the following result.

\begin{proposition}
\label{odd-equal}
 Let $\pi$ be an induced representation of $GL_{2n+1}$ of Langlands type. As $\mathbb{C}[q^{\pm s}]$-fractional ideals in $\mathbb{C}(q^{-s})$, we have the equalities $\mathcal{J}_{(2m+2)}(\pi)=\mathcal{J}_{(2m+1)}(\pi) $ for $n > 1$ and $0 \leq m < n-1$.
\end{proposition}

We assume $n > 0$ again. We proceed by repeating the steps on the proof of Proposition \ref{J2m-char} with replacing $I_{2m-1}$ by $I_{2m}$ of the integral $J_{(2m-1)}(s,W)$ and $I_{2m}$ by $I_{2m+1}$ of the integral $J_{(2m)}(s,W)$ in the lower right hand block. We do not reduce but remain the integrations over $K_{n-m}$, $Z_{n-m}$, $N_{n-m} \backslash P_{n-m}$ and $\mathcal{N}_{n-m} \backslash \mathcal{M}_{n-m}$
of $J_{(2m-1)}(s,W)$, and the same integrations over $N_{n-m} \backslash GL_{n-m}$ and $\mathcal{N}_{n-m} \backslash \mathcal{M}_{n-m}$ of $J_{(2m)}(s,W)$. Combining Proposition \ref{odd-equal} gives rise to the following.

 \begin{proposition}
 \label{filtration-odd}
Let $\pi$ be an induced representation of $GL_{2n+1}$ of Langlands type. Let $\mathcal{J}_{(0)}(\pi)$ denote $\mathcal{J}(\pi)$. For $n > 1$ and $0 < m < n$, we have the filtration 
\[
  \mathcal{J}_{(2n-1)}(\pi)  \subset\hspace*{-1mm} \left[\mathcal{J}_{(2n-2)}(\pi)=\mathcal{J}_{(2n-3)}(\pi)\right] \hspace*{-1mm}\subset \hspace*{-1mm}  \dotsm\hspace*{-1mm} \subset \hspace*{-1mm}\left[ \mathcal{J}_{(2m)}(\pi) = \mathcal{J}_{(2m-1)}(\pi) \right]\hspace*{-1mm} \subset\hspace*{-1mm} \dotsm\hspace*{-1mm} \subset \mathcal{J}_{(0)}(\pi)
\]
 of $\mathbb{C}[q^{\pm s}]$-fractional ideals in $\mathbb{C}(q^{-s})$. In addition for $n=1$, we have the inclusion $ \mathcal{J}_{(1)}(\pi)   \subset \mathcal{J}_{(0)}(\pi)$
of $\mathbb{C}[q^{\pm s}]$-fractional ideals in $\mathbb{C}(q^{-s})$. 
The bottom of this filtration $\mathcal{J}_{(2n-1)}(\pi)$ is $\mathbb{C}[q^{\pm s}]$.
For $n > 0$ and $0 \leq m < n$, the fractional ideal $\mathcal{J}_{(2m+1)}(\pi)$ is linearly spanned by the integral $J_{(2m)}(s,W)$ for $W \in \mathcal{W}(\pi_{(2m),2},\psi)$.

\end{proposition}

\subsection{The exceptional poles of the odd derivatives.} 
In this section, we study the link between the exceptional poles of $\mathcal{J}(\pi^{(2m+1)})$ and the inclusion of fractional ideals $\mathcal{J}_{(2m+1)}(\pi) \subset \mathcal{J}_{(2m)}(\pi)$ in $\mathbb{C}(q^{-s})$.

\par
 Assume that $n > 0$. Let $\pi=\mathrm{Ind}(\Delta_1 \otimes \dotsm \otimes \Delta_t)$ be an induced representation of $GL_{2n+1}$ of Langlgnads. Let $0 \leq k \leq 2n-1$. Since the fractional ideal $\mathcal{J}_{(k)}(\pi)$ is linearly spanned by the $J_{(k)}(s,W)$ it will be sufficient to understand the poles of these integrals. Suppose that there is a rational function $J_{(k)}(s,W)$ with $W \in \mathcal{W}(\pi,\psi)$ in $\mathcal{J}_{(k)}(\pi)$ having a pole at $s=s_0$ of order $d$ and that this is the maximal order with which the pole $s=s_0$ occurs in the family. Then this will have a Laurent expansion at $s=s_0$ of the form
\[
J_{(k)}(s,W)=\frac{\Lambda_{(k),s_0}(W)}{(q^s-q^{s_0})^d}+\mathrm{higher\; order\; terms.}
\]
By a realization of Whittaker model for $\pi_{(k)}$ in Proposition \ref{mira-model}, as the integral $J_{(k)}(s,W)$ only depends on the functions $W$ through their restriction to $P_{2n+1-k}$, $\Lambda_{(k),s_0}$ belongs to $\mathrm{Hom}_{P_{2n+1-k} \cap S_{2n+1-k}}(\mathcal{W}(\pi_{(k)},\psi), |\cdot|^{-\frac{s_0-1}{2}}\Theta)$.

\begin{definition*}
For $0 \leq k \leq 2n-1$, a linear functional $\Lambda_{(k),s}$ on $\mathcal{W}(\pi_{(k)},\psi)$ is called a twisted Shalika functional with respect to a Shlika subgroup $P_{2n+1-k} \cap S_{2n+1-k}$ if it satisfies the following quasi-invariance 
\[
\Lambda_{(k),s}(\pi_{(k)}(\begin{pmatrix} h &\\ & I_{k} \end{pmatrix})W)=|\mathrm{det}(h)|^{-\frac{s-1}{2}}\Theta(h)\Lambda_{(k),s}(W)
\]
 for all $h \in P_{2n+1-k} \cap S_{2n+1-k}$.
\end{definition*}

To proceed, we would make the following assumption.

\begin{assumption}
For $n > 0$, let $\pi=\mathrm{Ind}(\Delta_1 \otimes \dotsm \otimes \Delta_t)$ be an irreducible generic representation of $GL_{2n+1}$. For $0 \leq m \leq 2n+1$, we assume that all derivatives $\pi^{(m)}$ of $\pi$ are completely reducible with generic subquotients of the form $\pi_{i}^{(m)}=\mathrm{Ind}(\Delta_1^{(m_1)} \otimes \dotsm \otimes \Delta_t^{(m_t)})$ with $m=m_1+\dotsm+m_t$.  
\end{assumption}

We will remove this condition in Section $7$ by a deformation argument. Under this assumption, we essentially repeat the proof of Theorem \ref{prod-L0} to deduce the following statement.

\begin{theorem}
\label{prod-L-odd}
For $n > 0$, let $\pi$ be an irreducible generic representation of $GL_{2n+1}$ such that all derivatives $\pi^{(m)}$ of $\pi$ are completely reducible with generic Jordan--H\"older constituents of the form $\pi_{i}^{(m)}=\mathrm{Ind}(\Delta_1^{(m_1)} \otimes \dotsm \otimes \Delta_t^{(m_t)})$ with $m=m_1+\dotsm+m_t$. For each $m$, $i$ is indexing the partition of $m$. Then
\[
L(s,\pi,\wedge^2)^{-1}=l.c.m_{m,i}\{ L_{ex}(s,\pi_i^{(2m+1)},\wedge^2)^{-1} \}
\]
where the least common multiple is with respect to divisibility in $\mathbb{C}[q^{\pm s}]$ and is taken over all $m$ with $m=0,1, \dotsm, n-1$ and for each $m$ all constituents $\pi_i^{(2m)}$ of
$\pi^{(2m+1)}$. 
\end{theorem}

\begin{proof}
The proof is the same up to the following extra argument that any pole at $s=s_0$ of order $d$ of $L(s,\pi_i^{(2m+1)},\wedge^2)$ for any $m$ and $i$ occur with the same pole of order at least $d$ of the reciprocal of $l.c.m_{k,\ell}\{ L_{ex}(s,\pi_\ell^{(2k+1)},\wedge^2)^{-1} \}$, where $k$ takes over all $0 \leq k < n$ and $\pi_\ell^{(2k)}$ are irreducible constituents of $\pi^{(2k)}$. 
\par
Suppose that $L(s,\pi_i^{(2m+1)},\wedge^2)$ has a pole of order $d$ at $s=s_0$ for some $i$, and $m$ between $0$ and $n-1$. In light of Theorem \ref{prod-L}, we have equality of $L$-functions $L(s,\pi_i^{(2m+1)},\wedge^2)^{-1}=l.c.m_{k,j_i}\{ L_{ex}(s,(\pi_i^{(2m+1)})_{j_i}^{(2k)},\wedge^2)^{-1} \}$ where the least common multiple is taken over all $k$ with $0 \leq k < n-m$ and for each $k$ all constituents $(\pi_i^{(2m+1)})_{j_i}^{(2k)}$ of $(\pi_i^{(2m+1)})^{(2k)}$. However as all irreducible constituents $(\pi_i^{(2m+1)})_{j_i}^{(2k)}$ appear amongst those $\pi_{\ell}^{(2m+1+2k)}$ of the derivatives $\pi^{(2m+1+2k)}$ of $\pi$, $L(s,\pi_i^{(2m+1)},\wedge^2)^{-1}$ divides $l.c.m_{k,\ell}\{ L_{ex}(s,\pi_{\ell}^{(2k+1)},\wedge^2)^{-1} \}$ in $\mathbb{C}[q^{\pm s}]$, where $k$ runs over $0 \leq k < n$ and $\pi_{\ell}^{(2k+1)}$ are irreducible constituents of $\pi^{(2k+1)}$. Therefore the reciprocal of $l.c.m_{k,\ell}\{ L_{ex}(s,\pi_\ell^{(2k+1)},\wedge^2)^{-1} \}$ has the pole of order at least $d$ at $s=s_0$. This completes the proof.
\end{proof}

 If we interpret the above statement diagrammatically, we obtain
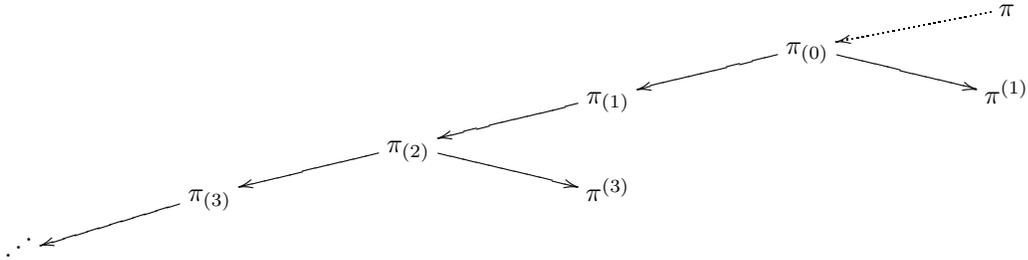
\[
  \xymatrix@C=4.4pc@L=.2pc@R=.2pc{    
                              &                                       &                                   &                                            &                                         &   \pi \ar@{.>}[dl]   \\ 
                             &                                      &                                      &                                        &         \pi_{(0)} \ar[dl]  \ar[dr] &                   \\
                             &                                      &                                      & \pi_{(1)}  \ar[dl]                &                                           & \pi^{(1)}   \\                                                                           
                            &                                       & \pi_{(2)}  \ar[dl]  \ar[dr] &                                        &                                            &                \\
                            &\pi_{(3)} \ar[dl]                &                                     &   \pi^{(3)}                        &                                              &                 \\
                \iddots &                                       &                                    &                                       &                                              &  
  }
\]
where the leftward dotted arrow illustrates the restriction to $P_{2n+1}$, the rest of all leftward arrow represent an application of $\Phi^-$, and the rightward arrows an application of $\Psi^-$. The filtration    $\mathcal{J}_{(2n-1)}(\pi)  \subset \left[\mathcal{J}_{(2n-2)}(\pi)=\mathcal{J}_{(2n-3)}(\pi)\right]\subset  \dotsm \subset \left[ \mathcal{J}_{(2m)}(\pi) = \mathcal{J}_{(2m-1)}(\pi) \right] \subset \dotsm \subset \mathcal{J}_{(0}(\pi)$ in $\mathbb{C}(q^{-s})$ for all $0 \leq m \leq n-1$ is associated to all leftward arrow of functors. All rightward arrow of functors describe all odd derivatives of $\pi$ which contribute all poles of local exterior square $L$-function $L(s,\pi,\wedge^2)$ for $GL_{2n+1}$. 
\par
We observe that all the poles of the exterior square $L$-function $L(s,\pi,\wedge^2)$ for $GL_m$ are completely determined by the exceptional poles of derivatives which are representations of the smaller group $GL_{2k}$, where $2k \leq m$ is an even integer, regardless of the even or odd case $GL_m$.

\subsection{A filtration of $\mathbb{C}[q^{\pm s}]$-fractional ideals in $\mathbb{C}[q^{\pm s}]$ for Rankin-Selberg convolution}

In this paragraph, we recall the notations from \cite{CoPe,JaPSSh83} and deduce analogous filtration of $\mathbb{C}[q^{\pm s}]$-fractional ideals associated with Rankin-Selberg convolution $L$-functions.
Let $(\pi ,V_{\pi})$ be an irreducible generic representation of $GL_n$, and $(\sigma, V_{\sigma})$ an irreducible generic representation of $GL_m$. We do not require the assumption that $\pi$ and $\sigma$ are completely reducible. We do the case $n=m$, the case $m < n$ being similar. For each $W \in \mathcal{W}(\pi_{(n-k-1)},\psi)$ and $W^{\prime} \in \mathcal{W}(\sigma_{(n-k-1)}, \psi^{-1})$, we associate an integral
\[
I_{(n-k-1)}(s\; ; W, W^{\prime})=\int_{N_k \backslash GL_k} W\begin{pmatrix}g& \\ & I_{n-k} \end{pmatrix}W^{\prime}\begin{pmatrix}g& \\ & I_{n-k} \end{pmatrix}  |\mathrm{det}(g)|^{s-(n-k)} \;dg
\]
with $0 \leq k < n$. Let us denote by $\mathcal{I}_{(n-k-1)}(\pi,\sigma)$ the span of the rational functions defined by integrals $I_{(n-k-1)}(s\; ; W, W^{\prime})$ as in Section 2.2 of \cite{CoPe}. Recall from Proposition \ref{asym} that the Whittaker functions $W \in \mathcal{W}(\pi_{(n-k-1),2},\psi)$ are characterized by the fact if we view them as functions on $GL_{k}$ by $W\begin{pmatrix}g& \\ & I_{n-k} \end{pmatrix}$ that their support in the last rows of $g$ is compact, modulo $N_k$. If we use a partial Iwasawa decomposition to $g \in GL_{k}$, for a fixed $W \in \mathcal{W}(\pi_{(n-k-1),2},\psi)$ and $W^{\prime}  \in \mathcal{W}(\sigma_{(n-k-1),2}, \psi^{-1})$, our integral $I_{(n-k-1)}(s\; ; W, W^{\prime})$ becomes a finite sum of the form
\[
\begin{split}
 &I_{(n-k-1)}(s\; ; W, W^{\prime})\\
 &=\sum_i c q^{-\beta_i s} \int_{N_{k-1} \backslash GL_{k-1}} W_i \begin{pmatrix}g& \\ & I_{n-k+1} \end{pmatrix}W_i^{\prime}\begin{pmatrix}g& \\ & I_{n-k+1} \end{pmatrix}  |\mathrm{det}(g)|^{s-(n-k+1)} \;dg
\end{split}
 \]
with $c > 0$ a volume term. As both $W$ and $W^{\prime}$ enter into these integrals through their restriction to $GL_{k-1}$, they only depend on the images of $W$ and $W^{\prime}$ in $\mathcal{W}(\pi_{(n-k)},\psi)$ and $\mathcal{W}(\sigma_{(n-k)},\psi^{-1})$ respectively. Moreover, by Lemma 9.2 of \cite{JaPSSh83}, we can see that each $I_{(n-k)}(s\; ; W_i, W_i')$ actually occurs as a $I_{(n-k+1)}(s;W_{\circ},W'_{\circ})$ for appropriate choice of $W_{\circ}$ and $W'_{\circ}$ in $\mathcal{W}(\pi_{(n-k+1)},\psi)$ and $\mathcal{W}(\sigma_{(n-k+1)},\psi^{-1})$ respectively. If we combine Proposition 2.2 in \cite{CoPe} with our analysis of the $\mathbb{C}[q^s,q^{-s}]$-fractional ideal, we have the filtration $\mathcal{I}_{(n-1)}(\pi,\sigma) \subset \mathcal{I}_{(n-2)}(\pi,\sigma) \subset \dotsm \subset \mathcal{I}_{(k)}(\pi,\sigma) \subset \dotsm \subset \mathcal{I}_{(0)}(\pi,\sigma) \subset \mathcal{I}(\pi,\sigma)$ in $\mathbb{C}(q^{-s})$ for $0 \leq k < n$, where $\mathcal{I}(\pi,\sigma)$ is spanned by the integral
\[
I(s\; ; W, W^{\prime},\Phi)=\int_{N_n \backslash GL_n} W(g)W^{\prime}(g) \Phi(e_ng) |\mathrm{det}(g)|^{s} \;dg
\]
with $W \in \mathcal{W}(\pi,\psi)$, $W^{\prime} \in \mathcal{W}(\sigma, \psi^{-1})$ and $\Phi \in \mathcal{S}(F^n)$. The filtration associated with the Rankin-Selberg integrals is analogous to the one related with the Jacquet-Shalika integrals for exterior square $L$-functions. Diagrammatically,
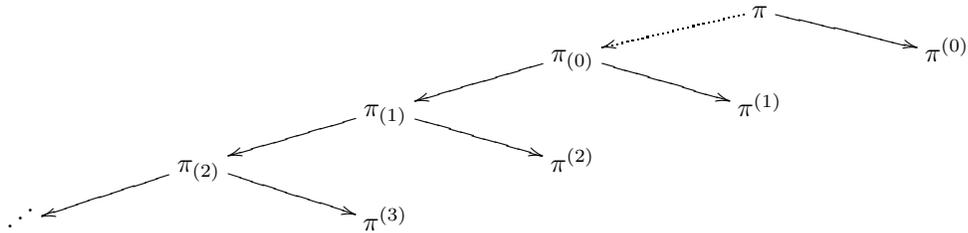
\[
  \xymatrix@C=4pc@L=.2pc@R=.2pc{    
                                                    &                                       &                                        &                                            &\pi \ar@{.>}[dl] \ar[dr] &     \\
                                                    &                                      &                                        &         \pi_{(0)} \ar[dl]  \ar[dr] &                            & \pi^{(0)}    \\
                                                    &                                      & \pi_{(1)}  \ar[dl] \ar[dr]     &                                           & \pi^{(1)}                &      \\                                                                           
                                                     & \pi_{(2)}  \ar[dl]  \ar[dr] &                                        &     \pi^{(2)}                            &                            &   \\
                  \iddots                         &                                     &   \pi^{(3)}                        &                                              &                           &     \\
  }
\]
where the leftward dotted and the first rightward arrow illustrate the restriction to $P_{n}$ and identity functor respectively, the rest of all leftward arrow represent an application of $\Phi^-$, and the remain of the rightward arrows an application of $\Psi^-$. This diagram is exactly the same for $\sigma$. The filtration $\mathcal{I}_{(n-1)}(\pi,\sigma) \subset \mathcal{I}_{(n-2)}(\pi,\sigma) \subset \dotsm \subset \mathcal{I}_{(k)}(\pi,\sigma) \subset \dotsm \subset \mathcal{I}_{(0)}(\pi,\sigma) \subset \mathcal{I}(\pi,\sigma)$ in $\mathbb{C}(q^{-s})$ for $0 \leq k < n$ is associated to all left arrow of functors. All rightward arrow of functors describe all derivatives of $\pi$ and $\sigma$ which contribute all poles of Rankin-Selberg convolution $L$-function $L(s,\pi \times \sigma)$ for $GL_n$.

\section{Computation of the Local Exterior Square $L$-Functions}

\subsection{Exterior square $L$-functions for supercuspidal representations}
 Let $Q=MN_Q$ be a proper standard parabolic subgroup of $GL_r$, where $M$ is the Levi subgroup of $Q$ and $N_Q$ the unipotent radical of $Q$. We know from Bernstein and Zelevinsky \cite{BeZe76} that an admissible representation $(\rho,V_{\rho})$ of $GL_r$ is called supercuspidal if it is killed by all Jacquet functors, that is, $r_M(\rho)=0$ for all $M$ where $r_M(\rho)$ is the natural representation of $M$ on $V_{\rho} / V_{\rho}(N_Q,\bf{1})$. In this section, we compute the $L$-functions $L(s,\rho,\wedge^2)$ for irreducible supercuspidal representations $\rho$ by means of the derivative theory developed by Bernstein and Zelevinsky.

\par

 We recall Theorem 4.2 in \cite{Ke11} which gives a characterization of poles. The proof of Theorem 4.2 in \cite{Ke11} follows the computation appearing in the proof of Ginzburg, Rallis, and Soudry \cite[Section 3.2, Theorem 2]{GiRaSo99}.

\begin{theorem}[Kewat]
\label{Kewat}
Let $\pi$ be an irreducible square integrable representation of $GL_{2n}$. Then $L(s,\pi,\wedge^2)$ has a pole at $s=0$ if and only if $\omega_{\pi}$ is trivial and there exists $W \in \mathcal{W}(\pi,\psi)$ such that $\Lambda(W) \neq 0$, where
\[
   \Lambda(W)=\int_{Z_nN_n \backslash GL_n} \int_{\mathcal{N}_n \backslash \mathcal{M}_n} W \left( \sigma_{2n} \begin{pmatrix} I_n & X\\ & I_n \end{pmatrix}
   \begin{pmatrix} g& \\ & g \end{pmatrix} \right) \psi^{-1}(\mathrm{Tr}X) dX dg.
\]

\end{theorem}

In fact, we can also characterize the poles at $s=0$ of the exterior square $L$-functions $L(s,\Delta,\wedge^2)$ for an irreducible quasi-square integrable representation $\Delta$ by $\mathrm{Hom}_{S_{2n}}(\Delta,\Theta)\neq 0$ in Proposition \ref{prop-supercusp}. We have the analogue result of Proposition 3.6 in \cite{Ma10} for $L(s,\rho,\wedge^2)$, where $\rho$ is an irreducible supercuspidal representation of $GL_r$.

 \begin{theorem}
\label{supercusp}
Let $\rho$ be an irreducible supercuspidal representation of $GL_r$. 
\begin{enumerate}
\item[$(\mathrm{i})$] If $r=2n$, $L(s,\rho,\wedge^2)$ has simple poles and we obtain
\[
   L(s,\rho,\wedge^2)=L_{ex}(s,\rho,\wedge^2)=\prod (1-\alpha q^{-s})^{-1}
\]
with the product over all $\alpha=q^{s_0}$ such that $\omega_{\rho}\nu^{ns_0}=1$ and $\Lambda_{s_0}(W) \neq 0$ for some $W \in \mathcal{W}(\rho,\psi)$, where
\[
\begin{split}
   &\Lambda_{s_0}(W)\\
   &=\int_{Z_nN_n \backslash GL_n} \int_{\mathcal{N}_n \backslash \mathcal{M}_n} W \left( \sigma_{2n} \begin{pmatrix} I_n & X\\ & I_n \end{pmatrix}
   \begin{pmatrix} g& \\ & g \end{pmatrix} \right) \psi^{-1}(\mathrm{Tr}X) |\mathrm{det} (g)|^{s_0} dX dg.
\end{split}
\]
\item[$(\mathrm{ii})$] If $r=2n+1$, we have $L(s,\rho,\wedge^2)=1$.

\end{enumerate} 

\end{theorem}

 \begin{proof}
 Let $\rho$ be an irreducible supercuspidal representation of $GL_r$. The derivatives of supercuspidal representations have been computed by Bernstein and Zelevinsky. $\rho^{(0)}=\rho, \rho^{(k)}=0$ for $1 \leq k \leq r-1$ and $\rho^{(r)}=\textbf{1}$ in Theorem \ref{deriviative}.
 \par
 We first investigate the case that $r=2n$. As we have seen in Theorem \ref{prod-L0}, since $\rho^{(i)}=0$ unless $i=0,\;2n$, $L_{(0)}(s,\rho,\wedge^2)=1$. Applying Proposition \ref{L0 function}, $L(s,\rho,\wedge^2)=L^{(0)}(s,\rho,\wedge^2)=L_{ex}(s,\rho,\wedge^2)$ and $L(s,\rho,\wedge^2)=L^{(0)}(s,\rho,\wedge^2)$ has simple poles. As $L(s+s_0,\rho,\wedge^2)$ is equal to $L(s,\rho\nu^{\frac{s_0}{2}},\wedge^2)$,
$L(s,\rho,\wedge^2)$ has a pole at $s=s_0$ if and only if $L(s,\rho\nu^{\frac{s_0}{2}},\wedge^2)$ has a pole at $s=0$. 
\par
We assume that $L(s,\rho,\wedge^2)$ has a pole at $s=s_0$. Since $s=0$ is an exceptional pole of $L(s,\rho\nu^{\frac{s_0}{2}},\wedge^2)$, we see from Section 3.1 that a nonzero Shalika functional $\Lambda$ exists, and $\Lambda$ belongs to $\text{Hom}_{S_{2n}}(\rho\nu^{\frac{s_0}{2}},\Theta)$. This implies that $\rho\nu^{\frac{s_0}{2}}$ has necessarily a trivial central character $\omega_{\rho\nu^{\frac{s_0}{2}}}$, because
\[
  \omega_{\rho\nu^{\frac{s_0}{2}}}(z)\Lambda(W)=\Lambda(\rho\nu^{\frac{s_0}{2}}(zI_{2n})W)=\Theta(zI_{2n})\Lambda(W)=\Lambda(W)
\]
for all $zI_{2n} \in Z_{2n}$ and $W \in \mathcal{W}(\rho\nu^{\frac{s_0}{2}},\psi)$. $\rho\nu^{\frac{s_0}{2}}$ is indeed an irreducible unitary supercuspidal representation. As $\mathcal{W}(\rho\nu^{\frac{s_0}{2}},\psi) \simeq \mathcal{W}(\rho,\psi) \otimes \nu^{\frac{s_0}{2}}$, we can deduce from Theorem \ref{Kewat} that there exists $W \in \mathcal{W}(\rho,\psi)$ such that 
\[
  \int_{Z_nN_n \backslash GL_n} \int_{\mathcal{N}_n \backslash \mathcal{M}_n} W \left( \sigma_{2n} \begin{pmatrix} I_n & X\\ & I_n \end{pmatrix}
   \begin{pmatrix} g& \\ & g \end{pmatrix} \right) \psi^{-1}(\mathrm{Tr}X) |\mathrm{det} (g)|^{s_0} dX dg \neq 0,
\] 
 because $\left|\mathrm{det} \left(\sigma_{2n} \begin{pmatrix} I_n & X\\ & I_n \end{pmatrix}
   \begin{pmatrix} g& \\ & g \end{pmatrix} \right)\right|^{\frac{s_0}{2}}= |\mathrm{det} (g)|^{s_0}$. Conversely, we suppose that 
 \[
  \int_{Z_nN_n \backslash GL_n} \int_{\mathcal{N}_n \backslash \mathcal{M}_n} W \left( \sigma_{2n} \begin{pmatrix} I_n & X\\ & I_n \end{pmatrix}
   \begin{pmatrix} g& \\ & g \end{pmatrix} \right) \psi^{-1}(\mathrm{Tr}X) |\mathrm{det} (g)|^{s_0} dX dg \neq 0
\]
 for some $W \in \mathcal{W}(\rho,\psi)$ and $\omega_{\rho}\nu^{ns_0}=1$. As the central character  $\omega_{\rho\nu^{\frac{s_0}{2}}}$ of $\rho\nu^{\frac{s_0}{2}}$ is trivial, $\rho\nu^{\frac{s_0}{2}}$ is again an irreducible unitary supercuspidal representation. 
 If, for any $W \in \mathcal{W}(\rho,\psi)$, we set $(W\otimes \nu^{\frac{s_0}{2}})(g)=W(g) \nu^{\frac{s_0}{2}}(g), g \in GL_{2n}$, we see that $W\otimes \nu^{\frac{s_0}{2}}$ is an element of $\mathcal{W}(\rho\nu^{\frac{s_0}{2}},\psi)$ such that $\Lambda(W\otimes \nu^{\frac{s_0}{2}}) \neq 0$, where $\Lambda$ is defined in Theorem \ref{Kewat}. We apply Theorem \ref{Kewat} to conclude that $L(s,\rho\nu^{\frac{s_0}{2}},\wedge^2)$ has a pole at $s=0$, which means that $L(s,\rho,\wedge^2)$ has a pole at $s=s_0$.

 \par
 Now we assume that $r=2n+1$. As we have seen, $\rho^{(i)}=0$ unless $i=0,\;2n+1$. Taking into account Theorem \ref{prod-L-odd}, we have $L(s,\rho,\wedge^2)=1$.

 \end{proof}

The following result about Shalika functionals of Jacquet and Rallis in \cite[Proposition 6.1]{JaRa96} plays a crucial rule in our study of poles and representations.

\begin{theorem}[Jacquet and Rallis]
\label{Jacquet-Rallis}
Let $\pi$ be an irreducible admissible representation of $GL_{2n}$. Then the dimension of the space of Shalika functionals is at most one.
If $\pi$ has a non-zero Shalika functional, then $\pi \simeq \widetilde{\pi}$. 
\end{theorem}

We would like to give a characterization of the existence of the local twisted Shalika functional on the supercuspidal representations in terms of the occurrence of exceptional poles of local exterior square $L$-functions. More generally, this property is true under the following condition, which are satisfied by irreducible generic unitary representations.

\begin{lemma}
\label{char-Shalika}
Let $\pi$ be an irreducible generic representation of $GL_{2n}$. Suppose that the space $\mathrm{Hom}_{S_{2n}}(\pi,\Theta) \neq 0$ and $\mathrm{Hom}_{P_{2n} \cap S_{2n}}(\pi,\Theta)$ is of dimension at most $1$. Then $L_{ex}(s,\pi,\wedge^2)$ has a pole at $s=0$.
\end{lemma}

\begin{proof}
The proof is similar to those of Theorem 1.4 in \cite{AnKaTa04}, Theorem 3.1 in \cite{Ma10}, or Proposition 4.7 in \cite{Ma15}. Theorem \ref{Jacquet-Rallis} of Jacquet and Rallis shows that  the vector space $\mathrm{Hom}_{S_{2n}}(\pi,\Theta)$ is of at most dimension one. Since any non-trivial $S_{2n}$-quasi invariant form in $\mathrm{Hom}_{S_{2n}}(\pi,\Theta)$ is, in particular, a non-trivial $P_{2n}\cap S_{2n}$-quasi invariant form in $\mathrm{Hom}_{P_{2n}\cap S_{2n}}(\pi,\Theta)$ by restriction, the $1$-dimensionality for both spaces $\mathrm{Hom}_{S_{2n}}(\pi,\Theta)$ and $\mathrm{Hom}_{P_{2n}\cap S_{2n}}(\pi,\Theta)$ imply that $\mathrm{Hom}_{S_{2n}}(\pi,\Theta)=\mathrm{Hom}_{P_{2n}\cap S_{2n}}(\pi,\Theta)$. 
\par

Since $\pi$ has a central character and $\mathrm{Hom}_{S_{2n}}(\pi,\Theta)$ is non-zero by assumption, $\pi$ has necessarily a trivial central character, because $\omega_{\pi}(z) \Lambda(W)=\displaystyle \Lambda ( \pi( zI_{2n} )W)=\Theta (zI_{2n}) \Lambda(W)=\Lambda(W)$ for all $zI_{2n} \in Z_{2n}$, and $W \in \mathcal{W}(\pi,\psi)$ from the $S_{2n}$-intertwining property. For $\mathrm{Re}(s) \ll 0$, by applying the partial  Iwasawa decomposition in Section 2.2, one has  
\[
\label{Pn-iwasawa}
\tag{6.1}
\begin{split}
J(1-s, \varrho(\tau_{2n})\widetilde{W},\hat{\Phi})&=\int_{K_n} \int_{N_n \backslash P_n} \int_{\mathcal{N}_n \backslash \mathcal{M}_n} \varrho(\tau_{2n})\widetilde{W}
 \left( \sigma_{2n} \begin{pmatrix} I_n & X \\ & I_n \end{pmatrix} \begin{pmatrix} pk & \\ & pk \end{pmatrix} \right)\\
 &\phantom{********} \int_{F^{\times}} |a|^{n(1-s)} \hat{\Phi}(e_nak) d^{\times}a \;  |\mathrm{det}(p)|^{-s} \psi(\mathrm{Tr} X) dX dp dk,
 \end{split}
\]
where $\varrho$ is right translation and $\tau_{2n}$ is $\begin{pmatrix} & I_{n}\\I_{n} & \end{pmatrix}$. We define a functional on $\mathcal{W}(\pi,\psi)$
\[
\Lambda_{1-s,\widetilde{\pi}} : W \mapsto \frac{J_{(0)}(1-s,\varrho(\tau_{2n})\widetilde{W})}{L(1-s,\widetilde{\pi},\wedge^2)}.
\]
In the realm of convergence, we can find
\[
\begin{split}
  \Lambda_{1-s,\widetilde{\pi}}(\varrho(h)W)&=\frac{J_{(0)}(1-s,\varrho(\tau_{2n})\widetilde{\varrho(h)W})}{L(1-s,\widetilde{\pi},\wedge^2)}
  =\frac{J_{(0)}(1-s,\varrho(\tau_{2n})\varrho(^t{h}^{-1})\widetilde{W})}{L(1-s,\widetilde{\pi},\wedge^2)} \\
  &=\frac{J_{(0)}(1-s,\varrho(^t{h}^{-\tau_{2n}})\varrho(\tau_{2n})\widetilde{W})}{L(1-s,\widetilde{\pi},\wedge^2)}
  =|h|^{-\frac{s}{2}}\Theta(h)\frac{J_{(0)}(1-s,\varrho(\tau_{2n})\widetilde{W})}{L(1-s,\widetilde{\pi},\wedge^2)}\\
  &=|h|^{-\frac{s}{2}}\Theta(h)\Lambda_{1-s,\widetilde{\pi}}(W),
\end{split}
\]
the fourth equality from \eqref{Pn-invariant}, where $h$ belongs to $P_{2n} \cap S_{2n}$ and $^t{h}^{-\tau_{2n}}$ is $^t(h^{-1})^{\tau_{2n}}$. $\Lambda_{1-s,\widetilde{\pi}}$ is an element of $\mathrm{Hom}_{P_{2n} \cap S_{2n}}(\pi,|\cdot|^{-\frac{s}{2}}\Theta)$.
Observe that the local exterior square $L$-function is the minimal function of the form $P(q^{-s})^{-1}$, with $P(X)$ a polynomial satisfying $P(0)=1$, such that the ratios
$\displaystyle \frac{J(1-s, \varrho(\tau_{2n})\widetilde{W},\hat{\Phi})}{L(1-s,\widetilde{\pi},\wedge^2)}$ are polynomials in $\mathbb{C}[q^{\pm s}]$ and hence entire functions of $s$ for all $\widetilde{W} \in \mathcal{W}(\widetilde{\pi},\psi^{-1})$ and $\hat{\Phi} \in \mathcal{S}(F^n)$. Let $K_{n,r} \subset K_n$ be a compact open congruent subgroup which stabilizes $\varrho(\tau_{2n})\widetilde{W}$. We take $\hat{\Phi}_{\circ} \in \mathcal{S}_0(F^n)$ the characteristic function of $e_nK_{n,r}$. With this choice of $\hat{\Phi}_{\circ}$, we have $J(1-s, \varrho(\tau_{2n})\widetilde{W},\hat{\Phi}_{\circ})=\mathrm{Vol}(e_nK_{n,r})J_{(0)}(1-s, \varrho(\tau_{2n})\widetilde{W})$ from \eqref{zp-invariance}. This implies that for a fixed $W$ in $\mathcal{W}(\pi,\psi)$, the functions $s \mapsto \Lambda_{1-s,\widetilde{\pi}}(W)$ is a polynomial in $\mathbb{C}[q^{\pm s}]$ and therefore entire function in $s$. We may write \eqref{Pn-iwasawa} as
\begin{equation}
\label{Iwasawa-distinct}
\tag{6.2}
\begin{split}
  &\frac{J(1-s,\varrho(\tau_{2n})\widetilde{W},\hat{\Phi})}{L(1-s,\widetilde{\pi},\wedge^2)}\\
  &\phantom{********}=\int_{K_n} \frac{J_{(0)}(1-s, \varrho\begin{pmatrix} k &\\&k \end{pmatrix}\varrho(\tau_{2n})\widetilde{W})}{L(1-s,\widetilde{\pi},\wedge^2)} \int_{F^{\times}} |a|^{n(1-s)} \hat{\Phi}(e_nak)  d^{\times}a dk.
\end{split}
\end{equation}
Let $K_{\circ} \subset K_n$ be a sufficiently small compact open subgroup which stabilizes $\varrho(\tau_{2n})\widetilde{W}$ and $\hat{\Phi}$. Write $K_n=\cup_ik_iK_{\circ}$ and let $\widetilde{W}_i=\varrho\begin{pmatrix} k_i&\\&k_i \end{pmatrix}\varrho(\tau_{2n})\widetilde{W}$ and $\hat{\Phi}_i=R\begin{pmatrix} k_i&\\&k_i \end{pmatrix}\hat{\Phi}$. Then the second member of the equality in \eqref{Iwasawa-distinct} is actually a finite sum:
\[
 c \sum_i \frac{J_{(0)}(1-s,\widetilde{W}_i)}{L(1-s,\widetilde{\pi},\wedge^2)} \int_{F^{\times}} |a|^{n(1-s)} \hat{\Phi}_i(e_na) d^{\times}a,
\]
with $c > 0$ the volume of $K_{\circ}$.
Notice that for $\mathrm{Re}(s) < 1$, the Tate integral
\[
  \int_{F^{\times}} |a|^{n(1-s)} \hat{\Phi}_i(e_na) d^{\times}a
\]
is absolutely convergent, and the ratios
\[
 \frac{J(1-s,\varrho(\tau_{2n})\widetilde{W},\hat{\Phi})}{L(1-s,\widetilde{\pi},\wedge^2)} \quad \text{or} \quad \frac{J_{(0)}(1-s,\widetilde{W}_i)}{L(1-s,\widetilde{\pi},\wedge^2)}
\]
are polynomials in $\mathbb{C}[q^{\pm s}]$. Hence the equality in \eqref{Iwasawa-distinct} is true for all $s$ with $\mathrm{Re}(s) < 1$ because of Tate integral. For any $k \in K_n$, we have
\[
\label{lambda-invariance}
\tag{6.3}
\begin{split}
  \frac{J_{(0)}(1-s, \varrho\begin{pmatrix} k &\\&k \end{pmatrix}\varrho(\tau_{2n})\widetilde{W})}{L(1-s,\widetilde{\pi},\wedge^2)}
  &=\frac{J_{(0)}(1-s, \varrho(\tau_{2n})\varrho\begin{pmatrix} k &\\&k \end{pmatrix}\widetilde{W})}{L(1-s,\widetilde{\pi},\wedge^2)} \\
  &=\Lambda_{1-s,\widetilde{\pi}}\left(\varrho\begin{pmatrix} ^tk^{-1} &\\&^tk^{-1} \end{pmatrix}W\right).
\end{split}
\]
Combining \eqref{lambda-invariance} with \eqref{Iwasawa-distinct}, for $s=0$ we arrive at
\[
 \frac{J(1,\varrho(\tau_{2n})\widetilde{W},\hat{\Phi})}{L(1,\widetilde{\pi},\wedge^2)}=\int_{K_n}\Lambda_{1,\widetilde{\pi}}\left(\varrho\begin{pmatrix} ^tk^{-1} &\\&^tk^{-1} \end{pmatrix}W\right)
 \int_{F^{\times}}  \hat{\Phi}(e_nak) |a|^nd^{\times}a dk.
\]
We have seen in the first paragraph  that a $(P_{2n} \cap S_{2n},\Theta)$-quasi invariant form $\Lambda_{1,\widetilde{\pi}}$ on $\mathcal{W}(\pi,\psi)$ is actually $(S_{2n},\Theta)$-quasi invariant form. 
Finally
 \[
\frac{J(1,\varrho(\tau_{2n})\widetilde{W},\hat{\Phi})}{L(1,\widetilde{\pi},\wedge^2)}
=\Lambda_{1,\widetilde{\pi}}(W) \int_{K_n}   \int_{F^{\times}}  \hat{\Phi}(e_nak) |a|^nd^{\times}a dk,
\]
because $\displaystyle   \Lambda_{1,\widetilde{\pi}} \left(\varrho\begin{pmatrix} ^tg^{-1} &\\&^tg^{-1} \end{pmatrix}W\right)=\Theta\begin{pmatrix} ^tg^{-1} &\\&^tg^{-1} \end{pmatrix}
\Lambda_{1,\widetilde{\pi}}(W)=\Lambda_{1,\widetilde{\pi}}(W)$  for all $g \in GL_n$ from the $S_{2n}$-quasi invariance. By Lemma \ref{pkdec}, we have
\[
  \int_{K_n}\int_{F^{\times}}  \hat{\Phi}(e_nak) |a|^nd^{\times}a dk=\int_{F^n-\{ 0\}}  \hat{\Phi}(u) du=\int_{F^n} \hat{\Phi}(u) du= \Phi(0)
\]
the last equality from Fourier inversion.

\par
Now let us write down the functional equation of exterior square $L$-function at $s=0$ from Theorem \ref{local functional eq}:
\[
\label{Pn-functionalequation}
\tag{6.4}
  \Lambda_{1,\widetilde{\pi}}(W) \Phi(0)= \frac{J(1,\varrho(\tau_{2n})\widetilde{W},\hat{\Phi})}{L(1,\widetilde{\pi},\wedge^2)}= \varepsilon(0,\pi,\wedge^2) \frac{J(0,W,\Phi)}{L(0,\pi,\wedge^2)}.
\]
By Lemma \ref{nonvanishing-even}, there exists $W_{\circ} \in \mathcal{W}(\pi,\psi)$ such that $J_{(0)}(s,W_{\circ})$ is a non-zero constant. Let $K_{n,s} \subset K_n$ be a sufficiently small compact open subgroup which stabilizes $W_{\circ}$. Now choose $\Phi^{\circ} \in \mathcal{S}_0(F^n)$ to be the characteristic function of $e_nK_{n,s}$. With these choices of $W_{\circ}$ and $\Phi^{\circ}$, the integral reduces to $J(0,W_{\circ},\Phi^{\circ})=\mathrm{Vol}(e_nK_{n,s}) J_{(0)}(0,W_{\circ}) \neq 0$. For $\Phi^{\circ} \in  \mathcal{S}_0(F^n)$, $\Lambda_{1,\widetilde{\pi}}(W_{\circ}) \Phi^{\circ}(0)$ is $0$. Since $\varepsilon(0,\pi,\wedge^2)$ and $J(0,W_{\circ},\Phi^{\circ})$ are non-zero constants, the factor $L(s,\pi,\wedge^2)$ has a pole at $s=0$.

\par
Let $d$ be the order of pole in $L(s,\pi,\wedge^2)$ at $s=0$. Then the pole at $s=0$ must occur with order $d$ for some function $J(s,W,\Phi)$ with $W \in \mathcal{W}(\pi,\psi)$. If we perform the Laurent expansion of $J(s,W,\Phi)$ at $s=0$, then $J(s,W,\Phi)$ will have the form
\[
 J(s,W,\Phi)=\frac{B_0(W,\Phi)}{(q^s-1)^d}+\mathrm{higher \; order \; terms.}
\]
with $B_0(W,\Phi) \neq 0$. We replace Jacquet-Shalika integral $J(s,W,\Phi)$ by its Laurent expansion in \eqref{Pn-functionalequation}, we are able to see that 
\[
 \Lambda_{1,\widetilde{\pi}}(W) \Phi(0)= \frac{J(1,\varrho(\tau_{2n})\widetilde{W},\hat{\Phi})}{L(1,\widetilde{\pi},\wedge^2)}= \frac{1}{\kappa} \varepsilon(0,\pi,\wedge^2,\psi) B_0(W,\Phi),
\]
where $\displaystyle \kappa=\lim_{s \to 0} (1-q^{-s})^dL(s,\pi,\wedge^2)$ is a nonzero constant. Since $B_0(W,\Phi)=0$ for all $\Phi \in \mathcal{S}_0(F^n)$, the pole $s=0$ must be an exceptional pole.

\end{proof}

Let $\rho$ be an irreducible supercuspidal representation of $GL_{2n}$. We have already seen that if $L(s,\rho,\wedge^2)$ has a pole at $s=0$, then $\mathrm{Hom}_{S_{2n}}(\rho,\Theta) \neq 0$. We would like to prove the converse of this statement. We introduce the following uniqueness result for unitary representations, obtained from the argument in the proof of Corollary 4.18 of \cite{Ma15}.

\begin{proposition}[Matringe]
\label{mira-multiplicity one}
Let $\pi$ be an irreducible unitary generic representation of $GL_{2n}$. Then the vector space $\mathrm{Hom}_{P_{2n} \cap H_{2n}}(\pi,\delta_{2n}^{-s_0})$ is of dimension at most one for any real number $s_0 > 0$.
\end{proposition}

The following Proposition is a direct consequence of Lemma \ref{char-Shalika} and Proposition \ref{mira-multiplicity one}. This Proposition is originally stated in Proposition 6.1 of Matringe \cite{Ma14}.

 \begin{proposition}
\label{prop-supercusp}
Let $\Delta$ be an irreducible quasi-square integrable representation of $GL_{2n}$. Then $\mathrm{Hom}_{S_{2n}}(\Delta,\Theta) \neq 0$ if and only if $L_{ex}(s,\Delta,\wedge^2)$ has a pole at $s=0$.
\end{proposition}

\begin{proof}
We have already seen in Section 3.1 that if $L_{ex}(s,\Delta,\wedge^2)$ has a pole at $s=0$, then $\mathrm{Hom}_{S_{2n}}(\Delta,\Theta) \neq 0$. 
Conversely suppose that $\mathrm{Hom}_{S_{2n}}(\Delta,\Theta) \neq 0$. Since $\Delta$ has a central character and $\mathrm{Hom}_{S_{2n}}(\Delta,\Theta)$ is non-zero, $\Delta$ has necessarily a trivial central character, and so $\Delta$ is an irreducible square integrable representation. From Proposition \ref{even-embedding}, the vector space $\mathrm{Hom}_{P_{2n} \cap S_{2n}}(\Delta,\Theta)$ embeds as a subspace of $\mathrm{Hom}_{P_{2n} \cap H_{2n}}(\Delta,\delta_{2n}^{-s_0})$ for some positive real number $s_0$. Then the later space $\mathrm{Hom}_{P_{2n} \cap H_{2n}}(\Delta,\delta_{2n}^{-s_0})$ is one dimensional by Proposition \ref{mira-multiplicity one}. The results follows immediately from Lemma \ref{char-Shalika}.
\end{proof}

Let $\rho$ be an irreducible supercuspidal representation of $GL_{2n}$. We set $(W \otimes |\cdot|^{s_1})(g)=W(g)|\mathrm{det}(g)|^{s_1}$ for $g \in GL_{2n}$.
Since $J(s,W,\Phi)=J(s-2s_1,W \otimes |\cdot|^{s_1},\Phi)$, it can be seen that $L(s,\rho,\wedge^2)= L(s-2s_1,\rho\nu^{s_1},\wedge^2)$.
For $s_1=\frac{s_0}{2}$, $\rho\nu^{\frac{s_0}{2}}$ is still a supercuspidal representation. Replacing $\rho$ by $\rho\nu^{\frac{s_0}{2}}$ in Proposition \ref{prop-supercusp} implies that the exceptional pole at $s=s_0$ exists, is necessarily simple, if and only if $\mathrm{Hom}_{S_{2n}}(\rho\nu^{\frac{s_0}{2}},\Theta) \neq 0$. Since these are the only possible poles of the exterior square $L$-function, we therefore have another description for our $L$-function.

\begin{theorem}
\label{supercusp2}
Let $\rho$ be an irreducible supercuspidal representation of $GL_r$. 
\begin{enumerate}
\item[$(\mathrm{i})$] If $r=2n$, $L(s,\rho,\wedge^2)$ has simple poles and we have
\[
   L(s,\rho,\wedge^2)=\prod (1-\alpha q^{-s})^{-1}
\]
with the product over all $\alpha=q^{s_0}$ such that $\mathrm{Hom}_{S_{2n}}(\rho\nu^{\frac{s_0}{2}},\Theta) \neq 0$. 
\item[$(\mathrm{ii})$] If $r=2n+1$, we have $L(s,\rho,\wedge^2) = 1$. 
\end{enumerate}
\end{theorem}

 \subsection{Exterior square $L$-function for quasi-square-integrable representations.}

Let $\pi$ be an irreducible quasi-square-integrable representation for the segment $\Delta=[\rho,\dotsm,\rho \nu^{\ell-1}]$ with $\rho$ an irreducible supercuspidal representation of $GL_r$ and $n=\ell r$. The derivative of quasi-square-integrable representations have been computed by Bernstein and Zelevinsky in Theorem \ref{deriviative}. $\Delta^{(k)}=0$ if $k$ is not a multiple of $r$, $\Delta^{(0)}=\Delta$, $\Delta^{(kr)}=[\rho \nu^k,\dotsm,\rho \nu^{\ell-1}]$ for $1 \leq k \leq \ell-1$, and $\Delta^{(\ell r)}=\textbf{1}$. Note that all non-zero derivatives are irreducible and quasi-square-integrable. In this section, we compute the $L$-functions $L(s,\Delta,\wedge^2)$ for quasi-square-integrable representations by means of the derivative theory developed by Bernstein and Zelevinsky.

\par

In this paragraph, we recall the results from \cite{ CoKiMu04, JaPSSh83,CoPe} about $L$-function for Rankin-Selberg convolution. Let $(\rho,V_{\rho})$ be a supercuspidal representation of $GL_r$ and $(\rho^{\prime},V_{\rho^{\prime}})$ a supercuspidal representation of $GL_{r^{\prime}}$. If $r \neq r^{\prime}$, then $L(s,\rho \times \rho^{\prime})=1$. Otherwise, it is proven in \cite{CoPe,JaPSSh83} that the poles of $L(s,\rho \times \rho^{\prime})$ can only occur at those $s=s_0$ where $\widetilde{\rho} \simeq \rho^{\prime} \nu^{s_0}$ and those are necessarily simple. Therefore, we have
\[
  L(s,\rho \times \rho^{\prime})=\prod (1-\alpha q^{-s})^{-1},
\]
where $\alpha$ runs over all $\alpha=q^{s_0}$ with $\widetilde{\rho} \simeq \rho^{\prime} \nu^{s_0}$. Let $(\pi,V_{\pi})$ be an irreducible quasi-square-integrable representation for $GL_n$ and $(\sigma,V_{\sigma})$ an irreducible quasi-square-integrable representation for $GL_n$. Let $\pi$ be the segment $\Delta=[\rho,\rho\nu,\dotsm,\rho\nu^{\ell-1}]$, $n=\ell r$. Likewise let  us take $\sigma$ with the segment $\Delta^{\prime}=[\rho^{\prime},\rho^{\prime}\nu, \dotsm, \rho^{\prime}\nu^{\ell^{\prime}-1}]$ with $n=\ell^{\prime}r^{\prime}$. We assume that $\ell \geq \ell^{\prime}$. It is proven in  \cite{CoPe} that exceptional poles can only occur at those $s=s_0$ where $(\Delta^{(n-k)})^{\sim}\simeq\Delta'^{(n-k)}\nu^{s_0}$ and those are necessarily again simple. Hence we have
\[
L_{ex}(s,\Delta^{(n-k)} \times \Delta'^{(n-k)}) = \prod(1-\alpha q^{-s})^{-1}
\]
where $\alpha$ accounts for all $\alpha=q^{s_0}$ with $(\Delta^{(n-k)})^{\sim}\simeq\Delta'^{(n-k)}\nu^{s_0}$ and exceptional $L$-functions $L_{ex}(s,\Delta^{(n-k)} \times \Delta'^{(n-k)})$ are defined in Section 2 of \cite{CoPe}.
By considering the derivatives of the associated segments that Bernstein and Zelevinsky compute and the calculation of supercuspidal representations in \cite{CoPe}, we  can conclude that
\[
L_{ex}(s,\Delta^{(n-k)} \times \Delta'^{(n-k)})=L(\ell-1+j+s, \;\rho \times \rho^{\prime})
\]
for appropriate $i$ and $j$ such that $\ell-i=\ell^{\prime}-j$. Since this $L$-function runs over all poles of $L(s, \Delta \times \Delta')$ as $j$ accounts for all permissible values, namely  $0 \leq j \leq \ell^{\prime}-1$, we arrive at the following result.
\[
L(s, \Delta \times \Delta')=\prod_{j=0}^{\ell^{\prime}-1}L(\ell-1+j+s,\;\rho \times \rho^{\prime}).
\]
Furthermore, $L(s, \Delta \times \Delta')$ has simple poles. Let us investigate $L^{(0)}(s,\Delta,\wedge^2)$ defined in Proposition \ref{L0 function} to see how $L^{(0)}(s,\Delta,\wedge^2)$ can be used to establish the divisibility of $L$-functions.

\begin{lemma}
\label{quasi-key lemma 1}
Let $\Delta$ be an irreducible quasi-square-integrable representation of $GL_{2n}$ with the segment $\Delta=[\rho,\rho\nu,\dotsm,\rho\nu^{\ell-1}]$ and $\rho$ an irreducible supercuspidal representation of $GL_r$ with $2n=\ell r$. Then $L^{(0)}(s,\Delta,\wedge^2)^{-1}$ divides $L(\ell-1+s,\rho \times \rho)^{-1}$ in $\mathbb{C}[q^{\pm s}]$. 
\end{lemma}

\begin{proof}
We assume that $L(s,\Delta,\wedge^2)$ has a pole at $s=s_0$. As $L^{(0)}(s+s_0,\Delta,\wedge^2)$ is equal to $L^{(0)}(s,\Delta\nu^{\frac{s_0}{2}},\wedge^2)$, $s=0$ is an exceptional pole of $L^{(0)}(s,\Delta\nu^{\frac{s_0}{2}},\wedge^2)$. Then we have seen in Section 3.1 that there is a non-zero twisted Shalika functional which belong to $\mathrm{Hom}(\Delta\nu^{\frac{s_0}{2}},\Theta)$. Theorem \ref{Jacquet-Rallis} of Jacquet and Rallis asserts that $\Delta \nu^{\frac{s_0}{2}}$ is self-contragredient. On the other hand, the segment of $\Delta\nu^{\frac{s_0}{2}}$ can be written as $[\rho\nu^{\frac{s_0}{2}},\dotsm,\rho\nu^{\ell-1+\frac{s_0}{2}}]$ and $(\Delta\nu^{\frac{s_0}{2}})^{\sim}$ is a quasi-square-integrable representation associated to $[\tilde{\rho}\nu^{-\ell+1-\frac{s_0}{2}},\dotsm,\tilde{\rho}\nu^{-\frac{s_0}{2}}]$. Hence we see that $(\Delta\nu^{\frac{s_0}{2}})^{\sim} \simeq \Delta\nu^{\frac{s_0}{2}}$ if and only if $\tilde{\rho}\nu^{-\ell+1-\frac{s_0}{2}} \simeq \rho\nu^{\frac{s_0}{2}}$. The last condition is exactly the same condition that the $L$-function $L(\ell-1+s,\rho \times \rho)$ has a pole at $s=s_0$, and this pole will be simple. Since both of the $L$-functions $L^{(0)}(s,\Delta,\wedge^2)$ and $L(\ell-1+s,\rho \times \rho)$ have simple poles, $L^{(0)}(s,\Delta,\wedge^2)^{-1}$ divides $L(\ell-1+s,\rho \times \rho)^{-1}$ in $\mathbb{C}[q^{\pm s}]$. 
\end{proof}

$L^{(0)}(s,\Delta,\wedge^2)$ in previous Lemma \ref{quasi-key lemma 1} can be replaced by the exceptional $L$-function $L_{ex}(s,\Delta,\wedge^2)$. We utilize the derivative theory of Bernstein and Zelevinsky in Theorem \ref{deriviative} for irreducible quasi-square-integrable representations and factorization formula in terms of exceptional $L$-functions for even derivatives of $\Delta$ in Theorem \ref{prod-L} to prove the following result.

\begin{lemma}
\label{quasi-key lemma 2}
Let $\Delta$ be an irreducible quasi-square-integrable representation of $GL_{2n}$ with the segment $\Delta=[\rho,\rho\nu,\dotsm,\rho\nu^{\ell-1}]$ and $\rho$ an irreducible supercuspidal representation of $GL_r$ with $2n=\ell r$. Then
\[
 L_{ex}(s,\Delta,\wedge^2)=L^{(0)}(s,\Delta,\wedge^2).
\]
Moreover $L_{ex}(s,\Delta,\wedge^2)^{-1}$ divides $L(\ell-1+s,\rho \times \rho)^{-1}$ in $\mathbb{C}[q^{\pm s}]$. 
\end{lemma}

\begin{proof}
We first investigate the case that $r$ is even. Suppose that $\ell=1$. Then $\Delta=\rho$ is an irreducible supercuspidal representation. By Theorem \ref{supercusp}, we can conclude that $L_{ex}(s,\rho,\wedge^2)=L^{(0)}(s,\rho,\wedge^2)$. 
\par
We show the statement of Lemma by induction on $\ell$. We now assume that the following statement is satisfied for all positive integers $1 \leq k \leq \ell$
\[
 L_{ex}(s,[\rho,\dotsm,\rho\nu^{k-1}],\wedge^2)=L^{(0)}(s,[\rho,\dotsm,\rho\nu^{k-1}],\wedge^2).
\]
Throughout this proof the least common multiple is always taken in terms of divisibility in $\mathbb{C}[q^{\pm s}]$. According to \ref{prod-L0}, we have
\[
  L_{(0)}(s,[\rho,\dotsm,\rho\nu^{\ell}],\wedge^2)^{-1}=\underset{1\leq k \leq \ell}{l.c.m.} \{L_{ex}(s,[\rho\nu^k,\dotsm,\rho\nu^{\ell}],\wedge^2)^{-1} \}.
\]
We can rewrite $[\rho^{\prime},\dotsm,\rho^{\prime}\nu^{\ell-k}]$ for $[\rho\nu^k,\dotsm,\rho\nu^{\ell}]$ with $\rho^{\prime}=\rho\nu^k$ an irreducible supercuspidal representation. Induction hypothesis on the length of  $[\rho^{\prime},\dotsm,\rho^{\prime}\nu^{\ell-k}]$ is applicable and we have
\[
  L_{(0)}(s,[\rho,\dotsm,\rho\nu^{\ell}],\wedge^2)^{-1}=\underset{1\leq k \leq \ell}{l.c.m.} \{L^{(0)}(s,[\rho\nu^k,\dotsm,\rho\nu^{\ell}],\wedge^2)^{-1} \}.
  \]
  We deduce from Lemma \ref{quasi-key lemma 1} that $L^{(0)}(s,[\rho\nu^k,\dotsm,\rho\nu^{\ell}],\wedge^2)^{-1}$ divides $L(s+\ell-k,\rho\nu^k \times \rho\nu^k)^{-1}$ for $0 \leq k \leq \ell$.  As $L(s+\ell-k,\rho\nu^k \times \rho\nu^k)^{-1}=L(s+\ell+k,\rho \times \rho)^{-1}$ for different $k$ with $0 \leq k \leq \ell$ are relative prime in $\mathbb{C}[q^{\pm s}]$, so are $L^{(0)}(s,[\rho,\dotsm,\rho\nu^{\ell}],\wedge^2)^{-1}$ and $L_{(0)}(s,[\rho,\dotsm,\rho\nu^{\ell}],\wedge^2)^{-1}$. We notice from Proposition \ref{L0 function} that $L(s,[\rho,\dotsm,\rho\nu^{\ell}],\wedge^2)=L^{(0)}(s,[\rho,\dotsm,\rho\nu^{\ell}],\wedge^2)L_{(0)}(s,[\rho,\dotsm,\rho\nu^{\ell}],\wedge^2)$. Since $L$-functions $L^{(0)}(s,[\rho,\dotsm,\rho\nu^{\ell}],\wedge^2)$ and $L_{(0)}(s,[\rho,\dotsm,\rho\nu^{\ell}],\wedge^2)$ does not share any common poles, we arrives at
  \[
   L_{ex}(s,[\rho,\dotsm,\rho\nu^{\ell}],\wedge^2)=L^{(0)}(s,[\rho,\dotsm,\rho\nu^{\ell}],\wedge^2).
  \]
  \par
  Now we assume that $r$ is odd. For $\ell=2$, we know from Theorem \ref{prod-L0} that $ L_{(0)}(s,[\rho,\rho\nu],\wedge^2)$ is $1$. From $ L(s,[\rho,\rho\nu],\wedge^2)= L^{(0)}(s,[\rho,\rho\nu],\wedge^2) L_{(0)}(s,[\rho,\rho\nu],\wedge^2)= L^{(0)}(s,[\rho,\rho\nu],\wedge^2) $, we get
  \[
    L_{ex}(s,[\rho,\rho\nu],\wedge^2)=L^{(0)}(s,[\rho,\rho\nu],\wedge^2).
    \]
  
    \par
We again apply induction on the even number $\ell$. We assume that $\ell$ is an even number and
\[
 L_{ex}(s,[\rho,\dotsm,\rho\nu^{k-1}],\wedge^2)=L^{(0)}(s,[\rho,\dotsm,\rho\nu^{k-1}],\wedge^2).
\]
  is true for all even number $k$ with $2 \leq k \leq \ell$. We make use of Theorem \ref{prod-L0} to get
  \[
   L_{(0)}(s,[\rho,\dotsm,\rho\nu^{\ell+1}],\wedge^2)^{-1}=\underset{1\leq k \leq \frac{\ell}{2}}{l.c.m.} \{L_{ex}(s,[\rho\nu^{2k},\dotsm,\rho\nu^{\ell+1}],\wedge^2)^{-1} \}.
   \]
    The length of the segment $[\rho\nu^{2k},\dotsm,\rho\nu^{\ell+1}]$ is $\ell+2-2k$, and so for $1\leq k \leq \frac{\ell}{2}$, we can utilize the induction. Then we obtain  
       \[
     L_{(0)}(s,[\rho,\dotsm,\rho\nu^{\ell+1}],\wedge^2)^{-1}=\underset{1\leq k \leq \frac{\ell}{2}}{l.c.m.} \{L^{(0)}(s,[\rho\nu^{2k},\dotsm,\rho\nu^{\ell+1}],\wedge^2)^{-1} \}.
     \]
     In a similar manner to the even case, each $L^{(0)}(s,[\rho\nu^{2k},\dotsm,\rho\nu^{\ell+1}],\wedge^2)^{-1}$ divides $L(s+\ell+1-2k,\rho\nu^{2k} \times \rho\nu^{2k})^{-1}=L(s+\ell+1+2k,\rho \times \rho)^{-1}$ and that
     $L(s+\ell+1+2k,\rho \times \rho)^{-1}$ are relative prime in $\mathbb{C}[q^{\pm s}]$ for different $k$ with $0 \leq k \leq \frac{\ell}{2}$. Hence $L^{(0)}(s,[\rho,\dotsm,\rho\nu^{\ell+1}],\wedge^2)$ and $L_{(0)}(s,[\rho,\dotsm,\rho\nu^{\ell+1}],\wedge^2)$ does not have any common poles. We can conclude that 
     \[
     L_{ex}(s,[\rho,\dotsm,\rho\nu^{\ell+1}],\wedge^2)=L^{(0)}(s,[\rho,\dotsm,\rho\nu^{\ell+1}],\wedge^2).
     \]
     \par
     The last assertion immediately follows from Lemma \ref{quasi-key lemma 1}. 
     
\end{proof}

 We begin with computing local exterior square $L$-functions for irreducible quasi-square-integrable representations of even $GL_{2n}$.

\begin{proposition}
\label{even-quasisquare}
Let $\Delta$ be an irreducible quasi-square-integrable representation of $GL_{2n}$ with the segment $\Delta=[\rho,\rho\nu,\dotsm,\rho\nu^{\ell-1}]$ and $\rho$ an irreducible supercuspidal representation of $GL_r$ with $2n=\ell r$. 
\begin{enumerate}
\item[$(\mathrm{i})$] Suppose that $r$ is even. Then
\[
  L(s,\Delta,\wedge^2)=\prod_{i=1}^nL_{ex}(s,\Delta^{(2n-2i)},\wedge^2)=\prod_{k=0}^{\ell-1} L_{ex}(s,[\rho\nu^k,\rho\nu^{k+1},\dotsm,\rho\nu^{\ell-1}],\wedge^2).
\]
\item[$(\mathrm{ii})$] Suppose that $r$ is odd. Then
\[
  L(s,\Delta,\wedge^2)=\prod_{i=1}^nL_{ex}(s,\Delta^{(2n-2i)},\wedge^2)=\prod_{k=0}^{\frac{\ell}{2}-1} L_{ex}(s,[\rho\nu^{2k},\rho\nu^{2k+1},\dotsm,\rho\nu^{\ell-1}],\wedge^2).
\]
\end{enumerate}
\end{proposition}

\begin{proof}
 Throughout this proof, the least common multiple is always taken in terms of divisibility in $\mathbb{C}[q^{\pm s}]$. We have seen in Theorem \ref{prod-L} that the poles of $L_{(0)}(s,\Delta,\wedge^2)$ are precisely the poles of the exceptional contribution to the $L$-functions for even derivatives. In fact,
 \[
   L(s,\Delta,\wedge^2)^{-1}=\underset{0 \leq k \leq \ell-1}{l.c.m.} \{L_{ex}(s,[\rho\nu^k,\dotsm,\rho\nu^{\ell-1}],\wedge^2)^{-1} \}
 \]
if $r$ is even or
\[
 L(s,\Delta,\wedge^2)^{-1}=\underset{0 \leq k \leq \frac{\ell}{2}-1}{l.c.m.} \{L_{ex}(s,[\rho\nu^{2k},\dotsm,\rho\nu^{\ell-1}],\wedge^2)^{-1} \}
\]
if $r$ is odd. By Lemma \ref{quasi-key lemma 2}, $L_{ex}(s,[\rho\nu^k,\dotsm,\rho\nu^{\ell-1}],\wedge^2)^{-1}$ divides $L(s+\ell-1-k,\rho\nu^k \times \rho\nu^k)^{-1}=L(s+\ell-1+k,\rho \times \rho)^{-1}$, or $L_{ex}(s,[\rho\nu^{2k},\dotsm,\rho\nu^{\ell-1}],\wedge^2)^{-1}$ divides $L(s+\ell-1-2k,\rho\nu^{2k} \times \rho\nu^{2k})^{-1}=L(s+\ell-1+2k,\rho \times \rho)^{-1}$ in $\mathbb{C}[q^{\pm s}]$. As $L(s+\ell-1+k,\rho \times \rho)^{-1}$ for different $k$ are relatives prime in $\mathbb{C}[q^{\pm s}]$, the least common multiple in $L(s,\Delta,\wedge^2)$ can be reduced to
\[
  L(s,\Delta,\wedge^2)^{-1}=\prod_{k=0}^{\ell-1}  L_{ex}(s,[\rho\nu^k,\dotsm,\rho\nu^{\ell-1}],\wedge^2)^{-1} 
\]
if $r$ is even or
\[
  L(s,\Delta,\wedge^2)^{-1}=\prod_{k=0}^{\frac{\ell}{2}-1} L_{ex}(s,[\rho\nu^{2k},\dotsm,\rho\nu^{\ell-1}],\wedge^2)^{-1}
\]
if $r$ is odd. This completes the proof.

\end{proof}

We establish the analogous results of Proposition \ref{even-quasisquare} for odd $GL_{2n+1}$.

\begin{proposition}
\label{odd-quasisquare}
Let $\Delta$ be an irreducible quasi-square integrable representation of $GL_{2n+1}$ with the segment $\Delta=[\rho,\rho \nu, \dotsm, \rho \nu^{\ell-1}]$ and $\rho$ an irreducible supercuspidal representation of 
$GL_r$ with $2n+1=r \ell$ and $\ell > 1$. Then 
\[
  L(s,\Delta,\wedge^2)=\prod_{i=1}^{n}L_{ex}(s,\pi^{(2n+1-2i)},\wedge^2)=\prod_{k=1}^{\frac{\ell-1}{2}}L_{ex}(s,[\rho\nu^{2k-1},\rho\nu^{2k},\dotsm,\rho\nu^{\ell-1}],\wedge^2).
\]
\end{proposition}

\begin{proof}
From a factorization of local exterior square $L$-function in terms of exceptional $L$-functions for odd derivatives of $\Delta$ in Theorem \ref{prod-L-odd}, we know that
\[
   L(s,\Delta,\wedge^2)^{-1}= \underset{1 \leq k \leq \frac{\ell-1}{2}}{l.c.m.} \{ L_{ex}(s,[\rho\nu^{2k-1},\rho\nu^{2k},\dotsm,\rho\nu^{\ell-1}],\wedge^2)^{-1} \},
\]
where the least common multiple is always taken in terms of divisibility in $\mathbb{C}[q^{\pm s}]$. For $1 \leq k \leq \frac{\ell-1}{2}$,
$L_{ex}(s,[\rho\nu^{2k-1},\rho\nu^{2k},\dotsm,\rho\nu^{\ell-1}],\wedge^2)^{-1}$ divides $L(s+\ell-2k,\rho\nu^{2k-1}\times \rho\nu^{2k-1})^{-1}=L(s+\ell+2k-2,\rho \times \rho)^{-1}$ in $\mathbb{C}[q^{\pm s}]$ because we can apply Lemma \ref{quasi-key lemma 2}. As $L(s+\ell+2k-2,\rho \times \rho)^{-1}$ for different $k$ are relative prime in $\mathbb{C}[q^{\pm s}]$, we have
\[
  L(s,\Delta,\wedge^2)^{-1}=\prod_{k=1}^{\frac{\ell-1}{2}}L_{ex}(s,[\rho\nu^{2k-1},\rho\nu^{2k},\dotsm,\rho\nu^{\ell-1}],\wedge^2).
\]

\end{proof}

Let $\Delta$ be an irreducible square integrable representation of $GL_{2n}$ with the segment $\Delta=[\rho \nu^{-\frac{\ell-1}{2}},\dotsm,\rho \nu^{\frac{\ell-1}{2}}]$ and $\rho$ an irreducible unitary supercuspidal representation of $GL_r$ with $2n=\ell r$. Then
\[
\label{RS-Lfunction}
\tag{6.5}
  L(s,\Delta \times \Delta)=\prod_{k=0}^{\ell-1}L(s+\ell-1+k,\rho\nu^{-\frac{\ell-1}{2}} \times \rho \nu^{-\frac{\ell-1}{2}})=\prod_{k=0}^{\ell-1}L(s+k,\rho \times \rho).
\]
As in Proposition \ref{even-quasisquare}, we have
\[
   L(s,\Delta,\wedge^2)=\prod_{k=0}^{\ell-1}L_{ex}(s,[\rho \nu^{k+\frac{1-\ell}{2}},\dotsm,\rho \nu^{\frac{\ell-1}{2}}],\wedge^2)
\]
if $r$ is even or
\[
  L(s,\Delta,\wedge^2)=\prod_{k=0}^{\frac{\ell}{2}-1}L_{ex}(s,[\rho \nu^{2k+\frac{1-\ell}{2}},\dotsm,\rho \nu^{\frac{\ell-1}{2}}],\wedge^2)
\]
if $r$ is odd. Using Lemma \ref{quasi-key lemma 2}, $L_{ex}(s,[\rho \nu^{k+\frac{1-\ell}{2}},\dotsm,\rho \nu^{\frac{\ell-1}{2}}],\wedge^2)^{-1}$ divides $L(s+\ell-1-k,\rho \nu^{k+\frac{1-\ell}{2}}\times \rho \nu^{k+\frac{1-\ell}{2}})^{-1}=L(s+k,\rho \times \rho)^{-1}$, and hence $\displaystyle L(s,\Delta,\wedge^2)^{-1}$ divides $\displaystyle L(s,\Delta \times \Delta)^{-1}=\prod_{k=0}^{\ell-1}L(s+k,\rho \times \rho)^{-1}$ in $\mathbb{C}[q^{\pm s}]$. $L(s,\Delta,\wedge^2)$ has simple poles because $L(s,\Delta \times \Delta)$ has simple poles. Since $\rho$ is unitary, the pole of $L(s,\rho \times \rho)$ must lie on the line $\mathrm{Re}(s)=0$. Thus the pole of $L(s,\Delta,\wedge^2)$ will lie on the line $\mathrm{Re}(s+k)=0$ or $\mathrm{Re}(s)=-k$ for $k=0, \dotsm, \ell-1$. Since these are the only possible poles of the exterior square $L$-function, we get as a proposition the main result of \cite{Ke11}.

\begin{proposition}
\label{even-square}
Let $\Delta$ be an irreducible square integrable representation of $GL_{2n}$ with the segment $\Delta=[\rho \nu^{-\frac{\ell-1}{2}},\dotsm,\rho \nu^{\frac{\ell-1}{2}}]$ and $\rho$ a supercuspidal representation of $GL_r$ with $2n=\ell r$. Then $L(s,\Delta,\wedge^2)$ has no poles in $\mathrm{Re}(s) > 0$ and has simple poles. 
\begin{enumerate}
\item[$(\mathrm{i})$] Suppose $r$ is even. Then
\[
  L(s,\Delta,\wedge^2)=\prod_{i=1}^nL_{ex}(s,\Delta^{(2n-2i)},\wedge^2)=\prod_{k=0}^{\ell-1}L_{ex}(s,[\rho \nu^{k+\frac{1-\ell}{2}},\dotsm,\rho \nu^{\frac{\ell-1}{2}}],\wedge^2).
\]
\item[$(\mathrm{ii})$] Suppose $r$ is odd. Then
\[
  L(s,\Delta,\wedge^2)=\prod_{i=1}^nL_{ex}(s,\Delta^{(2n-2i)},\wedge^2)=\prod_{k=0}^{\frac{\ell}{2}-1}L_{ex}(s,[\rho \nu^{2k+\frac{1-\ell}{2}},\dotsm,\rho \nu^{\frac{\ell-1}{2}}],\wedge^2).
\]
\end{enumerate}
Moreover $L(s,\Delta,\wedge^2)^{-1}$ divides $L(s,\Delta \times \Delta)^{-1}$ in $\mathbb{C}[q^{\pm s}]$. 
\end{proposition}



The following Proposition characterizes the irreducible square integrable representation of $GL_{2n}$ admitting a Shalika functional in terms of poles of the Jacquet-Shalika exterior square $L$-function. This result is also presented in the following result in Section 6 of \cite{Ma14}.

\begin{proposition}
\label{excep-Shalika}
Let $\Delta$ be an irreducible square integrable representation of $GL_{2n}$. Then $\mathrm{Hom}_{S_{2n}}(\Delta,\Theta) \neq 0$ if and only if $L_{ex}(s,\Delta,\wedge^2)$ has a pole at $s=0$, or equivalently if and only if $L(s,\Delta,\wedge^2)$ has a pole at zero.
\end{proposition}

\begin{proof}
The first equivalence is proved in Proposition \ref{prop-supercusp}. Now $L_{ex}(s,\Delta,\wedge^2)$ is the only factor having pole at $\mathrm{Re}(s)=0$, because we can apply Lemma \ref{quasi-key lemma 2} and Proposition \ref{even-square} so that $L_{ex}(s,[\rho \nu^{k+\frac{1-\ell}{2}},\dotsm,\rho \nu^{\frac{\ell-1}{2}}],\wedge^2)^{-1}$ divides the factor $L(s+k,\rho \times \rho)^{-1}$. This asserts the second equivalence condition.


\end{proof}

Let $\Delta$ be an irreducible square integrable representation of $GL_{2n+1}$ with the segment $\Delta=[\rho \nu^{-\frac{\ell-1}{2}}, \dotsm, \rho \nu^{\frac{\ell-1}{2}}]$ and $\rho$ an unitary supercuspidal representation of $GL_r$ with $2n+1=r \ell$ and $\ell > 1$. We deduce from Proposition \ref{quasi-key lemma 2} and Proposition \ref{odd-quasisquare} that $\displaystyle L(s,\Delta,\wedge^2)^{-1}=\prod_{k=1}^{\frac{\ell-1}{2}}L_{ex}(s,[\rho\nu^{2k-1+\frac{1-\ell}{2}},\dotsm,\rho\nu^{\frac{\ell-1}{2}}],\wedge^2)^{-1}$ divides Rankin-Selberg $L$-function $\displaystyle \prod_{k=1}^{\frac{\ell-1}{2}} L(s+\ell-2k,\rho\nu^{2k-1+\frac{1-\ell}{2}} \times \rho\nu^{2k-1+\frac{1-\ell}{2}})^{-1}=\prod_{k=1}^{\frac{\ell-1}{2}} L(s+2k-1,\rho \times \rho)^{-1}$ in $\mathbb{C}[q^{\pm s}]$. Then $L(s,\Delta,\wedge^2)^{-1}$ divides $L(s,\Delta\times \Delta)^{-1}$ in $\mathbb{C}[q^{\pm s}]$ from \eqref{RS-Lfunction}, which also implies that $L(s,\Delta,\wedge^2)$ has simple poles. Moreover, since $\rho$ is unitary, the pole of $L(s,\rho \times \rho)$ must lie on the line $\mathrm{Re}(s)=0$. Hence the only possible poles of the exterior square $L$-function will lie on the line $\mathrm{Re}(s+2k-1)=0$ or $\mathrm{Re}(s)=-2k+1$ for $k=1,\dotsm, \frac{\ell-1}{2}$. Therefore we get as a proposition the main result of \cite{Ke11}.

\begin{proposition}
\label{odd-square}
Let $\Delta$ be an irreducible square integrable representation of $GL_{2n+1}$ with the segment $\Delta=[\rho \nu^{-\frac{\ell-1}{2}}, \dotsm, \rho \nu^{\frac{\ell-1}{2}}]$ and $\rho$ an irreducible unitary supercuspidal representation of $GL_r$ with $2n+1=r \ell$ and $\ell > 1$. Then $L(s,\Delta,\wedge^2)$ has simple poles and has no poles in $\mathrm{Re}(s) \geq 0$. Moreover,
\[
  L(s,\Delta,\wedge^2)=\prod_{i=1}^{n}L_{ex}(s,\Delta^{(2n+1-2i)},\wedge^2)=\prod_{k=1}^{\frac{\ell-1}{2}}L_{ex}(s,[\rho\nu^{2k-1+\frac{1-\ell}{2}},\dotsm,\rho\nu^{\frac{\ell-1}{2}}],\wedge^2).
\]
Moreover $L(s,\Delta,\wedge^2)^{-1}$ divides $L(s,\Delta \times \Delta)^{-1}$ in $\mathbb{C}[q^{\pm s}]$. 
\end{proposition}


\par
 We wish to express exceptional $L$-functions $L_{ex}(s,[\rho,\rho\nu,\dotsm,\rho\nu^{\ell-1}],\wedge^2)$ for a segment in terms of $L$-function for the supercuspidal representation $\rho$ in Section 8.

\section{Deformation and Local $L$-Functions}

\subsection{Deformation and general position}

The aim of this section is to define what we mean by $``$in general position$"$ for a representation of $GL_n$. We first recall a few facts from \cite[Section 3]{CoPe} about deformations of representations and behavior of the derivatives. Let $\pi=\text{Ind}_Q^{GL_n}(\Delta_1 \otimes \dotsm \otimes \Delta_t)$ be a representation of Whittaker type of $GL_n$, where each $\Delta_i$ is a quasi-square-integrable representation of $GL_{n_i}$, $n=n_1+\dotsm+n_t$, and the induction is normalized parabolic induction from the standard parabolic subgroup $Q$ of $GL_n$ associated to the partition $(n_1,\dotsm,n_t)$. Let $M \simeq GL_{n_1} \times \dotsm \times GL_{n_t}$ denote the Levi subgroup of $Q$. If $u=(u_1,\dotsm,u_t) \in \mathbb{C}^t$ then $u$ defines an unramified character $\nu^u$ of $M$ via $\nu^u(m)=\nu^u(g_1,\dotsm,g_t)=\nu(g_1)^{u_1}\dotsm\nu(g_t)^{u_t}$ for $m \in M$. Let $\mathcal{D}_{\pi}$ denote the complex manifold $(\mathbb{C} \slash \frac{2\pi i}{\text{log}(q)}\mathbb{Z})^t$. The map $u \mapsto (q^{u_1},\dotsm,q^{u_t})$ induces a group isomorphism between $\mathcal{D}_{\pi}$ and $(\mathbb{C}^{\times})^t$. For convenience, we will let $q^u$ denote $(q^{u_1},\dotsm,q^{u_t})$. We will use $q^{\pm u}$ as short for $(q^{\pm u_1},\dotsm,q^{\pm u_t})$. For each $u \in \mathcal{D}_{\pi}$, we may define the representation $\pi_u$ by
\[
  \pi_u=\text{Ind}(\Delta_1\nu^{u_1}\otimes \dotsm \otimes \Delta_t\nu^{u_t}).
\]
This is the family of deformation of $\pi=\pi_0$ we are interested in. Note that each representation $\pi_u$ is of Whittaker type.

\par

 Suppose that the quasi-square-integrable representation $\Delta_i$ corresponds to the segment $[\rho_i,\rho_i\nu,\dotsm,\rho_i\nu^{\ell_i-1}]$ where $\rho_i$ is a supercuspidal representation of $GL_{r_i}$. Then $\Delta_i$ is a representation of $GL_{n_i}$ where $n_i=r_i\ell_i$ and $n=\sum n_i=\sum r_i\ell_i$. We say that $\Delta_i=[\rho_i,\rho_i\nu,\dotsm,\rho_i\nu^{\ell_i-1}]$ and $\Delta_j=[\rho_j,\rho_j\nu,\dotsm,\rho_j\nu^{\ell_j-1}]$ are linked if as sets $\Delta_i \not\subseteq \Delta_j$, $\Delta_j \not\subseteq \Delta_i$, and $\Delta_i \cup \Delta_j$ is again a segment. Zelevinsky in \cite{Ze80} establishes that the representations $\pi$ is irreducible as long as the segments corresponding to the $\Delta_{i}$ are unlinked.

\par

Now consider the derivative of $\pi_u$. If we twist $\Delta_i$ with the determinantal character $\nu^{u_i}$ then, setting $\Delta_{i,u_i}=\Delta_i\nu^{u_i}$, we see that $\Delta_{i,u_i}$ is again quasi-square-integrable representation associated to the segment $[(\rho_i\nu^{u_i}),(\rho_i\nu^{u_i})\nu,\dotsm,(\rho_i\nu^{u_i})\nu^{\ell_i-1}]$. By the result of Bernstein and Zelevinsky in Theorem \ref{deriviative}, the $k^{th}$ derivatives $\pi_u^{(k)}$ will be glued from the representation $\text{Ind}(\Delta^{(k_1)}_{1,u_1}\otimes \dotsm \otimes \Delta^{(k_t)}_{t,u_t})$ for all possible partitions $k=k_1+\dotsm+k_t$ with $0 \leq k_i \leq n_i$. Let us set
\[
  \pi_u^{(k_1,\dotsm,k_t)}=\text{Ind}(\Delta^{(k_1)}_{1,u_1}\otimes \dotsm \otimes \Delta^{(k_t)}_{t,u_t}).
\]
If we consider $\Delta_{i,u_i}^{(k_i)}$ which is associated to the segment $[(\rho_i\nu^{u_i}),\dotsm,(\rho_i\nu^{u_i})\nu^{\ell_i-1}]$ we see that $\Delta^{(k_i)}_{i,u_i}$ is zero unless $k_i=a_ir_i$ with $0 \leq a_i \leq \ell_i$ and $\Delta_{i,u_i}^{(a_ir_i)}$ is the quasi-square-integrable representation attached to the segment $[(\rho_i\nu^{u_i})\nu^{a_i},\dotsm,(\rho_i\nu^{u_i})\nu^{\ell_i-1}]$. So we have that $\pi_u^{(k_1,\dotsm,k_t)}=0$ unless $(k_1,\dotsm,k_t)=(a_1r_1,\dotsm,a_tr_t)$ with $0 \leq a_i \leq \ell_i$.

\begin{definition}
Let $\pi=\mathrm{Ind}(\Delta_1 \otimes \dotsm \otimes \Delta_t)$ be a representation of Whittaker type on $GL_n$. We assume that $\Delta_i=[\rho_i,\rho_i\nu,\dotsm,\rho_i\nu^{\ell_i-1}]$ is a representation of $GL_{n_i}$ where $\rho_i$ is an irreducible supercuspidal representation of $GL_{r_i}$, $n_i=r_i \ell_i$ and $n=\sum n_i=\sum r_i\ell_i$. We say that $u=(u_1,\dotsm,u_t) \in \mathcal{D}_{\pi}$ is in general position if it satisfies the following properties:
\begin{enumerate}
\item[$(\mathrm{1})$] For every sequence of nonnegative integers $(k_1,\dotsm,k_t)$ such that the representation 
\[
 \pi_u^{(k_1,\dotsm,k_t)}=\mathrm{Ind}(\Delta_1^{(k_1)}\nu^{u_1} \otimes \dotsm \otimes \Delta_t^{(k_t)}\nu^{u_t}) 
\]
this representation is irreducible.
\item[$(\mathrm{2})$] If $(a_1r_1,\dotsm,a_tr_t)$ and $(b_1r_1,\dotsm,b_tr_t)$ are two different sequences such that 
\[
\sum_{i=1}^t a_ir_i=\sum_{i=1}^t b_ir_i,
\]
 then the two representations
\[
  \mathrm{Ind}(\Delta_1^{(a_1r_1)}\nu^{u_1} \otimes \dotsm \otimes \Delta_t^{(a_tr_t)}\nu^{u_t}) \quad \mathrm{and} \quad \mathrm{Ind}(\Delta_1^{(b_1r_1)}\nu^{u_1} \otimes \dotsm \otimes \Delta_t^{(b_tr_t)}\nu^{u_t})
\]
have distinct central characters.
\item[$(\mathrm{3})$] If $(i,j) \in \{1,\dotsm,t\}$ with $i \neq j$, then $L(s,\Delta_i\nu^{u_i},\wedge^2)$ and $L(s,\Delta_j\nu^{u_j},\wedge^2)$ have no common poles.
\item[$(\mathrm{4})$] If $(i,j,k,l) \in \{1,\dotsm,t\}$ with $\{i,j\} \neq \{k,l \}$, $L(s,\Delta_i\nu^{u_i} \times \Delta_j\nu^{u_j})$ and $L(s,\Delta_k\nu^{u_k} \times \Delta_l\nu^{u_l})$ have no common poles.
\item[$(\mathrm{5})$] If $(i,j,k) \in \{ 1,\dotsm,t \}$ with $i \neq j$, then $L(s,\Delta_i\nu^{u_i} \times \Delta_j\nu^{u_j})$ and $L(s,\Delta_k\nu^{u_k},\wedge^2)$ have no common poles.
\item[$(\mathrm{6})$] If $1 \leq i \neq j \leq t$ and $(\Delta_i^{(a_ir_i)})^{\sim} \simeq \Delta_j^{(a_jr_j)}\nu^e$ for some complex number $e$, then the space
\[
  \mathrm{Hom}_{P_{2(n_i-a_ir_i)} \cap S_{2(n_i-a_ir_i)}}(\mathrm{Ind}(\Delta_i^{(a_ir_i)}\nu^{\frac{u_i-u_j+e}{2}}\otimes(\Delta_i^{(a_ir_i)}\nu^{\frac{u_i-u_j+e}{2}})^{\sim}),\Theta)
\]
is of dimension at most $1$.
\end{enumerate}
\end{definition}

The important point is the fact that  the notion $u \in \mathcal{D}_{\pi}$ in general position depends only on the representation $\pi$. The purpose of $(2)$ is that outside of a finite number on hyperplanes in the $u \in \mathcal{D}_{\pi}$ the central characters of the $\pi_u^{(a_1r_1,\dotsm,a_tr_t)}$ will be distinct and so there are only trivial extension among these representations. Therefore off these hyperplanes, the derivatives $\pi_u^{(k)}=\oplus\pi_u^{(a_1r_1,\dotsm,a_tr_t)}$ are completely reducible, where $k=\sum a_ir_i$ and each $\pi_u^{(a_1r_1,\dotsm,a_tr_t)}$ is irreducible. The reason for condition $(6)$ is that the occurrence of an exceptional pole of $L(s,\pi,\wedge^2)$ at $s=0$ is related to the occurrence of Shalika functional for $\pi$ from Lemma \ref{char-Shalika} under this condition $(6)$.

\par

 We now check that outside a finite number of hyperplanes in $u$, the deformed representation $\pi_u$ is in general position.

 \begin{proposition}
 \label{general position}
 Let $\pi=\mathrm{Ind}(\Delta_1 \otimes \dotsm \otimes \Delta_t)$ be a representation of Whittaker type on $GL_n$. We assume that $\Delta_i=[\rho_i,\rho_i\nu,\dotsm,\rho_i\nu^{\ell_i-1}]$ is a representation of $GL_{n_i}$ where $\rho_i$ is an irreducible supercuspidal representation of $GL_{r_i}$, $n_i=r_i \ell_i$ and $n=\sum n_i=\sum r_i\ell_i$. The elements $u \in \mathcal{D}_{\pi}$ which is not in general position belong to a finite number of affine hyperplanes.
 \end{proposition}

 \begin{proof}
 The condition $(1), (2)$ and $(4)$ are explained in detail in \cite[Proposition 5.2]{Ma15}.
 \par

 If the condition $(3)$ is not satisfied, then $L(s,\Delta_i\nu^{u_i},\wedge^2)$ and $L(s,\Delta_j\nu^{u_j},\wedge^2)$ have a common pole for some $i \neq j$. According to Proposition \ref{even-quasisquare} and \ref{odd-quasisquare}, this implies that the two factors $L_{ex}(s,\Delta^{(a_ir_i)}_i\nu^{u_i},\wedge^2)$ and $L_{ex}(s,\Delta^{(a_jr_j)}_j\nu^{u_j},\wedge^2)$ for which $n_i-a_ir_i$ and $n_j-a_jr_j$ are even numbers, have a common pole at $s=s_0$. The existence of non-trivial Shalika functional assures the following twisted self-contragrediance condition from Theorem \ref{Jacquet-Rallis}:
\[
  (\Delta^{(a_ir_i)}_i\nu^{u_i+\frac{s_0}{2}})^{\sim} \simeq \Delta^{(a_ir_i)}_i\nu^{u_i+\frac{s_0}{2}} \quad \mathrm{and} \quad (\Delta^{(a_jr_j)}_j\nu^{u_j+\frac{s_0}{2}})^{\sim} \simeq \Delta^{(a_jr_j)}_j\nu^{u_j+\frac{s_0}{2}}.
\] 
 Then the equality of the central character at $\varpi$ gives the relation
 \[
   q^{(n_i-a_ir_i)(n_j-a_jr_j)(u_i-u_j)}\omega_{\Delta_j^{(a_jr_j)}}^{n_i-a_ir_i}(\varpi)\omega_{\Delta_i^{(a_ir_i)}}^{-(n_j-a_jr_j)}(\varpi)=1.
 \]
 Hence $u \in \mathcal{D}_{\pi}$ belongs to a finite number of hyperplanes if $(3)$ is not satisfied.
 \par

 If condition $(5)$ is not satisfied, then $L(s,\Delta_i\nu^{u_i} \times \Delta_j\nu^{u_j})$  and $L(s,\Delta_k\nu^{u_k},\wedge^2)$ have a common pole for some $i \neq j$. This implies that the two factors $L_{ex}(s,\Delta^{(a_ir_i)}_i\nu^{u_i} \times \Delta^{(a_jr_j)}_j\nu^{u_j})$ and $L_{ex}(s,\Delta^{(a_kr_k)}_k\nu^{u_k},\wedge^2)$ for which $n_k=a_kr_k$ is an even number, have a common pole at $s=s_0$ according to Proposition \ref{even-quasisquare} and \ref{odd-quasisquare}. The existence of non trivial Shalika functional in turn implies the following twisted self-contragrediance condition by Theorem \ref{Jacquet-Rallis}. Then this gives the conditions:
 \[
    (\Delta^{(a_ir_i)}_i)^{\sim} \simeq \Delta^{(a_jr_j)}_j\nu^{u_i+u_j+s_0} \quad \mathrm{and} \quad (\Delta^{(a_kr_k)}_k\nu^{u_k+\frac{s_0}{2}})^{\sim} \simeq \Delta^{(a_kr_k)}_k\nu^{u_k+\frac{s_0}{2}}
 \]
 We obtain the equation
 \[
 \begin{split}
   &q^{(n_i+n_j-a_ir_i-a_jr_j)(n_k-a_kr_k)(u_i+u_j-2u_k)}\omega_{\Delta_k^{(a_kr_k)}}^{2(n_i+n_j-a_ir_i-a_jr_j)}(\varpi)\omega_{\Delta_i^{(a_ir_i)}}^{-(n_k-a_kr_k)}(\varpi)\omega_{\Delta_j^{(a_jr_j)}}^{-(n_k-a_kr_k)}(\varpi \hspace*{-.5mm})\\
   &=1.
 \end{split}
 \]
Hence $u \in \mathcal{D}_{\pi}$ belongs to a finite number of hyperplanes if $(5)$ is not satisfied.
 \par

 The last case $(6)$ needs to be treated more carefully. We assume that $\widetilde{\Delta}_i \simeq \Delta_j\nu^e$ for some complex number $e$. We want to show that outside of a finite number of hyperplanes in $u \in \mathcal{D}_{\pi}$, the space
 \[
    \mathrm{Hom}_{P_{2n_i} \cap S_{2n_i}}(\mathrm{Ind}(\Delta_i\nu^{\frac{u_i-u_j+e}{2}}\otimes(\Delta_i\nu^{\frac{u_i-u_j+e}{2}})^{\sim}),\Theta)
\]
is of dimension at most $1$. To simplify notation, let $\pi$ denote the induced representation $\mathrm{Ind}(\Delta_i\nu^{\frac{u_i-u_j+e}{2}}\otimes(\Delta_i\nu^{\frac{u_i-u_j+e}{2}})^{\sim})$ and let us take $s$ to be $\frac{u_i-u_j+e}{2}$. According to Proposition \ref{even-embedding},
the space $\mathrm{Hom}_{P_{2n_i} \cap S_{2n_i}}(\mathrm{Ind}(\Delta_i\nu^{s}\otimes(\Delta_i\nu^{s})^{\sim}),\Theta)$ embeds as a subspace of the vector space $\mathrm{Hom}_{P_{2n_i} \cap H_{2n_i}} (\mathrm{Ind}(\Delta_i\nu^{s}\otimes(\Delta_i\nu^{s})^{\sim}),\delta_{2n_i}^{-s_0})$, where a real number $s_0 > r_{\pi}$ is defined in Proposition \ref{even-embedding}. The restriction of $\mathrm{Ind}(\Delta_i\nu^{s}\otimes(\Delta_i\nu^{s})^{\sim})$  to $P_{2n_i}$ has a filtration by derivatives with each successive quotient of the form
 \[
 \begin{split}
 & (\Phi^+)^{(c_i+d_i)r_i-1}\Psi^+(\mathrm{Ind}([(\rho_i\nu^s)\nu^{c_i},(\rho_i\nu^s)\nu^{c_i+1},\dotsm,(\rho_i\nu^s)\nu^{\ell_i-1}] \\
&\phantom{*********************} \otimes[(\widetilde{\rho}_i\nu^{-s})\nu^{d_i+1-\ell_i},(\widetilde{\rho}_i\nu^{-s})\nu^{d_i+2-\ell_i},\dotsm,(\widetilde{\rho}_i\nu^{-s})]))
 \end{split}
 \]
 for $c_i+d_i \leq 2\ell_i$. From \eqref{even-func-points}, one has an isomorphism between the space
 \[
 \begin{split}
  & \mathrm{Hom}_{P_{2n_i} \cap H_{2n_i}}( (\Phi^+)^{(c_i+d_i)r_i-1}\Psi^+(\mathrm{Ind}([(\rho_i\nu^s)\nu^{c_i},(\rho_i\nu^s)\nu^{c_i+1},\dotsm,(\rho_i\nu^s)\nu^{\ell_i-1}] \\
   &\phantom{*****************} \otimes[(\widetilde{\rho}_i\nu^{-s})\nu^{d_i+1-\ell_i},(\widetilde{\rho}_i\nu^{-s})\nu^{d_i+2-\ell_i},\dotsm,(\widetilde{\rho}_i\nu^{-s})])),\delta_{2n_i}^{-s_0})
 \end{split}
 \]
 and
 \[
  \begin{split}
  & \mathrm{Hom}_{GL_{2n_i-(c_i+d_i)r_i}}(\mathrm{Ind}([(\rho_i\nu^s)\nu^{c_i},(\rho_i\nu^s)\nu^{c_i+1},\dotsm,(\rho_i\nu^s)\nu^{\ell_i-1}] \\
   &\phantom{****} \otimes[(\widetilde{\rho}_i\nu^{-s})\nu^{d_i+1-\ell_i},(\widetilde{\rho}_i\nu^{-s})\nu^{d_i+2-\ell_i},\dotsm,(\widetilde{\rho}_i\nu^{-s})]),\mu_{2n_i-(c_i+d_i)r_i}\delta^{-s_0}_{2n_i-(c_i+d_i)r_i}\nu^{-\frac{1}{2}}),
 \end{split}
 \]
where $\mu_{2n_i-(c_i+d_i)r_i}=\delta_{2n_i-(c_i+d_i)r_i}$ when $2n_i-(c_i+d_i)r_i$ is even, and $\mu_{2n_i-(c_i+d_i)r_i}$ is trivial when $2n_i-(c_i+d_i)r_i$ is odd. For $2n_i-(c_i+d_i)r_i=0$, the dimension of linear form
 $ \mathrm{Hom}_{P_{2n_i} \cap H_{2n_i}}( (\Phi^+)^{2n_i-1}\Psi^+(\textbf{1}),\delta_{2n_i}^{-s_0})$ is of at most $1$. 
 
\par

 Suppose that $c_i=d_i < \ell$. The central character of
 \[
 \begin{split}
   &\mathrm{Ind}([(\rho_i\nu^s)\nu^{c_i},(\rho_i\nu^s)\nu^{c_i+1},\dotsm,(\rho_i\nu^s)\nu^{\ell_i-1}]  \\
  &\phantom{**********************}\otimes [(\widetilde{\rho}_i\nu^{-s})\nu^{d_i+1-\ell_i},(\widetilde{\rho}_i\nu^{-s})\nu^{d_i+2-\ell_i},   \dotsm,(\widetilde{\rho}_i\nu^{-s})])
\end{split}
 \] 
 is given by $z \mapsto |z|^{r_ic_i(\ell_i-c_i)}$. The space
 \[
 \begin{split}
  & \mathrm{Hom}_{GL_{2n_i-(c_i+d_i)r_i}}(\mathrm{Ind}([(\rho_i\nu^s)\nu^{c_i},(\rho_i\nu^s)\nu^{c_i+1},\dotsm,(\rho_i\nu^s)\nu^{\ell_i-1}] \\
   &\phantom{****} \otimes[(\widetilde{\rho}_i\nu^{-s})\nu^{d_i+1-\ell_i},(\widetilde{\rho}_i\nu^{-s})\nu^{d_i+2-\ell_i},\dotsm,(\widetilde{\rho}_i\nu^{-s})]),\mu_{2n_i-(c_i+d_i)r_i}\delta^{-s_0}_{2n_i-(c_i+d_i)r_i}\nu^{-\frac{1}{2}})
   \end{split}
 \]
 is reduced to zero, because the character
 \[
   z \mapsto \mu_{2n_i-(c_i+d_i)r_i}\delta^{-s_0}_{2n_i-(c_i+d_i)r_i}\nu^{-\frac{1}{2}}(zI_{2n_i-(c_i+d_i)r_i})=|z|^{-(n_i-c_ir_i)}=|z|^{(c_i-\ell_i)r_i}
 \]  
 has negative parts and $z \mapsto |z|^{r_ic_i(\ell_i-c_i)}$ has positive parts.
 
 \par
 
  If $c_i \neq d_i$, the central character of
 \[
 \begin{split}
    &\mathrm{Ind}([(\rho_i\nu^s)\nu^{c_i},(\rho_i\nu^s)\nu^{c_i+1},\dotsm,(\rho_i\nu^s)\nu^{\ell_i-1}] \\
    &\phantom{**********************} \otimes[(\widetilde{\rho}_i\nu^{-s})\nu^{d_i+1-\ell_i},(\widetilde{\rho}_i\nu^{-s})\nu^{d_i+2-\ell_i},   \dotsm,(\widetilde{\rho}_i\nu^{-s})])
\end{split}
 \]
 is equal to $z \mapsto |z|^{sr_i(d_i-c_i)}\omega_{\rho_i}^{d_i-c_i}|z|^{r_i(c_i+\dotsm+(\ell_i-1)+(1-\ell_i+d_i)+\dotsm+(-1))}$, hence the uniqueness of functionals in the space $\mathrm{Hom}_{P_{2n_i} \cap S_{2n_i}}(\mathrm{Ind}(\Delta_i\nu^{s}\otimes(\Delta_i\nu^{s})^{\sim}),\Theta)$ fails when we obtain the relation from the equality of the central characters at $\varpi$:
 \[
 \begin{split}
 & q^{-sr_i(d_i-c_i)}\omega_{\rho_i}^{d_i-c_i}(\varpi)q^{-r_i(c_i+\dotsm+(\ell_i-1)+(1-\ell_i+d_i)+\dotsm+(-1))}\\
&=\mu_{2n_i-(c_i+d_i)r_i}\delta^{-s_0}_{2n_i-(c_i+d_i)r_i}\nu^{-\frac{1}{2}}(\varpi I_{2n_i-(c_i+d_i)r_i})\\
&=  \left\{  
    \begin{array}{l l}
   q^{\frac{2n_i-(c_i+d_i)r_i}{2}} &\quad \text{if \quad $2n_i-(c_i+d_i)r_i$ is even} \\
   q^{s_0+\frac{2n_i-(c_i+d_i)r_i}{2}} &\quad \text{if \quad $2n_i-(c_i+d_i)r_i$ is odd.}
     \end{array} 
     \right.
 \end{split}
 \]
 for all $c_i+d_i \leq 2\ell_i$ with $c_i \neq d_i$. Reasoning as in the proof of Proposition \ref{even-multi one}, on the intersection of the complements of the hyperplanes
\[
\label{hyper-plane}
\tag{7.1}
  q^{-r_i(d_i-c_i)\frac{u_i-u_j+e}{2}}\omega_{\rho_i}^{d_i-c_i}(\varpi)q^{-r_i(c_i+\dotsm+\ell_i-1+(1-\ell_i+d_i)+\dotsm+(-1))}
=q^{\frac{2n_i-(c_i+d_i)r_i}{2}},
\]
where $2n_i-(c_i+d_i)r_i$ is an even number with $c_i+d_i \leq 2\ell_i$ and $c_i \neq d_i$, the dimension of the space $\mathrm{Hom}_{P_{2n_i} \cap S_{2n_i}}(\mathrm{Ind}(\Delta_i\nu^{\frac{u_i-u_j+e}{2}}\otimes(\Delta_i\nu^{\frac{u_i-u_j+e}{2}})^{\sim}),\Theta)$ is at most one.

\par

 We now assume that $(\Delta_i^{(a_ir_i)})^{\sim} \simeq \Delta_j^{(a_jr_j)} \nu^f$ for some complex number $f$. As $ \Delta_i^{(a_ir_i)}=[\rho_i\nu^{a_i},\dotsm,\rho_i\nu^{\ell_i-1}]$ by Theorem \ref{deriviative}, we write
 $\Delta_i^{(a_ir_i)}=[\rho_i',\dotsm,\rho_i'\nu^{\ell_i-a_i-1}]$ where $\rho_i'=\rho_i\nu^{a_i}$ is an irreducible supercuspidal representation of $GL_{r_i}$. The length of the segment becomes $\ell_i-a_i$ and $\mathrm{Ind}(\Delta^{(a_ir_i)}_i\nu^{\frac{u_i-u_j+e}{2}}\otimes(\Delta^{(a_ir_i)}_i\nu^{\frac{u_i-u_j+e}{2}})^{\sim})$ is a representation of $GL_{2(n_i-a_ir_i)}$. The equation \eqref{hyper-plane} amounts to
 \[
    q^{-r_i(d_i-c_i)\frac{u_i-u_j+e}{2}}\omega_{\rho'_i}^{d_i-c_i}(\varpi)q^{-r_i(c_i+\dotsm+(\ell_i-a_i-1)+(1-\ell_i+a_i+d_i)+\dotsm+(-1))}
=q^{\frac{2(n_i-a_ir_i)-(c_i+d_i)r_i}{2}},
 \]
 where $2(n_i-a_ir_i)-(c_i+d_i)r_i$ is an even number with $c_i+d_i \leq 2\ell_i-2a_i$, $c_i \neq d_i$. But we have $\omega_{\rho'_i}^{d_i-c_i}(\varpi)=\omega_{\rho_i}^{d_i-c_i}(\varpi)q^{-r_ia_i(d_i-c_i)}$ and
 \[
 \begin{split}
  &q^{-r_i(c_i+\dotsm+(\ell_i-a_i-1)+(1-\ell_i+a_i+d_i)+\dotsm+(-1))}\\
  &=q^{a_ir_i(\ell_i-a_i-c_i)-a_ir_i(\ell_i-a_i-d_i)}q^{-r_i((a_i+c_i)+\dotsm+(\ell_i-1)+(1-\ell_i+d_i)+\dotsm+(-a_i))}\\
& =q^{r_ia_i(d_i-c_i)}q^{-r_i((a_i+c_i)+\dotsm+(\ell_i-1)+(1-\ell_i+d_i)+\dotsm+(-a_i))}
 \end{split}
 \]
 In a similar fashion to the previous case when $\widetilde{\Delta}_i \simeq \Delta_j\nu^e$ for some complex number $e$, 
 on the intersection of the complements of the hyperplanes
\[
  q^{-r_i(d_i-c_i)\frac{u_i-u_j+e}{2}}\omega_{\rho_i}^{d_i-c_i}(\varpi)q^{-r_i((a_i+c_i)+\dotsm+(\ell_i-1)+(1-\ell_i+d_i)+\dotsm+(-a_i))}=q^{\frac{2(n_i-a_ir_i)-(c_i+d_i)r_i}{2}},
\]
where $2(n_i-a_ir_i)-(c_i+d_i)r_i$ is an even number with $c_i+d_i \leq 2\ell_i-2a_i$, $c_i \neq d_i$, 
 the dimension of 
 \[
 \mathrm{Hom}_{P_{2(n_i-a_ir_i)} \cap S_{2(n_i-a_ir_i)}}(\mathrm{Ind}(\Delta^{(a_ir_i)}_i\nu^{\frac{u_i-u_j+e}{2}}\otimes(\Delta^{(a_ir_i)}_i\nu^{\frac{u_i-u_j+e}{2}})^{\sim}),\Theta)
 \]
  is at most one.

  \end{proof}

 The interest of considering deformed representations $\pi_u$ in general position is that $L$-functions $L(s,\pi_u,\wedge^2)$ are easy to analyze for those representations. In Section 8, we eventually specialize the result to $u=0$ which may not be in general position.

\subsection{Rationality properties of integrals under deformation.}

In this section we would like to show that the Jacquet-Shalika integrals defining the $L$-functions $L(s,\pi_u,\wedge^2)$ for deformed representations are rational in $\mathbb{C}(q^{-u},q^{-s})$ via Bernstein's method \cite{Ba98}. We essentially follow notations and the argument of rationality of solutions of polynomial system in \cite{CoPe} for the Rankin-Selberg convolution $L$-function, which is utilized for the Asai $L$-function in \cite{Ma09}. 
Let $\pi=\text{Ind}(\Delta_1 \otimes \dotsm \otimes \Delta_t)$ be a representation of Whittaker type on $GL_m$ and the induction is normalized parabolic induction from the standard parabolic subgroup $Q$ of $GL_m$. Let $M \simeq GL_{m_1} \times \dotsm \times GL_{m_t}$ denote the Levi subgroup of $Q$. If $u=(u_1,\dotsm,u_t) \in \mathcal{D}_{\pi}$, we defines an unramified character $\nu^u$ of $M$ by $\nu^u(m)=\nu^u(g_1,\dotsm,g_t)=\nu(g_1)^{u_1}\dotsm\nu(g_t)^{u_t}$ for $m \in M$. We apologize for the double use of $m$, but we hope that the reader can separate them by context. For each $u \in \mathcal{D}_{\pi}$ we define the deformed representation by $\pi_u=\text{Ind}(\Delta_1\nu^{u_1} \otimes \dotsm \otimes \Delta_t\nu^{u_t})$. We would like to describe that the family of deformation of representations $\pi_u$ has the structure of a trivial vector bundle over $\mathcal{D}_{\pi}$. We realize each quasi-square-integrable representation $\Delta_i$ in its Whittaker model $\mathcal{W}(\Delta_i,\psi)$. 
The space $V_{\pi}$ of $\pi$ can be realized as the space of smooth functions 
\[
  f : GL_m \rightarrow \mathcal{W}(\Delta_1,\psi) \otimes \dotsm \otimes \mathcal{W}(\Delta_t,\psi),
\]
which we write as $f(m;g)$ with $g \in GL_m$ and $m \in M$ satisfying
\[
  f(m;hg)=\delta_Q^{\frac{1}{2}}(m_h)f(mm_h;g)
\]
for $h \in Q, h=nm_h$ with $m_h \in M$ and $n$ in the unipotent radical $N_Q$ of $Q$. Since $GL_m=QK_m$, each function $f$ is determined by its restriction to $K_m$ and the value of $f$ is
\[
  f(m;g)=f(m;nm_gk_g)=\delta_Q^{\frac{1}{2}}(m_g)f(mm_g;k_g)
\]
for $g \in GL_m, g=nm_gk_g$ with $n \in N_Q, m_g \in M$ and $k_g \in K_m$. We have the so-called compact realization of $\pi$ on the space of smooth functions
\[
  f : K_m \rightarrow \mathcal{W}(\Delta_1,\psi) \otimes \dotsm \otimes \mathcal{W}(\Delta_t,\psi),
\]
which we again write as $f(m;k)$. Let us denote this space by $\mathcal{F}_{\pi}$. Now the action of $\pi(g)$ is given by
\[
  (\pi(g)f)(m;k)=\delta_Q^{\frac{1}{2}}(m')f(mm';k')
\]
where $kg=n'm'k'$ with $n' \in N_Q$, $m' \in M$ and $k' \in K_m$. Then each $\pi_u$ can also be realized on $\mathcal{F}_{\pi}$ with actions being given by
\[
  (\pi_u(g)f)(m;k)=\nu^u(m')\delta_Q^{\frac{1}{2}}(m')f(mm';k')
\]
where the decomposition $n'm'k'$ of $kg$ is as above. We may now form a bundle of representations over $\mathcal{D}_{\pi}$ where the fiber over $u \in \mathcal{D}_{\pi}$ is the representation $(\pi_u,\mathcal{F}_{\pi})$. As a vector bundle, this is a trivial bundle $\mathcal{D}_{\pi} \times \mathcal{F}_{\pi}$, with different actions of $GL_m$ in each fibre. In the usual model of representation $\pi_u$ as smooth function on $GL_m$, $f$ determines the function $f_u$ defined by $f_u(m;g)=f_u(m;nm_gk_g)=\nu^u(m_g)\delta_Q^{\frac{1}{2}}(m_g)f(mm_g;k_g)$ where $g=nm_gk_g$ with $n \in N_Q, m_g \in M$ and $k_g \in K_m$. Then the function $f_u$ satisfies $f_u(m;g)=(\pi_u(g)f)(m;e)$ for $g \in GL_m$.

\par

We clarify the Weyl element $w_Q$ and the unipotent radical $N_Q$ defining the Whittaker functions in  \cite[Section 3.1]{CoPe}.
Let $w_M$ be the longest element in the Weyl groups of $M\simeq GL_{m_1} \times \dotsm \times GL_{m_t}$ of the form
\[
 w_M=\begin{pmatrix} w_{\ell,1} &&&\\ &w_{\ell,2}&& \\ &&\ddots& \\&&&w_{\ell,t} \end{pmatrix}
\] 
with each $w_{\ell,i}$ the representative of the long Weyl element of $GL_{m_i}$. We set $w_Q:=ww_M$. Let $M':=w_QMw_Q^{-1}$ be the standard Levi subgroup of type $(m_t,\dotsm,m_2,m_1)$ and $Q'$ the standard parabolic subgroup of type $(m_t,\dotsm,m_2,m_1)$, ${N_Q}'$ its unipotent radical.
As in \cite[Section 1.5]{CaSh80} for minimal parabolic subgroups or in \cite[Section 3.3]{Sh14}, a Whittaker functional on $\mathcal{F}_{\pi}$ is defined by
\[
   \lambda_u(f)=\int_{{N_f}'} (\pi_u(n)f)(e;w_Q^{-1}) \psi^{-1}(n)\; dn,
\]
where ${N_f}'$ is a sufficiently large compact open subgroup of the unipotent radical ${N_Q}'$. ${N_f}'$ depends on $f$ but is independent of $u \in \mathcal{D}_{\pi}$. For each $f \in \mathcal{F}_{\pi}$, we define the Whittaker function by
\[
  W_{f,u}(g)=\lambda_u(\pi_u(g)f).
\]
We let $\mathcal{W}^{(0)}_{\pi}=\mathcal{W}^{(0)}$ denote the complex vector space generated by the functions $(u,g) \mapsto W_{f,u}(gg')$ for $g' \in GL_m$ and $f \in \mathcal{F}_{\pi}$. It is shown in \cite[Section 3.1]{CoPe} that the action of the group $GL_m$ on $\mathcal{W}^{(0)}_{\pi}$ by right translation is a smooth representation. Note that
$
  W_{\pi_u(g')f,u}(g)=\lambda_u(\pi_u(g)\pi_u(g')f)=\lambda_u(\pi_u(gg')f)=W_{f,u}(gg')
$
and 
$\mathcal{W}^{(0)}=\langle W_{\pi_u(g')f,u}(g)=W_{f,u}(gg')\; |\; f \in \mathcal{F},\; g' \in GL_m \rangle$. Let $\mathcal{W}_{(0)}$ denote the space of restrictions of functions of $\mathcal{W}^{(0)}$ to $P_m$, that is,
$\mathcal{W}_{(0)}=\{W(p)\;|\;W \in \mathcal{W}^{(0)},\;p \in P_m\}$. It is also proven in \cite[Section 3.1]{CoPe} that for fixed $g$, the function $u \mapsto W_{f,u}(g)$ belongs to $\mathbb{C}[q^{\pm u}]$. Let $\mathcal{P}_0$ be the vector subspace of $\mathbb{C}[q^{\pm u}]$ consisting of all polynomials of the form $u \mapsto W_{f,u}(I_m)$ for some $W_{f,u} \in \mathcal{W}^{(0)}$, that is,
\[
  \mathcal{P}_0=\langle W_u(I_m)\;|\; W_u \in  \mathcal{W}^{(0)}  \rangle. 
\]
 If $V$ is any complex vector space, let us denote by $\mathcal{S}_{\psi}(P_m,V)$ the space of smooth functions $\varphi : P_m \rightarrow V$ which satisfy $\varphi(np)=\psi(n)\varphi(p)$
for $n \in N_m$ and $p \in P_m$ and which are compactly supported modulo $N_m$. We now state Proposition 3.1 of \cite{CoPe}, which we will use later.

\begin{proposition}[Cogdell and Piatetski-Shapiro]
\label{poly-GeKa}
Let $\pi=\mathrm{Ind}(\Delta_1 \otimes \dotsm \otimes \Delta_t)$ be a representation of $GL_m$ of Whittaker type. The complex vector space $\mathcal{W}_{(0)}$ defined above contains the space $\mathcal{S}_{\psi}(P_m,\mathcal{P}_0)$.
\end{proposition}


We prove that the local integrals $J(s,W_u,\Phi)$ are rational functions in $q^{\pm u}$ and $q^{\pm s}$ if we deform $\pi$ to $\pi_u$ and form the corresponding spaces of functions $\mathcal{F}_{\pi}$ and $\mathcal{W}^{(0)}_{\pi}$. As in \cite{CoPe}, our method will be to employ the following theorem of Bernstein about solutions of a polynomial family of affine equations. Let $V$ be a complex vector space of countable dimension over $\mathbb{C}$ and let $V^*=\text{Hom}_{\mathbb{C}}(V,\mathbb{C})$ denote its algebraic dual. Let $\Xi$ be a collection $\{(x_r,c_r)\;|\;r \in R \}$ where $x_r \in V$, $c_r \in \mathbb{C}$, and $R$ is an index set. A linear functional $\lambda \in V^*$ is said to be a solution of the system $\Xi$ if $\lambda(x_r)=c_r$ for all $r \in R$. Let $\mathcal{D}$ denote an irreducible algebraic variety over $\mathbb{C}$ and suppose that for each $d \in \mathcal{D}$, we are given a system of equations $\Xi_d=\{(x_r(d),c_r(d))\;|\;r\in R \}$ with index set $R$ independent of $d \in \mathcal{D}$. We will say that a family of system $\{\Xi_d\;|\;d \in \mathcal{D} \}$ is a polynomial family of system of equation if, for each $d \in \mathcal{D}$, the functions $x_r(d)$ and $c_r(d)$ vary polynomially in $d$ and belong respectively to $\mathbb{C}[\mathcal{D}] \otimes_{\mathbb{C}}V$ and $\mathbb{C}[\mathcal{D}]$. Let $\mathcal{M}=\mathbb{C}(\mathcal{D})$ be the field of fractions of $\mathbb{C}[\mathcal{D}]$. We denote by $V_{\mathcal{M}}$ the space $\mathcal{M}\otimes_{\mathbb{C}}V$ and by $V_{\mathcal{M}}^*$  the space $\text{Hom}_{\mathcal{M}}(V_{\mathcal{M}},\mathcal{M})$. The following statement is a consequence of Bernstein's theorem and its corollary appeared first in \cite{CoPe} or \cite[Theorem 2.11]{Ma15,Ba98}.

\begin{theorem}[Bernstein]
\label{Bernstein}
With the above notation, suppose that $V$ has countable dimension over $\mathbb{C}$ and suppose that there exists an non-empty subset $\Omega \subset \mathcal{D}$, open in the usual complex topology of $\mathcal{D}$, such that for each $d \in \Omega$ the system $\Xi_d$ has a unique solution $\lambda_d$. Then the system $\Xi=\{(x_r(d),c_r(d))\;|\; r \in R \}$ viewed as a system over the field $\mathcal{M}=\mathbb{C}(\mathcal{D})$  has a unique solution $\lambda (d) \in V_{\mathcal{M}}^*$ and $\lambda(d)=\lambda_d$ is the unique solution of $\Xi$ on $\Omega$. 
\end{theorem}

\subsection{The even case $m=2n$}

Let $\pi=\text{Ind}(\Delta_1 \otimes \dotsm \otimes \Delta_t)$ be a representation of Whittaker type on $GL_{2n}$ and for each $u \in \mathcal{D}_{\pi}$ we define the deformed representation by $\pi_u=\text{Ind}(\Delta_1\nu^{u_1} \otimes \dotsm \otimes \Delta_t\nu^{u_t})$. We realize each of these representations on the common vector space $\mathcal{F}_{\pi}$ and form the representation of $GL_{2n}$ on the Whittaker space $\mathcal{W}_{\pi}^{(0)}$. Observe that since $\pi$ is an admissible representation, the vector space is countable dimensional over $\mathbb{C}$. Consider the local integral for the even case $2n$. In this case the local integrals involves not just the Whittaker functions associated to $\pi_u$ but also a Schwartz-Bruhat function on $F^n$. For $W_u \in \mathcal{W}_{\pi}^{(0)}$, and $\Phi \in \mathcal{S}(F^n)$ these integrals are
 \[
 \begin{split}
   &J(s,W_u,\Phi)\\
   &=\int_{N_n \backslash GL_n} \int_{\mathcal{N}_n \backslash \mathcal{M}_n} W_u \left( \sigma_{2n} \begin{pmatrix} I_n&X \\ &I_n\end{pmatrix} \begin{pmatrix} g&\\&g \end{pmatrix} \right) 
   \psi^{-1}(\mathrm{Tr}X)dX \Phi(e_ng) |\mathrm{det}(g)|^s dg.
\end{split}
 \]
 We apply Bernstein's theorem to show that $J(s,W_u,\Phi)$ for $W_u \in \mathcal{W}^{(0)}_{\pi}$ and $\Phi \in \mathcal{S}(F^n)$ are rational functions in $q^{\pm u}$ and $q^{\pm s}$.

\begin{proposition}
\label{even-rational}
Let $\pi=\mathrm{Ind}(\Delta_1 \otimes \dotsm \otimes \Delta_t)$ be a representation of Whittaker type on $GL_{2n}$. With the above notation, for $W_u \in \mathcal{W}^{(0)}_{\pi}$ and $\Phi \in \mathcal{S}(F^n)$ the integral $J(s,W_u,\Phi)$ is a rational function of $q^{-u}$ and $q^{-s}$.
\end{proposition}

\begin{proof}
We will need to view these integrals as defining a polynomial system of equations for the underlying vector space $\mathcal{V}=\mathcal{F}_{\pi} \otimes \mathcal{S}(F^n)$. The space $\mathcal{V}$ is still countable dimensional over $\mathbb{C}$. First note that any $W_u \in \mathcal{W}^{(0)}_{\pi}$ is the Whittaker function attached to a finite linear combination of translated elements of $\mathcal{F}_{\pi}$, that is, $W_u=W_{f,u}$ where $f=\sum \pi_u(g_i)f_i$ with $g_i \in GL_{2n}$ and $f_i \in \mathcal{F}_{\pi}$. The local integrals have well known quasi-invariance properties from \eqref{invariance}, namely
\[
\begin{split}
  &J\left(s,\pi_u\left( \begin{pmatrix} I_n &Z \\ &I_n \end{pmatrix}\begin{pmatrix} g&\\&g \end{pmatrix} \right)W_u, R\left( \begin{pmatrix} I_n &Z \\ &I_n \end{pmatrix}\begin{pmatrix} g&\\&g \end{pmatrix} \right)\Phi \right) \\
  &\phantom{********************************}=|\mathrm{det}(g)|^{-s} \psi(\mathrm{Tr}Z) J(s,W_u,\Phi)
\end{split}
\]
for each $g \in GL_n$ and $Z \in \mathcal{M}_n$, where the action $R$ of the Shalika group $S_{2n}$ on $\mathcal{S}(F^n)$ is defined in Section 3.1. Let us choose a basis $\{f_i\}$ of $\mathcal{F}_{\pi}$ and a basis $\Phi_k$ of $\mathcal{S}(F^n)$. We let $\mathcal{D}=\mathcal{D}_{\pi} \times \mathcal{D}_s$, where $\mathcal{D}_s=(\mathbb{C}/\frac{2\pi i}{\log(q)}) \simeq \mathbb{C}^{\times}$. For a fixed $d=(u,s) \in \mathcal{D}$,  the quasi-invariance property gives the system $\Xi^{\prime}_d$
\begin{equation}
\tag{$\Xi^{\prime}_d$} 
\begin{split}
&\left\{ \left( \pi_u \begin{pmatrix} g &\\ & g \end{pmatrix} \pi_u(g_i)f_i\otimes R\begin{pmatrix} g &\\ & g \end{pmatrix}\Phi_k-|\mathrm{det}(g)|^{-s}\pi_u(g_i)f_i \otimes \Phi_k, 0   \right) \right. \\ 
&\phantom{***********************************}\left| \; g \in GL_n,\; g_i \in GL_{2n} \right\}\\
&\bigcup \left\{ \left( \pi_u \begin{pmatrix} I_n & Z\\ & I_n \end{pmatrix} \pi_u(g_i)f_i\otimes R\begin{pmatrix} I_n &Z \\ & I_n \end{pmatrix}\Phi_k-\psi(\mathrm{Tr}Z)\pi_u(g_i)f_i \otimes \Phi_k, 0   \right) \right. \\
&\phantom{************************************}\left| \; Z \in \mathcal{M}_n,\; g_i \in GL_{2n} \right\}.
\end{split}
\end{equation}
Suppose that an irreducible quasi-square integrable representation $\Delta_i$ of $GL_{n_i}$ is the segment $[\rho_i,\dotsm,\rho_i\nu^{\ell_i-1}]$ with $\rho_i$ an irreducible supercuspidal representation of $GL_{r_i}$ and $n_i=r_i\ell_i$. From Theorem \ref{integral-rational} and Proposition \ref{even-multi one}, we have the linear forms $L_{(a_1,\dotsm,a_t)}(u,s)$ such that
\begin{equation}
\tag{7.2}
\label{even-absconv}
 q^{L_{(a_1,\dotsm,a_t)}(u,s)}=q^{ks+\sum_{i=1}^t(\ell_i-a_i)r_iu_i} \prod_{i=1}^t \omega^{-1}_{\Delta_i^{(a_ir_i)}}(\varpi)
 \end{equation}
for any $0 \leq a_i \leq \ell_i$ satisfying $\sum a_ir_i=2n-2k$ for some $1 \leq k \leq n$. The integral $J(s,W_u,\Phi)$ converges absolutely and the space of solutions of the system $\Xi^{\prime}_d$ is of dimension $1$, when $d=(u,s) \in \mathcal{D}$ satisfies $\mathrm{Re}(L_{(a_1,\dotsm,a_t)}(u,s)) > 0$ for all $0 \leq a_i \leq \ell_i$ with $\sum a_ir_i=2n-2k$ and $1 \leq k \leq n$. In fact, the statement in Proposition \ref{even-multi one} is for irreducible generic $\pi$ with each $\pi^{(a_1r_1,\dotsm,a_tr_t)}$ irreducible rather than families $\pi_u$. But by Proposition \ref{general position}, outside of a finite number of hyperplanes in $u$ each $\pi^{(a_1r_1,\dotsm,a_tr_t)}_u$ will be irreducible and generic. We define $\Omega \subset \mathcal{D}$ by the conditions that $\Omega$ is the intersection of the complements of the hyperplanes on which each deformed representation $\pi^{(a_1r_1,\dotsm,a_tr_t)}_u$ is reducible, and of the half spaces $\mathrm{Re}(L_{(a_1,\dotsm,a_t)}(u,s)) > 0$ on which our integrals converge absolutely and the uniqueness up to scalar holds.

\par
To be able to apply Bernstein's theorem, we must add one equation to ensure uniqueness on $\Omega$. This is a normalization equation. Using the partial Iwasawa decomposition in Section 2.2, we see that for any choice of $W_u \in \mathcal{W}^{(0)}_{\pi}$ and $\Phi \in \mathcal{S}(F^n)$, we can decompose our local integral as
\[
 \begin{split}
&J(s,W_u,\Phi)
=\int_{K_n} \int_{N_n\backslash P_n} \int_{\mathcal{N}_n \backslash \mathcal{M}_n} W_u \begin{pmatrix} \sigma_{2n} \begin{pmatrix} I_n & X \\ & I_n \end{pmatrix} \begin{pmatrix} pk &  \\ & pk \end{pmatrix} \end{pmatrix}\\
&\phantom{***************}\times  |\mathrm{det}(p)|^{s-1} \psi^{-1}(\mathrm{Tr} X) dX \int_{F^{\times}} \omega_{\pi}(a) |a|^{ns} \Phi (e_nak) d^{\times} a\; dp\; dk.
\end{split}
\]
The Proof of Lemma \ref{nonvanishing-even} or \cite[Lemma 2.3]{Be14} also shows that there exists $\varphi \in \mathcal{S}_{\psi}(P_n,\mathbb{C})$ such that
\[
  \int_{N_n \backslash P_n} \int_{\mathcal{N}_n \backslash \mathcal{M}_n} \varphi \left( \sigma_{2n} \begin{pmatrix} I_n & X \\ & I_n \end{pmatrix}  \begin{pmatrix} p &  \\ & p \end{pmatrix} \sigma_{2n}^{-1} \right) 
  |\mathrm{det}(p)|^{s-1} \psi^{-1}(\mathrm{Tr} X) dX dp
\]
is a nonzero constant. We notice that this constant and $\varphi$ do not depend on $u$. We may first choose any $f \in \mathcal{F}_{\pi}$ such that $W_{f,u}(I_{2n}) \neq 0$. Let $P(q^{\pm u})=W_{f,u}(I_{2n})$ be corresponding polynomial in $\mathbb{C}[q^{\pm u}]$ by \cite[Section 3.2]{CoPe}. According to Proposition \ref{poly-GeKa}, we can find $W_u^{\circ} \in \mathcal{W}_{\pi}^{(0)}$ such that the restriction of $W_u^{\circ}$ to $P_{2n} \times \mathcal{D}_{\pi}$ is of the form $W_u^{\circ}(p)=\varphi(p)P(q^{\pm u})$. Setting $W_u^{\prime}=\pi_u(\sigma_{2n}^{-1})W_u^{\circ} \in \mathcal{W}_{\pi}^{(0)}$, we obtain
\[
 \begin{split}
&J(s,W_u^{\prime},\Phi)
=\int_{K_n} \int_{N_n\backslash P_n} \int_{\mathcal{N}_n \backslash \mathcal{M}_n} W_u^{\prime} \begin{pmatrix} \sigma_{2n} \begin{pmatrix} I_n & X \\ & I_n \end{pmatrix} \begin{pmatrix} pk &  \\ & pk \end{pmatrix} \end{pmatrix}\\ 
&\phantom{***************}\times |\mathrm{det}(p)|^{s-1} \psi^{-1}(\mathrm{Tr} X) dX\int_{F^{\times}} \omega_{\pi}(a) |a|^{ns} \Phi (e_nak) d^{\times} a\; dp\; dk.
\end{split}
\]
Let $K_{n,r}$ be a sufficiently small open compact congruence subgroup of $K_n$. Now choose $\Phi$ to be the characteristic function $\Phi^{\prime}$ of $e_nK_{n,r}$. With theses choices the integral is reduced to
\[
\begin{split}
 &J(s,W_u^{\prime},\Phi^{\prime})\\
 &\phantom{*****}=c_1\int_{N_n\backslash P_n} \int_{\mathcal{N}_n \backslash \mathcal{M}_n} W_u^{\prime} \begin{pmatrix} \sigma_{2n} \begin{pmatrix} I_n & X \\ & I_n \end{pmatrix} \begin{pmatrix} p &  \\ & p \end{pmatrix} \end{pmatrix} |\mathrm{det}(p)|^{s-1} \psi^{-1}(\mathrm{Tr} X) dX dp
\end{split}
\] 
with $c_1$ a volume of $e_nK_{n,r}$ and we get
\[
\begin{split}
  &J(s,W_u^{\prime},\Phi^{\prime})\\
  &=c_1P(q^{\pm u}) \int_{N_n\backslash P_n} \int_{\mathcal{N}_n \backslash \mathcal{M}_n}  \varphi \left( \sigma_{2n} \begin{pmatrix} I_n & X \\ & I_n \end{pmatrix}  \begin{pmatrix} p &  \\ & p \end{pmatrix} \sigma_{2n}^{-1} \right) 
  |\mathrm{det}(p)|^{s-1} \psi^{-1}(\mathrm{Tr} X) dX dp\\
  &=c_2P(q^{\pm u})
\end{split}
\]
for some constant $c_2$. Since $W_u^{\prime} \in \mathcal{W}_{\pi}^{(0)}$ it is a linear combination of $GL_{2n}$ translation of functions in $\mathcal{F}_{\pi}$. So we have
\[
  W_u'=\sum_i \pi_u(g_i)W_{h_i,u}
\]
for appropriate $g_i \in GL_{2n}$ and $h_i \in \mathcal{F}_{\pi}$. Therefore to remove the scalar ambiguity in our system of equations we add the single normalization equation
\begin{equation}
\tag{$\mathrm{N}_d$}
 \left(  \sum_i \pi_u(g_i)h_i \otimes \Phi^{\prime}, c_2P(q^{\pm u}) \right).
\end{equation}

\par

If $\Xi_d$ is the system $\Xi^{\prime}_d$ with the equation (N$_d$) adjoined, the system $\Xi_d$ has a unique solution which we denote by $\lambda_d : f_u \otimes \Phi \mapsto J(s,W_{f,u},\Phi)$. Now we have a system which satisfies the hypotheses of Bernstein's Theorem \ref{Bernstein}. Hence we may conclude that there is a unique solution $\lambda \in V^*_{\mathbb{C}(\mathcal{D})}$ of $\Xi$
such that $\lambda(d)=\lambda_d$ on $\Omega$. Putting another way, for $d=(u,s)$ in $\Omega$, we have
\[
  \lambda(u,s)(f_u,\Phi)=J(s,W_{f,u},\Phi)
\]
and the left hand side $\lambda(u,s)(f_u,\Phi)$ defines a rational function in $\mathbb{C}(q^{-u},q^{-s})$.

\par
We fix any $u \in \mathcal{D}_{\pi}$. We then define an open set $\mathcal{U}$ of $\mathcal{D}_s$ by the condition that $\max\{\mathrm{Re}(L_{(a_1,\dotsm,a_t)}(u,s))\} > 0$ on which the maximum runs over all $0 \leq a_i \leq \ell_i$ satisfying $\sum a_ir_i=2n-2k$ for some $1 \leq k \leq n$. From \eqref{even-absconv}, $\mathcal{U}$ is not empty. If $s$ lies in $\mathcal{U}$, we obtain
\[
\tag{7.3}
\label{rational-even-equal}
  \lambda(u,s)(f_u,\Phi)=J(s,W_{f,u},\Phi),
\]
and both are rational functions in $\mathbb{C}(q^{-s})$. Hence they are equal on this open set $\mathcal{U}$. We then extend the equality in \eqref{rational-even-equal} meromorphically to the whole space $\mathcal{D}_s$. This concludes that 
\[
  \lambda(u,s)(f_u,\Phi)=J(s,W_{f,u},\Phi)
\]
on $\mathcal{D}_{\pi} \times \mathcal{D}_s$ and hence $J(s,W_u,\Phi)$ defines a rational function in $\mathbb{C}(q^{-u},q^{-s})$.

\end{proof}

If we now look at the local functional equation for the even $GL_{2n}$ case from Theorem \ref{local functional eq}, it reads
\[
  J(1-s,\varrho(\tau_{2n})\widetilde{W_u},\hat{\Phi})=\gamma(s,\pi_u,\wedge^2,\psi)J(s,W_u,\Phi)
\]
where
\[
 \gamma(s,\pi_u,\wedge^2,\psi)=\frac{\varepsilon(s,\pi_u,\wedge^2,\psi)L(1-s,{(\pi_u)}^{\iota},\wedge^2)}{L(s,\pi_u,\wedge^2)}
\]
for $W_u \in \mathcal{W}(\pi_u,\psi)$ and $\Phi \in \mathcal{S}(F^n)$. If $W_u \in \mathcal{W}^{(0)}_{\pi}$, then $\varrho(\tau_{2n})\widetilde{W_u} \in \mathcal{W}^{(0)}_{{\pi}^{\iota}}$. Under deformation, $(\pi_u)^{\iota}={\pi^{\iota}}_{u^{\iota}}$, if $u=(u_1,\dotsm,u_t)$ then $u^{\iota}=(-u_t,\dotsm,-u_1)$. Note that in the Whittaker model, $W^{(0)}_{\pi^{\iota}}$ should be taken with respect to the character $\psi^{-1}$. From Proposition \ref{even-rational} the integrals $J(1-s,\varrho(\tau_{2n})\widetilde{W_u},\hat{\Phi})$ appearing in the functional equation are rational functions of $q^{-u}$ and $q^{-s}$ for $W_u \in \mathcal{W}_{\pi}^{(0)}$. Therefore $\gamma(s,\pi_u,\wedge^2,\psi)$ must also be rational.

\begin{corollary}
\label{gamma-rational}
Let $\pi=\mathrm{Ind}(\Delta_1 \otimes \dotsm \otimes \Delta_t)$ be a representation of Whittaker type on $GL_{2n}$. Then the gamma factor $\gamma(s,\pi_u,\wedge^2,\psi)$ belongs to $\mathbb{C}(q^{-u},q^{-s})$.
\end{corollary}

\subsection{The odd case $m=2n+1$}

Let $\pi=\text{Ind}(\Delta_1 \otimes \dotsm \otimes \Delta_t)$ be a representation of Whittaker type on $GL_{2n+1}$ and for each $u \in \mathcal{D}_{\pi}$ we define the deformed representation by $\pi_u=\text{Ind}(\Delta_1\nu^{u_1} \otimes \dotsm \otimes \Delta_t\nu^{u_t})$. The odd case $2n+1$ runs along the same lines. For $W_u \in \mathcal{W}^{(0)}_{\pi}$ and $\Phi \in \mathcal{S}(F^n)$, the local integral, which occur in the functional equation for $GL_{2n+1}$, is defined by
 \[
\begin{split}
&J(s,W_u,\Phi)\\
&\phantom{*****}=\int_{N_n \backslash GL_n} \int_{\mathcal{N}_n \backslash \mathcal{M}_n} \int_{F^n} W_u \left( \sigma_{2n+1} \begin{pmatrix} I_n&X&\\&I_n&\\&&1 \end{pmatrix} 
\begin{pmatrix}g&&\\&g&\\&&1 \end{pmatrix} \begin{pmatrix} I_n&&\\&I_n&\\&x&1 \end{pmatrix}  \right) \\
&\phantom{*******************************}\times\Phi(x) dx \psi^{-1}(\mathrm{Tr}X) dX |\mathrm{det}(g)|^{s-1} dg.
\end{split}
\]
We apply Bernstein's theorem to prove the rationality of Jacquet-Shalika integrals.

\begin{proposition}
\label{odd-rational}
Let $\pi=\mathrm{Ind}(\Delta_1 \otimes \dotsm \otimes \Delta_t)$ be a representation of Whittaker type on $GL_{2n+1}$. For $W_u \in \mathcal{W}_{\pi}^{(0)}$ and $\Phi \in \mathcal{S}(F^n)$, the integral $J(s,W_u,\Phi)$ is a rational function of $q^{-u}$ and $q^{-s}$.
\end{proposition} 

\begin{proof}
We once again should write down a system of equations which are polynomials in $q^{\pm u}$ and $q^{\pm s}$ which characterize these functionals. In this case, our underlying vector space is $\mathcal{V}=\mathcal{F}_{\pi}$ which is countable dimensional over $\mathbb{C}$. For fixed $d=(u,s) \in \mathcal{D}=\mathcal{D}_{\pi} \times \mathcal{D}_s$, we can taking the bases of $\mathcal{F}_{\pi}$ and the bases of $\mathcal{S}(F^n)$ as above in Proposition \ref{even-rational}. From the quasi-invariance properties in \eqref{invariance-odd}, our system of equation $\Xi_d^{\prime}$ expressing the invariance of the local integral is
\begin{equation}
\tag{$\Xi_d^{\prime}$}
\begin{split}
&\left\{ \left( \pi_u \begin{pmatrix} g &&\\ & g&\\&&1 \end{pmatrix} \pi_u(g_i)f_i\otimes R\begin{pmatrix} g &&\\ & g&\\&&1 \end{pmatrix}\Phi_k -|\mathrm{det}(g)|^{-s}\pi_u(g_i)f_i\otimes \Phi_k, 0   \right) \right. \\
&\phantom{********************************}\left| \; g \in GL_n,\; g_i \in GL_{2n} \right\}\\
\end{split}
\end{equation}
\[
\begin{split}
&\bigcup \left\{ \left( \pi_u \begin{pmatrix} I_n & Z&\\ & I_n&\\&&1 \end{pmatrix} \pi_u(g_i)f_i\otimes R\begin{pmatrix} I_n & Z&\\ & I_n&\\&&1 \end{pmatrix}\Phi_k-\pi_u(g_i)f_i\otimes\Phi_k, 0   \right) \right. \\
&\phantom{********************************} \left| \; Z \in \mathcal{M}_n,\; g_i \in GL_{2n} \right\} \\
&\bigcup \left\{ \left( \pi_u \begin{pmatrix} I_n & &\\ & I_n&\\&x&1 \end{pmatrix} \pi_u(g_i)f_i\otimes R\begin{pmatrix} I_n & &\\ & I_n&\\&x&1 \end{pmatrix}\Phi_k -\pi_u(g_i)f_i\otimes \Phi_k, 0   \right) \right. \\
&\phantom{********************************} \left| \; x \in \mathcal{M}_{n,1},\; g_i \in GL_{2n} \right\} \\
&\bigcup \left\{ \left( \pi_u \begin{pmatrix} I_n & &y\\ & I_n&\\&&1 \end{pmatrix} \pi_u(g_i)f_i\otimes R\begin{pmatrix} I_n & &y\\ & I_n&\\&&1 \end{pmatrix} \Phi_k-\pi_u(g_i)f_i\otimes \Phi_k, 0   \right) \right. \\
&\phantom{********************************} \left| \; y \in \mathcal{M}_{1,n},\; g_i \in GL_{2n} \right\}. \\
\end{split}
\]
Suppose that an irreducible quasi-square integrable representation $\Delta_i$ of $GL_{n_i}$ is the segment $[\rho_i,\dotsm,\rho_i\nu^{\ell_i-1}]$ with $\rho_i$ an irreducible supercuspidal representation of $GL_{r_i}$ and $n_i=r_i\ell_i$. As in Proposition \ref{even-rational}, we have the following linear forms $L_{(a_1,\dotsm,a_t)}(u,s)$ from Lemma \ref{convergence-odd} and Proposition \ref{odd-multi one} such that
\begin{equation}
\tag{7.4}
\label{odd-absconv}
  q^{L_{(a_1,\dotsm,a_t)}(u,s)}=q^{ks+\sum_{i=1}^t(\ell_i-a_i)r_iu_i} \prod_{i=1}^t \omega^{-1}_{\Delta_i^{(a_ir_i)}}(\varpi)
 \end{equation}
for any $0 \leq a_i \leq r_i$ satisfying $\sum a_ir_i=2n+1-2k$ for some $1 \leq k \leq n$.  By Proposition \ref{general position} as before, we know that each deformed representation $\pi_u^{(a_1r_1,\dotsm,a_tr_t)}$ is irreducible off of a finite number of hyperplanes in $(u,s)$. If we define $\Omega \subset \mathcal{D}$ by the condition that $\Omega$ is the intersection of the complements of the hyperplanes on which the deformed representation $\pi_u^{(a_1r_1,\dotsm,a_tr_t)}$ is reducible, and of the half spaces $\mathrm{Re}(L_{(a_1,\dotsm,a_t)}(u,s)) > 0$ for all $\sum_i a_i r_i=2n+1-2k$ with $1 \leq k \leq n$ and  on which our integrals converge absolutely and the uniqueness up to scalar holds.

 \par
 To be able to apply Bernstein's theorem, we must add one equation to ensure the uniqueness on $\Omega$. This is again a normalization equation, which is slightly complicated  simple in this situation. To give it, we may first take $f \in \mathcal{F}_{\pi}$ such that $W_{f,u}(I_{2n+1})\neq 0$. Let $P(q^{\pm u})=W_{f,u}(I_{2n+1})$ be the corresponding polynomial in $\mathbb{C}[q^{\pm u}]$ by \cite[Section 3.2]{CoPe}. The proof of Lemma \ref{nonvanishing-odd} also shows that there exists $\varphi \in \mathcal{S}_{\psi}(P_n,\mathbb{C})$ such that
\[
  \int_{N_n \backslash GL_n} \int_{\mathcal{N}_n \backslash \mathcal{M}_n} \varphi \left( \sigma_{2n+1} \begin{pmatrix} I_n&X&\\&I_n&\\&&1  \end{pmatrix} \begin{pmatrix} g&&\\&g& \\&&1 \end{pmatrix} \sigma_{2n+1}^{-1} \right)  \psi^{-1}(\mathrm{Tr}X) dX |\mathrm{det}g|^{s-1} dg
\]
is a nonzero constant. We notice that $\varphi$ and this constant do not depend on $u$. According to Proposition \ref{poly-GeKa}, we can find $W_u^{\circ} \in \mathcal{W}_{\pi}^{(0)}$ such that the restriction of $W_u^{\circ}$ to $P_{2n+1} \times \mathcal{D}_{\pi}$ is of the form $W_u^{\circ}(p)=\varphi(p)P(q^{\pm u})$. Setting $W_u'=\pi_u(\sigma_{2n+1}^{-1})W_u^{\circ} \in \mathcal{W}_{\pi}^{(0)}$, there exist a compact open set $U$ of $F^n$ such that
\[
  W_u' \left( g \begin{pmatrix} I_n&&\\&I_n& \\ &x&1 \end{pmatrix} \right)=W_u' ( g )
\]
for all $x \in U$. Now we choose $\Phi$ to be the characteristic function $U$ of $F^n$. With these choices of $W_u'$ and $\Phi$, our integral is reduced to
\[
\begin{split}
 &J(s,W_u',\Phi)
 =c_3\int_{N_n \backslash GL_n} \int_{\mathcal{N}_n \backslash \mathcal{M}_n}W_u' \left( \sigma_{2n+1} \begin{pmatrix} I_n&X&\\&I_n&\\&&1  \end{pmatrix} \begin{pmatrix} g&&\\&g& \\&&1 \end{pmatrix}  \right) \\
 &\phantom{************************************} \psi^{-1}(\mathrm{Tr}X) dX |\mathrm{det}g|^{s-1} dg
\end{split}
\]
with $c_3$ a volume of $U$ and we get
\[
\begin{split}
 J(s,W_u^{\prime},\Phi)
 &=c_3P(q^{\pm u})\int_{N_n \backslash GL_n} \int_{\mathcal{N}_n \backslash \mathcal{M}_n} \varphi \left( \sigma_{2n+1} \begin{pmatrix} I_n&X&\\&I_n&\\&&1  \end{pmatrix} \begin{pmatrix} g&&\\&g& \\&&1 \end{pmatrix} \sigma_{2n+1}^{-1} \right) \\
 &\phantom{*****************************}  \psi^{-1}(\mathrm{Tr}X) dX |\mathrm{det}g|^{s-1} dg \\
 &=c_4P(q^{\pm u})
\end{split}
\]
for some constant $c_4$.  Now, as $W^{\prime} \in \mathcal{W}^{(0)}_{\pi}$, it can be expressed as a finite linear combination
\[
  W^{\prime}=\sum_i \pi_u(g_i) W_{h_i,u}
\]
for appropriate $g_i \in GL_n$ and $h_i \in \mathcal{F}_{\pi}$. Thus our normalization equation can be written 
\begin{equation}
\tag{$\mathrm{N}_d$}
 \left(  \sum_i \pi_u(g_i)h_i\otimes \Phi, c_4P(q^{\pm u}) \right).
\end{equation}

\par

If $\Xi_d$ is the system $\Xi^{\prime}_d$ with the equation (N$_d$) adjoined, the system $\Xi_d$ has a unique solution which we denote by $\lambda_d : f_u\otimes \Phi \mapsto J(s,W_{f,u},\Phi)$. Now we have a system which satisfies the hypotheses of Bernstein's Theorem. Hence we may conclude that there is a unique solution $\lambda \in V^*_{\mathbb{C}(\mathcal{D})}$ of $\Xi$
such that $\lambda(d)=\lambda_d$ on $\Omega$, because $\lambda_d$ is holomorphic in $d$ on $\Omega$. Putting another way, for $d=(u,s)$ in $\Omega$, we have
\[
  \lambda(u,s)(f_u,\Phi)=J(s,W_{f,u},\Phi)
\]
and the left hand side $\lambda(u,s)(f_u)$ defines a rational function in $\mathbb{C}(q^{-u},q^{-s})$.

\par
We fix any $u$ belonging to $\mathcal{D}_{\pi}$. We then define an open set $\mathcal{U}$ by the condition that $\max\{\mathrm{Re}(L_{(a_1,\dotsm,a_t)}(u,s))\} > 0$ on which the maximum is taken over all $0 \leq a_i \leq r_i$ satisfying $\sum a_ir_i=2n+1-2k$ for some $1 \leq k \leq n$. From \eqref{odd-absconv}, $\mathcal{U}$ is not empty. If $s$ lies in $\mathcal{U}$, we obtain
\[
\tag{7.5}
\label{rational-odd-equal}
  \lambda(u,s)(f_u,\Phi)=J(s,W_{f,u},\Phi)
\]
and both are rational functions in $\mathbb{C}(q^{-s})$. Hence they are equal on this open set $\mathcal{U}$. We then extend the equality in \eqref{rational-odd-equal} meromorphically to the whole space $\mathcal{D}_s$. This concludes that 
\[
  \lambda(u,s)(f_u,\Phi)=J(s,W_{f,u},\Phi)
\]
on $\mathcal{D}_{\pi} \times \mathcal{D}_s$ and hence $J(s,W_u,\Phi)$ defines a rational function in $\mathbb{C}(q^{-u},q^{-s})$.

 \end{proof}

From the local functional equation for the odd $GL_{2n+1}$ case in Theorem \ref{local-func-odd}, we obtain
\[
  J(1-s,\varrho(\tau_{2n+1})\widetilde{W_u},\hat{\Phi})=\gamma(s,\pi_u,\wedge^2,\psi)J(s,W_u,\Phi),
\]
where
\[
 \gamma(s,\pi_u,\wedge^2,\psi)=\frac{\varepsilon(s,\pi_u,\wedge^2,\psi)L(1-s,{(\pi_u)}^{\iota},\wedge^2)}{L(s,\pi_u,\wedge^2)}
\]
for $W_u \in \mathcal{W}(\pi_u,\psi)$ and $\Phi \in \mathcal{S}(F^n)$. If $W_u \in \mathcal{W}_{\pi}^{(0)}$, then $\varrho(\tau_{2n+1})\widetilde{W_u} \in \mathcal{W}^{(0)}_{\pi^{\iota}}$. For $W_u \in \mathcal{W}_{\pi}^{(0)}$ and $\Phi \in \mathcal{S}(F^n)$, the integrals $J(1-s,\varrho(\tau_{2n+1})\widetilde{W_u},\hat{\Phi})$ involved in the functional equation are again rational functions in $\mathbb{C}(q^{-u},q^{-s})$ by proceeding Proposition \ref{odd-rational} and hence so must $\gamma(s,\pi_u,\wedge^2,\psi)$ be.

\begin{corollary}
\label{gamma-rational-odd}
Let $\pi=\mathrm{Ind}(\Delta_1 \otimes \dotsm \otimes \Delta_t)$ be a representation of Whittaker type on $GL_{2n+1}$. Then the gamma factor $\gamma(s,\pi_u,\wedge^2,\psi)$ belongs to $\mathbb{C}(q^{-u},q^{-s})$.
\end{corollary}

\subsection{$L$-functions for deformed representations in general position.}

Let $\mathrm{Ind}(\Delta_1\nu^{u_1}\otimes \Delta_2\nu^{u_2})$ be the normalized induced representation from two segments $\Delta_1\nu^{u_1}$ and $\Delta_2\nu^{u_2}$ in general position under deformation.
The purpose of this section is to provide the equality of exceptional $L$-functions $L_{ex}(s,\mathrm{Ind}(\Delta_1^{(k_1)}\nu^{u_1}\otimes \Delta_2^{(k_2)}\nu^{u_2}),\wedge^2)$ and Rankin-Selberg $L$-functions $L_{ex}(s,\Delta_1^{(k_1)}\nu^{u_1} \times \Delta_2^{(k_2)}\nu^{u_2})$ defined in Section 2 of \cite{CoPe}. If $\pi_u=\mathrm{Ind}(\Delta_1\nu^{u_1} \otimes \dotsm \otimes \Delta_t\nu^{u_t})$ is a deformed representation of $\pi$ in general position, we really need only to consider those exceptional $L$-functions $L_{ex}(s,\mathrm{Ind}(\Delta_i^{(k_i)}\nu^{u_i}\otimes \Delta_j^{(k_j)}\nu^{u_j}),\wedge^2)$ or $L_{ex}(s,\Delta_i^{(k_i)}\nu^{u_i},\wedge^2)$, which contribute poles of $L_{ex}(s,\pi_u,\wedge^2)$.  We will explain this claim in Section 8. Hence we are only holding at $2$ segments $\mathrm{Ind}(\Delta_1^{(k_1)}\nu^{u_1}\otimes \Delta_2^{(k_2)}\nu^{u_2})$ to compute the exceptional $L$-functions. The following Hartogs' Theorem plays an important roles to determine poles of $L_{ex}(s,\mathrm{Ind}(\Delta_1^{(k_1)}\nu^{u_1}\otimes \Delta_2^{(k_2)}\nu^{u_2}),\wedge^2)$. We quote a statement in \cite{Lo91}, Chapter III, Section 4.3.

\begin{theorem}[Hartogs]
\label{Hartog}
Let $\mathcal{M}$ be an $n$-dimensional complex manifold with $n \geq 2$. If $\mathcal{N} \subset \mathcal{M}$ is an analytic subset of codimension $2$ or more, then every holomorphic function on $\mathcal{M}-\mathcal{N}$ extends to a holomorphic function on $\mathcal{M}$.

\end{theorem}

 The following Lemma is originally stated in \cite{Cog94} and is proved by Matringe in \cite[Proposition 5.3]{Ma} to understand the Shalika functional on the representation of the form $\mathrm{Ind}(\Delta_i \otimes \Delta_j)$ with $(\Delta_i)^{\sim} \simeq \Delta_j$ up to a twist.

  \begin{lemma}[Matringe]
 \label{Shalika periods}
Let $\Delta$ be a square-integrable representation of $GL_m$. The normalized induced representation of the form $\mathrm{Ind}(\Delta\nu^s \otimes \widetilde{\Delta}\nu^{-s})$ has a nontrivial Shalika functional for any complex number $s \in \mathbb{C}$. 
 \end{lemma}

 As a consequence of this Lemma, we obtain the following Proposition that will be important for what follows.

 \begin{proposition}
 \label{twosegment-L0}
Let $\pi=\mathrm{Ind}(\Delta_1 \otimes \Delta_2)$ be a representation of Whittaker type of $GL_m$. We let $u=(u_1,u_2) \in \mathcal{D}_{\pi}$ be in general position and $\pi_u=\mathrm{Ind}(\Delta_1\nu^{u_1}\otimes \Delta_2\nu^{u_2})$ an irreducible generic representation of $GL_m$. We assume that $\Delta_i=[\rho_i,\rho_i\nu,\dotsm,\rho_i\nu^{\ell_i-1}]$ is an irreducible quasi-square-integrable representation of $GL_{m_i}$ where $\rho_i$ is an irreducible supercuspidal representation of $GL_{r_i}$, $m_i=r_i \ell_i$ for $i=1,2$, and $m=m_1+m_2$. 
 Then we have
 \[
   L^{(0)}(s,\mathrm{Ind}(\Delta_1\nu^{u_1} \otimes \Delta_2\nu^{u_2}),\wedge^2)=L_{ex}(s,\Delta_1\nu^{u_1}  \times \Delta_2\nu^{u_2} ).
 \]
 \end{proposition}

\begin{proof}

The exceptional Rankin-Selberg $L$-function $L_{ex}(s,\Delta_1\nu^{u_1}\times\Delta_2\nu^{u_2})$ can have a pole only at those $s_0$ for 
which $(\Delta_1\nu^{u_1})^{\sim} \simeq \Delta_2\nu^{u_2+s_0}$, that is, $(\Delta_1)^{\sim} \simeq \Delta_2\nu^e$ with $e=u_1+u_2+s_0$. 
We write $\Delta_1\nu^{u_1+\frac{s_0}{2}}=\Delta_0 \otimes \nu^t$ where $\Delta_0$ is a square integrable representation with $t$ a real number. 
With replaced $\Delta_1\nu^{u_1+\frac{s_0}{2}}$ by $\Delta_0 \otimes \nu^t$, according to Lemma \ref{Shalika periods}, $\mathrm{Ind}(\Delta_1\nu^{u_1+\frac{s_0}{2}}\otimes\Delta_2\nu^{u_2+\frac{s_0}{2}}) \simeq \mathrm{Ind}(\Delta_0 \nu^t \otimes \widetilde{\Delta_0 }\nu^{-t})$ has a nontrivial Shalika functional. However by the condition $(6)$ of the Definition of general position, this implies that the vector space $\mathrm{Hom}_{S_{2m_1} \cap P_{2m_1}}(\mathrm{Ind}(\Delta_1\nu^{u_1+\frac{s_0}{2}}\otimes\Delta_2\nu^{u_2+\frac{s_0}{2}}),\Theta)$ is of dimension at most one. 
The representation $\mathrm{Ind}(\Delta_1\nu^{u_1}\otimes\Delta_2\nu^{u_2})$ being irreducible and generic, Lemma \ref{char-Shalika} implies that $L^{(0)}(s,\nu^{\frac{s_0}{2}}\mathrm{Ind}(\Delta_1\nu^{u_1}\otimes\Delta_2\nu^{u_2}),\wedge^2)$ has a pole at $s=0$, and so $L^{(0)}(s,\mathrm{Ind}(\Delta_1\nu^{u_1}\otimes\Delta_2\nu^{u_2}),\wedge^2)$ has a pole at $s=s_0$. 
As both $L$-functions $L^{(0)}(s,\mathrm{Ind}(\Delta_1\nu^{u_1}\otimes\Delta_2\nu^{u_2}),\wedge^2)$ and $L_{ex}(s,\Delta_1\nu^{u_1}\times\Delta_2\nu^{u_2})$ have the same simple poles at $s=s_0$, we obtain that
\[
  L_{ex}(s,\Delta_1\nu^{u_1}\times\Delta_2\nu^{u_2})^{-1} \quad \text{divides} \quad L^{(0)}(s,\mathrm{Ind}(\Delta_1\nu^{u_1}\otimes\Delta_2\nu^{u_2}),\wedge^2)^{-1},
\]
where we take the divisibility in $\mathbb{C}[q^{\pm s}]$.

\par
We claim that in fact
\[
 L^{(0)}(s,\mathrm{Ind}(\Delta_1\nu^{u_1}\otimes\Delta_2\nu^{u_2}),\wedge^2)=L_{ex}(s,\Delta_1\nu^{u_1}\times\Delta_2\nu^{u_2}).
\]
To see this we explain that $L^{(0)}(s,\mathrm{Ind}(\Delta_1\nu^{u_1}\otimes\Delta_2\nu^{u_2}\hspace*{-.5mm}),\wedge^2)^{-1}$ divides $L_{ex}(s,\Delta_1\nu^{u_1}\times\Delta_2\nu^{u_2}\hspace*{-.7mm})^{-1}$ with respect to divisibility in $\mathbb{C}[q^{\pm s}]$. Suppose that $L^{(0)}(s,\mathrm{Ind}(\Delta_1\nu^{u_1}\otimes\Delta_2\nu^{u_2}),\wedge^2)$ has a pole at $s=s_0$. The existence of exceptional pole of $L(s,\mathrm{Ind}(\Delta_1\nu^{u_1}\otimes\Delta_2\nu^{u_2}),\wedge^2)$  at $s=s_0$ ensures a nontrivial Shalika functional.  According to Theorem \ref{Jacquet-Rallis} of Jacquet and Rallis, the existence of the nontrivial Shalika functional imposes the twisted self contragredient condition 
\[
\mathrm{Ind}(\Delta_1\nu^{u_1}\otimes\Delta_2\nu^{u_2})^{\sim} \simeq \mathrm{Ind}(\Delta_1\nu^{u_1}\otimes\Delta_2\nu^{u_2})\nu^{s_0}.
\]
As the induced representation $\mathrm{Ind}(\Delta_1\nu^{u_1}\otimes\Delta_2\nu^{u_2})$ is irreducible, the way that this is possible is that $(\Delta_1\nu^{u_1})^{\sim} \simeq \Delta_1\nu^{u_1+s_0}$ and $(\Delta_2\nu^{u_2})^{\sim} \simeq \Delta_2\nu^{u_2+s_0}$, or $(\Delta_1\nu^{u_1})^{\sim} \simeq \Delta_2\nu^{u_2+s_0}$. 

\par
Now let us see how these conditions vary in $u_1$ and $u_2$. If we consider now the Jacquet and Shalika integral $J(s,W_u,\Phi)$ with $W_u \in \mathcal{W}^{(0)}_{\pi}$, and $\Phi \in \mathcal{S}(F^n)$ if $m=2n$ or $m=2n+1$, then these integrals define rational functions in $\mathbb{C}(q^{-u_1},q^{-u_2},q^{-s})$. For $u=(u_1,u_2)$ in the Zariski open subset of general position, these rational functions $J(s,W_u,\Phi)$ can have poles coming from $L^{(0)}(s,\mathrm{Ind}(\Delta_1\nu^{u_1}\otimes\Delta_2\nu^{u_2}),\wedge^2)$. The poles of $L^{(0)}(s,\mathrm{Ind}(\Delta_1\nu^{u_1}\otimes\Delta_2\nu^{u_2}),\wedge^2)$ can lie along the locus defined by one equation $q^{-(u_1+u_2+s)m_1}\omega_{\Delta_1}(\varpi)\omega_{\Delta_2}(\varpi)=1$ from $(\Delta_1\nu^{u_1})^{\sim} \simeq \Delta_2\nu^{u_2+s}$, or two independent equations $q^{-(2u_1+s)m_1}\omega^2_{\Delta_1}(\varpi)=1$ and $q^{-(2u_2+s)m_2}\omega^2_{\Delta_2}(\varpi)=1$ from $(\Delta_1\nu^{u_1})^{\sim} \simeq \Delta_1\nu^{u_1+s}$ and $(\Delta_2\nu^{u_2})^{\sim} \simeq \Delta_2\nu^{u_2+s}$, respectively. By Theorem \ref{Hartog} of Hartogs, viewing the local integrals $J(s,W_u,\Phi)$ as meromorphic functions of $u_1$, $u_2$ and $s$, these integrals cannot have singularities of codimension equal to two. Hence every singularities of our integrals cannot be accounted for the form $(\Delta_1\nu^{u_1})^{\sim} \simeq \Delta_1\nu^{u_1+s}$ and $(\Delta_2\nu^{u_2})^{\sim} \simeq \Delta_2\nu^{u_2+s}$. 
 So if $L^{(0)}(s,\mathrm{Ind}(\Delta_1\nu^{u_1}\otimes\Delta_2\nu^{u_2}),\wedge^2)$ has a pole at $s=s_0$, then $(\Delta_1\nu^{u_1})^{\sim} \simeq \Delta_2\nu^{u_2+s_0}$ which implies that the Rankin-Selberg exceptional $L$-function $L_{ex}(s,\Delta_1\nu^{u_1}\times\Delta_2\nu^{u_2})$ has a pole at $s=s_0$. As both $L$-functions $L^{(0)}(s,\mathrm{Ind}(\Delta_1\nu^{u_1}\otimes\Delta_2\nu^{u_2}),\wedge^2)$ and $L_{ex}(s,\Delta_1\nu^{u_1}\times\Delta_2\nu^{u_2})$ have the same simple poles at $s=s_0$, we obtain that
\[
  L^{(0)}(s,\mathrm{Ind}(\Delta_1\nu^{u_1}\otimes\Delta_2\nu^{u_2}),\wedge^2)^{-1}  \quad \text{divides} \quad L_{ex}(s,\Delta_1\nu^{u_1}\times\Delta_2\nu^{u_2})^{-1},
\]
where we take the divisibility in $\mathbb{C}[q^{\pm s}]$.

 \par

Upon both $L$-functions $L^{(0)}(s,\mathrm{Ind}(\Delta_1\nu^{u_1}\otimes\Delta_2\nu^{u_2}),\wedge^2)$ and $L_{ex}(s,\Delta_1\nu^{u_1}\times\Delta_2\nu^{u_2})$ normalized to be an standard Euler factor, in other words, $L^{(0)}(s,\mathrm{Ind}(\Delta_1\nu^{u_1}\otimes\Delta_2\nu^{u_2}),\wedge^2)$ and $L_{ex}(s,\Delta_1\nu^{u_1}\times\Delta_2\nu^{u_2})$ being of the form $P(q^{-s})^{-1}$ with $P(X) \in \mathbb{C}[X]$ having $P(0)=1$,  we have
\[
  L^{(0)}(s,\mathrm{Ind}(\Delta_1\nu^{u_1}\otimes\Delta_2\nu^{u_2}),\wedge^2)=L_{ex}(s,\Delta_1\nu^{u_1}\times\Delta_2\nu^{u_2}).
\]

\end{proof}

 If $\pi_u=\mathrm{Ind}(\Delta_1\nu^{u_1}\otimes \Delta_2\nu^{u_2})$ is an irreducible generic representation of $GL_m$ in general position, various derivatives $\mathrm{Ind}(\Delta_1^{(a_1r_1)}\nu^{u_1}\otimes \Delta_2^{(a_2r_2)}\nu^{u_2})$ is again the irreducible representation in general position. Each segment $\Delta_i^{(a_ir_i)}=[\rho_i\nu^{a_i},\dotsm,\rho_i\nu^{\ell_i-1}]$ can be written as
 $\Delta_i^{(a_ir_i)}=[\rho'_i,\dotsm,\rho_i'\nu^{\ell_i-a_i-1}]$ for $\rho_i'=\rho_i\nu^{a_1}$ and $0 \leq a_i < \ell_i-1$. Hence we have the following result.

 \begin{corollary}
 \label{twosegment-L0-Cor}
 With the notation in Proposition \ref{twosegment-L0}, we have
 \[
   L^{(0)}(s,\mathrm{Ind}(\Delta_1^{(a_1r_1)}\nu^{u_1} \otimes \Delta_2^{(a_2r_2)}\nu^{u_2}),\wedge^2)=L_{ex}(s,\Delta^{(a_1r_1)}_1\nu^{u_1}  \times \Delta^{(a_2r_2)}_2\nu^{u_2} )
 \]
 for $0 \leq a_1 < \ell_1$ and $0 \leq a_2 < \ell_2$.
 \end{corollary}

We furthermore can translate exceptional poles of $L$-functions $L(s,\mathrm{Ind}(\Delta_1^{(k_1)}\otimes \Delta_2^{(k_2)}),\wedge^2)$ for the derivative of a pair of fixed segments $\Delta_1$ and $\Delta_2$ in terms of exceptional poles of Rankin-Selberg $L$-functions $L(s,\Delta_1^{(k_1)} \times \Delta_2^{(k_2)})$.

 \begin{proposition}
 \label{two blocks-prop}
 Let $\pi=\mathrm{Ind}(\Delta_1 \otimes \Delta_2)$ be a representation of Whittaker type of $GL_m$. We let $u=(u_1,u_2) \in \mathcal{D}_{\pi}$ be in general position and $\pi_u=\mathrm{Ind}(\Delta_1\nu^{u_1} \otimes \Delta_2\nu^{u_2})$ an irreducible generic representation of $GL_m$. We assume that $\Delta_i=[\rho_i,\rho_i\nu,\dotsm,\rho_i\nu^{\ell_i-1}]$ is an irreducible quasi-square-integrable representation of $GL_{m_i}$, where $\rho_i$ is an irreducible supercuspidal representation of $GL_{r_i}$, $m_i=r_i \ell_i$ for $i=1,2$ and $m=m_1+m_2$. Then we have 
  \[
   L_{ex}(s,\mathrm{Ind}(\Delta_1\nu^{u_1} \otimes \Delta_2\nu^{u_2}),\wedge^2)=L_{ex}(s,\Delta_1\nu^{u_1}  \times \Delta_2\nu^{u_2} ).
 \]
 \end{proposition}

\begin{proof}
 We begin with the case when $\ell_1=\ell_2=1$, that is, $\Delta_1\nu^{u_1}=\rho_1\nu^{u_1}$ and $\Delta_2\nu^{u_2}=\rho_2\nu^{u_2}$ are also irreducible supercuspidal representations. If $r_1 \neq r_2$, then 
 $L_{ex}(s,\mathrm{Ind}(\Delta_1\nu^{u_1} \otimes \Delta_2\nu^{u_2}),\wedge^2)=L^{(0)}(s,\mathrm{Ind}(\Delta_1\nu^{u_1} \otimes \Delta_2\nu^{u_2}),\wedge^2)=L_{ex}(s,\Delta_1\nu^{u_1}  \times \Delta_2\nu^{u_2} )=1$. As in Proposition \ref{L0 function},
\[
\begin{split}
 &L(s,\mathrm{Ind}(\rho_1\nu^{u_1} \otimes \rho_2\nu^{u_2}),\wedge^2)\\
  &=L^{(0)}(s,\mathrm{Ind}(\rho_1\nu^{u_1} \otimes \rho_2\nu^{u_2}),\wedge^2)L_{(0)}(s,\mathrm{Ind}(\rho_1\nu^{u_1} \otimes \rho_2\nu^{u_2}),\wedge^2).
\end{split}
\]
We first treat the derivatives for $\mathrm{Ind}(\rho_1\nu^{u_1}  \otimes \rho_2\nu^{u_2})$. As $u$ is in general position, and because of Proposition \ref{prod-L0}, we have
\[
\begin{split}
 L_{(0)}(s,\mathrm{Ind}(\rho_1\nu^{u_1} \otimes \rho_2\nu^{u_2}),\wedge^2)&=L_{ex}(s,\rho_1\nu^{u_1},\wedge^2)L_{ex}(s,\rho_2\nu^{u_2},\wedge^2)\\
 &=L(s,\rho_1\nu^{u_1},\wedge^2)L(s,\rho_2\nu^{u_2},\wedge^2).
\end{split}
\]
If $r_1=r_2$ is an odd number, Theorem \ref{supercusp} ensures that
\[
L_{(0)}(s,\mathrm{Ind}(\rho_1\nu^{u_1} \otimes \rho_2\nu^{u_2}),\wedge^2)=1.
\]
According to Proposition \ref{twosegment-L0},  $L^{(0)}(s,\mathrm{Ind}(\rho_1\nu^{u_1} \otimes \rho_2\nu^{u_2}),\wedge^2)=L_{ex}(s,\rho_1\nu^{u_1}\times\rho_2\nu^{u_2})=L(s,\rho_1\nu^{u_1} \times\rho_2\nu^{u_2})$. Since two $L$-functions $L^{(0)}(s,\mathrm{Ind}(\rho_1\nu^{u_1} \otimes \rho_2\nu^{u_2}),\wedge^2)=L(s,\rho_1\nu^{u_1} \times \rho_2\nu^{u_2})$ and $L_{(0)}(s,\mathrm{Ind}(\rho_1\nu^{u_1} \otimes \rho_2\nu^{u_2}),\wedge^2)
 =L(s,\rho_1\nu^{u_1},\wedge^2)L(s,\rho_2\nu^{u_2},\wedge^2)$ dose not have common poles because of the condition $(5)$ of the Definition of general position, we get 
 \[
 \begin{split}
   L_{ex}(s,\Delta_1\nu^{u_1}  \times \Delta_2\nu^{u_2} )&=L^{(0)}(s,\mathrm{Ind}(\Delta_1\nu^{u_1}\otimes\Delta_2\nu^{u_2}),\wedge^2)\\
   &=L_{ex}(s,\mathrm{Ind}(\Delta_1\nu^{u_1}\otimes\Delta_2\nu^{u_2}),\wedge^2).
\end{split}
\]

\par

Now we assume that $\ell_1,\ell_2 > 1$. It is sufficient to prove the case when $r_1=r_2$ and $\ell_1=\ell_2$. Otherwise $L_{ex}(s,\mathrm{Ind}(\Delta_1\nu^{u_1} \otimes \Delta_2\nu^{u_2}),\wedge^2)=L^{(0)}(s,\mathrm{Ind}(\Delta_1\nu^{u_1} \otimes \Delta_2\nu^{u_2}),\wedge^2)=L_{ex}(s,\Delta_1\nu^{u_1}  \times \Delta_2\nu^{u_2} )=1$ according to the previous Corollary \ref{twosegment-L0-Cor}. We assume that the following statement is satisfied for all integers $k$ with $1 \leq k \leq \ell-1$ 
\[
\begin{split}
 &L_{ex}(s,\mathrm{Ind}([\rho_1,\dotsm,\rho_1\nu^{k}]\nu^{u_1} \otimes [\rho_2,\dotsm,\rho_2\nu^{k}]\nu^{u_2}),\wedge^2)\\
 &=L_{ex}(s,[\rho_1,\dotsm,\rho_1\nu^{k}]\nu^{u_1} \times [\rho_2,\dotsm,\rho_2\nu^{k}]\nu^{u_2}),
\end{split}
\]
where $\rho_i$ is an irreducible supercuspidal representation of $GL_r$ for $i=1,2$.
We show this by induction on the length of segment $\ell$. Let $\Delta_i=[\rho_i,\dotsm,\rho_i\nu^{\ell}]$ be an irreducible quasi-square-integrable representation of $GL_{m}$ with $m=(\ell+1)r$ for $i=1,2$.
Since $u \in \mathcal{D}_{\pi}$ for $\mathrm{Ind}(\Delta_1\nu^{u_1} \otimes \Delta_2\nu^{u_2})$ is in general position, according to Proposition \ref{prod-L0} we find
\[
  L_{(0)}(s,\mathrm{Ind}(\Delta_1\nu^{u_1}\otimes \Delta_2\nu^{u_2}),\wedge^2)^{-1}=l.c.m.\{ L_{ex}(s,\mathrm{Ind}(\Delta_1^{(a_1r)}\nu^{u_1}\otimes \Delta^{(a_2r)}_2\nu^{u_2}),\wedge^2)^{-1} \},
\]
where the least common multiple is taken over all $0 \leq a_1, a_2 \leq \ell+1$ and $a_1+a_2 > 0$ such that $2m-(a_1+a_2)r$ is an even number. We consider $L_{ex}(s,\mathrm{Ind}(\Delta_1^{(a_1r)}\nu^{u_1}\otimes \Delta^{(a_2r)}_2\nu^{u_2}),\wedge^2)$ for three different cases. If $0 \leq a_1 \neq a_2 < \ell+1$, 
\[
\begin{split}
L_{ex}(s,\mathrm{Ind}(\Delta_1^{(a_1r)}\nu^{u_1} \otimes \Delta_2^{(a_2r)}\nu^{u_2}),\wedge^2)&=L^{(0)}(s,\mathrm{Ind}(\Delta_1^{(a_1r)}\nu^{u_1} \otimes \Delta_2^{(a_2r)}\nu^{u_2}),\wedge^2)\\
&=L_{ex}(s,\Delta_1^{(a_1r)}\nu^{u_1}  \times \Delta_2^{(a_2r)}\nu^{u_2} )=1
\end{split}
\]
according to the previous Corollary \ref{twosegment-L0-Cor}. For $a_1=\ell+1$ or $a_2=\ell+1$, as the order of the derivatives varies, by Proposition \ref{even-quasisquare} and \ref{odd-quasisquare}, we see that the least common multiple of the factors $L_{ex}(s,\Delta_i^{(a_ir)}\nu^{u_i},\wedge^2)$ will precisely contribute to
\[
  L(s,\Delta_i\nu^{u_i},\wedge^2)^{-1}=l.c.m.\{ L_{ex}(s,\Delta_i^{(a_ir)}\nu^{u_i},\wedge^2)^{-1} \},
\]
where the least common multiple is taken over all $0 \leq a_i \leq \ell+1$ such that $m-a_ir$ is the even number. We are left with the case when $0 < a_1=a_2 < \ell+1$. The length of the segment $\Delta_i^{(a_ir)}=[\rho\nu^{a_i},\dotsm,\rho\nu^{\ell}]$ is $\ell-a_i+1$ and $\mathrm{Ind}(\Delta_1^{(a_1r)}\nu^{u_1}\otimes \Delta^{(a_2r)}_2\nu^{u_2})$ is still an irreducible generic representation in general position. The induction hypothesis on the length $k$ of segments for $1 \leq k \leq \ell$ then enable us to obtain the least common multiple
\[
\begin{split}
 &\underset{0 < a_1=a_2 < \ell+1}{l.c.m.}\{ L_{ex}(s,\mathrm{Ind}(\Delta^{(a_1r)}_1\nu^{u_1}\otimes \Delta^{(a_2r)}_2\nu^{u_2}),\wedge^2)^{-1} \} \\
 &= \underset{0 < a_1=a_2 < \ell+1}{l.c.m.}\{ L_{ex}(s,\Delta^{(a_1r)}_1\nu^{u_1} \times \Delta^{(a_2r)}_2\nu^{u_2})^{-1} \}
\end{split}
\]
which is equal to 
\[
  \prod_{0 < a_1=a_2 < \ell+1}L_{ex}(s,\Delta^{(a_1r)}_1\nu^{u_1} \times \Delta^{(a_2r)}_2\nu^{u_2})^{-1}.
\]
Putting all three contributions to $L_{(0)}(s,\mathrm{Ind}(\Delta_1\nu^{u_1}\otimes \Delta_2\nu^{u_2}),\wedge^2)$ together we write the $L$-function as
\[
\begin{split}
  &L_{(0)}(s,\mathrm{Ind}(\Delta_1\nu^{u_1}\otimes \Delta_2\nu^{u_2}),\wedge^2)\\
  &=L(s,\Delta_1\nu^{u_1},\wedge^2)L(s,\Delta_2\nu^{u_2},\wedge^2)\prod_{0 < a_1=a_2 < \ell+1}L_{ex}(s,\Delta^{(a_1r)}_1\nu^{u_1} \times \Delta^{(a_2r)}_2\nu^{u_2}),
\end{split}
\]
which does not have any common poles with the $L$-function $L^{(0)}(s,\mathrm{Ind}(\Delta_1\nu^{u_1}\otimes \Delta_2\nu^{u_2}),\wedge^2)\hspace*{-.5mm}=L_{ex}(s,\Delta_1\nu^{u_1} \times \Delta_2\nu^{u_2})$
because $u$ is in general position. As a result, we have
\[
   L_{ex}(s,\mathrm{Ind}(\Delta_1\nu^{u_1} \otimes \Delta_2\nu^{u_2}),\wedge^2)=L^{(0)}(s,\mathrm{Ind}(\Delta_1\nu^{u_1} \otimes \Delta_2\nu^{u_2}),\wedge^2),
\]
from Proposition \ref{L0 function}, and hence $L_{ex}(s,\mathrm{Ind}(\Delta_1\nu^{u_1} \otimes \Delta_2\nu^{u_2}),\wedge^2)=L_{ex}(s,\Delta_1\nu^{u_1} \times \Delta_2\nu^{u_2})$
which is what we want.

\end{proof}

As we mention in the proof, $\mathrm{Ind}(\Delta_1^{(a_1r)}\nu^{u_1}\otimes \Delta^{(a_2r)}_2\nu^{u_2})$ is an irreducible generic representation in general position with the segment $\Delta_i^{(a_ir_i)}=[\rho_i\nu^{a_i},\dotsm,\rho_i\nu^{\ell_i-1}]$ of length $\ell_i-a_i$. Hence the previous Proposition \ref{two blocks-prop} is still applicable for  various derivatives $\mathrm{Ind}(\Delta_1^{(a_1r)}\nu^{u_1}\otimes \Delta^{(a_2r)}_2\nu^{u_2})$ and we have the following Corollary.

\begin{corollary}
\label{two blocks}
Let $\pi=\mathrm{Ind}(\Delta_1 \otimes \Delta_2)$ be a representation of Whittaker type of $GL_m$. We let $u=(u_1,u_2) \in \mathcal{D}_{\pi}$ be in general position and $\pi_u=\mathrm{Ind}(\Delta_1\nu^{u_1} \otimes \Delta_2\nu^{u_2})$ an irreducible generic representation of $GL_m$. We assume that $\Delta_i=[\rho_i,\rho_i\nu,\dotsm,\rho_i\nu^{\ell_i-1}]$ is an irreducible quasi-square-integrable representation of $GL_{m_i}$, where $\rho_i$ is an irreducible supercuspidal representation of $GL_{r_i}$, $m_i=r_i \ell_i$ for $i=1,2$ and $m=m_1+m_2$. Then we have
 \[
   L_{ex}(s,\mathrm{Ind}(\Delta_1^{(a_1r_1)}\nu^{u_1} \otimes \Delta_2^{(a_2r_2)}\nu^{u_2}),\wedge^2)=L_{ex}(s,\Delta^{(a_1r_1)}_1\nu^{u_1}  \times \Delta^{(a_2r_2)}_2\nu^{u_2} )
 \]
 for $0 \leq a_1 < \ell_1$ and $0 \leq a_2 < \ell_2$.

\end{corollary}

In Section 8, this will be the first step to prove multiplicativity result of exterior square $L$-function for deformed representations in general position.

\section{Multiplicativity of $L$-Functions for Representations in General Position}

\subsection{Exterior square $L$-functions for irreducible admissible representations.}

The purpose of considering deformed representations $\pi_u$ of $\pi$ in general position is that the $L$-function is easy to compute for those representations. In general, the constituents of the derivatives are not quasi-square-integrable, hence we do not have a good way of explicitly analyzing whether a given possible pole actually occurs. We resolve the difficulty in this section by use of a deformation of representations and an argument utilizing Hartogs' Theorem. We specialize the result of exterior square $L$-function for the family of representations in general position to our original representation $\pi$. We then concern $L(s,\pi,\wedge^2)$ for irreducible generic representation $\pi$ or representation of Langlands type $\pi$. In order to complete the computation of $L(s,\pi,\wedge^2)$ for an irreducible admissible representation $\pi$ of $GL_m$, we need to define the exterior square $L$-function for non-generic representation $\pi$.

\begin{definition*}
 Let $\pi$ be an irreducible admissible representation of $GL_m$. Let $\Pi=\mathrm{Ind}(\Delta_1\nu^{u_1} \otimes \dotsm \otimes \Delta_t\nu^{u_t} )$ be the Langlands induced representation  with each $\Delta_i$ an irreducible square integrable representation, the $u_i$ real and ordered so that $u_1 \geq \dotsm \geq u_t$ such that $\pi$  is the unique irreducible quotient of $\Pi$. Then the $L$-function $L(s,\pi,\wedge^2)$ is defined by
 \[
    L(s,\pi,\wedge^2)=L(s,\Pi,\wedge^2).
 \]
\end{definition*}

In this section, we closely follow the argument of Section 4 of \cite{CoPe}. Let us begin with representations $\pi=\text{Ind}(\Delta_1\otimes \dotsm \otimes \Delta_t)$ of $GL_m$ with the $\Delta_i$ irreducible quasi-square-integrable. For now, $\pi$ need not to be irreducible. We deform the representation to families $\pi_u$ with $u=(u_1,\dotsm,u_t) \in \mathcal{D}_{\pi}$. For $u$ in general position, $\pi_u$ is irreducible, and its higher derivatives are completely reducible. Then we can compute the local exterior square $L$-functions for $\pi_u$ in general position in the following way.

\begin{theorem}
\label{general-prod}
Let $\pi=\mathrm{Ind}(\Delta_1 \otimes \dotsm \otimes \Delta_t)$ be a representation of $GL_m$ of Whittaker type. We let $u=(u_1,\dotsm,u_t) \in \mathcal{D}_{\pi}$ be in general position and $\pi_u=\mathrm{Ind}(\Delta_1\nu^{u_1}\otimes \dotsm \otimes \Delta_t\nu^{u_t})$ the deformed representation of $\pi$ on $GL_m$. Then we have
\[
 L(s,\pi_u,\wedge^2)=\prod_{1 \leq k \leq t} L(s+2u_k,\Delta_k,\wedge^2) \prod_{1 \leq i < j \leq t} L(s+u_i+u_j, \Delta_i \times \Delta_j).
\]
\end{theorem}

\begin{proof}

Let us fix points $u$ in general position. Consider the Jacquet and Shalika local integrals for $\pi_u$, namely the $J(s,W_u,\Phi)$ if $m=2n$ or $m=2n+1$ with $W_u \in \mathcal{W}^{(0)}_{\pi}$, and $\Phi \in \mathcal{S}(F^n)$, and the fractional ideals $\mathcal{J}(\pi_u) \subset \mathbb{C}(q^{-s})$ they generate. Let us now take $\Delta_i$ to be associated to the segment $[\rho_i,\dotsm,\rho_i\nu^{\ell_i-1}]$ with $\rho_i$ an irreducible supercuspidal representation of $GL_{r_i}$ $m_i=r_i\ell_i$ and $m=\sum r_i\ell_i$. Then by Theorem \ref{prod-L} and Theorem \ref{prod-L-odd}, we know that the poles of these families are precisely the poles of the exceptional contributions to the $L$-functions of the form
\[
  L_{ex}(s,\pi_u^{(a_1r_1,\dotsm,a_tr_t)},\wedge^2)
\]
 such that $0 \leq a_i \leq \ell_i$ and $m-\sum a_ir_i$ is an even number. In fact,
 \[
   L(s,\pi_u,\wedge^2)^{-1}=l.c.m. \{ L_{ex}(s,\pi_u^{(a_1r_1,\dotsm,a_tr_t)},\wedge^2)^{-1} \}
 \]
where $0 \leq a_i \leq \ell_i$, $m-\sum a_ir_i$ is an even number, and the least common multiple is taken in terms of divisibility in $\mathbb{C}[q^{\pm s}]$. Each such exceptional $L$-function can have poles which lie along the twisted self-contragrediance requirement
\[
  (\pi_u^{(a_1r_1,\dotsm,a_tr_t)})^{\sim} \simeq \pi_u^{(a_1r_1,\dotsm,a_tr_t)}\nu^s
\]
or equivalently
\[
\label{deformed-self contra}
\tag{8.1}
 ( \mathrm{Ind}(\Delta_{1,u_1}^{(a_1r_1)} \otimes \dotsm \otimes \Delta_{t,u_t}^{(a_tr_t)}))^{\sim} \simeq \mathrm{Ind}(\Delta_{1,u_1}^{(a_1r_1)}\otimes \dotsm \otimes \Delta_{t,u_t}^{(a_tr_t)} )\nu^s.
\]
For $x$ a positive real number, let $\lfloor x \rfloor$ be the greatest integer less than or equal to $x$. Since the induced representations are irreducible, the only way that \eqref{deformed-self contra} is possible is that there is a reordering of those indices $i$ for which $a_i \neq \ell_i$ and under this reordering there exists an integer $p$ between $0$ and $\lfloor t/2 \rfloor$ such that $(\Delta_{i_j}^{(a_{i_j}r_{i_j})}\nu^{u_{i_j}})^{\sim} \simeq \Delta_{i_{j+1}}^{(a_{i_{j+1}}r_{i_{j+1}})}\nu^{u_{i_{j+1}}}\nu^s$ or equivalently $(\Delta_{i_j}^{(a_{i_j}r_{i_j})})^{\sim} \simeq \Delta_{i_{j+1}}^{(a_{i_{j+1}}r_{i_{j+1}})}\nu^{u_{i_j}+u_{i_{j+1}}+s}$ for $i_j < i _{j+1}$ and $j=1,3,\dotsm,2p-1$ if $p \neq 0$, and $(\Delta_{i_j}^{(a_{i_j}r_{i_j})}\nu^{u_{i_j}})^{\sim} \simeq \Delta_{i_j}^{(a_{i_j}r_{i_j})}\nu^{u_{i_j}}\nu^s$ or equivalently $(\Delta_{i_j}^{(a_{i_j}r_{i_j})})^{\sim} \simeq \Delta_{i_j}^{(a_{i_j}r_{i_j})}\nu^{2u_{i_j}+s}$ for $j > 2p$.

\par
Now let us see how the condition \eqref{deformed-self contra} varies in $u=(u_1,\dotsm,u_t) \in \mathcal{D}_{\pi}$. Consider now the Jacquet-Shalika local integral $J(s,W_u,\Phi)$ if $m=2n$ or $m=2n+1$ with $W_u \in \mathcal{W}^{(0)}_{\pi}$, and $\Phi \in \mathcal{S}(F^n)$. Then these local integrals define rational functions in $\mathbb{C}(q^{-s},q^{-u})$. For $u$ in the Zariski open subset of general position, these rational functions can have poles coming from the exceptional contributions to the exterior square $L$-functions from
\[
   L(s,\pi_u,\wedge^2)^{-1}=l.c.m. \{ L_{ex}(s,\pi_u^{(a_1r_1,\dotsm,a_tr_t)},\wedge^2)^{-1} \}
 \]
where $0 \leq a_i \leq \ell_i$ and $m-\sum a_ir_i$ is an even number. Each such exceptional $L$-function can have poles which lie along the locus defined by a finite number of hyperplanes, where there is one equation 
\[
q^{-(u_i+u_j+s)(m_i-a_ir_i)}\omega_{\Delta_i^{(a_ir_i)}}(\varpi)\omega_{\Delta_j^{(a_jr_j)}}(\varpi)=1
\] 
for every pair $(i,j)$ of indices $i < j$ such that $r_i \neq \ell_i$, $r_j \neq \ell_j$, and $(\hspace*{-.5mm}\Delta_i^{(a_ir_i)}\hspace*{-.4mm})^{\sim}\hspace*{-1mm} \simeq\hspace*{-1mm} \Delta_j^{(a_jr_j)}\nu^{u_i+u_j+s}$, or one equation
\[
 q^{-(2u_i+s)(m_i-a_ir_i)}\omega_{\Delta_i^{(a_ir_i)}}^2(\varpi)=1
\]
for every indices $i$ such that $r_i \neq \ell_i$ and $(\Delta_i^{(a_ir_i)})^{\sim} \simeq \Delta_i^{(a_ir_i)}\nu^{2u_i+s}$. So if there is more than one pair of indices $i < j$ such that $r_i \neq \ell_i$, $r_j \neq \ell_j$, and $(\Delta_i^{(a_ir_i)})^{\sim} \simeq \Delta_j^{(a_jr_j)}\nu^{u_i+u_j+s}$, or more than one index $i$ satisfying $(\Delta_i^{(a_ir_i)})^{\sim} \simeq \Delta_i^{(a_ir_i)}\nu^{2u_i+s}$ and $r_i \neq \ell_i$, as we are assuming, then this singular locus will be defined by $2$ or more independent equations and hence will be of codimension greater than or equal to two. By Theorem of Hartogs in Theorem \ref{Hartog}, we know that a function of several complex variables on a domain that is holomorphic outside an analytic set of codimension two or more extends uniquely to a function holomorphic on that domain. Hence every singularity of our local integrals must be accounted by an exceptional contribution of the form $L_{ex}(s,\mathrm{Ind}(\Delta_i^{(a_ir_i)}\nu^{u_i}\otimes \Delta_j^{(a_jr_j)}\nu^{u_j} ),\wedge^2)$ for $i < j$ or of the form $L_{ex}(s,\Delta_i^{(a_ir_i)}\nu^{u_i},\wedge^2)$. According to Proposition \ref{two blocks}, the first type is the same contribution as the Rankin-Selberg $L$-function
\[
\begin{split}
  L_{ex}(s,\mathrm{Ind}(\Delta_i^{(a_ir_i)}\nu^{u_i}\otimes \Delta_j^{(a_jr_j)}\nu^{u_j} ),\wedge^2)&=L_{ex}(s,\Delta_i^{(a_ir_i)}\nu^{u_i} \times \Delta_j^{(a_jr_j)}\nu^{u_i})\\
  &=L(s+u_i+u_j,\Delta_i^{(a_ir_i)}\times \Delta_j^{(a_jr_j)}).
\end{split}
\]
For fixed $i$ and $j$ with $i < j$, we see that the least common multiple of the inverse of these factors $L(s+u_i+u_j,\Delta_i^{(a_ir_i)}\times \Delta_j^{(a_jr_j)})$
will contribute a factor of $L(s+u_i+u_j,\Delta_i \times \Delta_j)^{-1}$ to $L(s,\pi_u,\wedge^2)^{-1}$. For fixed $i$, by the analysis in Proposition \ref{even-quasisquare} and \ref{odd-quasisquare}, the least common multiple of the inverse of the second types of factors $L_{ex}(s,\Delta_i^{(a_ir_i)}\nu^{u_i},\wedge^2)$ contributes a factor of $L(s,\Delta_i\nu^{u_i},\wedge^2)^{-1}=L(s+2u_i,\Delta_i,\wedge^2)^{-1}$ to $L(s,\pi_u,\wedge^2)^{-1}$. Now, still for $u$ in general position, the $L(s+u_i+u_j,\Delta_i \times \Delta_j)^{-1}$ with $i < j$ and $L(s+2u_k,\Delta_k,\wedge^2)^{-1}$ will be relatively prime as $i,j$ and $k$ vary. So their contribution to the least common multiple of the exceptional contribution to the exterior square $L$-functions will be their product. In this way, we obtain that
\[
  L(s,\pi_u,\wedge^2)=\prod_{1 \leq k \leq t} L(s+2u_k,\Delta_k,\wedge^2) \prod_{1 \leq i < j \leq t} L(s+u_i+u_j, \Delta_i \times \Delta_j)
\]
 because $L(s,\pi_u,\wedge^2)^{-1}$ is the least common multiple of
$
  \{ L_{ex}(s,\Delta_k^{(a_kr_k)}\nu^{u_k},\wedge^2)^{-1}\; | \; 1\leq k  \leq t \} 
$
 and
$
    \{ L_{ex}(s,\text{Ind}(\Delta_i^{(a_ir_i)}\nu^{u_i} \otimes \Delta_j^{(a_jr_j)}\nu^{u_j}),\wedge^2)^{-1}\; | \; 1\leq i < j \leq t \}.
$

\end{proof}

 We recall some results from complex analytic geometry in \cite{Fi82}. Let $\mathcal{M}$ be a complex manifold and $\mathcal{X}$ an analytic hypersurface of $\mathcal{M}$. We denote the set of regular points of $\mathcal{X}$ by $Reg(\mathcal{X})$. We say that $\mathcal{X}$ is reducible if we can find analytic subsets $\mathcal{Y}, \mathcal{Z}$ of $\mathcal{M}$, neither of which equals to $\mathcal{X}$, such that $\mathcal{X}=\mathcal{Y} \cup \mathcal{Z}$. If $\mathcal{X}$ is not reducible, we say $\mathcal{X}$ is irreducible. We state the following results from \cite{Fi82}.

 \begin{theorem}[\cite{Fi82}, Chapter IV, Section 4, Theorem 4.6.7] 
 \label{hypersurfacedec}
 Let $\mathcal{X}$ be an analytic hypersurface of the complex manifold $\mathcal{M}$ and $\{ \mathcal{X}_i \;| \; i \in I \}$ denote the set of connected component of $Reg(\mathcal{X})$. Then $\mathcal{X}=\cup_{i \in I} \mathcal{X}_i$ and this decomposition of $\mathcal{X}$ as a countable union of irreducible analytic hypersurfaces is unique up to order.
 
 \end{theorem}
 
  We call the hypersurfaces $\mathcal{X}_i$ in Theorem \ref{hypersurfacedec} the irreducible component of $\mathcal{X}$. We would like to specialize the result in Theorem \ref{general-prod} to $u$ not in general position.

 \begin{proposition}
\label{deform-l factor}
Let $\pi=\mathrm{Ind}(\Delta_1 \otimes \dotsm \otimes \Delta_t)$ be a representation of $GL_m$ of Whittaker type. Then 
$\displaystyle L(s,\pi,\wedge^2) \in \prod_{1 \leq k \leq t} L(s,\Delta_k,\wedge^2) \prod_{1 \leq i < j \leq t}  L(s, \Delta_i \times \Delta_j) \mathbb{C}[q^s,q^{-s}]$.
\end{proposition}

 \begin{proof}
 
 We know from Proposition \ref{even-rational} and Proposition \ref{odd-rational} that for $W_u \in \mathcal{W}^{(0)}_{\pi}$ and $\Phi \in \mathcal{S}(F^n)$ with $m=2n$ or $m=2n+1$, the ratio
 \[
 \label{deform-rational}
 \tag{8.2}
   \frac{J(s,W_u,\Phi)}{\prod L(s+2u_k,\Delta_k,\wedge^2) \prod L(s+u_i+u_j, \Delta_i \times \Delta_j)} 
    \]
    is a rational functions in $\mathbb{C}(q^{-u},q^{-s})$. Let $P(q^{\pm u},q^{\pm s}) \in \mathbb{C}[q^{\pm u},q^{\pm s}]$ be the denominator of the rational function in \eqref{deform-rational}. We let $\mathcal{D}_s=(\mathbb{C}/\frac{2\pi i}{\log(q)}\mathbb{Z}) \simeq \mathbb{C}^{\times}$. Knowing that the exterior square $L$-function for the deformed representation in general position $\pi_u$ is given by the product in Theorem \ref{general-prod}, which is the inverse of a Laurent polynomial in $\mathbb{C}[q^{\pm u},q^{\pm s}]$, we know that for the rational function in \eqref{deform-rational} has no poles on the Zariski open set of $u$ in general position. Proposition \ref{general position} asserts that the removed hyperplanes defining general position, $H_1,\dotsm,H_p \in \mathcal{D}_{\pi}=(\mathbb{C}/\frac{2\pi i}{\log(q)}\mathbb{Z})^t$, do not depend on $s$, but only on $u$. If $P(q^{\pm u},q^{\pm s})$ is not unit in $\mathbb{C}[q^{\pm u},q^{\pm s}]$, then the set $Z(P)$ of zeroes of exponential polynomial $P(q^{\pm u},q^{\pm s})$ which defines the hypersurface in $\mathcal{D}_{\pi} \times \mathcal{D}_s$ is contained in the union $\cup_i\;H_i \times \mathcal{D}_s$. By Theorem \ref{hypersurfacedec}, one of the irreducible components of $Z(P)$ is $H \times \mathcal{D}_s$  for $H$ an affine hyperplane of $\mathcal{D}_{\pi}$. However $H \times \mathcal{D}_s$ cannot lie entirely in $Z(P)$, as for any fixed $u \in H$, for $s$ large enough, the rational functions $J(s,W_u,\Phi)$ are defined by absolutely convergent integrals according to \eqref{even-absconv} or \eqref{odd-absconv}, hence is holomorphic, and the inverse of $\prod L(s+2u_k,\Delta_k,\wedge^2) \prod L(s+u_i+u_j, \Delta_i \times \Delta_j)$ is polynomial in $\mathbb{C}[q^{\pm u},q^{\pm s}]$, thus without poles. Therefore the ratios in \eqref{deform-rational} is entire and hence lie in $\mathbb{C}[q^{\pm u},q^{\pm s}]$. If we now specialize $u=0$, we find that
\[
   \frac{J(s,W,\Phi)}{\prod L(s,\Delta_k,\wedge^2) \prod L(s, \Delta_i \times \Delta_j)}
   \]
   have no poles for all $W \in \mathcal{W(\pi,\psi)}$. This completes the proof.

    \end{proof}

Let us consider again the behavior of the gamma factor for the representation $\pi_u=\mathrm{Ind}(\Delta_1\nu^{u_1}\otimes \dotsm \otimes \Delta_t\nu^{u_t})$ on $GL_m$ under deformation. From the previous Proposition \ref{deform-l factor}, we cannot conclude multiplicativity of the local $\gamma$-factor, but we only have the following weak version of multiplicativity of the $\gamma$-factor.

\begin{proposition}
\label{deform-gamma}
Let $\pi=\mathrm{Ind}(\Delta_1 \otimes \dotsm \otimes \Delta_t)$ be a representation of $GL_m$ of Whittaker type. We let $u=(u_1,\dotsm,u_t)$ be an element of $\mathcal{D}_{\pi}$ and $\pi_u=\mathrm{Ind}(\Delta_1\nu^{u_1}\otimes \dotsm \otimes \Delta_t\nu^{u_t})$ the deformed representation of $\pi$. Then $\gamma(s,\pi_u,\wedge^2,\psi)$ and 
\[
  \prod_{1 \leq k \leq t}\gamma(s+2u_k,\Delta_k,\wedge^2,\psi) \prod_{1 \leq i < j \leq t}\gamma(s+u_i+u_j,\Delta_i \times \Delta_j,\psi)
  \]
are equal up to a unit in $\mathbb{C}[q^{\pm u},q^{\pm s}]$.
\end{proposition}

\begin{proof}

We know from  Corollary \ref{gamma-rational} and Corollary \ref{gamma-rational-odd} that $\gamma(s,\pi_u,\wedge^2,\psi)$ is a rational function in $\mathbb{C}(q^{-u},q^{-s})$. The local $\varepsilon$-factor satisfies
\[
  \gamma(s,\pi_u,\wedge^2,\psi)=\frac{\varepsilon(s,\pi_u,\wedge^2,\psi)L(1-s,(\pi_u)^{\iota},\wedge^2)}{L(s,\pi_u,\wedge^2)}.
\]
For fixed $u$ we know, by applying the functional equations twice, that $\varepsilon$-factor $\varepsilon(s,\pi_u,\wedge^2,\psi)$ is of the form $\alpha q^{-\beta s}$, that is, it is a unit in $\mathbb{C}[q^{\pm s}]$. 
If we now define a variant of the $\varepsilon$-factor by
\[
\label{variantgamma}
\tag{8.3}
\begin{split}
& \gamma(s,\pi_u,\wedge^2,\psi)=  \varepsilon^{\circ}(s,\pi_u,\wedge^2,\psi) \\
&\phantom{***}\times \frac{\prod_{1 \leq k \leq t}L(1-s-2u_k,\widetilde{\Delta}_k,\wedge^2) \prod_{1 \leq i < j \leq t}L(1-s-u_i-u_j,\widetilde{\Delta}_i \times \widetilde{\Delta}_j)}
 {\prod_{1 \leq k \leq t}L(s+2u_k,\Delta_k,\wedge^2)\prod_{1 \leq i < j \leq t}L(s+u_i+u_j,\Delta_i \times \Delta_j)}
\end{split}
\]
then $\varepsilon^{\circ}(s,\pi_u,\wedge^2,\psi) \in \mathbb{C}(q^{-u},q^{-s})$ and for $u$ in general position we obtain $\varepsilon^{\circ}(s,\pi_u,\wedge^2,\psi)=\varepsilon(s,\pi_u,\wedge^2,\psi)$. Since the $\varepsilon^{\circ}$-factor is the elements of $\mathbb{C}(q^{-u},q^{-s})$, let $P(\hspace*{-.3mm}q^{\pm u},q^{\pm s}\hspace*{-.3mm})\hspace*{-1mm} \in \hspace*{-.7mm}\mathbb{C}[q^{\pm u},q^{\pm s}]$ be the denominator of the rational function $\varepsilon^{\circ}(s,\pi_u,\wedge^2,\psi)$. As the $\varepsilon$-factor $\varepsilon(s,\pi_u,\wedge^2,\psi)$ is a unit in $\mathbb{C}[q^{\pm s}]$, this implies that $\varepsilon^{\circ}(s,\pi_u,\wedge^2,\psi)$ has no poles on the Zariski open set of $u$ in general position. Proposition \ref{general position} asserts that the removed hyperplanes defining general position, $H_1,\dotsm,H_p \in \mathcal{D}_{\pi}=(\mathbb{C}/\frac{2\pi i}{\log(q)}\mathbb{Z})^t$, do not depend on $s$, but only on $u$. If $P(q^{\pm u},q^{\pm s})$ is not unit in $\mathbb{C}[q^{\pm u},q^{\pm s}]$, then the set $Z(P)$ of zeroes of exponential polynomial $P(q^{\pm u},q^{\pm s})$ which defines the hypersurface in $\mathcal{D}_{\pi} \times \mathcal{D}_s$ is contained in the union $\cup_i\;H_i \times \mathcal{D}_s$. By Theorem \ref{hypersurfacedec}, one of irreducible components of $Z(P)$ is $H \times \mathcal{D}_s$  for $H$ an affine hyperplane of $\mathcal{D}_{\pi}$. Thus  the hypersurface $Z(P)$ necessarily contains a set of the from $H \times \mathcal{D}_s$, for $H$ an affine hyperplane of $\mathcal{D}_{\pi}$. Since we know from the local functional equation in Theorem \ref{local functional eq} and Theorem \ref{local-func-odd} that $J(1-s,\varrho(\tau_{m})\widetilde{W_u},\hat{\Phi})=\gamma(s,\pi_u,\wedge^2,\psi)J(s,W_u,\Phi)$, \eqref{variantgamma} can be rewritten as
\[
\label{ration of J and L}
\tag{8.4}
\begin{split}
 & \frac{J(1-s,\varrho(\tau_{m})\widetilde{W_u},\hat{\Phi})}{\prod_{1 \leq k \leq t}L(1-s-2u_k,\widetilde{\Delta}_k,\wedge^2) \prod_{1 \leq i < j \leq t}L(1-s-u_i-u_j,\widetilde{\Delta}_i \times \widetilde{\Delta}_j)}\\
 &=\varepsilon^{\circ}(s,\pi_u,\wedge^2,\psi) \frac{J(s,W_u,\Phi)}{\prod_{1 \leq k \leq t}L(s+2u_k,\Delta_k,\wedge^2)\prod_{1 \leq i < j \leq t}L(s+u_i+u_j,\Delta_i \times \Delta_j)}
  \end{split}
\]
Because of the previous Proposition \ref{deform-l factor}, the left hand side of the equality \eqref{ration of J and L} is entire and so $H \times \mathcal{D}_s$ is a subset of the zeros of the polynomials 
\[
\frac{J(s,W_u,\Phi)}{\prod_{1 \leq k \leq t}L(s+2u_k,\Delta_k,\wedge^2)\prod_{1 \leq i < j \leq t}L(s+u_i+u_j,\Delta_i \times \Delta_j)}
\]
for all $W_u \in \mathcal{W}^{(0)}_{\pi}$ and $\Phi \in \mathcal{S}(F^n)$ with $m=2n$ or $m=2n+1$. 
We recall that $\mathcal{P}_0=\langle W_u(I_n)\;|\; W_u \in  \mathcal{W}^{(0)}  \rangle$. Reasoning as finding the normalized equation in proof of Proposition \ref{even-rational} or Proposition \ref{odd-rational}, for any $P$ in $\mathcal{P}_0 \subset \mathbb{C}[q^{\pm u}]$, the hyperplane $H \times \mathcal{D}_s$ is a subset of the zeroes of the polynomials
\[
 \frac{P(q^{\pm u})}{\prod_{1 \leq k \leq t}L(s+2u_k,\Delta_k,\wedge^2)\prod_{1 \leq i < j \leq t}L(s+u_i+u_j,\Delta_i \times \Delta_j)}.
\]
As we consider Lemma \ref{kirillov} or Gelfand and Kazdan Theorem F in \cite{GeKa72}, there exists $P$ in $\mathcal{P}_0$ such that $P(q^{\pm u})=1$ for each fixed $u \in \mathcal{D}_{\pi}$, because 
$\mathcal{W}_{(0)}$ contains $S_{\psi}(P_n,\mathcal{P}_0)$ from Proposition \ref{poly-GeKa}. For fixed $u \in H$, this would imply that $\{u \} \times \mathcal{D}_s$ is a subset of the zeros of
\[
\label{inverseLufunction}
\tag{8.5}
 \frac{1}{\prod_{1 \leq k \leq t}L(s+2u_k,\Delta_k,\wedge^2)\prod_{1 \leq i < j \leq t}L(s+u_i+u_j,\Delta_i \times \Delta_j)}.
\]
This is absurd, as for any fixed $u \in H$, \eqref{inverseLufunction} is the nonzero polynomial in $\mathbb{C}[q^{\pm s}]$. Therefore $\varepsilon^{\circ}(s,\pi_u,\wedge^2,\psi)$ has no poles and hence  lies in $\mathbb{C}[q^{\pm u},q^{\pm s}]$.

\par
If we apply the functional equation twice, we find that for $u$ in general position
\[
 \varepsilon^{\circ}(s,\pi_u,\wedge^2,\psi)\varepsilon^{\circ}(1-s,(\pi_u)^{\iota},\wedge^2,\psi^{-1})=1.
\]
Since both sides of this equality are polynomials in $\mathbb{C}[q^{\pm u},q^{\pm s}]$ and agree on the Zariski open subset of $u$ in general position, we have that they agree for all $u$ and $s$. Thus $\varepsilon^{\circ}(s,\pi_u,\wedge^2,\psi)$ is a unit in $\mathbb{C}[q^{\pm u},q^{\pm s}]$, that is, a monomial of the form $\varepsilon^{\circ}(s,\pi_u,\wedge^2,\psi)=\alpha q^{-(\sum_{i=1}^t \beta_i u_i)-\delta s}$.

\par

We return to our consideration of the behavior of $\gamma(s,\pi_u,\wedge^2,\psi)$. The equalities
\[
  \gamma(s+2u_k,\Delta_k,\wedge^2,\psi)=\frac{\varepsilon(s+2u_k,\Delta_k,\wedge^2,\psi)L(1-s-2u_k,\widetilde{\Delta}_k,\wedge^2)}{L(s+2u_k,\Delta_k,\wedge^2)}
\]
and
\[
  \gamma(s+u_i+u_j,\Delta_i \times \Delta_j,\psi)=\frac{\varepsilon(s+u_i+u_j,\Delta_i\times \Delta_j,\psi )L(1-s-u_i-u_j,\widetilde{\Delta}_i \times \widetilde{\Delta}_j)}{L(s+u_i+u_j,\Delta_i \times \Delta_j)},
\]
together with \eqref{variantgamma} imply that
\[
\begin{split}
  &\gamma(s,\pi_u,\wedge^2,\psi)
  =\left\{   \frac{\varepsilon^{\circ}(s,\pi_u,\wedge^2,\psi)}{\prod_{1 \leq k \leq t}\varepsilon(s+2u_k,\Delta_k,\wedge^2,\psi)\prod_{1\leq i<j \leq t}\varepsilon(s+u_i+u_j,\Delta_i \times \Delta_j,\psi) } \right\} 
  \\
 &\phantom{*************} \times \prod_{1 \leq k \leq t}\gamma(s+2u_k,\Delta_k,\wedge^2,\psi) \prod_{1 \leq i < j \leq t}\gamma(s+u_i+u_j,\Delta_i \times \Delta_j,\psi).
  \end{split}
\]

\end{proof}

Under deformation the $\gamma$-factor is multiplicative up to a monomial factor. If we now specialize $u=0$, then we have the following corollary.

\begin{corollary}
\label{multiplicativity}
Let $\pi=\mathrm{Ind}(\Delta_1 \otimes \dotsm \otimes \Delta_t)$ be a representation of Whittaker type on $GL_m$. Then $\gamma(s,\pi,\wedge^2,\psi)$ and $\displaystyle
  \prod_{1 \leq k \leq t}\gamma(s,\Delta_k,\wedge^2,\psi) \prod_{1 \leq i < j \leq t}\gamma(s,\Delta_i \times \Delta_j,\psi)$
are equal up to a unit in $\mathbb{C}[q^{\pm s}]$.

\end{corollary}

From this point on, let us recall that a representation $\pi$ of $GL_m$ will be called of induced representations of Langlands type if $\pi$ is an induced representation of the form
\[
 \pi=\text{Ind}(\Delta_1\nu^{u_1} \otimes \dotsm \otimes \Delta_t\nu^{u_t}),
\]
where each $\Delta_i$ is an irreducible square-integrable representation of $GL_{n_i}$, $m=m_1+\dotsm+m_t$, each $u_i$ is real and they are ordered so that $u_1 \geq u_2 \geq \dotsm \geq u_t$. Note that these are representations of Whittaker type. Every irreducible generic representation $\pi$ can be rearranged to be in Langlands order without changing $\pi$, and in fact every irreducible admissible representation occurs as the unique quotient module of such. We need the following Lemma, which is analogue of Lemma 9.3 of \cite{JaPSSh83} or Lemma 5.10 of \cite{Ma15}.

\begin{lemma}
\label{inclusion}
Suppose that $\pi=\mathrm{Ind}(\pi_1 \otimes \pi_2)$ is a representation of Langlands type of $GL_{n_1+n_2}$ with each $\pi_i$ an induced representation of $GL_{n_i}$ of Langlands type and the induction is
the normalized parabolic induced representation from a standard upper maximal parabolic. Then $L(s,\pi_2,\wedge^2)^{-1}$ divides $L(s,\pi,\wedge^2)^{-1}$, that is, $L(s,\pi_2,\wedge^2)=Q(q^{-s})L(s,\pi,\wedge^2)$ with $Q(X) \in \mathbb{C}[X]$.  
\end{lemma}

\begin{proof}
In Proposition 9.1 of \cite{JaPSSh83} they establish that if $\pi=\mathrm{Ind}(\pi_1 \otimes \pi_2)$ is a representation of $GL_{n_1+n_2}$ with each $\pi_i$ an induced representation of Whittaker type of $GL_{n_i}$, then for every $W_2 \in \mathcal{W}(\pi_2,\psi)$ and $\Phi \in \mathcal{S}(F^{n_2})$ there is a $W \in \mathcal{W}(\pi,\psi)$ such that
\[
  W \begin{pmatrix} h &\\ & I_{n_1} \end{pmatrix}=W_2(h) \Phi(e_{n_2}h)|\mathrm{det}(h)|^{\frac{n_1}{2}}.
\]
for $h \in GL_{n_2}$.
In the case that $2n=n_1+n_2$ is an even number, we know in Proposition \ref{filtration-even} that $\mathbb{C}[q^{\pm s}]$-fractional ideal $\mathcal{J}(\pi)$ contains the $\mathbb{C}[q^{\pm s}]$-fractional ideal $\mathcal{J}_{(2m-1)}(\pi)$ generated by the local integrals
\[
\begin{split}
&J_{(2m-1)}(s,W)\\
&=\int_{N_{n-m} \backslash {GL}_{n-m}}  \int_{\mathcal{N}_{n-m} \backslash \mathcal{M}_{n-m}} W \begin{pmatrix}  \sigma_{2n-2m+1}  \begin{pmatrix} I_{n-m} & X & \\ & I_{n-m} &\\ &&1  \end{pmatrix} \begin{pmatrix}  g&&  \\ &g& \\ &&1 \end{pmatrix} & \\& I_{2m-1} \end{pmatrix} \\
&\phantom{*********************************} \psi^{-1}(\mathrm{Tr} X) dX\; |\mathrm{det}(g)|^{s-2m} dg
\end{split}
\]
with $W \in \mathcal{W}(\pi,\psi)$. In the case of $n_1=2m$, if we let $W_2 \in \mathcal{W}(\pi_2,\psi)$, $\Phi \in \mathcal{S}(F^{n-m})$ and take $W$ to be associated element of $\mathcal{W}(\pi,\psi)$, we see from $\sigma_{2n-2m+1}=\begin{pmatrix} \sigma_{2n-2m} & \\ & 1 \end{pmatrix}$ that this integral turns out to be
\[
\begin{split}
 &J_{(2m-1)}(s,W)=\int_{N_{n-m} \backslash {GL}_{n-m}}  \int_{\mathcal{N}_{n-m} \backslash \mathcal{M}_{n-m}} W_2 \begin{pmatrix} \sigma_{2n-2m} \begin{pmatrix} I_{n-m} & X \\ & I_{n-m}  \end{pmatrix} \begin{pmatrix} g &  \\ & g  \end{pmatrix} \end{pmatrix}  \\
&\phantom{*****************************}  \psi^{-1}(\mathrm{Tr} X) dX \Phi(e_{n-m}g)|\mathrm{det}(g)|^s dg.
 \end{split} 
\]
In the case of $n_1=2m-1$, we can choose any $\Phi \in \mathcal{S}(F^{n_2})$ such that $\Phi(e_{n_2})=1$.
For every $W_2 \in \mathcal{W}(\pi_2,\psi)$ and the associated elements $W$ of $\mathcal{W}(\pi,\psi)$, our integral $J_{(2m-1)}(s,W)$ becomes
\[
\begin{split}
 &J_{(2m-1)}(s,W)\\
 &=\int_{N_{n-m} \backslash {GL}_{n-m}}  \int_{\mathcal{N}_{n-m} \backslash \mathcal{M}_{n-m}} W_2  \begin{pmatrix}  \sigma_{2n-2m+1}  \begin{pmatrix} I_{n-m} & X & \\ & I_{n-m} &\\ &&1  \end{pmatrix} \begin{pmatrix}  g&&  \\ &g& \\ &&1 \end{pmatrix} \end{pmatrix} \\
&\phantom{**********************************}  \psi^{-1}(\mathrm{Tr} X) dX |\mathrm{det}(g)|^{s-1} dg,
 \end{split} 
\]
because $\Phi(e_{n_2})=\Phi(e_{n_2}h)$ is $1$ for $h= \begin{pmatrix}  g&&  \\ &g& \\ &&1 \end{pmatrix}$. Combined, as every integral $J(s,W_2,\Phi)$ or $J(s,W_2)$ for $W_2 \in \mathcal{W}(\pi_2,\psi)$ and $\Phi \in \mathcal{S}(F^{n-m})$ if necessary occurring in $\mathcal{J}(\pi_2)$ is actually an integral $J_{(2m-1)}(s,W)$ for $W \in \mathcal{W}(\pi,\psi)$ in $\mathcal{J}_{(2m-1)}(\pi)$, we obtain $\mathcal{J}(\pi_2) \subset \mathcal{J}(\pi)$ and the conclusion follows.

\par
The case when $2n+1=n_1+n_2$ runs along the same lines, but the local integrals are different. In this case we know in Proposition \ref{filtration-odd} that $\mathbb{C}[q^{\pm s}]$-fractional ideal $\mathcal{J}(\pi)$ contains the $\mathbb{C}[q^{\pm s}]$-fractional ideal $\mathcal{J}_{(2m)}(\pi)$ spanned by the local integrals
\[
\begin{split}
&J_{(2m)}(s,W)\\
&=\int_{N_{n-m} \backslash {GL}_{n-m}}  \int_{\mathcal{N}_{n-m} \backslash \mathcal{M}_{n-m}} W \begin{pmatrix}  \sigma_{2n-2m+1}  \begin{pmatrix} I_{n-m} & X & \\ & I_{n-m} &\\ &&1  \end{pmatrix} \begin{pmatrix}  g&&  \\ &g& \\ &&1 \end{pmatrix} & \\& I_{2m} \end{pmatrix} \\
&\phantom{********************************} \psi^{-1}(\mathrm{Tr} X) dX\; |\mathrm{det}(g)|^{s-2m-1} dg
\end{split}
\]
Performing the same steps as in the even case, for every pairs $(W_2,\Phi)$ in $\mathcal{W}(\pi_2,\psi) \times \mathcal{S}(F^{n_2})$, there is a $W$ in $\mathcal{W}(\pi,\psi)$ such that for any $h$ in $GL_{n_2}$, we have the equality
\[
  W \begin{pmatrix} h &\\ & I_{n_1} \end{pmatrix}=W_2(h) \Phi(e_{n_2}h)|\mathrm{det}(g)|^{\frac{n_1}{2}}.
\]
For this $W$, we can rewrite the integral $J_{(2m)}(s,W)$ as
\[
\begin{split}
 &J_{(2m)}(s,W)=\int_{N_{n-m} \backslash {GL}_{n-m}}  \int_{\mathcal{N}_{n-m} \backslash \mathcal{M}_{n-m}} W_2 \begin{pmatrix} \sigma_{2n-2m} \begin{pmatrix} I_{n-m} & X \\ & I_{n-m}  \end{pmatrix} \begin{pmatrix} g &  \\ & g  \end{pmatrix} \end{pmatrix}  \\
&\phantom{******************************}  \psi^{-1}(\mathrm{Tr} X) dX \Phi(e_{n-m}g)|\mathrm{det}(g)|^s dg
 \end{split} 
\]
if $n_1=2m+1$, or
\[
\begin{split}
 &J_{(2m)}(s,W)\\
 &=\int_{N_{n-m} \backslash {GL}_{n-m}}  \int_{\mathcal{N}_{n-m} \backslash \mathcal{M}_{n-m}} W_2  \begin{pmatrix}  \sigma_{2n-2m+1}  \begin{pmatrix} I_{n-m} & X & \\ & I_{n-m} &\\ &&1  \end{pmatrix} \begin{pmatrix}  g&&  \\ &g& \\ &&1 \end{pmatrix} \end{pmatrix} \\
&\phantom{**********************************}  \psi^{-1}(\mathrm{Tr} X) dX |\mathrm{det}(g)|^{s-1} dg
 \end{split} 
\]
if $n_1=2m$ for $\Phi(e_{n_2})=1$. Thus once again $\mathcal{J}(\pi_2) \subset \mathcal{J}(\pi)$ and we get the stated divisibility.

\end{proof}

We can now state the main Theorem of this section, which states the computation of $L$-functions for the induced representation of Langlands type.

\begin{theorem}
\label{langlands-multi}
Let $\pi=\mathrm{Ind}(\Delta_1\nu^{u_1} \otimes \dotsm \otimes \Delta_t\nu^{u_t})$ be an induced representation of $GL_m$ of Langlands type. Then
\[
  L(s,\pi,\wedge^2)=\prod_{k=1}^t L(s+2u_k,\Delta_k,\wedge^2) \prod_{1\leq i < j \leq t} L(s+u_i+u_j,\Delta_i \times \Delta_j).
\]
\end{theorem}

\begin{proof}
We introduce the notation for two rational functions $P(q^{-s})$ and $Q(q^{-s})$ that $P \sim Q$ denotes that the ratio is a unit in $\mathbb{C}[q^{\pm s}]$, that is, a monomial factor $\alpha q^{-\beta s}$. Let $u_1 \geq u_2 \geq \dotsm \geq u_t$ be fixed. By the deformation argument in Proposition \ref{deform-l factor}, we know that
\[
  L(s,\pi,\wedge^2)=P(q^{-s}) \prod_{k=1}^t L(s+2u_k,\Delta_k,\wedge^2) \prod_{1\leq i < j \leq t} L(s+u_i+u_j,\Delta_i \times \Delta_j)
\]
and
\[
  L(1-s,\pi^{\iota},\wedge^2)=\widetilde{P}(q^{-s}) \prod_{k=1}^t L(1-s-2u_k,\widetilde{\Delta}_k,\wedge^2) \prod_{1\leq i < j \leq t} L(1-s-u_i-u_j,\widetilde{\Delta}_i \times \widetilde{\Delta}_j).
\]
for some elements $P(X)$ and $\widetilde{P}(X)$ of $\mathbb{C}[X]$. According to Proposition \ref{deform-gamma}, we have
\[
 \gamma(s,\pi,\wedge^2,\psi) \sim \prod_{k=1}^t \gamma(s+2u_k,\Delta_k,\wedge^2,\psi) \prod_{1 \leq i < j \leq t}\gamma(s+u_i+u_j,\Delta_i \times \Delta_j,\psi).
\]
Replacing the $\gamma$-factors by their deformation we also have
\[
  \frac{L(1-s,\pi^{\iota},\wedge^2)}{L(s,\pi,\wedge^2)} \sim \prod_{k=1}^t \frac{L(1-s-2u_k,\widetilde{\Delta}_k,\wedge^2)}{L(s+2u_k,\Delta_k,\wedge^2)}
  \prod_{1\leq i < j \leq t} \frac{L(1-s-u_i-u_j,\widetilde{\Delta}_i \times \widetilde{\Delta}_j)}{L(s+u_i+u_j,\Delta_i \times \Delta_j)},
\]
which then gives $P(q^{-s}) \sim \widetilde{P}(q^{-s})$. We will show that the set of zeros of $P$ and $\widetilde{P}$ are disjoint to conclude that they are units in $\mathbb{C}[q^{\pm s}]$. 
\par
We proceed by induction on $t$ the equality
\[
\begin{split}
  &L(s,\mathrm{Ind}(\Delta_1\nu^{u_1} \otimes \dotsm \otimes \Delta_t\nu^{u_t}),\wedge^2)\\
  &\phantom{************}=P(q^{-s}) \prod_{k=1}^t L(s+2u_k,\Delta_k,\wedge^2) \prod_{1\leq i < j \leq t} L(s+u_i+u_j,\Delta_i \times \Delta_j).
\end{split}
\]
If $t=1$ then $\pi$ is a quasi-square-integrable representation and that $L(s,\pi,\wedge^2)=L(s+2u_1,\Delta_1,\wedge^2)$. 
\par
For $t > 1$, we exploit transitivity of induction to write $\pi=\mathrm{Ind}(\Delta_1\nu^{u_1}\otimes  \pi_2)$ where $\pi_2=\mathrm{Ind}(\Delta_2\nu^{u_2}\otimes \dotsm \otimes \Delta_t\nu^{u_t})$ and $\pi^{\iota}=\mathrm{Ind}(\widetilde{\Delta}_t\nu^{-u_t}\otimes \pi_2^{\iota})$ where $\pi_2^{\iota}=\mathrm{Ind}(\widetilde{\Delta}_{t-1}\nu^{-u_{t-1}}\otimes \dotsm \otimes \widetilde{\Delta}_1\nu^{-u_1})$.  By the previous Lemma \ref{inclusion}, there is a polynomial $Q(X) \in \mathbb{C}[X]$ such that
\[
 L(s,\pi_2,\wedge^2)=Q(q^{-s})L(s,\pi,\wedge^2).
\]
We thus are left with
\[
 L(s,\pi_2,\wedge^2)=P(q^{-s})Q(q^{-s}) \prod_{k=1}^t L(s+2u_k,\Delta_k,\wedge^2) \prod_{1\leq i < j \leq t} L(s+u_i+u_j,\Delta_i \times \Delta_j).
\] 
By induction we have
\[
 L(s,\pi_2,\wedge^2)=\prod_{k=2}^t L(s+2u_k,\Delta_k,\wedge^2) \prod_{2 \leq i < j \leq t} L(s+u_i+u_j,\Delta_i \times \Delta_j).
\]
Combined, we obtain the relation
\[
 P(q^{-s})Q(q^{-s})=L(s+2u_1,\Delta_1,\wedge^2)^{-1} \prod_{j=2}^t L(s+u_1+u_j,\Delta_1 \times \Delta_j)^{-1}.
\]
Proposition \ref{even-square} and Proposition \ref{odd-square} imply that $L(s+2u_1,\Delta_1,\wedge^2)^{-1}$ has its zero in the half plane $\mathrm{Re}(s) \leq -2u_1$ and Corollary to Theorem 8.2 of \cite{JaPSSh83}, that Rankin-Selberg $L$-function $\prod_{j \geq 2}L(s,\Delta_1 \times \Delta_j)^{-1}$ has zeros in the set $\cup_{j=2}^t \{\mathrm{Re}(s) \leq -u_1-u_j \}$. As $P(q^{-s})$ divides 
$L(s+2u_1,\Delta_1,\wedge^2)^{-1} \prod_{j=2}^t L(s+u_1+u_j,\Delta_1 \times \Delta_j)^{-1}$, $P(q^{-s})$ has its zeros in the half plane $\mathrm{Re}(s) \leq -u_1-u_t$. By the same argument applied to $\pi^{\iota}$, we find the polynomial $\widetilde{Q}(X)$ of $\mathbb{C}[X]$ such that
\[
  L(1-s,\pi_2^{\iota},\wedge^2)=\widetilde{Q}(q^{-s})L(1-s,\pi^{\iota},\wedge^2).
\]
Combined again, we get the relation
\[
 \widetilde{P}(q^{-s})\widetilde{Q}(q^{-s})=L(1-s-2u_t,\Delta_t,\wedge^2)^{-1} \prod_{j=1}^{t-1} L(1-s-u_t-u_j,\Delta_1 \times \Delta_j)^{-1}.
\]
We have the similar manner that $\widetilde{P}(q^{-s})$ has its zeroes in the half plane $\mathrm{Re}(s) \geq 1-u_1-u_t$. This implies that $P(q^{-s})$ and $\widetilde{P}(q^{-s})$ have no common zeros, but are equal up to a unit in $\mathbb{C}[q^{\pm s}]$. As $L(s,\pi,\wedge^2)$ and $L(1-s,\pi^{\iota},\wedge^2)$ are both normalized Euler factor, we have $P=\widetilde{P}\equiv 1$.
\end{proof}

As a consequence of this Theorem, we have the following.

\begin{corollary}
\label{generic-multi}
Let $\pi=\mathrm{Ind}(\Delta_1 \otimes \dotsm \otimes \Delta_t)$ be an irreducible generic representation of $GL_m$, so each $\Delta_i$ is quasi-square-integrable and no two of segments $\Delta_1, \dotsm, \Delta_t$ are linked. Then
\[
  L(s,\pi,\wedge^2)=\prod_{k=1}^t L(s,\Delta_k,\wedge^2) \prod_{1\leq i < j \leq t} L(s,\Delta_i \times \Delta_j).
\]
\end{corollary}

\begin{proof}
Since $\pi$ is irreducible and generic, two of segments $\Delta_1, \dotsm, \Delta_t$ are unlinked, and so the quasi-square-integrable representations $\Delta_1, \dotsm, \Delta_t$ are rearranged to be in Langlands order without changing $\pi$. Then the result is just a restatement of the Theorem above.
\end{proof}

We pass to the case of a tempered representation, which is defined in \cite[Section 8]{JaPSSh83}. A tempered representation is of the form 
 \[
  \pi=\mathrm{Ind}(\Delta_1 \otimes \dotsm \otimes \Delta_t)
  \]
 where $\Delta_i$ are irreducible square-integrable. As a tempered representation is automatically irreducible by \cite[Theorem 4.2]{Ze80} and generic, we have the following Corollary.
 
 \begin{corollary} Let $\pi=\mathrm{Ind}(\Delta_1 \otimes \dotsm \otimes \Delta_t)$ be a tempered representation of $GL_m$. Then we obtain
  \[
   L(s,\pi,\wedge^2)=\prod_{1 \leq k \leq t} L(s,\Delta_k,\wedge^2) \prod_{1 \leq i < j \leq t} L(s, \Delta_i \times \Delta_j).
 \]
 \end{corollary}

Let $B_m$ be the Borel subgroup of $GL_m$. Take $\psi : F \rightarrow \mathbb{C}$ our additive character to be unramified and non-trivial, so $\psi(\mathcal{O})=1$ but $\psi(\varpi) \neq 1$. Suppose that $\pi$ is irreducible generic and unramified, that is, it has vectors fixed under $K_m$. Then $\pi \simeq \text{Ind}_{B_m}^{GL_m}(\mu_1 \otimes \dotsm \otimes \mu_m)$ is a full induced representation from the Borel subgroup $B_m$ of unramified characters $\mu_i$ of $F^{\times}$, that is, each $\mu_i$ is invariant under the maximal compact subgroup $\mathcal{O}^{\times} \subset F^{\times}$. Applying the formalism of the exterior square $L$-function from the Corollary \ref{generic-multi}, we obtain 
 \[
  L(s,\pi,\wedge^2)=\prod_{k=1}^m L(s,\mu_k,\wedge^2) \prod_{1\leq i < j \leq m} L(s,\mu_i \times \mu_j).
\]
As $L(s,\mu_k,\wedge^2) \equiv 1$ for all $k$ from Proposition \ref{GL_1}, we arrives at
\[
 L(s,\pi,\wedge^2)=\prod_{1\leq i < j \leq m} L(s,\mu_i \times \mu_j)=\prod_{1\leq i < j \leq m} (1-\mu_i(\varpi)\mu_j(\varpi)q^{-s})^{-1}.
\]

\begin{corollary}
Let $\pi$ be an irreducible generic and unramified representation of $GL_m$. Suppose that $\pi=\mathrm{Ind}_{B_m}^{GL_m}(\mu_1 \otimes \dotsm \otimes \mu_m)$ is a full induced representation from unramified characters $\mu_i$ of $F^{\times}$. Then
\[
 L(s,\pi,\wedge^2)=\prod_{1\leq i < j \leq m} L(s,\mu_i \times \mu_j)=\prod_{1\leq i < j \leq m} (1-\mu_i(\varpi)\mu_j(\varpi)q^{-s})^{-1}.
\]
\end{corollary}

\subsection{Shalika functional for a normalized parabolically induced representation}

Matringe establishes the complete criteria that an irreducible generic representation $\pi$ admits a Shalika functional in \cite{Ma}. In this section, as the author in \cite{Ma} insinuates, we relate the existence of a Shalika functional on $\pi$ with the occurrence of poles of exterior square $L$-functions or Rankin-Selberg convolution $L$-functions. We then present how one can deduce inductivity relation of the exterior square $L$-functions in Theorem \ref{general-prod} without using Hartogs' Theorem. We introduce the main result of Corollary 1.1 in \cite{Ma}. 

\begin{theorem}[Matringe]
\label{char of Shalika functional}
Let $\pi=\mathrm{Ind}(\Delta_1 \otimes \dotsm \otimes \Delta_t)$ be an irreducible generic representation of $GL_{2n}$ where each $\Delta_i$ is an irreducible quasi-square integrable representation of $GL_{n_i}$ and $2n=\sum n_i$. Then $\pi$ admits a Shalika functional if and only if
\[
  \pi \simeq \mathrm{Ind}((\Delta_1 \otimes \widetilde{\Delta}_1) \otimes \dotsm \otimes (\Delta_s \otimes \widetilde{\Delta}_s) \otimes \Delta_{s+1} \otimes \dotsm \otimes \Delta_t ), \quad 0 \leq s \leq t,
\]
where $\Delta_i$ admits a Shalika functional and each $n_i$ is even for all $i > s$.
\end{theorem}

Now we provide a consequence of Theorem \ref{char of Shalika functional} in terms of $L$-functions which is analogous to Proposition 4.20 in \cite{Ma15}.

 \begin{proposition}
 \label{l-function-shalika}
 Let $\pi=\mathrm{Ind}(\Delta_1 \otimes \dotsm \otimes \Delta_t)$ be an irreducible generic representation of $GL_{2n}$ where each $\Delta_i$ is an irreducible quasi-square-integrable representation of $GL_{n_i}$ with $2n=\sum n_i$ and $t \geq 2$. Suppose that $L_{ex}(s,\pi,\wedge^2)$ has a pole at $s=s_0$. Then we are in the one of the following :
 \begin{enumerate}
\item[$(\mathrm{1})$] There are $(i,j) \in \{1,\dotsm,t\}$ with $i \neq j$ such that $n_i$ and $n_j$ are even, and exceptional $L$-functions $L_{ex}(s,\Delta_i,\wedge^2)$ and $L_{ex}(s,\Delta_j,\wedge^2)$ have $s=s_0$ as a common pole.
\item[$(\mathrm{2})$] There are $(i,j,k,\ell) \in \{1,\dotsm,t\}$ with $\{ i, j \} \neq \{ k,\ell \}$ such that $L_{ex}(s,\Delta_i \times \Delta_j)$ and $L_{ex}(s,\Delta_k \times \Delta_{\ell})$ have $s=s_0$ as a common pole.
\item[$(\mathrm{3})$] There are $(i,j,k) \in \{1,\dotsm,t\}$ with $i \neq j$ such that $n_k$ is even and $L_{ex}(s,\Delta_i \times \Delta_j)$ and $L_{ex}(s,\Delta_k,\wedge^2)$ have $s=s_0$ as a common pole.
 \end{enumerate}
 \end{proposition}

 \begin{proof}
 
 Suppose that $L_{ex}(s,\pi,\wedge^2)$ has a pole at $s=s_0$. 
 As $L_{ex}(s,\pi,\wedge^2)=L_{ex}(s-s_0,\pi\nu^{\frac{s_0}{2}},\wedge^2)$, $\pi\nu^{\frac{s_0}{2}}$ admits a non-trivial Shalika functional from Section 3.1. 
 we know from Theorem \ref{char of Shalika functional} that $\pi\nu^{\frac{s_0}{2}}$ is isomorphic to
 \[
 \begin{split}
   & \mathrm{Ind}((\Delta_1\nu^{\frac{s_0}{2}} \otimes (\Delta_1\nu^{\frac{s_0}{2}})^{\sim}) \otimes \dotsm \otimes (\Delta_s\nu^{\frac{s_0}{2}} \otimes (\Delta_s\nu^{\frac{s_0}{2}})^{\sim}) \otimes \Delta_{s+1}\nu^{\frac{s_0}{2}} \otimes \dotsm \otimes\Delta_t\nu^{\frac{s_0}{2}}), \\
   &\phantom{********************************************} 0 \leq s \leq t,
     \end{split}
 \] 
where $\Delta_i\nu^{\frac{s_0}{2}} $ admits a Shalika functional and each $n_i$ is even for all $i > s$. Hence $(\Delta_i\nu^{\frac{s_0}{2}})^{\sim} \simeq \Delta_j\nu^{\frac{s_0}{2}}$ with $i \neq j$ or 
 $\Delta_k\nu^{\frac{s_0}{2}}$ admits a Shalika functional with $n_k$ an even number.
 \par
 According to Section 6.2, $(\Delta_i\nu^{\frac{s_0}{2}})^{\sim} \simeq \Delta_j\nu^{\frac{s_0}{2}}$ or equivalently $\widetilde{\Delta}_i \simeq \Delta_j\nu^{s_0}$ if and only if $L_{ex}(s,\Delta_i \times \Delta_j)$ has a pole at $s=s_0$.

 \par
  If $\Delta_k\nu^{\frac{s_0}{2}}$ has the Shalika functional, the space $\mathrm{Hom}_{S_{n_k}}(\Delta_k\nu^{\frac{s_0}{2}},\Theta)$ is nontrivial and that the central character $\omega_{\Delta_k\nu^{\frac{s_0}{2}}}$ of $\Delta_k\nu^{\frac{s_0}{2}}$ is trivial. Since $\Delta_j\nu^{\frac{s_0}{2}}$ is the irreducible square-integrable representation, we obtain from Proposition \ref{excep-Shalika} that 
$L_{ex}(s,\Delta_k\nu^{\frac{s_0}{2}},\wedge^2)$ has a pole at $s=0$ or equivalently $L_{ex}(s,\Delta_k,\wedge^2)$ has a pole at $s=s_0$. Therefore $s=s_0$ is a common pole for either of three cases in Proposition \ref{l-function-shalika}.
 \end{proof}

 We now explains that Theorem \ref{char of Shalika functional} of Matringe essentially provides the same result as Theorem \ref{general-prod}.
Let $\pi=\mathrm{Ind}(\Delta_1 \otimes \dotsm \otimes \Delta_t)$ be a representation of $GL_m$ of Whittaker type. We let $u=(u_1,\dotsm,u_t) \in \mathcal{D}_{\pi}$ be in general position and $\pi_u=\mathrm{Ind}(\Delta_1\nu^{u_1}\otimes \dotsm \otimes \Delta_t\nu^{u_t})$ the deformed representation of $\pi$ on $GL_m$. Let us now take $\Delta_i$ to be associated to the segment $[\rho_i,\dotsm,\rho_i\nu^{\ell_i-1}]$ with $\rho_i$ an irreducible supercuspidal representation of $GL_{r_i}$ $m_i=r_i\ell_i$ and $m=\sum r_i\ell_i$. By Theorem \ref{prod-L} and Theorem \ref{prod-L-odd}, we have
 \[
   L(s,\pi_u,\wedge^2)^{-1}=l.c.m. \{ L_{ex}(s,\pi_u^{(a_1r_1,\dotsm,a_tr_t)},\wedge^2)^{-1} \}
 \]
where $0 \leq a_i \leq \ell_i$, $m-\sum a_ir_i$ is an even number, and the least common multiple is taken in terms of divisibility in $\mathbb{C}[q^{\pm s}]$. Suppose that $L_{ex}(s,\pi_u^{(a_1r_1,\dotsm,a_tr_t)},\wedge^2)$ has a pole at $s=s_0$.
\par

If the number of indices $i$ such that $r_i \neq \ell_i$ is more than $3$, we know from Proposition \ref{l-function-shalika} that our $L$-functions satisfy one of the following conditions:
\begin{enumerate}
\item[$(\mathrm{1})$] There are $(i,j) \in \{1,\dotsm,t\}$ with $i \neq j$ such that $m_i-a_ir_i$ and $m_j-a_jr_j$ are even, and $L(s,\Delta^{(a_ir_i)}_i\nu^{u_i},\wedge^2)$ and $L(s,\Delta^{(a_jr_j)}_j\nu^{u_j},\wedge^2)$ have $s=s_0$ as a common pole,
\item[$(\mathrm{2})$] There are $(i,j,k,\ell) \in \{1,\dotsm,t\}$ with $\{ i, j \} \neq \{ k,\ell \}$ such that $L(s,\Delta_i^{(a_ir_i)}\nu^{u_i} \times \Delta_j^{(a_jr_j)}\nu^{u_j})$ and $L(s,\Delta_k^{(a_kr_k)}\nu^{u_k} \times \Delta^{(a_{\ell}r_{\ell})}_{\ell}\nu^{u_{\ell}})$ have $s=s_0$ as a common pole,
\item[$(\mathrm{3})$] There are $(i,j,k) \in \{1,\dotsm,t\}$ with $i \neq j$ such that $m_k-a_kr_k$ is even and two $L$-functions $L(s,\Delta_i^{(a_ir_i)}\nu^{u_i} \times \Delta_j^{(a_jr_j)}\nu^{u_j})$ and $L(s,\Delta_k^{(a_kr_k)}\nu^{u_k},\wedge^2)$ have $s=s_0$ as a common pole,
 \end{enumerate}
However condition $3$, $4$, and $5$ of general position guarantee that the above situations cannot happen as $u$ is in general position because exceptional poles are poles of original $L$-functions. 

\par 

If there exists exactly one pair $(i,j)$ of indices $i \neq j$ such that $r_i \neq \ell_i$ and $r_j \neq \ell_j$, according to Corollary \ref{two blocks}, $L_{ex}(s,\mathrm{Ind}(\Delta_i^{(a_ir_i)}\nu^{u_i}\otimes \Delta_j^{(a_jr_j)}\nu^{u_j} ),\wedge^2)=L_{ex}(s,\Delta_i^{(a_ir_i)}\nu^{u_i}\times \Delta_j^{(a_jr_j)}\nu^{u_j})$. If $i$ is the only index such that $r_i \neq \ell_i$, we have $L_{ex}(s,\Delta_i^{(a_ir_i)}\nu^{u_i},\wedge^2)$.
\par

 Putting all together, for $i < j$  $L_{ex}(s,\pi_u^{(a_1r_1,\dotsm,a_tr_t)},\wedge^2)$ is equal to $L_{ex}(s,\Delta_i^{(a_ir_i)}\nu^{u_i}\times \Delta_i^{(a_jr_j)}\nu^{u_j}))$, or $L_{ex}(s,\Delta_i^{(a_ir_i)}\nu^{u_i},\wedge^2)$. In the proof of Theorem \ref{general-prod} we employ Hartogs' Theorem rather than Theorem \ref{char of Shalika functional}. Using Hartogs' Theorem is more the original sprit of the method in \cite{Cog94} or \cite{CoPe}. Following the rest of the proof in Theorem \ref{general-prod}, we finally arrive at
 \[
 \begin{split}
  L(s,\pi_u,\wedge^2)&=\prod_{1 \leq k \leq t} L(s,\Delta_k\nu^{u_k},\wedge^2) \prod_{1 \leq i < j \leq t} L(s, \Delta_i\nu^{u_i} \times \Delta_j\nu^{u_j}) \\
   &=\prod_{1 \leq k \leq t} L(s+2u_k,\Delta_k,\wedge^2) \prod_{1 \leq i < j \leq t} L(s+u_i+u_j, \Delta_i \times \Delta_j).
 \end{split}
 \]

\subsection{Exceptional and symmetric square $L$-functions.}
 We introduce symmetric square $L$-functions defined by the ratio of Rankin-Selberg $L$-functions for $GL_m$ by exterior square $L$-functions for $GL_m$. We first recall results from \cite{CoPe,JaPSSh79} about $L$-function for Rankin-Selberg convolution. Let $\pi$ be an irreducible admissible representation of $GL_m$ and $\Pi=\mathrm{Ind}(\Delta_1\nu^{u_1} \otimes \dotsm \otimes \Delta_t\nu^{u_t} )$ the Langlands induced representation such that $\pi$  is the unique irreducible quotient of $\Pi$. It is proven in \cite{CoPe,JaPSSh79} that
 \[
   L(s,\pi \times \pi)=\prod_{1 \leq i,j\leq t} L(s+u_i+u_j,\Delta_i \times \Delta_j).
 \]
 Applying the formalism of the exterior square $L$-functions in Theorem \ref{langlands-multi}, we obtain
 \[
   \tag{8.6}
   \label{RSprod}
    \frac{L(s,\pi \times \pi)}{L(s,\pi,\wedge^2)}=\prod_{1 \leq k \leq t} \frac{L(s+2u_k,\Delta_k \times \Delta_k)}{L(s+2u_k,\Delta_k,\wedge^2)} \prod_{1 \leq j < i \leq t} L(s, \Delta_i \times \Delta_j).
 \]
 
Since $L(s+2u_i,\Delta_i,\wedge^2)^{-1}$ divides $L(s+2u_i,\Delta_i \times \Delta_i)^{-1}$ in $\mathbb{C}[q^{\pm s}]$ by Proposition \ref{even-square} and \ref{odd-square}, \eqref{RSprod} insures that $L(s,\pi,\wedge^2)^{-1}$ divides $L(s,\pi \times \pi)^{-1}$ in $\mathbb{C}[q^{\pm s}]$. Motivated by this divisibility, we define symmetric square $L$-functions as the ratio of $L(s,\pi \times \pi)$ by $L(s,\pi,\wedge^2)$.

\begin{definition}
Let $\pi$ be an irreducible admissible representation of $GL_m$. We let $\Pi=\mathrm{Ind}(\Delta_1\nu^{u_1} \otimes \dotsm \otimes \Delta_t\nu^{u_t} )$ be the Langlands induced representation with each $\Delta_i$ an irreducible square integrable representation, the $u_i$ real and ordered so that $u_1 \geq \dotsm \geq u_t$ such that $\pi$  is the unique irreducible quotient of $\Pi$. We denote by $L(s,\pi,\mathrm{Sym}^2)$ the symmetric square $L$-function of $\pi$ defined by
\[
 L(s,\pi,\mathrm{Sym}^2)= \frac{L(s,\pi \times \pi)}{L(s,\pi,\wedge^2)}.
\]
\end{definition}

We emphasize that this definition of $L$-function $L(s,\pi,\mathrm{Sym}^2)$ agrees with Artin's symmetric square $L$-function on the arithmetic side. We will be concerned with the proof of this result in the next section. We know that $L(s,\Delta_i \times \Delta_j)=L(s,\Delta_j \times \Delta_i)$ by the result of \cite{JaPSSh79}. From \eqref{RSprod}, we have the following analogous result of Theorem \ref{langlands-multi} for symmetric square $L$-function $L(s,\pi,\mathrm{Sym}^2)$.

\begin{proposition}
\label{symmetric-multi}
Let $\pi$ be an irreducible admissible representation of $GL_m$ such that $\pi$  is the unique irreducible quotient of representations of Langlands type $\Pi=\mathrm{Ind}(\Delta_1\nu^{u_1} \otimes \dotsm \otimes \Delta_t\nu^{u_t} )$. Then
\[
 L(s,\pi,\mathrm{Sym}^2)=\prod_{k=1}^t L(s+2u_k,\Delta_k,\mathrm{Sym}^2) \prod_{1\leq i < j \leq t} L(s+u_i+u_j,\Delta_i \times \Delta_j).
\]
\end{proposition}

Let $\pi$ be irreducible generic and unramified representation of $GL_m$. Then it is full induced representations from unramified characters of the Borel subgroup $B_m$ by \cite{Ze80}. Hence let us write $\pi \simeq \mathrm{Ind}_{B_m}^{GL_m}(\mu_1 \otimes \dotsm \otimes \mu_m)$ with $\mu_i$ unramified character of $F^{\times}$. Since $\pi$ is assumed generic, $\mu_1,\dotsm,\mu_m$ are arranged to be in the Langlands order without changing $\pi$. We know from previous Proposition \ref{symmetric-multi} that 
\[
  L(s,\pi,\mathrm{Sym}^2)=\prod_{k=1}^m L(s,\mu_k,\mathrm{Sym}^2) \prod_{1\leq i < j \leq m} L(s,\mu_i \times \mu_j).
\]
As we have seen $L(s,\mu_k,\wedge^2) \equiv 1$ for all $k$ from Proposition \ref{GL_1}, we can conclude that $L(s,\mu_k,\mathrm{Sym}^2)=L(s,\mu_k \times \mu_k)$. When we replace $L(s,\mu_k,\mathrm{Sym}^2)$ by $L(s,\mu_k \times \mu_k)$, we can obtain the following Corollary.

\begin{corollary}
Let $\pi$ be an irreducible generic and unramified representation of $GL_m$. Suppose that $\pi=\mathrm{Ind}_{B_m}^{GL_m}(\mu_1 \otimes \dotsm \otimes \mu_m)$ is a full induced representation from unramified characters $\mu_i$ of $F^{\times}$. Then
\[
 L(s,\pi,\wedge^2)=\prod_{1\leq i \leq j \leq m} L(s,\mu_i \times \mu_j)=\prod_{1\leq i \leq j \leq m} (1-\mu_i(\varpi)\mu_j(\varpi)q^{-s})^{-1}.
\]
\end{corollary}

Suppose that $\rho$ is an irreducible supercuspidal representation of $GL_r$. Then we can restate our characterization of $L$-function $L(s,\rho,\wedge^2)$ in Proposition \ref{supercusp2} in terms of symmetric square $L$-function $L(s,\rho,\mathrm{Sym}^2)$.

\begin{proposition}
\label{symmetric}
Let $\rho$ be an irreducible supercuspidal representation of $GL_r$. 
\begin{enumerate}
\item[$(\mathrm{i})$] Suppose that $r=2n$ is even. Then we obtain
\[
 L(s,\rho,\mathrm{Sym}^2)= \frac{L(s,\rho \times \rho)}{L(s,\rho,\wedge^2)}=\prod (1-\alpha q^{-s})^{-1},
\]
with the product over all $\alpha=q^{s_0}$ such that $\widetilde{\rho} \simeq \rho \nu^{s_0}$ and $\mathrm{Hom}_{S_{2n}}(\rho\nu^{\frac{s_0}{2}},\Theta)=0$. 
\item[$(\mathrm{ii})$] Suppose that $r=2n+1$ is odd. Then we have
\[
 L(s,\rho,\mathrm{Sym}^2)=L(s,\rho \times \rho).
\]
\end{enumerate}
\end{proposition}

Let $\rho$ be an irreducible unitary supercuspidal representation of $GL_r$. We denote by $\overline{x}$ the complex conjugate of $x \in \mathbb{C}$. Jacquet, Piatetski-Shapiro, and Shalika establish
\[
  L(s,\widetilde{\rho} \times \widetilde{\rho})=\overline{L(\overline{s},\rho \times \rho)}
\]  
in the proof of Theorem 8.2 of \cite{JaPSSh83} (page 446). The following is the analogue of this result for our $L$-functions $L(s,\rho,\wedge^2)$.

\begin{lemma}
 Let $\rho$ be an irreducible unitary supercuspidal representation of $GL_r$. Then we have
 \[
  L(s,\widetilde{\rho},\wedge^2)=\overline{L(\overline{s},\rho,\wedge^2 )} \quad \quad \text{and} \quad \quad L(s,\widetilde{\rho} \times \widetilde{\rho})=\overline{L(\overline{s},\rho \times \rho)}.
 \]
\end{lemma}

\begin{proof}
We start with adopting the idea in the proof of \cite[Proposition 2.2]{Chai17} or \cite[Section 1]{Sh84}. Let $\overline{\mathcal{W}(\rho,\psi)}$ be the complex conjugate of Whittaker model $\mathcal{W}(\rho,\psi)$ defined by
\[
  \overline{\mathcal{W}(\rho,\psi)}=\{ \overline{W}(g)\; | \; W \in \mathcal{W}(\rho,\psi), g \in GL_r \}.
\]
 Let $\overline{\rho}$ be right translation of $GL_r$ on the space $\overline{\mathcal{W}(\rho,\psi)}$. Then  $\overline{\rho}$ is an irreducible unitary supercuspidal representation of $GL_r$. As $\overline{W}(ng)=\psi^{-1}(n)\overline{W}(g)$ for all $n \in N_r$ and all $g \in GL_r$, the space of functions $\overline{\mathcal{W}(\rho,\psi)}$ is the Whittaker model $\mathcal{W}(\overline{\rho},\psi^{-1})$. Since $\rho$ is unitary, $\mathcal{W}(\rho,\psi)$ has a $GL_r$-invariant inner product $\langle \;\;, \; \rangle$, that is, $\langle \rho(g)W_1, \rho(g)W_2 \rangle=\langle W_1,W_2 \rangle$ for all $g \in GL_r$ and all $W_1,W_2 \in \mathcal{W}(\rho,\psi)$. The $GL_r$-intertwining linear map $T : \overline{W} \mapsto \ell_{W}$ for $\overline{W} \in \overline{\mathcal{W}(\rho,\psi)}$ induces an isomorphism from $\overline{\mathcal{W}(\rho,\psi)}$ to the smooth dual space $(\widetilde{\rho},V_{\widetilde{\rho}})$, where a smooth linear form $\ell_W$ on $\mathcal{W}(\rho,\psi) \simeq (\rho,V_{\rho})$ is given by $\ell_W : W' \mapsto \langle W', W \rangle$ for $W' \in \mathcal{W}(\rho,\psi)$. As we have $(\widetilde{\rho},V_{\widetilde{\rho}}) \simeq \mathcal{W}(\widetilde{\rho},\psi^{-1})$, we can identify the Whittaker model $\mathcal{W}(\widetilde{\rho},\psi^{-1})$ of the contragredient representation $\widetilde{\rho}$ with the Whittaker model $\mathcal{W}(\overline{\rho},\psi^{-1})$.

 Suppose that $r=2n$ is even. We recall from Theorem \ref{supercusp2} that $L(s,\rho,\wedge^2)=\prod_j(1-\alpha_jq^{-s})^{-1}$, where the product runs over all $\alpha_j=q^{s_0}$ such that $\omega_{\rho}\nu^{ns_0}=1$ and 
 \[
\tag{8.7}
 \label{residueint}
   \int_{Z_nN_n \backslash GL_n} \int_{\mathcal{N}_n \backslash \mathcal{M}_n} W \left( \sigma_{2n} \begin{pmatrix} I_n & X\\ & I_n \end{pmatrix}
   \begin{pmatrix} g& \\ & g \end{pmatrix} \right) \psi^{-1}(\mathrm{Tr}X) |\mathrm{det} (g)|^{s_0} dX dg \neq 0
 \]
 for some $W \in \mathcal{W}(\rho,\psi)$. Since $\rho$ is unitary and supercuspidal, the central character $\omega_{\rho}$ is unitary. Then the poles are purely imaginary, because $\nu^{ns_0}=\omega_{\rho}^{-1}$ is also a unitary character. Hence $\overline{s_0}=-s_0$. After applying complex conjugation for the central character $\omega_{\rho}\nu^{ns_0}$ and the residual integral in \eqref{residueint}, the previous condition \eqref{residueint} is equivalent to stating that $ \omega_{\overline{\rho}}\nu^{-ns_0}=1$ and
 \[
    \int_{Z_nN_n \backslash GL_n} \int_{\mathcal{N}_n \backslash \mathcal{M}_n} \overline{W} \left( \sigma_{2n} \begin{pmatrix} I_n & X\\ & I_n \end{pmatrix}
   \begin{pmatrix} g& \\ & g \end{pmatrix} \right) \psi(\mathrm{Tr}X) |\mathrm{det} (g)|^{-s_0} dX dg \neq 0
 \]
 for some $\overline{W} \in \mathcal{W}(\overline{\rho},\psi^{-1})$. Since these are the only possible poles of $L(s,\overline{\rho},\wedge^2)$, we take the product over all such $\alpha_j^{-1}=q^{-s_0}$, which establishes 
 \[
   L(s,\widetilde{\rho},\wedge^2)=L(s,\overline{\rho},\wedge^2)=\prod_j(1-\alpha^{-1}_jq^{-s})^{-1}=\overline{L(\overline{s},\rho,\wedge^2 )}.
 \]
 If $r=2n+1$ is odd, then by Theorem \ref{supercusp} we obtain $L(s,\widetilde{\rho},\wedge^2)=\overline{L(\overline{s},\rho,\wedge^2 )}=1$.
 
\end{proof}

 As a consequence of this Lemma, we have the following result.
 
 \begin{lemma}
 \label{unitary-l}
  Let $\rho$ be an irreducible unitary supercuspidal representation of $GL_r$. Up to multiplication by a monomial in $q^{-s}$, $L(-s,\rho,\wedge^2)$ and $L(s,\widetilde{\rho},\wedge^2)$ are equal and similarly for the two functions $L(-s,\rho \times \rho)$ and $L(s,\widetilde{\rho}\times \widetilde{\rho})$. 
 \end{lemma}

 \begin{proof}
 Suppose that $L(s,\rho,\wedge^2)=\prod_j(1-\alpha_jq^{-s})^{-1}$. Since $\rho$ is unitary and supercuspidal, $|\alpha_j|=1$ for all $j$. Then $L$-function 
 \[
 \begin{split}
   L(s, \widetilde{\rho},\wedge^2)&=\prod_j(1-\alpha_j^{-1}q^{-s})^{-1}=\prod_j (-\alpha_j^{-1}q^{-s})(1-\alpha_jq^{s})^{-1} \\
   &=\left[ \prod_j (-\alpha_j^{-1}q^{-s}) \right] L(-s,\rho,\wedge^2)
 \end{split}
 \]
 is equal to $L(-s,\rho,\wedge^2)$ up to a unit $\prod_j (-\alpha_j^{-1}q^{s})$ in $\mathbb{C}[q^{\pm s}]$. The case of $L$-function $L(-s,\rho \times \rho)$ can be treated in the same way.

 \end{proof}

We are now in the position to complete the proof of the main Theorem in this section. 

\begin{theorem}
\label{square-multip}
 Let $\Delta$ be an irreducible square integrable representation of $GL_{\ell r}$ with the segment $\Delta=[\rho\nu^{-\frac{\ell-1}{2}},\dotsm,\rho\nu^{\frac{\ell-1}{2}}]$ and $\rho$ an irreducible unitary supercuspidal representation of $GL_r$.
 \begin{enumerate}
\item[$(\mathrm{i})$] Suppose that $\ell$ is even. Then we have
\[
\begin{split}
& L(s,\Delta,\wedge^2)=\prod_{i=0}^{\frac{\ell}{2}-1} L(s+2i+1,\rho,\wedge^2) L(s+2i,\rho,\mathrm{Sym}^2),\\
 &L(s,\Delta,\mathrm{Sym}^2)=\prod_{i=0}^{\frac{\ell}{2}-1} L(s+2i,\rho,\wedge^2) L(s+2i+1,\rho,\mathrm{Sym}^2).
\end{split}
\]
\item[$(\mathrm{ii})$] Suppose that $\ell$ is odd. Then we obtain
\[
\begin{split}
  &L(s,\Delta,\wedge^2)=\prod_{i=0}^{\frac{\ell-1}{2}}L(s+2i,\rho,\wedge^2) \prod_{i=1}^{\frac{\ell-1}{2}}L(s+2i-1,\rho,\mathrm{Sym}^2),\\
 &L(s,\Delta,\mathrm{Sym}^2)=\prod_{i=1}^{\frac{\ell-1}{2}}L(s+2i-1,\rho,\wedge^2) \prod_{i=0}^{\frac{\ell-1}{2}}L(s+2i,\rho,\mathrm{Sym}^2).
\end{split}
\]
\end{enumerate}
\end{theorem}

\begin{proof}
 We start with the Whittaker model of $\Delta$. As in 9.1 of Zelevinsky \cite{Ze80}, $\Delta$ is the unique irreducible quotient of the normalized induced representation $\mathrm{Ind}(\rho\nu^{-\frac{\ell-1}{2}} \otimes \rho\nu^{-\frac{\ell-3}{2}} \otimes \dotsm \otimes \rho\nu^{\frac{\ell-1}{2}})$. This induces a short exact sequence of smooth representation of $GL_{\ell r}$, namely
 \[
   0 \rightarrow V_S \rightarrow \mathrm{Ind}(\rho\nu^{-\frac{\ell-1}{2}} \otimes \rho\nu^{-\frac{\ell-3}{2}} \otimes \dotsm \otimes \rho\nu^{\frac{\ell-1}{2}}) \rightarrow \Delta \rightarrow 0
 \]
where $V_S$ denote the kernel. Since $\Delta$ is generic, there is essentially unique Whittaker functional on $\Delta$. This will then induce a non-zero Whittaker functional on $ \mathrm{Ind}(\rho \otimes \rho\nu \otimes \dotsm \otimes \rho\nu^{\ell-1})$. Since this representation also has a unique Whittaker functional, this must be it and we can conclude that 
\[
 \mathcal{W}(\Delta,\psi)=\mathcal{W}(\mathrm{Ind}(\rho\nu^{-\frac{\ell-1}{2}} \otimes \rho\nu^{-\frac{\ell-3}{2}} \otimes \dotsm \otimes \rho\nu^{\frac{\ell-1}{2}}),\psi).
\]
 \par

The following argument is extracted from the proof of the main Proposition 8.1 in \cite{Sh92}, taking into account the change from the multiplicativity of the $\gamma$-factor via Langland-Shahidi method to the weak version of the multiplicativity of the $\gamma$-factor in Corollary \ref{multiplicativity}. We employ the notation for two rational functions $P(q^{-s})$ and $Q(q^{-s})$ that $P \sim Q$ denotes that the ratio is a unit in $\mathbb{C}[q^{\pm s}]$, that is, a monomial factor $\alpha q^{-\beta s}$. Applying Corollary \ref{multiplicativity} to the Whittaker model $\mathcal{W}(\mathrm{Ind}(\rho\nu^{-\frac{\ell-1}{2}} \otimes  \dotsm \otimes \rho\nu^{\frac{\ell-1}{2}}),\psi)$ implies that
\[
  \gamma(s,\Delta,\wedge^2,\psi) \sim \prod_{i=0}^{\ell-1} \gamma(s,\rho\nu^{i+\frac{1-\ell}{2}},\wedge^2,\psi) \prod_{0 \leq i < j \leq \ell-1} \gamma(s,\rho\nu^{i+\frac{1-\ell}{2}}\times \rho\nu^{j+\frac{1-\ell}{2}},\psi).
\]
We may unwind the unramified twist to get
\[
 \gamma(s,\Delta,\wedge^2,\psi) \sim \prod_{i=0}^{\ell-1} \gamma(s+1-\ell+2i,\rho,\wedge^2,\psi) \prod_{0 \leq i < j \leq \ell-1} \gamma(s+1-\ell+i+j, \rho \times \rho,\psi).
\]
In terms of $L$-function, we restate it as
\[ 
 \gamma(s,\Delta,\wedge^2,\psi) \sim \prod_{i=0}^{\ell-1} \frac{L(-s+\ell-2i,\widetilde{\rho},\wedge^2)}{L(s+1-\ell+2i,\rho,\wedge^2)}
  \prod_{0 \leq i < j \leq \ell-1} \frac{L(-s+\ell-i-j, \widetilde{\rho} \times \widetilde{\rho})}{L(s+1-\ell+i+j, \rho \times \rho)}.
\]
After we apply Lemma \ref{unitary-l}, we find
\[
\tag{8.8}
\label{lratio}
 \gamma(s,\Delta,\wedge^2,\psi) \sim \prod_{i=0}^{\ell-1} \frac{L(s-\ell+2i,\rho,\wedge^2)}{L(s+1-\ell+2i,\rho,\wedge^2)}
  \prod_{0 \leq i < j \leq \ell-1} \frac{L(s-\ell+i+j, \rho \times \rho)}{L(s+1-\ell+i+j, \rho \times \rho)}.
\]
Then
\[
\tag{8.9}
\label{minyoung}
\begin{split}
& \prod_{0 \leq i < j \leq \ell-1} \frac{L(s-\ell+i+j, \rho \times \rho)}{L(s+1-\ell+i+j, \rho \times \rho)}=\prod_{0 \leq i < \ell-1} \frac{L(s-\ell+2i+1,\rho \times \rho)}{L(s+i,\rho \times \rho)} \\
&=\frac{\prod_{-\ell < i < \ell-1}L(s+i,\rho \times \rho)}{\prod_{0 \leq i < \ell-1} L(s+i,\rho \times \rho) \prod_{1\leq i \leq \ell-1} L(s-\ell+2i,\rho \times \rho)}=\prod_{i=1}^{\ell-1} 
\frac{L(s-\ell+i,\rho \times \rho)}{L(s-\ell+2i,\rho \times \rho)}.
 \end{split}
 \]
Inserting all the way to the right term in \eqref{minyoung} into \eqref{lratio}, we derive
\[
\begin{split}
 &\gamma(s,\Delta,\wedge^2,\psi) \\
 & \sim  \prod_{i=0}^{\ell-1} \frac{L(s-\ell+2i,\rho,\wedge^2)}{L(s+1-\ell+2i,\rho,\wedge^2)} \prod_{i=1}^{\ell-1} 
\frac{L(s-\ell+i,\rho,\wedge^2 )L(s-\ell+i,\rho,\mathrm{Sym}^2 )}{L(s-\ell+2i,\rho,\wedge^2)L(s-\ell+2i,\rho,\mathrm{Sym}^2)},
\end{split}
\]
where we have replaced the $L$-function $L(s,\rho \times \rho)$ by the product $L(s,\rho,\wedge^2)L(s,\rho,\mathrm{Sym}^2)$ by definition. 
\par
First assume $\ell$ is even. We do the case $\ell$ even, the case $\ell$ odd being similar. Let us now cancel common factors. Then our quotient simplifies to
\[
  \frac{L(1-s,\widetilde{\Delta},\wedge^2)}{L(s,\Delta,\wedge^2)}  \sim\ \gamma(s,\Delta,\wedge^2,\psi)
 \sim  \prod_{i=0}^{\frac{\ell}{2}-1}\hspace*{-1mm} \frac{L(s-\ell+2i,\rho,\wedge^2)}{L(s+2i+1,\rho,\wedge^2)} \prod_{i=0}^{\frac{\ell}{2}-1}\hspace*{-1mm} \frac{L(s-\ell+2i+1,\rho,\mathrm{Sym}^2)}{L(s+2i,\rho,\mathrm{Sym}^2)}
\]
where we have used the equality of $L(1-s,\widetilde{\Delta},\wedge^2)/L(s,\Delta,\wedge^2)$ and $\gamma(s,\Delta,\wedge^2,\psi)$ up to units in $\mathbb{C}[q^{\pm s}]$ according to the functional equation. Proposition \ref{even-square} and Proposition \ref{odd-square} imply that $L(1-s,\widetilde{\Delta},\wedge^2)^{-1}$ has zeros in the half plane $\mathrm{Re}(s) \geq 1$ while $L(s,\Delta,\wedge^2)^{-1}$ has its zeros contained in the region $\mathrm{Re}(s) \leq 0$. Since the half planes $\mathrm{Re}(s) \geq 1$ and $\mathrm{Re}(s) \leq 0$ are disjoint, they do not have common factors in $\mathbb{C}[q^{\pm s}]$. As $\rho$ is unitary, the poles of $L$-function $\prod_{i=0}^{\frac{\ell}{2}-1} L(s-\ell+2i,\rho,\wedge^2)   L(s-\ell+2i+1,\rho,\mathrm{Sym}^2)$ must lie on the line $\mathrm{Re}(s)=\ell-i$ for $i=-\ell+1,\dotsm,0$, while the poles of $L$-function $\prod_{i=0}^{\frac{\ell}{2}-1} L(s+2i+1,\rho,\wedge^2) L(s+2i,\rho,\mathrm{Sym}^2)$ will lie on the line $\mathrm{Re}(s)=-i$ for $i=-\ell+1,\dotsm,0$. Therefore they have no factors in common. We conclude that
\[
 L(s,\Delta,\wedge^2) \sim \prod_{i=0}^{\frac{\ell}{2}-1} L(s+2i+1,\rho,\wedge^2) L(s+2i,\rho,\mathrm{Sym}^2).
\]
Both sides of the above $L$-functions are exactly the inverse of the normalized polynomial and hence are indeed equal. For symmetric square $L$-function $L(s,\Delta,\mathrm{Sym}^2)$, we know from \cite{CoPe,JaPSSh83} that
\[
  L(s,\Delta \times \Delta)=\prod_{i=0}^{\ell-1}L(i+s,\rho \times \rho)=\prod_{i=0}^{\frac{\ell}{2}-1}L(s+2i+1,\rho \times \rho)L(s+2i,\rho \times \rho). 
  \]
  As symmetric square $L$-function is defined by the ratio of Rankin-Selberg $L$-functions by exterior square $L$-function, we obtain
  \[
  \begin{split}
    L(s,\Delta,\mathrm{Sym}^2)
    &=\prod_{i=0}^{\frac{\ell}{2}-1}\frac{L(s+2i,\rho \times \rho)}{L(s+2i,\rho,\mathrm{Sym}^2)} \frac{L(s+2i+1,\rho \times \rho)}{L(s+2i+1,\rho,\wedge^2)}\\
    &=\prod_{i=0}^{\frac{\ell}{2}-1} L(s+2i,\rho,\wedge^2) L(s+2i+1,\rho,\mathrm{Sym}^2).
  \end{split}
  \]

 \end{proof}

We remark that the product formula in Theorem \ref{square-multip} agrees with that of Langland-Shahidi exterior or symmetric square $L$-function for square integrable representation in \cite[Proposition 8.1]{Sh92}. On the arithmetic side, the formula is also consistent with the expression of Artin's exterior square $L$-function for the Langlands parameter of square integrable representation in \cite[Proposition 6.2]{Ma14}. We also note that the result easily follows from the agreement of exterior square $L$-function by means of integral representation and the Langland-Shahidi exterior square $L$-function for square integrable representation in Kewat and Raghunathan of \cite{KeRa12}, assuming that symmetric square $L$-functions are defined in our way. However our proof is completely local as opposed to the global argument in \cite{KeRa12}, using a global
exterior square $L$-function for an unitary cuspidal automorphic representation.

 \par
 
 Let $\Delta$ be an irreducible square integrable representation of $GL_{\ell r}$ with the segment $\Delta=[\rho\nu^{-\frac{\ell-1}{2}},\dotsm,\rho\nu^{\frac{\ell-1}{2}}]$ and $\rho$ an irreducible supercuspidal representation of $GL_r$. We wish to express the exceptional $L$-function $L_{ex}(s,\Delta,\wedge^2)$ in terms of $L$-functions of irreducible supercuspidal representations. 
 Theorem \ref{square-multip} leads quickly to the following Theorem which is originally conjectured in \cite{Cog94}.
  
 \begin{theorem}
 \label{excep-square}
Let $\Delta$ be an irreducible square integrable representation of $GL_{\ell r}$ with the segment $\Delta=[\rho \nu^{-\frac{\ell-1}{2}},\dotsm,\rho \nu^{\frac{\ell-1}{2}}]$ and $\rho$ an irreducible unitary supercuspidal representation of $GL_r$.
 \begin{enumerate}
\item[$(\mathrm{i})$] Suppose that $r$ is even. Then
\[
L_{ex}(s,[\rho \nu^{-\frac{\ell-1}{2}},  \dotsm, \rho \nu^{\frac{\ell-1}{2}}],\wedge^2)= \left\{  \begin{array}{l l} L(s,\rho,\wedge^2) &\quad \text{if $\ell$ is odd} \\
   L(s,\rho,\mathrm{Sym}^2)  &\quad \text{if $\ell$ is even.}
     \end{array} \right.
\]
\item[$(\mathrm{ii})$] Suppose that $r$ is odd. Then
\[
 L_{ex}(s,[\rho \nu^{-\frac{\ell-1}{2}},  \dotsm, \rho \nu^{\frac{\ell-1}{2}}],\wedge^2)=L(s,\rho,\mathrm{Sym}^2) \qquad \text{if $\ell$ is even.}
\]
 \end{enumerate}
\end{theorem}

 \begin{proof}
 For $\ell=1$, $L_{ex}(s,\rho,\wedge^2)=L(s,\rho,\wedge^2)$, which follows from Proposition \ref{supercusp}. We now begin with the simplest non-supercuspidal  square integrable representation $\Delta=[\rho \nu^{-\frac{1}{2}},\rho \nu^{\frac{1}{2}}]$ of $GL_{2r}$ with $\rho$ an irreducible unitary supercuspidal representation of $GL_r$ and $r$ an even number. The derivatives of $\Delta$ are $\Delta^{(0)}=\Delta$ and $\Delta^{(r)}=\rho \nu^{\frac{1}{2}}$. Then Proposition \ref{even-square} gives
 \[
 \begin{split}
   L(s,\Delta,\wedge^2)&=L_{ex}(s,[\rho \nu^{-\frac{1}{2}},\rho \nu^{\frac{1}{2}}],\wedge^2)L_{ex}(s,\rho \nu^{\frac{1}{2}},\wedge^2)\\
   &=L_{ex}(s,[\rho \nu^{-\frac{1}{2}},\rho \nu^{\frac{1}{2}}],\wedge^2)L(s+1,\rho,\wedge^2).
\end{split}
 \]
According to \ref{square-multip}, we know
\[
  L(s,\Delta,\wedge^2)=L(s,\rho,\mathrm{Sym}^2)L(s+1,\rho,\wedge^2).
\] 
 Therefore we can conclude that
 \[
  L_{ex}(s,[\rho \nu^{-\frac{1}{2}},\rho \nu^{\frac{1}{2}}],\wedge^2)=L(s,\rho,\mathrm{Sym}^2).
\]
 for the base case $\ell=2$.

  \par
 We are going to prove by induction on $\ell$. We assume that the following statement is satisfied for all integers $k$ with $1 \leq k \leq \ell$ 
 \[
L_{ex}(s,[\rho \nu^{-\frac{k-1}{2}},  \dotsm, \rho \nu^{\frac{k-1}{2}}],\wedge^2,s)= \left\{  \begin{array}{l l} L(s,\rho,\wedge^2) &\quad \text{if $k$ is odd} \\
   L(s,\rho,\mathrm{Sym}^2)  &\quad \text{if $k$ is even.}
     \end{array} \right.
\]
On the one hand, Theorem \ref{square-multip} gives the formula
\[
\tag{8.10}
\label{prod-square1}
  L(s,[\rho \nu^{-\frac{\ell}{2}},  \dotsm, \rho \nu^{\frac{\ell}{2}}],\wedge^2)=
  L(s,\rho,\wedge^2)\prod_{i=1}^{\frac{\ell}{2}}L(s+2i,\rho,\wedge^2)\prod_{i=1}^{\frac{\ell}{2}}L(s+2i-1,\rho ,\mathrm{Sym}^2).
  \]
if $\ell$ is even or
\[
\tag{8.11}
\label{prod-square2}
 L(s,[\rho \nu^{-\frac{\ell}{2}},  \dotsm, \rho \nu^{\frac{\ell}{2}}],\wedge^2)=
 L(s,\rho,\mathrm{Sym}^2) \prod_{i=0}^{\frac{\ell-1}{2}} L(s+2i+1,\rho,\wedge^2) \prod_{i=1}^{\frac{\ell-1}{2}}L(s+2i,\rho,\mathrm{Sym}^2)
 \]
 if $\ell$ is odd. On the other hand, if we combine the first part of Proposition \ref{even-square} with our induction hypothesis, we can obtain the following equality of $L$-functions
\[
\tag{8.12}
\label{prod-square3}
\begin{split}
  &L(s,[\rho \nu^{-\frac{\ell}{2}},  \dotsm, \rho \nu^{\frac{\ell}{2}}],\wedge^2)\\
  &=L_{ex}(s,[\rho \nu^{-\frac{\ell}{2}},  \dotsm, \rho \nu^{\frac{\ell}{2}}],\wedge^2) \\
  &\phantom{********} \times \prod_{i=1}^{\frac{\ell}{2}}L_{ex}(s,[\rho \nu^{2i-\frac{\ell}{2}},  \dotsm, \rho \nu^{\frac{\ell}{2}}],\wedge^2)L_{ex}(s,[\rho \nu^{2i-1-\frac{\ell}{2}},  \dotsm, \rho \nu^{\frac{\ell}{2}}],\wedge^2) \\
  &=L_{ex}(s,[\rho \nu^{-\frac{\ell}{2}},  \dotsm, \rho \nu^{\frac{\ell}{2}}],\wedge^2)\prod_{i=1}^{\frac{\ell}{2}}L(s+2i,\rho, \wedge^2)\prod_{i=1}^{\frac{\ell}{2}}L(s+2i-1,\rho,\mathrm{Sym}^2)
  \end{split}
\]
if $\ell$ is even or
\[
\tag{8.13}
\label{prod-square4}
\begin{split}
  &L(s,[\rho \nu^{-\frac{\ell}{2}},  \dotsm, \rho \nu^{\frac{\ell}{2}}],\wedge^2)\\
  &=L_{ex}(s,[\rho \nu^{-\frac{\ell}{2}},  \dotsm, \rho \nu^{\frac{\ell}{2}}],\wedge^2) \\
  &\phantom{*****}\times \prod_{i=1}^{\frac{\ell+1}{2}}L_{ex}(s,[\rho \nu^{2i-1-\frac{\ell}{2}},  \dotsm, \rho \nu^{\frac{\ell}{2}}],\wedge^2)\prod_{i=1}^{\frac{\ell-1}{2}}L_{ex}(s,[\rho \nu^{2i-\frac{\ell}{2}},  \dotsm, \rho \nu^{\frac{\ell}{2}}],\wedge^2) \\
  &=L_{ex}(s,[\rho \nu^{-\frac{\ell}{2}},  \dotsm, \rho \nu^{\frac{\ell}{2}}],\wedge^2)\prod_{i=0}^{\frac{\ell-1}{2}}L(s+2i+1,\rho, \wedge^2)\prod_{i=1}^{\frac{\ell-1}{2}}L(s+2i,\rho,\mathrm{Sym}^2)
  \end{split}
\]
if $\ell$ is odd. We conclude the desired equality from \eqref{prod-square1} and \eqref{prod-square3}, or  \eqref{prod-square2} and \eqref{prod-square4} that
 \[
L_{ex}(s,[\rho \nu^{-\frac{\ell}{2}},  \dotsm, \rho \nu^{\frac{\ell}{2}}],\wedge^2)= \left\{  \begin{array}{l l} L(s,\rho,\wedge^2) &\quad \text{if $\ell$ is even} \\
   L(s,\rho,\mathrm{Sym}^2)  &\quad \text{if $\ell$ is odd.}
     \end{array} \right.
\]

\par
Now we assume that $r$ is odd, $\ell$ is even, and $\Delta$ is an irreducible square integrable representation of $GL_{\ell r}$ with the segment $\Delta=[\rho\nu^{-\frac{\ell-1}{2}},\dotsm,\rho\nu^{\frac{\ell-1}{2}}]$. Then the formula in Theorem \ref{square-multip} is reduced to 
\[
\tag{8.14}
\label{odd-symprod}
  L(s,\Delta,\wedge^2)=\prod_{i=0}^{\frac{\ell}{2}-1} L(s+2i+1,\rho,\wedge^2) L(s+2i,\rho,\mathrm{Sym}^2)=\prod_{i=0}^{\frac{\ell}{2}-1} L(s+2i,\rho,\mathrm{Sym}^2),
\]
because $L(s+2i+1,\rho,\wedge^2)=1$ from Theorem \ref{supercusp}.
For $\ell=2$, the second part of Proposition \ref{even-square} implies that $L(s,\Delta,\wedge^2)=L_{ex}(s,[\rho\nu^{-\frac{1}{2}},\rho\nu^{\frac{1}{2}}],\wedge^2)$. According to \eqref{odd-symprod}, we obtain
\[
L_{ex}(s,[\rho\nu^{-\frac{1}{2}},\rho\nu^{\frac{1}{2}}],\wedge^2)=L(s,\rho,\mathrm{Sym}^2).
\]
We again apply induction on the even number $\ell$. Assuming that $\ell$ is an even number and
\[
 L_{ex}(s,[\rho \nu^{-\frac{k-1}{2}},  \dotsm, \rho \nu^{\frac{k-1}{2}}],\wedge^2)=L(s,\rho,\mathrm{Sym}^2)
\] 
is true for all even numbers $k$ with $2 \leq k \leq \ell$, we see from Proposition \ref{even-square}
\[
\begin{split}
  &L(s,[\rho \nu^{-\frac{\ell+1}{2}},  \dotsm, \rho \nu^{\frac{\ell+1}{2}}],\wedge^2)\\
  &=L_{ex}(s,[\rho \nu^{-\frac{\ell+1}{2}},  \dotsm, \rho \nu^{\frac{\ell+1}{2}}],\wedge^2)   \prod_{i=1}^{\frac{\ell}{2}}
  L_{ex}(s,[\rho \nu^{2i+\frac{1-\ell}{2}},  \dotsm, \rho \nu^{\frac{\ell+1}{2}}],\wedge^2) \\
  &=L_{ex}(s,[\rho \nu^{-\frac{\ell+1}{2}},  \dotsm, \rho \nu^{\frac{\ell+1}{2}}],\wedge^2)  \prod_{i=1}^{\frac{\ell}{2}} L(s+2i,\rho,\mathrm{Sym}^2).
  \end{split}
\]
Using \eqref{odd-symprod}, we have
\[
 L_{ex}(s,[\rho \nu^{-\frac{\ell+1}{2}},  \dotsm, \rho \nu^{\frac{\ell+1}{2}}],\wedge^2)=L(s,\rho,\mathrm{Sym}^2).
\]

\end{proof}

From Theorem \ref{prod-L} and Theorem \ref{prod-L-odd}, $L(s,\Delta,\wedge^2)$ are completely determined by the exceptional $L$-functions $L_{ex}(s,\Delta^{(i)},\wedge^2)$ for the derivatives $\Delta^{(i)}$ which are representations of $GL_{m-i}$ with $m-i$ even numbers. Hence we do not consider the case when $\ell$ and $r$ are odd because $\Delta=[\rho\nu^{-\frac{\ell-1}{2}},\dotsm,\rho\nu^{\frac{\ell-1}{2}}]$ is a representation of odd $GL_{\ell r}$. Theorem \ref{excep-square} will allow us to compute $L(s,\pi,\wedge^2)$ for all irreducible admissible representations of $GL_n$ in terms of symmetric or exterior square $L$-functions for irreducible supercuspidal representations.

\subsection{The relationships with the local Langlands correspondence.}

 In this section, we wish to explain the equality of the exterior square arithmetic (Artin) and analytic (Jacquet-Shalika) $L$-functions for $GL_m$ through the local Langlands correspondence. Let $\Phi_F$ denote a choice of geometric Frobenius element of $\mathrm{Gal}(\overline{F}/F)$. Let $\phi$ be an $m$-dimensional $\Phi_F$-semisimple representation of $W_F'$, the Weil-Deligne group for $\overline{F}/F$. Let $\pi=\pi(\phi)$ be the irreducible admissible representation of $GL_m$ associated to $\phi$ by the local Langlands correspondence by Harris-Taylor \cite{HaTa01}, Henniart \cite{He02}. Let $\wedge^2$ and $\mathrm{Sym}^2$ denote the exterior and symmetric square representations of $GL_n(\mathbb{C})$, respectively. Zelevinsky in \cite[Proposition 10.2]{Ze80} asserts that $\phi$ can be decomposed into $\phi=\phi_1 \oplus \dotsm \oplus \phi_t$, where $\phi_i$ is of dimension $m_i$ with $m=m_1 + \dotsm + m_t$ and $\Delta_i=\Delta_i(\phi_i)$ is the associated irreducible quasi-square-integrable representation of $GL_{m_i}$ under the local Langlands correspondence. Then we have the following isomorphism
 \[
   \wedge^2(\phi)=\wedge^2(\phi_1 \oplus \dotsm \oplus \phi_t)=\bigoplus_{k=1}^t \wedge^2(\phi_k) \bigoplus_{1 \leq i < j \leq t} \phi_i \otimes \phi_j.
 \]
Taking Artin's $L$-function, this gives the formula
\[
  L(s, \wedge^2(\phi))=\prod_{k=1}^t L(s, \wedge^2(\phi_k)) \prod_{1 \leq i < j \leq t} L(s, \phi_i \otimes \phi_j).
\]
 The local Langlands correspondence again implies 
 \[
   \prod_{1 \leq i < j \leq t} L(s, \phi_i \otimes \phi_j)=\prod_{1 \leq i < j \leq t} L(s, \Delta_i(\phi_i) \times \Delta_j(\phi_j))=\prod_{1 \leq i < j \leq t} L(s, \Delta_i \times \Delta_j).
  \]
We employ the agreement of arithmetic (Artin) $L$-function and analytic (Jacquet-Shalika) $L$-function for the discrete series from Kewat and Raghunathan in \cite{KeRa12} and Henniart in \cite{He10} to obtain
  \[
   \prod_{k=1}^t L(s, \wedge^2(\phi_k))=\prod_{k=1}^t L(s, \Delta_k(\phi_k),\wedge^2)=\prod_{k=1}^t L(s,\Delta_k,\wedge^2).
  \]
 Under the local Langlands correspondence, as the decomposition $\phi=\phi_1 \oplus \dotsm \oplus \phi_t$ is unique up to permutation, we can always reorder the $\phi_i$ without changing $\phi$ so that $\phi$ corresponds to $\pi=\pi(\phi)$, where $\pi$ is the unique irreducible quotient of the induced representation of Langlands type $\Pi=\text{Ind}(\Delta_1(\phi_1) \otimes \dotsm \otimes \Delta_t(\phi_t))$. Applying the formalism of the exterior square $L$-function in Theorem \ref{langlands-multi}, we obtain
\[
\begin{split}
 L(s, \wedge^2(\phi))&=\prod_{k=1}^t L(s, \Delta_k(\phi_k),\wedge^2)  \prod_{1 \leq i < j \leq t} L(s, \Delta_i(\phi_i) \times \Delta_j(\phi_j))\\
 & =\prod_{k=1}^t  L(s,\Delta_k,\wedge^2) \prod_{1 \leq i < j \leq t}  L(s, \Delta_i \times \Delta_j)=L(s,\Pi,\wedge^2)=L(s,\pi(\phi),\wedge^2).
 \end{split}
\]

  \begin{theorem} 
  \label{agree-exterior}
  Let $\phi$ be a $\Phi_F$-semisimple $m$-dimensional complex representation of Weil-Deligne group $W_F'$ and $\pi=\pi(\phi)$ an irreducible admissible representation of $GL_m$ associated to $\phi$ under the local Langlands correspondence. Then we have
  \[
    L(s,\wedge^2(\phi))=L(s,\pi(\phi),\wedge^2).
  \]
  \end{theorem}

The analogous results for the Asai $L$-functions \cite{Ma09,Ma11} or the Bump-Friedberg $L$-function \cite{Ma15} are established by Matringe. For the symmetric square $L$-function $L(s,\pi,\mathrm{Sym}^2)$, we begin with 
\[
  L(s,\pi(\phi) \times \pi(\phi))=L(s,\pi(\phi),\wedge^2)L(s,\pi(\phi),\mathrm{Sym}^2).
\]
On the arithmetic side, we have
\[
 L(s,\phi \times \phi)=L(s,\wedge^2(\phi))L(s,\mathrm{Sym}^2(\phi)).
\]
The local Langlands correspondence asserts that $L(s,\pi(\phi) \times \pi(\phi))=L(s,\phi \times \phi)$. Then we obtain the following results from the previous Theorem \ref{agree-exterior}. 

\begin{corollary}
Let $\phi$ be a $\Phi_F$-semisimple $m$-dimensional complex representation of Weil-Deligne group $W_F'$ and $\pi=\pi(\phi)$ an irreducible admissible representation of $GL_m$ associated to $\phi$ under the local Langlands correspondence. Then we have
  \[
    L(s,\mathrm{Sym}^2(\phi))=L(s,\pi(\phi),\mathrm{Sym}^2).
  \]
\end{corollary}
This Corollary explains that our definition of symmetric $L$-function is consistent with Artin's symmetric $L$-function in arithmetic side.

\par

We return to local exterior square $\gamma$- and $\varepsilon$- factor. Let $\phi^{\vee}$ be the dual representation of $\phi$. The relation between local Artin's $\varepsilon$ and $\gamma$-factor is
\[
 \gamma(s,\wedge^2(\phi),\psi)=\frac{\varepsilon(s,\wedge^2(\phi),\psi)L(1-s,\wedge^2(\phi^{\vee}))}{L(s,\wedge(\phi))}
\] 
in \cite{De72}. On the analytic side, for irreducible admissible representation $\pi(\phi)$ associated to $\phi$, $\gamma$-factor is simply defined in Section 3 by
\[
 \gamma(s,\pi(\phi),\wedge^2,\psi)=\frac{\varepsilon(s,\pi(\phi),\wedge^2,\psi)L(1-s,(\pi(\phi))^{\iota},\wedge^2)}{L(s,\pi(\phi),\wedge^2)}.
 \]
In the view of the Local Langlands correspondence, $\phi^{\vee}$ corresponds to $(\pi(\phi))^{\iota}$. Since we have the equality of $L$-functions $L(s,\wedge(\phi)=L(s,\pi(\phi),\wedge^2)$ and $L(1-s,\wedge^2(\phi^{\vee})=L(1-s,(\pi(\phi))^{\iota},\wedge^2)$ from Theorem \ref{agree-exterior}, we arrive at
\[
   \gamma(s,\wedge^2(\phi),\psi)=\frac{\varepsilon(s,\wedge^2(\phi),\psi)}{\varepsilon(s,\pi(\phi),\wedge^2,\psi)} \gamma(s,\pi(\phi),\wedge^2,\psi). 
   \]
Summing up, Artin's $\gamma$-factor $\gamma(s,\wedge^2(\phi),\psi)$
and $ \gamma(s,\pi(\phi),\wedge^2,\psi)$ are equal up to units in $\mathbb{C}[q^{\pm s}]$. 

\begin{proposition}
Let $\phi$ be a $\Phi_F$-semisimple $m$-dimensional complex representation of Weil-Deligne group $W_F'$ and $\pi=\pi(\phi)$ an irreducible admissible representation of $GL_m$ associated to $\phi$ under the local Langlands correspondence. Then there is a monomial of the form $\alpha q^{-\beta s}$, that is, a unit in $\mathbb{C}[q^{\pm s}]$ such that
\[
  \gamma(s,\wedge^2(\phi),\psi)=\alpha q^{-\beta s} \gamma(s,\pi(\phi),\wedge^2,\psi).
\]
\end{proposition} 

Let $\gamma_{Sh}(s,\pi,\wedge^2,\psi)$ denote Shahidi's $\gamma$-factor, $\varepsilon_{Sh}(s,\pi,\wedge^2,\psi)$ Shahidi's $\varepsilon$-factor, and $L_{Sh}(s,\pi,\wedge^2)$ Shahidi's $L$-factor from the Langlands-Shahidi methods in \cite{Sh90, Sh14}. The relation between the local $\varepsilon$-factor and the local $\gamma$-factor is given in \cite{Sh90}
\[
  \gamma_{Sh}(s,\pi,\wedge^2,\psi)=\frac{\varepsilon_{Sh}(s,\pi,\wedge^2,\psi)L_{Sh}(1-s,\pi^{\iota},\wedge^2)}{L_{Sh}(s,\pi,\wedge^2)}.
\] 
Cogdell, Shahidi, and Tsai \cite{CoShTs17} prove that Artin's $\gamma$-factor $\gamma(s,\wedge^2(\phi),\psi)$ and Shahidi $\gamma$-factor $\gamma_{Sh}(s,\pi(\phi),\wedge^2,\psi)$ are equal. Then they deduce the similar identity for $\varepsilon$-factor
\[
  \varepsilon(s,\wedge^2(\phi),\psi)=\varepsilon_{Sh}(s,\pi(\phi),\wedge^2,\psi).
\]
It is still an open problem to show that the two defintions of local $\gamma$-factors $\gamma(s,\wedge^2(\phi),\psi)$ and $\gamma(s,\pi(\phi),\wedge^2,\psi)$, or local $\varepsilon$-factors $\varepsilon(s,\wedge^2(\phi),\psi)$ and $\varepsilon(s,\pi(\phi),\wedge^2,\psi)$ coincide. We expect that the proof require multiplicativity result for the exterior square $
\gamma$-factor $\gamma(s,\pi,\wedge^2,\psi)$ which is stronger than Corollary \ref{multiplicativity}.

\vskip.1in
\noindent
\textbf{Acknowledgments.} 
This work was completed as an extension of the study embarked by Cogdell and Piatetski-Shapiro \cite{Cog94} and owes much to \cite{CoPe} and \cite{Ma15}. This paper is the main part of my Ph.D thesis.
I would like to thank my advisor Professor James W. Cogdell for his encouragement, invaluable comments, and many careful reading of several preliminary versions of this paper. I also thank Professor Nadir Matringe for notifying errors in his papers to me and suggesting me how to prove Proposition 6.1 in \cite{Ma14}.

\bibliographystyle{amsplain}
\bibliography{bibfile}

\end{document}